\documentclass[10pt]{article}
\usepackage{amsfonts,exscale}
\usepackage{mleftright}
 \usepackage{srcltx,hyperref}
\usepackage[curve,matrix,arrow,cmtip]{xy}
\newdir^{ (}{{}*!/-3pt/\dir^{(}}    
\newdir^{  }{{}*!/-3pt/\dir^{}}    
\newdir_{ (}{{}*!/-3pt/\dir_{(}}    
\newdir_{  }{{}*!/-3pt/\dir_{}}    
 
\def\?{\ 
{\bf\color{red}???}\ 
\immediate\write16{}
\immediate\write16{Warning: There was still a question mark . . . }
\immediate\write16{}}

\long\def\forget#1{}

\usepackage{color}
\newcounter{commentcounter}

\usepackage{geometry}
\usepackage{amsmath}
\usepackage{amssymb}
\usepackage{amscd}
\usepackage{textcomp}

\newcommand{\DS}{\displaystyle}
\newcommand{\TS}{\textstyle}
\newcommand{\SC}{\scriptstyle}
\newcommand{\SSC}{\scriptscriptstyle}

\usepackage{amsthm}
\theoremstyle{plain}
\newtheorem{Lemma}{Lemma}[section]
\newtheorem{Theorem}[Lemma]{Theorem}
\newtheorem{Proposition}[Lemma]{Proposition}
\newtheorem{Corollary}[Lemma]{Corollary}
\newtheorem{Conjecture}[Lemma]{Conjecture}

\theoremstyle{definition}
\newtheorem{Definition}[Lemma]{Definition}
\newtheorem{Example}[Lemma]{Example}
\newtheorem{Remark}[Lemma]{Remark}
\newtheorem{Notation}[Lemma]{Notation}
\newtheorem{Situation}[Lemma]{Situation}
\newtheorem{Convention}[Lemma]{Convention}
\newtheorem{Point}[Lemma]{}

\def\theenumi{(\alph{enumi})}

\def\p@enumii{\theenumi}

\newcommand{\Art}{{\rm Art}}

\newcommand{\Betti}{{\rm Betti}}
\newcommand{\dR}{{\rm dR}}

\DeclareMathOperator{\QHom}{QHom}
\DeclareMathOperator{\QEnd}{ QEnd}

\DeclareMathOperator{\Aut}{Aut}

\DeclareMathOperator{\End}{End}

\DeclareMathOperator{\Frob}{Frob}
\DeclareMathOperator{\Fil}{Fil}

\DeclareMathOperator{\Gal}{Gal}
\DeclareMathOperator{\GL}{GL}

\DeclareMathOperator{\Koh}{H}

\DeclareMathOperator{\Hom}{Hom}

\renewcommand{\Im}{\FI\Fm}

\DeclareMathOperator{\Ind}{Ind}

\DeclareMathOperator{\Lie}{Lie}

\DeclareMathOperator{\Nm}{Norm}

\DeclareMathOperator{\Pic}{Pic}
\DeclareMathOperator{\PGL}{PGL}

\DeclareMathOperator{\Res}{Res}

\DeclareMathOperator{\Spec}{Spec}

\newcommand{\T}{{\rm T}}
\DeclareMathOperator{\Var}{V}
\DeclareMathOperator{\Quot}{Quot}
\DeclareMathOperator{\Frac}{Frac}

\DeclareMathOperator{\aff}{aff}
\newcommand{\alg}{{\rm alg}}

\DeclareMathOperator{\charakt}{char}

\DeclareMathOperator{\coker}{coker}

\DeclareMathOperator{\di}{div}
\newcommand{\et}{{\rm\acute{e}t}}
\newcommand{\fppf}{{\it fppf\/}}

\newcommand{\id}{{\rm id}}

\renewcommand{\mod}{\;{\rm mod}\;}

\DeclareMathOperator{\ord}{ord}

\newcommand{\red}{{red}}

\newcommand{\rig}{{\rm rig}}
\DeclareMathOperator{\rk}{rk}
\newcommand{\sep}{{\rm sep}}

\DeclareMathOperator{\Cant}{\BC\,{-antilinear}}

\let\setminus\smallsetminus

\newcommand{\es}{\enspace}

\newcommand{\notdiv}{\nmid}
\newcommand{\dual}{^\vee}

\newcommand{\mal}{^{\SSC\times}}
\newcommand{\fdot}{\;\,_{_{\bullet}}\,\;}
\newcommand{\dbl}{{\mathchoice{\mbox{\rm [\hspace{-0.15em}[}}
                              {\mbox{\rm [\hspace{-0.15em}[}}
                              {\mbox{\scriptsize\rm [\hspace{-0.15em}[}}
                              {\mbox{\tiny\rm [\hspace{-0.15em}[}}}}
\newcommand{\dbr}{{\mathchoice{\mbox{\rm ]\hspace{-0.15em}]}}
                              {\mbox{\rm ]\hspace{-0.15em}]}}
                              {\mbox{\scriptsize\rm ]\hspace{-0.15em}]}}
                              {\mbox{\tiny\rm ]\hspace{-0.15em}]}}}}
\newcommand{\dpl}{{\mathchoice{\mbox{\rm (\hspace{-0.15em}(}}
                              {\mbox{\rm (\hspace{-0.15em}(}}
                              {\mbox{\scriptsize\rm (\hspace{-0.15em}(}}
                              {\mbox{\tiny\rm (\hspace{-0.15em}(}}}}
\newcommand{\dpr}{{\mathchoice{\mbox{\rm )\hspace{-0.15em})}}
                              {\mbox{\rm )\hspace{-0.15em})}}
                              {\mbox{\scriptsize\rm )\hspace{-0.15em})}}
                              {\mbox{\tiny\rm )\hspace{-0.15em})}}}}
\newcommand{\invlim}[1][]{\ifthenelse{\equal{#1}{}}
{\DS \lim_{\longleftarrow}}
{\DS \lim_{\underset{#1}{\longleftarrow}}}
}
\newcommand{\dirlim}[1][]{\ifthenelse{\equal{#1}{}}
{\DS \lim_{\longrightarrow}}
{\DS \lim_{\underset{#1}{\longrightarrow}}}
}
\newcommand{\ul}[1]{{\underline{#1}}}
\newcommand{\ol}[1]{{\overline{#1}}}

\newcommand{\wt}[1]{{\widetilde{#1}}}
\newcommand{\wh}[1]{{\widehat{#1}}}

\usepackage{xcolor}
\usepackage{graphicx}
\newcommand{\Bmu}{\mbox{$\raisebox{-0.59ex}{$l$}\hspace{-0.16em}\mu\hspace{-0.91em}\raisebox{-0.95ex}{\scalebox{2}{$\color{white}.$}}\hspace{-0.59em}\raisebox{+0.78ex}{\scalebox{2}{$\color{white}.$}}\hspace{0.46em}$}{}}

\newcommand{\BOne} {{\mathchoice{\hbox{\rm1\kern-2.7pt l\kern.9pt}}
                              {\hbox{\rm1\kern-2.7pt l\kern.9pt}}
                              {\hbox{\scriptsize\rm1\kern-2.3pt l\kern.4pt}}
                              {\hbox{\scriptsize\rm1\kern-2.4pt l\kern.5pt}}}}
\newcommand{\BA}{{\mathbb{A}}}
\newcommand{\BB}{{\mathbb{B}}}
\newcommand{\BC}{{\mathbb{C}}}

\newcommand{\BF}{{\mathbb{F}}}
\newcommand{\BG}{{\mathbb{G}}}

\newcommand{\BN}{{\mathbb{N}}}

\newcommand{\BP}{{\mathbb{P}}}
\newcommand{\BQ}{{\mathbb{Q}}}
\newcommand{\BR}{{\mathbb{R}}}

\newcommand{\BZ}{{\mathbb{Z}}}

\newcommand{\sC}{{\mathscr{C}}}

\newcommand{\sG}{{\mathscr{G}}}

\usepackage{mathrsfs}

\newcommand{\CA}{{\cal{A}}}

\newcommand{\CC}{{\cal{C}}}

\newcommand{\CG}{{\cal{G}}}
\newcommand{\CH}{{\cal{H}}}

\newcommand{\CJ}{{\cal{J}}}

\newcommand{\CL}{{\cal{L}}}
\newcommand{\CM}{{\cal{M}}}

\newcommand{\CO}{{\cal{O}}}
\newcommand{\CP}{{\cal{P}}}

\newcommand{\CR}{{\cal{R}}}

\newcommand{\CX}{{\cal{X}}}

\newcommand{\FC}{{\mathfrak{C}}}
\newcommand{\FD}{{\mathfrak{D}}}

\newcommand{\FI}{{\mathfrak{I}}}

\newcommand{\FP}{{\mathfrak{P}}}

\newcommand{\Fa}{{\mathfrak{a}}}

\newcommand{\Fd}{{\mathfrak{d}}}
\newcommand{\Ff}{{\mathfrak{f}}}

\newcommand{\Fp}{{\mathfrak{p}}}
\newcommand{\Fq}{{\mathfrak{q}}}
\newcommand{\Fm}{{\mathfrak{m}}}

\def\longto{\longrightarrow}
\def\into{\hookrightarrow}
\let\onto\twoheadrightarrow

\def\isoto{\stackrel{}{\mbox{\hspace{1mm}\raisebox{+1.4mm}{$\SC\sim$}\hspace{-3.5mm}$\longrightarrow$}}}

\newbox\mybox
\def\arrover#1{\mathrel{
       \setbox\mybox=\hbox spread 1.4em{\hfil$\scriptstyle#1$\hfil}
       \vbox{\offinterlineskip\copy\mybox
             \hbox to\wd\mybox{\rightarrowfill}}}}

\def\ulCC{{\underline{\CC\!}\,}}

\def\ulG{{\underline{G}}}
\def\ulM{{\underline{M\!}\,}}
\def\ulN{{\underline{N\!}\,}}
\def\ulHM{{\underline{\hat M\!}\,}}
\def\ulHN{{\underline{\hat N\!}\,}}

\newcommand{\tminus}[1]{\ell^{\SSC -}_{#1}}
\newcommand{\tplus}[1]{\ell^{\SSC +}_{#1}}
\newcommand{\ttplus}[1]{\tilde\ell^{\SSC +}_{#1}}
\newcommand{\tplusminus}{\ell}

\newcommand{\AChar}{A\text{\rm-char}}

\renewcommand{\phi}{\varphi}
\renewcommand{\epsilon}{\varepsilon}

\begin{document}
\author{Urs Hartl and Rajneesh Kumar Singh}
\title{Product Formulas for Periods of CM Abelian Varieties and the Function Field Analog}
\maketitle

\begin{abstract}
We survey Colmez's theory and conjecture about the Faltings height and a product formula for the periods of abelian varieties with complex multiplication, along with the function field analog developed by the authors. In this analog, abelian varieties are replaced by Drinfeld modules and $A$-motives. We also explain the necessary background on abelian varieties, Drinfeld modules and $A$-motives, including their cohomology theories and comparison isomorphisms and their theory of complex multiplication. 

\noindent
{\it Mathematics Subject Classification (2000)\/}: 
11G09,  
(11G15, 
11R42)  
\end{abstract}
\section{Introduction}\label{SectIntroduction}

One purpose of this survey is to give a brief introduction to abelian varieties with complex multiplication over number fields, some of their cohomology theories with comparison isomorphisms, and to explain Colmez's conjectures~\cite{ColmezPeriods} on a product formula for the periods and on the Faltings height of these abelian varieties. The second purpose is to explain the function field analog of this theory. There abelian varieties are replaced by Drinfeld modules \cite{Drinfeld,Goss} and their higher dimensional generalizations, so-called $A$-motives. So we give a brief introduction to Drinfeld modules and $A$-motives with complex multiplication, some of their cohomology theories with comparison isomorphisms, and explain the conjecture~\cite{HartlSingh} of the authors on periods of these $A$-motives. We point out that recently other surveys on Colmez's conjectures were written by Gross~\cite{GrossSurvey}, by Yuan~\cite{YuanSurvey}, and by Gao, van K{\"a}nel and Mocz \cite{Zhang20} based on a lecture of Shou-Wu Zhang. However, these do not discuss the function field analog that we are discussing in Part~\ref{Amot}. In \cite{GrossSurvey} it is explained how Colmez's conjectures generalize the Chowla-Selberg formula. And in \cite{YuanSurvey} the consequences of the recently proved averaged Colmez Conjecture for the Andr\'e-Oort Conjecture are explained. In \cite{Zhang20} in addition to these aspects, the proof of Yuan and Shou-Wu Zhang \cite{YuanZhang15} of the averaged Colmez conjecture, and the work of Yun and Wei Zhang \cite{YunZhang,YunZhang2} on the Gross-Zagier formula for intersection numbers in the Chow group of moduli spaces of $\PGL_2$-shtukas is discussed.

\begin{Point}
We begin with a review of product formulas for global fields. For a rational number  $\alpha\in\BQ\mal$, all of its absolute values $|\alpha|_v$ are linked by the \emph{product formula} $\prod_v |\alpha|_v = 1$ where only finitely many factors are different from $1$. Here $v$ runs through the set $\CP$ of \emph{places} of $\BQ$ consisting of all prime numbers $p$ together with $\infty$, and the \emph{$p$-adic absolute values} $|\,.\,|_p$ are normalized such that $|p|_p=p^{-1}$. This product formula extends to \emph{number fields}, i.e.\ finite extensions of $\BQ$, as follows. Let $\BQ^\alg$ be the algebraic closure of $\BQ$ in $\BC$, and if $p$ is a prime number let $\BQ_p$ be the completion of $\BQ$ with respect to $|\,.\,|_p$ and let $\BQ_p^\alg$ be an algebraic closure of $\BQ_p$. The $p$-adic absolute value $|\,.\,|_p$ extends canonically to $\BQ_p^\alg$. We denote by $|\,.\,|_\infty$ the usual absolute value on $\BC$. In addition to the embedding $\BQ^\alg\subset\BC$ we fix once and for all an embedding of $\BQ^\alg$ in $\BQ_p^\alg$ for every $p$ and consider the induced absolute value $|\,.\,|_p$ on $\BQ^\alg$. For a finite field extension $K$ of $\BQ$ we set $H_K:=\Hom_\BQ(K, \BQ^\alg)$. Then the product formula \cite[Chapter~V, \S\,1, bottom of page~99]{LangANT} for $0\ne\alpha\in K$ can be written  as 
\begin{equation}\label{EqProdFormulaAbVar}
\prod_{p\in \CP}\prod_{\eta\in H_K}|\eta(\alpha)|_p = 1.
\end{equation}
\end{Point}

\begin{Point}\label{Point1.2}
The product formula also holds for \emph{function fields}. More precisely, let $Q$ be a finitely generated field of transcendence degree one over the finite field $\BF_p=\BZ/p\BZ$. Let $\BF_q:=\{a\in Q\colon a\text{ is algebraic over }\BF_p\}\subset Q$ be the \emph{field of constants}, see \cite[Definition~2.1.3]{VillaSalvador}, which is a finite field with $q$ elements. Then $Q$ is the field of rational functions on a smooth, projective curve $C$ over $\BF_q$ by \cite[Chapter~7.3, Proposition~3.13]{Liu_AlgGeom} which is geometrically irreducible by \cite[IV$_2$, 4.3.1 and Proposition~4.5.9c)]{EGA_IV2}. Every closed point $v$ of $C$ is called a \emph{place}. We denote its residue field by $\BF_v$ and set $q_v:=\#\BF_v=q^{[\BF_v:\BF_q]}$. The local ring $\CO_{C,v}$ is a discrete valuation ring by \cite[Proposition~1.1]{Silverman86}. We denote the corresponding valuation also by $v$ and the corresponding absolute value on $Q$ by $|\,.\,|_v$. Both are normalized such that $v(z_v)=1$ and $|z_v|_v=q_v^{-1}$ for a uniformizing parameter $z_v\in Q$ at $v$. Then every $a\in Q\setminus\{0\}$ satisfies $\prod_v|a|_v=1$ where again only finitely many factors are different from $1$, see \cite[Chapter~II, \S\,12, Theorem]{CasselsFroehlich}. This can be reinterpreted in terms of \emph{divisors} on $C$. Namely, since $|a|_v= q_v^{-v(a)}$ we have $-\log\prod_v|a|_v=\sum_v v(a)\cdot[\BF_v:\BF_q]\cdot\log q=0$, because $\sum_v v(a)\cdot[\BF_v:\BF_q]$ is the degree of the principal divisor of $a$ which is zero, see \cite[Corollary~3.2.9]{VillaSalvador}.

Let $Q^\alg$ be a fixed algebraic closure of $Q$. For every place $v$ of $Q$ let $Q_v$ be the completion of $Q$ with respect to $|\,.\,|_v$ and let $Q_v^\alg$ be an algebraic closure of $Q_v$. The $v$-adic absolute value $|\,.\,|_v$ extends canonically to $Q_v^\alg$. We fix once and for all an embedding of $Q^\alg$ in $Q_v^\alg$ for every $v$ and consider the induced absolute value $|\,.\,|_v$ on $Q^\alg$. For a finite field extension $K$ of $Q$ we set $H_K:=\Hom_Q(K, Q^\alg)$. Then by transformations of equations as in \cite[Chapter~V, \S\,1, bottom of page~99]{LangANT} the product formula \cite[Chapter~II, \S\,12, Theorem]{CasselsFroehlich} for $0\ne a\in K$ can be written  as 
\begin{equation}\label{EqProdFormulaAMot}
\prod_{\text{all }v}\prod_{\;\eta\in H_K}|\eta(a)|_v = 1.
\end{equation}
\end{Point}

\begin{Point}\label{Point1.3}
In \cite{ColmezPeriods} P.~Colmez considers product formulas for periods of abelian varieties. Let $X$ be an abelian variety defined over a number field $K$ with complex multiplication by the ring of integers in a CM-field $E$ and of CM-type $\Phi$, see Section~\ref{SectCMAbVar} for explanations. Assume that $K$ contains $\psi(E)$ for every $\psi\in H_E$. For a $\psi\in H_E$ let $\omega_\psi\in\Koh^1_{\dR}(X,K)$ be a non-zero cohomology class such that $b^*\omega_\psi=\psi(b)\cdot\omega_\psi$ for all $b\in E$, see Section~\ref{SectDRAbVar}. For every embedding $\eta\colon K\into\BQ^\alg$, let $X^\eta:=X\times_{\Spec K,\Spec\eta}\Spec\eta(K)$ and $\omega_\psi^\eta\in\Koh^1_{\dR}(X^\eta,\eta(K))$ be deduced from $X$ and $\omega_\psi$ by base extension. Let $(u_\eta)_\eta\in\prod_{\eta\in H_K}\Koh_1(X^\eta(\BC),\BZ)$ be a family of cycles compatible with complex conjugation, see Section~\ref{SectBettiAbVar}. Let $v$ be a place of $\BQ$. If $v=\infty$ the de Rham isomorphism between Betti and de Rham cohomology (Theorem~\ref{ThmDeRhamIsom}) yields a complex number $\int_{u_\eta}\omega_\psi^\eta$ and its absolute value $\bigl|\int_{u_\eta}\omega_\psi^\eta\bigr|_\infty\in\BR$. If $v$ corresponds to a prime number $p\in\BZ$, Colmez~\cite{ColmezPeriods} associates a period $\int_{u_\eta}\omega_\psi^\eta$ in Fontaine's $p$-adic period field $\BB_{p,\dR}$, see Notation~\ref{NotAinf}, and an absolute value $\bigl|\int_{u_\eta}\omega_\psi^\eta\bigr|_v\in\BR$. He considers the product $\prod_v\prod_{\eta\in H_K}\bigl|\int_{u_\eta}\omega_\psi^\eta\bigr|_v$ and (after some modifications which we explain in Section~\ref{SectColmezConjAbVar}) conjectures that this product evaluates to~$1$; see Conjecture~\ref{ConjColmezAbVar} for the precise formulation. This conjecture implies a conjectural formula for the Faltings height of a CM abelian variety in terms of the logarithmic derivatives at $s=0$ of certain Artin $L$-functions. Colmez proves the conjectures when $E$ is an abelian extension of $\BQ$, see Theorem~\ref{ThmColmezAbelianE}. On the way, he computes $\prod_{\eta\in H_K}\bigl|\int_{u_\eta}\omega_\psi^\eta\bigr|_v$ at a finite place $v$ in terms of the local factor at $v$ of the Artin $L$-series associated with an Artin character $a^0_{E,\psi,\Phi}\colon\Gal(\BQ^\alg/\BQ)\to\BC$ that only depends on $E$, $\psi$ and $\Phi$ but not on $X$ and $v$; see Theorem~\ref{ThmColmezLocal}. There has been further progress on Colmez's conjecture on which we report in Section~\ref{SectColmezConjAbVar}.

We point out that Colmez's formulation generalizes various previous results. Namely, when $[E:\BQ]=2$ his Theorem~\ref{ThmColmezAbelianE} is equivalent to the formula proved by Lerch~\cite{Lerch1897} and rediscovered by Chowla-Selberg~\cite{ChowlaSelberg67} 
\begin{equation}\label{EqLerch}
\frac{\zeta'_E(0)}{\zeta_E(0)}\;=\;\frac{1}{12\,\#\Pic(\CO_E)}\sum_{[I]\in\Pic(\CO_E)}\log\bigl(\Delta(I)\Delta(I^{-1})\bigr)\,,
\end{equation}
where $\Delta(I)$ is the modular discriminant of the lattice $I\subset E\subset\BC$. A new geometric proof of \eqref{EqLerch} was given by Gross~\cite{Gross78}, who together with Deligne conjectured a generalization to a formula for the archimedean periods of certain CM motives up to multiplication by algebraic numbers. Anderson~\cite{AndersonLogarithmic} reformulated the Gross-Deligne conjecture in terms of the logarithmic derivative of an $L$-function at $s = 0$ and proved it when the CM field $E$ is abelian over $\BQ$. Colmez added the consideration of the non-archimedean periods and thus removed the ambiguity of the algebraic factors in Anderson's theorem.
\end{Point}

\begin{Point}
There is a beautiful analog to the theory of elliptic curves and abelian varieties in the ``Arithmetic of function fields''. Namely, Drinfeld~\cite{Drinfeld} invented the analog of elliptic curves under the name ``elliptic modules''. These are today called \emph{Drinfeld modules}, see Section~\ref{SectBasicDefAMot}. Since then, the arithmetic of function fields has evolved into an equally rich parallel world to the arithmetic of number fields. As higher dimensional generalizations of Drinfeld modules and analogs of abelian varieties, Anderson~\cite{Anderson} has defined \emph{abelian $t$-modules} and the dual notion of \emph{$t$-motives}, which are a kind of ``global Dieudonn\'e-modules'' for abelian $t$-modules, see Remark~\ref{RemAndersonAModule}. They can be slightly generalized to \emph{$A$-motives} as follows. In the notation of \S\,\ref{Point1.2} let $\infty$ be a fixed closed point on $C$ and let $A=\Gamma(C\setminus\{\infty\},\CO_C)=\{a\in A\colon v(a)\ge0\text{ for all }v\ne\infty\}$. Let $K\subset Q^\alg$ be a finite field extension of $Q$. We write $A_K:=A\otimes_{\BF_q}K$ and consider the endomorphism $\sigma^*:=\id_A\otimes\Frob_{q,K}$ of $A_K$, where $\Frob_{q,K}(b)=b^q$ for $b\in K$. For an $A_K$-module $M$ we set $\sigma^*M:=M\otimes_{A_K,\sigma^*}A_K$ and for a homomorphism $f\colon M\to N$ of $A_K$-modules we set $\sigma^*f:=f\otimes\id_{A_K}\colon\sigma^*M\to\sigma^*N$. Let $\gamma\colon A\to K$ be the inclusion $A\subset Q\subset K$, and set $\CJ:=(a\otimes1-1\otimes\gamma(a)\colon a\in A)\subset A_K$. Then $\gamma$ can be recovered as the homomorphism $A\to A_K/\CJ=K$. 
\end{Point}

\begin{Definition}
An  \emph{(effective) $A$-motive of rank $r$ and dimension $d$ over $K$} is a pair $\ulM=(M,\tau_M)$ consisting of a locally free $A_K$-module $M$ of rank $r$ and an $A_K$-homomorphism $\tau_M\colon\sigma^*M \to M$ such that
\begin{enumerate} 
\item $\dim_K(\coker \tau_M)=d$.
\item $(a-\gamma(a))^d \cdot \coker \tau_M =0$ for all $a\in A$.
\end{enumerate}
 We write $\rk\ulM:=r$ and $\dim\ulM:=d$. 
\end{Definition}

$A$-motives possess cohomology realizations in analogy with abelian varieties, see Section~\ref{UniCoh}. 

\begin{Point}
Let us now explain the analog of Colmez's theory from \S\,\ref{Point1.3} which was developed by the authors in \cite{HartlSingh}. Let $\ulM$ be a uniformizable $A$-motive defined over a finite extension $K\subset Q^\alg$ of $Q$ with complex multiplication by the ring of integers in a CM-algebra $E$ and of CM-type $\Phi$, see Sections~\ref{SectBettiAMot} and \ref{SectCMAMot} for explanations. Assume that $K$ contains $\psi(E)$ for every $\psi\in H_E:=\Hom_Q(E,Q^\alg)$. For a $\psi\in H_E$ let $\omega_\psi\in\Koh^1_{\dR}(\ulM,K\dbl z-\zeta\dbr)$ be a non-zero cohomology class such that $b^*\omega_\psi=\psi(b)\cdot\omega_\psi$ for all $b\in E$, see Section~\ref{SectDRAMot}. For every embedding $\eta\colon K\into Q^\alg$, let $\ulM^\eta:=\ulM\otimes_{K,\eta}\eta(K)$ and $\omega_\psi^\eta\in\Koh^1_{\dR}(\ulM^\eta,\eta(K)\dbl z-\zeta\dbr)$ be deduced from $\ulM$ and $\omega_\psi$ by base extension. Let $(u_\eta)_\eta\in\prod_{\eta\in H_K}\Koh_{1,\Betti}(\ulM^\eta,A)$ be a family of cycles, see Section~\ref{SectBettiAMot}. Let $v$ be a place of $Q$. If $v=\infty$ the comparison isomorphism between Betti and de Rham cohomology (Theorem~\ref{PeriodIso}) yields an element $\int_{u_\eta}\omega_\psi^\eta$ in the completion $\BC_\infty$ of $Q_\infty^\alg$ with respect to $|\,.\,|_\infty$ and its absolute value $\bigl|\int_{u_\eta}\omega_\psi^\eta\bigr|_\infty\in\BR$. If $v$ corresponds to a maximal ideal of $A$, the period isomorphism between $v$-adic and de Rham cohomology (Theorem~\ref{ThmHv,dR}) gives a period $\int_{u_\eta}\omega_\psi^\eta$ in the analog $\BC_v\dpl z_v-\zeta_v\dpr$ of Fontaine's $p$-adic period field $\BB_{p,\dR}$ and an absolute value $\bigl|\int_{u_\eta}\omega_\psi^\eta\bigr|_v\in\BR$, see Definition~\ref{DefNorm}. We consider the product $\prod_v\prod_{\eta\in H_K}\bigl|\int_{u_\eta}\omega_\psi^\eta\bigr|_v$ and (after some modifications analogous to Colmez's which we explain in Section~\ref{SectColmezConjAMot}) we conjecture that this product evaluates to~$1$; see Conjecture~\ref{ConjColmezAMot} for the precise formulation. In \cite{HartlSingh} we have computed $\prod_{\eta\in H_K}\bigl|\int_{u_\eta}\omega_\psi^\eta\bigr|_v$ at all finite places $v\ne\infty$ in terms of the local factor at $v$ of the Artin $L$-series associated with an Artin character $a^0_{E,\psi,\Phi}\colon\Gal(Q^\alg/Q)\to\BC$ that only depends on $E$, $\psi$ and $\Phi$ but not on $\ulM$ and $v$; see Theorem~\ref{ThmValueAtV}. 

If $\ulM$ is the $A$-motive associated with a Drinfeld module $\ulG$, then Conjecture~\ref{ConjColmezAMot} is equivalent to a formula for the Taguchi height (Definition~\ref{DefTagHeight}) of $\ulG$ in terms of the logarithmic derivatives at $s=0$ of an Artin $L$-function. This formula was established by Fu-Tsun Wei~\cite{Wei20} by first proving the function field analogs of Kronecker's limit theorem and Lerch's formula \ref{EqLerch}, see Theorem~\ref{ThmWei} below. Previously, formulas of Chowla-Selberg type expressing the periods at $\infty$ of CM Drinfeld modules in terms of $\Gamma$-values were obtained by Thakur~\cite{Thakur91} for certain CM-fields. Also when proving his results in \cite{AndersonLogarithmic} Anderson had considered the analogous case of $A$-motives, but without publishing his results. 

\end{Point}

This survey contains no new results, except for Theorems~\ref{ThmDriModIntegral} and \ref{ThmTagHeight} which give a formula for the Taguchi height of a Drinfeld module with complex multiplication. Our presentation summarizes material from various sources. But all shortcomings of the exposition are solely due to the authors. We describe the content of the individual sections of this survey. In Part~\ref{PartAbVar} we first define elliptic curves and abelian varieties and discuss their torsion points in Section~\ref{SectAVBasicDef}. Section~\ref{SectEmdonAbVar} is concerned with simple and semi-simple abelian varieties and their endomorphism rings. In Section~\ref{SectCohomAbVar} we review the singular (co-)homology, Tate modules and the $\ell$-adic (co-)homology, and the de Rham (co-)homology of abelian varieties and period isomorphisms between these (co-)homologies. The period isomorphism between $\ell$-adic and de Rham (co-)homology is explained in Section~\ref{SectPAdicPeriodIsom}. It is based on the concept of $p$-divisible groups, which we also review in this section. The definition of complex multiplication of abelian varieties, of CM-fields, CM-algebras and CM-types is explained in Section~\ref{SectCMAbVar}. A short review of the Faltings height of an abelian variety fills Section~\ref{SectFaltingsHeight}. Finally, in Section~\ref{SectColmezConjAbVar} we discuss Colmez's conjecture alluded to in \S\,\ref{Point1.3} above.

In Part~\ref{Amot} we discuss the analog of Colmez's theory in the ``Arithmetic of function fields''. We define Drinfeld modules and $A$-motives in Section~\ref{SectBasicDefAMot}, and isogenies and semi-simplicity in Section~\ref{SectEndomAMot}, where we also describe the endomorphism rings of semi-simple $A$-motives. The analytic theory of Drinfeld modules via lattices is explained in Section~\ref{SectAnalytTh}. Section~\ref{SectTorsDriMod} is devoted to torsion points and Tate modules of Drinfeld modules. In Section~\ref{UniCoh} we review the singular (co-)homology, Tate modules and the $v$-adic (co-)homology, and the de Rham (co-)homology of $A$-motives and period isomorphisms between these (co-)homologies. The period isomorphism between $v$-adic and de Rham (co-)homology is explained in Section~\ref{SectLocalShtukas}. It is based on the concept of $z$-divisible local Anderson modules and local shtukas, which we also review in this section. In Section~\ref{SectCMAMot} we introduce the concept of complex multiplication of $A$-motives and of their CM-types. Section~\ref{SectTagHeight} contains a brief review of the Taguchi height of a Drinfeld module. Then in Section~\ref{SectColmezConjAMot} we present the theory of the authors on the product formula for periods of $A$-motives analogous to Colmez's conjecture. In the last Section~\ref{SectExample} we compute an interesting example for this product formula where $Q$ and $C$ have genus $1$.

\tableofcontents

\part{ Abelian Varieties and Elliptic Curves}\label{PartAbVar}

Our exposition of the theory of abelian varieties and elliptic curves follows \cite{Mumford70,Milne84,MilneAbVar,Silverman86,DiamondShurman}, which serve as background material for this article.

\section{Basic Definitions}\label{SectAVBasicDef}
\setcounter{equation}{0}

\begin{Notation}\label{NotNumberFld}
As usual we denote by $\BQ$ and $\BR$ the fields of rational and real numbers, respectively, by $\BZ$ the ring of integers and by $\BN_0$, respectively $\BN_{>0}$ the set of non-negative, respectively positive integers. By a \emph{place} of $\BQ$ we mean either $\infty$ or a maximal ideal $v=(p)\subset\BZ$ for a prime number $p\in\BN_{>0}$. It defines a normalized absolute value $|\,.\,|_v\colon\BQ\to\BR_{\ge0}$ given for $v=\infty$ by the usual absolute value $|x|_\infty=x$ if $x\ge0$ and $|x|_\infty=-x$ if $x\le0$, and for $v=(p)$ by the $p$-adic absolute value $|x|_v:=|x|_p=p^{-v_p(x)}$ where $v_p(x)=n$ if $x=p^n\tfrac{a}{b}$ with $a,b\in\BZ$ and $p\notdiv ab$. Let $\BQ_v$ be the completion of $\BQ$ with respect to the valuation $v$, that is $\BQ_\infty=\BR$ and $\BQ_v=\BQ_p$ for $v=(p)$. Let $\BQ_v^\alg$ be a fixed algebraic closure of $\BQ_v$ and let $\BC_v$ be the completion of $\BQ_v^\alg$ with respect to the canonical extension of the absolute value $|\,.\,|_v$ to $\BQ_v^\alg$. Note that $\BC_v$ is algebraically closed. It equals the field of complex numbers $\BC$ when $v=\infty$, and is usually denoted $\BC_p$ when $v=(p)$. We also fix an algebraic closure $\BQ^\alg$ of $\BQ$ and an embedding $\BQ^\alg\into \BQ_v^\alg$ for every place $v$ of $\BQ$. We let $\CO_{\BC_p}$ be the ring of integers of $\BC_p$.
\end{Notation}

\begin{Definition}
Let $K$ be an arbitrary field, let $K^\alg$ be a fixed algebraic closure and let $K^\sep$ be the separable closure of $K$ in $K^\alg$, and $\sG_K : = \Gal(K^\sep/K)$.  We mean by a (smooth) \emph{group variety} over $K$ an irreducible smooth separated scheme $G$ of finite type over $K$ with a group law $mult: G\times_K G \to G$, an inverse map $inv:G\to G$ and a $K$-rational point $0\in G(K)$, the identity element, such that $mult$ and $inv$ are morphisms of varieties satisfying the usual axioms, see \cite[Chapter~III, \S\,11]{Mumford70}. A morphism of group varieties  is a morphism of varieties which is also a homomorphism of groups.  

For a group variety $G$ over $K$, let $\Lie(G)=\T_0G$ be the tangent space to $G$ at the identity element $0$. It is also called the \emph{Lie algebra} of $G$. For every endomorphism $f$ of $G$ we let $\Lie(f)$ be the induced endomorphism of $\Lie G$.
\end{Definition}

\begin{Definition}
An {\emph{elliptic curve}} over a field $K$ is a smooth projective curve $E$ of genus 1, together with a distinguished point $0\in E(K)$. Every such can be written as a smooth projective plane curve which is the zero locus of an equation
\begin{equation}\label{EqWeierstrass}
Y^2 Z + a_1 XYZ + a_3 YZ^2= X^3+a_2 X^2Z+a_4XZ^2+a_6Z^3\quad \text{with } a_i\in K
\end{equation}
and with distinguished point $0=(0:1:0)$. It carries a group law making it into a commutative group variety  with identity element $0$  (see \cite{Silverman86, Husemoller04}).  
\end{Definition}

Let $E$ be an elliptic curve over $\BC$. Then $E(\BC)$    inherits a complex structure as a sub-manifold of $\BP^2(\BC)$.  It is a complex manifold (because $E$ is nonsingular) and compact (because it is closed in the compact space $\BP^2(\BC)$). It is connected and carries a commutative group structure. Therefore, $E$ is  a compact connected complex Lie group of dimension 1. Let $\T_0 E(\BC)$ be the tangent space of $E(\BC)$ at the identity element 0. It is also called the Lie algebra of $E(\BC)$ and denoted $\Lie E$. Then there is a unique homomorphism 
\[
\exp : \T_0 E(\BC) \to E(\BC)
\]
of complex Lie groups such that, for each $v \in \T_0 E(\BC), \ z \mapsto \exp(zv)$  is the one parameter subgroup\footnote{ For a complex Lie group  $G$, a one parameter subgroup of $G$  is a holomorphic homomorphism
$f: \BC\to G$. In complex analysis one proves that for every tangent vector $v$ to $G$ at $e$, there is a unique one-parameter subgroup $f_v: \BC \to G$ such that $f_v(0) = e$ and $(df_v)(1) =v$, see \cite[pp.~79 and 195]{Hochschild}.} $f_v: \BC \to E(\BC)$  corresponding to $v$. The differential of $\exp$ at $0$ is the identity map
\[ 
\T_0 E(\BC) \to \T_0 E(\BC),
\]
and the map $\exp$ is surjective, and its kernel is a  lattice $\Lambda=\Lambda(E)$ in the complex vector space $\T_0 E(\BC)$. So $E(\BC) \cong \BC/ \Lambda$ as a complex Lie group (for more details see \cite[Chapter~I, \S\,1]{Mumford70}). 

Now we explain how one associates an elliptic curve with a lattice. Let $\Lambda$ be  a lattice in $\BC$, that is, a discrete $\BZ$-module $\Lambda\subset \BC$ which is free of rank $2$. With $\Lambda$, we associate its  Weierstrass $\wp$-function 
\begin{align}\label{ElEq1}
\wp_{\Lambda}(z) = \frac{1}{z^2}+\sum_{\omega\in \Lambda\setminus\{0\}} \frac{1}{(z-\omega)^2}- \frac{1}{\omega^2}.
\end{align}
Then $\wp_{\Lambda}(z)$ is $\Lambda$-invariant and meromorphic on $\BC$ with poles of order $2$ at all $\omega\in\Lambda$. It satisfies the equation 
\begin{align}\label{ElEq2}
{\wp'_{\Lambda}(z)}^2 = 4\wp^3_{\Lambda}(z)-g_2 (\Lambda)\wp_{\Lambda}(z)-g_3(\Lambda)
\end{align}
where $g_2(\Lambda) = 60 G_4(\Lambda)$ and $g_3(\Lambda)  = 140 G_6(\Lambda)$, and 
\[
G_k(\Lambda) = \sum_{\omega\in {\Lambda-\{ 0\} }}\frac{1}{\omega^k},
\]
is the Eisenstein series of the lattice $\Lambda$ for $k>2$ even. $g_2$ and $g_3$ satisfy the relation 
\begin{align}\label{ElEq3}
\Delta : = g_2^3- 27g_3^2\neq 0.
\end{align}
This means $\big( \wp_{\Lambda}(z), \wp'_{\Lambda}(z)\big)\in \BC^2$ for $z\notin\Lambda$ is a point on the smooth affine curve $E^{\aff}_\Lambda$ (since $\Delta \neq 0$) with equation
\begin{align}
Y^2 = 4X^3- g_2X-g_3
\end{align}
and $\big( \wp_{\Lambda}(z): \wp'_{\Lambda}(z): 1\big) \in \BP^2(\BC)$ for all $z\in \BC$ is a point on the projective model of the above curve with equation
\begin{align}
Y^2 Z= 4X^3- g_2XZ^2-g_3Z^3.
\end{align}
The above yields a biholomorphic isomorphism of the complex torus $\BC/\Lambda$  with $E_\Lambda(\BC)$, well-defined through its restriction to $(\BC\setminus\Lambda)/\Lambda$  by $z \mapsto  \big( \wp_{\Lambda}(z): \wp'_{\Lambda}(z): 1\big)$. Note that $E_\Lambda(\BC)$ inherits a group structure from $\BC/\Lambda$, which may however be defined in purely algebraic terms on the algebraic curve $E_\Lambda$, and which turns $E_\Lambda$ into an elliptic curve. This is the elliptic curve associated with the lattice $\Lambda$. In fact, each elliptic curve $E$ over $\BC$ has the form $E = E_\Lambda$ for some lattice $\Lambda$ as above, and two such, $E_\Lambda$ and $E_{\Lambda'}$, are isomorphic as elliptic curves (i.e., as algebraic curves through some isomorphism preserving the group structures) if and only if $\Lambda'$ and $\Lambda$ are homothetic, that is, $\Lambda' = c\Lambda$ for some $c \in \BC\mal$.

\medskip 

\begin{Definition}\label{DefAbVar}
An \emph{abelian variety} over a field $K$ is a smooth projective connected group variety. The group law is automatically commutative; see \cite[Chapter~II, \S\,4, Corollary~2]{Mumford70}. Abelian varieties are higher-dimensional generalizations of elliptic curves, which in turn are abelian varieties of dimension $1$. 

A \emph{homomorphism} $f\colon X\to Y$ between abelian varieties over $K$ is a morphism of varieties over $K$ which is compatible with the group structure. The abelian group of homomorphisms $f\colon X\to Y$ over $K$ is denoted $\Hom_K(X,Y)$ and we write $\End_K(X)=\Hom_K(X,X)$. We also write $\QHom_K(X,Y)= \Hom_K(X,Y)\otimes_\BZ \BQ$ and $\QEnd_K(X) =\QHom_K(X,X)= \End_K(X)\otimes_\BZ \BQ$. For an abelian variety $X$ over $K$ and an integer $m\in \BZ$, there is an endomorphism $[m]\in\End_K(X)$ given as the multiplication by $m$ on the points. Thus  if $m > 0$, then
\[
[m](P) = P + P + \cdots + P \ \ (m\ \text{times})
\]
For $m< 0$, we set $[m](P) = [-m](-P)$, and we define  $[0](P) = 0$.

A morphism $f\colon X\to Y$ between abelian varieties is an \emph{isogeny} if it is surjective with finite kernel. Every isogeny is finite, flat, surjective, see \cite[Proposition~8.2]{Milne84}. The \emph{degree} of an isogeny  $ f: X\to Y$ is its degree as a regular map, i.e., the degree of the field extension $[K(X): f^\ast K(Y)]$. If there exists an isogeny $X \to Y$ defined over $K$ we will say that $X$ and $Y$ are \emph{isogenous over $K$} and write $X \approx_K Y$. Note that $\approx_K$ is an equivalence relation. In fact, for every isogeny $f:X\to Y$ there is an isogeny $g : Y \to X$ such that $g\circ f = [n]$ on $X$ for some $n\in  \BZ$, see \cite[Remark~6.5]{MilneAbVar}. This means that $f$ becomes invertible in $\QHom_K(X,Y)$, in the sense that $f^{-1}:=g\otimes\tfrac{1}{n}\in\QHom_K(Y,X)$ is its inverse.
\end{Definition}

\begin{Remark}\label{RemQHomAbVar}
\begin{enumerate}
\item 
Let $X$ and $Y$ be abelian varieties over $K$. If $X$ and $Y$ are isogenous over $K$ via an isogeny $f$, then
\[
\QEnd_K (X) \cong \QHom_K(X,Y) \cong \QEnd_K(Y),\ \ h\mapsto f\circ h \mapsto f\circ h\circ f^{-1}.
\]
More precisely, $\QHom_K(X,Y)$ is a free right $\QEnd_K(X)$-module of rank $1$ and a free left $\QEnd_K(Y)$-module of rank $1$. If $X$ and $Y$ are not isogenous then $\QHom_K(X,Y)=(0)$.
\item 
The homomorphism $[m]\in \End_K(X)$ is an isogeny of degree $m^{2g}$, where $g=\dim X$. It is always \'etale when  $K$ has characteristic zero, and when $ K$  has characteristic $ p > 0$ it is \'etale  if and only if $p$ does not divide $m$, see \cite[Chapter~II, \S\,6]{Mumford70}.
\item 
The kernel $X[m]:=\ker([m]\colon X\to X)$ is a finite group scheme over $K$ of order $m^{2g}$.
\end{enumerate}
\end{Remark}

\begin{Definition} Let $X$ be an abelian variety  and let $m \in Z$ with $m \geq 1$. The {\emph{$m$-torsion subgroup}} of $X$, denoted by $X[m](K^\alg)$, is the subgroup of points of $X(K^\alg)$ of order $m$,
\[
X[m](K^\alg) =\{ P \in  X(K^\alg) : [m]P = 0\}.
\]
It equals the group of $K^\alg$-valued points of the finite group scheme $X[m]$.
\end{Definition}

\begin{Remark}\label{RemRkTateModAbVar}
 For any $m$ not divisible by the characteristic of $K, \ X[m](K^\alg)$ has order $m^{2g}$ and is contained in $X(K^\sep)$.  Since this is also true for any $n$ dividing $m, \  X[m](K^\alg)$ must be a free $\BZ/m\BZ$-module of rank  $2g$.
\end{Remark}

Finally, if $X$ is an abelian variety over $\BC$ of dimension $g$, then $X(\BC)$ is isomorphic to a complex torus $\BC^g/\Lambda$,
\[
X(\BC) \cong \BC^g/\Lambda
\]
for some lattice  $\Lambda=\Lambda(X)$ in $\BC^g$ under an isomorphism of complex manifolds which preserves the group structures. Here $\Lambda\subset\BC^g$ is a discrete $\BZ$-submodule which is free of rank $2g$. However, when $g > 1$, not every  lattice $\Lambda\subset\BC^g$ arises from an abelian variety, that is, the quotient $\BC^g/\Lambda$ of $\BC^g$  by an arbitrary lattice $\Lambda$  does not always arise from an abelian variety.  There is a criterion on $\Lambda$ for when $\BC^g/\Lambda$ is an algebraic  (hence abelian) variety, namely, that $(\BC^g, \Lambda)$  admits a Riemannian form\footnote{For a complex torus $V/\Lambda$ where $V$ is a complex vector space and
$\Lambda$ is a full lattice in $V$, a skew-symmetric form $F:\Lambda \times \Lambda \to \BZ$, that is $F(w,v)=-F(v,w)$, extended to  a skew-symmetric $\BR$-bilinear form $F_{\BR}: V\times V\to \BR$  is  a \emph{Riemannian form} if $F_\BR(iv,iw) =F_\BR(v,w)$ and the associated Hermitian form $H\colon V\times V\to\BC$ with $H(v,w):=F_\BR(iv,w)+i\, F_\BR(v,w)$ and $F_\BR(v,w)=\Im (H(v,w))$ is positive definite.}, see \cite[Chapter~I, \S\,3]{Mumford70}.

\section{Semi-simple Abelian Varieties}\label{SectEmdonAbVar}
\setcounter{equation}{0}

\begin{Theorem}\label{ThmHomAbVar}
For two abelian varieties $X$ and $Y$ over a field $K$ the $\BZ$-module $\Hom_K(X,Y)$ is finite projective of rank $\le (2\dim X)\cdot(2\dim Y)$.
\end{Theorem}

\begin{proof}
See for example \cite[Chapter~IV, \S\,19, Corollary~1]{Mumford70}.
\end{proof}

\begin{Definition}\label{DefSimpleAbVar}
Let $X$ be an abelian variety over $K$. Then $X$ is called
\begin{enumerate}
\item \emph{simple over $K$} if $X$ is non-trivial and there does not exist an abelian subvariety $Y \subset X$ over $K$ other than $(0)$ and $X$. 
\item \emph{semi-simple over $K$} if $X$ is isogenous over $K$ to a direct product of simple abelian varieties, i.e. $X \approx_K X_1\times_K\ldots\times_K X_n$ with $X_i$ simple.
\end{enumerate}
\end{Definition}

\begin{Remark}\label{RemPoincareWeil}
The Theorem of Poincar\'e and Weil \cite[Proposition~9.1]{MilneAbVar} states that any abelian variety is semi-simple over $K$. More precisely, for any abelian variety $X$ over $K$,  there are simple abelian subvarieties $X_1, \cdots, X_n \subset X$ such that the map $X_1 \times_K \cdots \times_K X_n \to X, \ (a_1, \cdots, a_n) \to a_1+\cdots+a_n$ is an isogeny. The  proof of this is analogous with a standard proof for the semi-simplicity of a representation of a finite group $G$ on a finite-dimensional vector space over $\BQ$, see \cite[Remark~9.2]{MilneAbVar}.
\end{Remark}

Let $X$ be a simple abelian variety, and let $0 \neq f\in  \End_K(X)$. Then $f$  is an  isogeny, because by the simplicity of $X$, the image of $f$ equals $X$ and the connected component of $\ker f$ equals $\{0\}$, as both are abelian subvarieties. So $f$ is surjective with finite kernel. From this it follows
that $\QEnd_K(X)$  is a division algebra or equivalently a skew-field, i.e., a ring, possibly non commutative, in which every nonzero element has an inverse. 

\bigskip

\begin{Remark} \label{RemEndOfSimpleAbVar}
Let $X$ be a simple abelian variety over $K$, and let $D = \QEnd_K(X)$. Then $\QEnd_K(X^n) =  M_n(D)$ is the ring of $n\times  n$  matrices with coefficients in $D$.

Now consider an arbitrary abelian variety $X$. Then $X$ is isogenous over $K$ to a product $X_1^{n_1}\times_K \cdots \times_K X_r^{n_r} $, where each $X_i$ is simple, and $X_i$ is not isogenous to $X_j$ for $ i\neq j$ over $K$. The above remarks show that
\[
\QEnd_K (X) \cong \prod  M_{n_i}(D_i), \ \  D_i = \QEnd_K (X_i).
\]
Since $\End_K(X)$ is a free $\BZ$-module of finite rank $\le (2\dim X)^2$ we know that  $\QEnd_K (X)$ is a finite dimensional $\BQ$-algebra.
\end{Remark}

In the following we recall a few facts about semi-simple algebras. Let $Q$ be a field, let $B$ be a semi-simple $Q$-algebra of finite dimension, and let $B = \prod B_i$  be its decomposition into a product of simple algebras $B_i$. A simple $Q$-algebra is isomorphic to a matrix algebra over a division $Q$-algebra. The center of each $B_i$ is a field $F_i$, and each degree $[B_i:F_i]$ is a square. The \emph{reduced degree} of $B$ over $Q$ is defined to be 
\[
[B:Q]_\red = \sum_i[B_i:F_i]^{1/2}[F_i:Q].
\]
For any field $Q'$ containing $Q$,
\[
[B:Q] = [B\otimes_Q Q':Q'] \quad \text{and} \quad [B:Q]_\red = [B\otimes_Q Q':Q']_\red.
\]

\begin{Proposition}[{\cite[Proposition~1.2]{MilneCM}}]\label{PropSemi-Simple}
Let $B$ be a semi-simple $Q$-algebra which is finite dimensional over $Q$. For any faithful $B$-module $M$,
\[
\dim_Q M\geq  [B:Q]_\red;
\]
and there exists a faithful module for which equality holds if and only if the simple factors of $B$ are matrix algebras over their centers. 
\end{Proposition}

\begin{Proposition}[{\cite[Proposition~1.3]{MilneCM}}]\label{PropEtSubalg}
Let $\charakt(Q)=0$ and let $B$ be a semi-simple $Q$-algebra. Every maximal \'etale $ Q$-subalgebra of $B$ has degree $[B:Q]_\red$ over $Q$. Here we mean by an \emph{\'etale $Q$-algebra}  a finite product of finite separable field extensions of $Q$.  
\end{Proposition}

\section{Cohomology}\label{SectCohomAbVar}
\setcounter{equation}{0}

\subsection{Singular Cohomology}\label{SectBettiAbVar}

Let $X$ be an abelian variety of dimension $g$ over $\BC$. Let $V$ be the tangent space  of $X$ at the identity element  and let  $\Lambda$ be the kernel of the exponential map $\exp : V \to X$. Now the space $V \cong \BC^g$ is simply connected, and $\exp : V \to X$ is a covering map, therefore it realizes $V$ as the universal covering space of $X$, and so $\pi_1(X)$ is its group of covering transformations, which is $\Lambda$. In particular, it is abelian. As for any good topological space we obtain for the \emph{singular cohomology} of $X$ 
\begin{align*}
 \Koh^1(X, \BZ) \cong \Hom_{\text{groups}}(\pi_1(X), \BZ) = \Hom_\BZ(\Lambda, \BZ).
\end{align*}

Since we have seen that $X$ is a complex torus of dimension $g$, it is isomorphic to $(\BR/\BZ)^{2g} = (S^1)^{2g}$ as a real Lie group, where $S^1$ is the circle group.  We claim that for all $r\in \BN_{>0}$
\[
 {\bigwedge}^{r}  \Koh^1(X, \BZ) \isoto \Koh^r(X, \BZ) 
\]
under the natural map defined by the cup product. Indeed, by the K\"unneth formula if the above map is an isomorphism for spaces $X_1$ and $X_2$ with finitely generated cohomologies, then it is an isomorphism for $X_1 \times_K X_2$. Since it is an isomorphism for $S^1$ for all $r\geq 0$, where the module is $(0)$ for $r\ge2$, the result for $X$ follows.

Since $X$ is compact and orientable and $\Koh^r(X, \BZ) $ is torsion free, the duality theorems gives  us for the \emph{singular homology} of $X$
\[
\Koh_r(X, \BZ)  \cong \Koh^r(X, \BZ)^{\vee}\qquad\text{and in particular}\qquad\Koh_1(X, \BZ)= \Lambda\,.
\] 

\subsection{$\ell$-adic Cohomology} \label{SectEtCohAbVar}

We follow \cite[\S\,15]{Milne84}. Let $X$ be an abelian variety of dimension $g$ over a field $K$, and let $\ell$  be a prime different from $\charakt(K)$.  Recall that, for any $m$ not divisible by the characteristic of $K$, $X[m](K^\sep)$ has order $m^{2g}$. Define  the \emph{$\ell$-adic Tate module} of $X$ as
\[
T_\ell(X) = \invlim \bigl(X[\ell^n](K^\sep), [\ell]\bigr).
\]
It follows that $T_\ell(X)$  is a free $\BZ_\ell$-module of rank $2g$. There is a continuous action of $\sG_K$ on this module.

Let $X$ and $Y$ be two abelian varieties over $K$. A homomorphism $f : X \to Y$ induces a homomorphism $X[\ell^n] \to Y[\ell^n]$, and hence a homomorphism 
\[
T_\ell(f): T_\ell(X)\to T_\ell(Y),  \quad (a_1, a_2, \cdots )\mapsto \big(f(a_1), f(a_2), \cdots\big).
\]
Therefore, $T_\ell$ is a functor from abelian varieties to $\BZ_\ell$-modules. It is easy to see that for any prime $\ell \neq \charakt(K)$, the natural map
\[
\Hom_K(X, Y)\to \Hom_{\BZ_\ell} ( T_\ell(X), T_\ell(Y))
\]
is injective. From this one obtains that the $\BZ$-algebra $\Hom_K(X, Y)$ of morphisms $X\to Y$ of group varieties is torsion-free. The following theorem was conjectured by Tate~\cite{TateEndoms} and proved by him for finite fields $K$. It was proved by Zarhin~\cite{Zarhin75} for fields of positive characteristic and by Faltings~\cite{FaltingsEndlichkeit,FaltingsComplements} for fields of characteristic zero. 

\begin{Theorem}[Tate conjecture for abelian varieties] \label{ThmTateConjAbVar}
Let $X$ and $Y$ be two abelian varieties over a finitely generated field $K$ and let $\ell$ be a prime different from the characteristic of $K$. Then the natural map 
\[
\Hom_{K}(X, Y) \otimes_\BZ\BZ_\ell \to \Hom_{\BZ_\ell[\sG_K]}( T_\ell X, T_\ell Y),\quad f\otimes a \mapsto a\cdot T_\ell(f)
\]
is an isomorphism of $\BZ_\ell$-modules.
\end{Theorem}
The theorem is known to fail for some classes of fields which are not finitely generated (e.g. local fields and of course algebraically closed fields).

Now we write $X_{K^\alg}:=X\times_K\Spec K^\alg$ and denote by $\pi^\et_1(X_{K^\alg},0)$ the \'etale fundamental group, then 
\[
 \Koh^1_{ \et}(X_{K^\alg}, \BZ_\ell) \cong \Hom^{\rm cont}(\pi^\et_1(X_{K^\alg},0),\BZ_\ell).
\]
For each $n$ the map $[\ell^n] : X\to X$  is a finite \'etale covering of  $X$ with group of covering transformations $X[\ell^n](K^\sep)$.  By definition $\pi^\et_1(X_{K^\alg},0)$  classifies such coverings, and therefore there is a canonical epimorphism $\pi^\et_1(X_{K^\alg},0) \twoheadrightarrow X[\ell^n]$. On passing to the inverse limit, we get an epimorphism $\pi^\et_1(X_{K^\alg},0) \twoheadrightarrow T_\ell(X)$, and consequently an injection
\[
\Hom_{\BZ_\ell}(T_\ell(X), \BZ_\ell)\into  \Koh^1_{ \et}(X_{K^\alg}, \BZ_\ell),
\]
which actually is an isomorphism, see \cite[Theorem~15.1]{Milne84}. So we obtain for  the  first $\ell$-adic homology group of $ X$
\begin{align*}
 \Koh_{1, \et}(X_{K^\alg}, \BZ_\ell) = T_\ell(X)  \  \text{and}\  \Koh_{1, \et}(X_{K^\alg}, \BQ_\ell) = T_\ell(X)\otimes_{\BZ_\ell}\BQ_\ell\,,
\end{align*}
and for the first $\ell$-adic cohomology group of $ X$
\begin{align*}
 \Koh^1_{ \et}(X_{K^\alg}, \BZ_\ell) =  \Koh_{1, \et}(X_{K^\alg}, \BZ_\ell)^\vee \ \text{and} \  \Koh^1_{\et}(X_{K^\alg}, \BQ_\ell) =  \Koh_{1, \et}(X_{K^\alg}, \BQ_\ell)^\vee.
\end{align*}
By \cite[Theorem~15.1]{Milne84} the cup product pairings define isomorphisms
\begin{align*}
 \Koh_{ r,\et}(X_{K^\alg}, \BZ_\ell)\cong {\bigwedge}^{r}\Koh_{1, \et}(X_{K^\alg}, \BZ_\ell) \ \text{and} \  \Koh^r_{\et}(X_{K^\alg}, \BQ_\ell) \cong {\bigwedge}^{r} \Koh^1_{ \et}(X_{K^\alg}, \BQ_\ell).
\end{align*}
Now, over the field $K=\BC$ the choice of an isomorphism $X(\BC) \cong \BC^g/\Lambda$ determines $X[m](\BC) \cong m^{-1}\Lambda/\Lambda$. Then
\begin{align*}
T_\ell(X) & = \invlim \bigl(X[\ell^n](\BC), [\ell]\bigr)  \cong\invlim  \bigl(\ell^{-n}\Lambda/\Lambda , \text{multiplication with }\ell\bigr) \\
              & \cong \invlim \bigl(\Lambda \otimes_\BZ (\BZ/\ell^n\BZ),\mod \ell^n\bigr)\\
              & \cong \Lambda \otimes_\BZ  \invlim(\BZ/\ell^n\BZ), \ \text{because}\ \Lambda\ \text{is a free}\ \BZ\text{-module}\\
             & \cong \Lambda \otimes_\BZ \BZ_\ell.
\end{align*}

Taking duals and exterior powers, we can summarize the results as a

\begin{Theorem} \label{ThmCompIsomBEt}
For every abelian variety $X$ over $\BC$ there are canonical \emph{comparison isomorphisms between singular and $\ell$-adic (co-)homology} 
\[
\Koh^r_{\et}(X, \BZ_\ell) \cong \Koh^r(X, \BZ) \otimes_\BZ \BZ_\ell\qquad \text{and} \qquad \Koh_{r, \et}(X, \BZ_\ell) \cong \Koh_r(X, \BZ) \otimes_\BZ \BZ_\ell\,. 
\]
\end{Theorem}

\begin{Example}\label{ExGmBettiEt}
Also for the multiplicative group scheme $\BG_m:=\BG_{m,\BQ}=\Spec\BQ[x,x^{-1}]$ there is a period isomorphism between $\Koh_1(\BG_m(\BC), \BZ)\otimes_\BZ \BZ_\ell$ and $\Koh_{1,\et}(\BG_{m,\BC}, \BZ_\ell)$. Namely, $\Koh_1(\BG_m(\BC), \BZ)=\BZ\cdot u$, where $u\colon[0,1]\to\BG_m(\BC)=\BC\mal$ is the cycle given by $u(s)=\exp(2\pi is)$. Also let $\epsilon_\ell^{(n)}:=\exp(2\pi i/\ell^n)\in\BQ^\alg\subset\BC$. It is a primitive $\ell^n$-th root of unity with $(\epsilon_\ell^{(n+1)})^\ell=\epsilon_\ell^{(n)}$ for all $n$. Let $\epsilon_\ell:=(\epsilon_\ell^{(n)})_{n\in\BN}\in T_\ell\BG_m$. Then $\Koh_{1,\et}(\BG_{m,\BC}, \BZ_\ell)=T_\ell\BG_m=\epsilon_\ell^{\BZ_\ell}$ and the comparison isomorphism 
\[
\Koh_1(\BG_m(\BC), \BZ)\otimes_\BZ \BZ_\ell \isoto \Koh_{1,\et}(\BG_{m,\BC}, \BZ_\ell)
\]
sends $u$ to $\epsilon_\ell$. This can be seen from the exact sequence $0\to\BZ=\pi_1(\BC\mal)\to\BC\xrightarrow{\exp(2\pi i\,\fdot\,)}\BC\mal\to0$ and the induced comparison isomorphism $\pi_1(\BC\mal)\otimes_\BZ\BZ_\ell\isoto T_\ell\BG_m,\,1 \mapsto\bigl(\exp(2\pi i/\ell^n)\bigr)_{n\in\BN}$.
 \end{Example}

\subsection{De Rham Cohomology}\label{SectDRAbVar}

We will now explain the construction  of the Dolbeault complex associated with $X$ which is an analog of the de Rham complex for complex manifolds. Let $X$ be an abelian variety over $\BC$. 

Let $ \sC^n =\DS \oplus_{p+q= r}\sC^{p,q}$ be the sheaf of $C^\infty$ complex valued  differential $n$-forms, where $\sC^{p,q}$ is the sheaf of $C^\infty$ complex valued differential forms of type $(p,q)$. In terms of local coordinates, let $(z_1,\cdots, z_g)$ be a holomorphic coordinate system.  First  we decompose the complex coordinates into their real and imaginary parts: $z_j=x_j+iy_j$ for each $j$. Letting  $\displaystyle dz_{j}=dx_{j}+idy_{j}, \  d{\bar {z}}_{j}=dx_{j}-idy_{j}, $  one sees that any differential $1$-form with complex coefficients can be written uniquely as a sum
\[
\displaystyle \sum _{j=1}^{n}\left(f_{j}dz_{j}+g_{j}d{\bar {z}}_{j}\right), 
\]
for $\BC$-valued $C^\infty$-functions $f_j$ and $g_j$. Let $\sC^{1,0}$ be the sheaf of $C^\infty$ complex valued differential $1$-forms where all $g_j$ are zero, and let $\sC^{0,1}$ be the sheaf of $C^\infty$ complex valued differential $1$-forms where all $f_j$ are zero. Then the space $\sC^{p,q}$ of type  $(p,q)$-forms is defined by taking linear combinations of the wedge products of $p$ elements from $\sC^{1,0}$ and $q$ elements from $\sC^{0,1}$. Symbolically, 
\[
\sC^{p,q} = \bigwedge^p\sC^{1,0}\wedge\bigwedge^q \sC^{0,1}
\]
In particular for each $n$ and each $p$ and $q$ with $p+q = n$,  there are canonical projection maps which we denote by 
\[
\pi^{(p,q)}: \sC^n \to \sC^{p,q}.
\]
The exterior derivative defines a map $d : \sC^n \to \sC^{n+1}$ i.e. if $\phi \in \sC^{p,q}$ then $d(\phi) \in \sC^{p+1,q}\oplus \sC^{p,q+1}$. Using $d$ and the projections maps, it is possible to define the operators:
\[
\partial =\pi ^{p+1,q}\circ d: \sC^{p,q} \to \sC^{p+1,q}, \quad  \bar \partial =\pi ^{p,q+1}\circ d : \sC^{p,q} \to \sC^{p,q+1}
\]
In terms of local coordinates $z =(z_1, \cdots z_g)$ we can write $\phi \in  \sC^{p,q}$ as
\[
{\displaystyle \phi =\sum _{\#I=p,\,\#J=q}\ f_{IJ}\,dz_{I}\wedge d{\bar {z}}_{J}\in  \sC^{p,q}} 
\]
where $I$ and $J$ are multi-indices and $dz_I=\bigwedge_{i\in I}dz_i$ and $d\bar z_I=\bigwedge_{i\in I}d\bar z_i$. Then
\begin{align*}
 \partial\phi & =\sum _{\#I=p,\,\#J=q}\sum _{i }{\frac {\partial f_{IJ}}{\partial z_{i }}}\,dz_{i }\wedge dz_{I}\wedge d{\bar {z}}_{J} \qquad\text{and}\qquad
\bar {\partial }\phi & =\sum _{\#I=p,\,\#J=q}\sum _{i }{\frac {\partial f_{IJ}}{\partial {\bar {z}}_{i }}}d{\bar {z}}_{i }\wedge dz_{I}\wedge d{\bar {z}}^{J}.
\end{align*}
It is not difficult to see the following properties:
\begin{align*}
 d& =\partial +{\bar {\partial }} \\
 \partial ^{2}& ={\bar {\partial }}^{2}=\partial {\bar {\partial }}+{\bar {\partial }}\partial =0.
\end{align*}
Then the Poincar\'e lemma gives that the complex
\[
0 \to \BC \to \sC^0 \xrightarrow{d} \sC^1   \xrightarrow{d} \cdots
\]
is a fine resolution of the constant sheaf $\BC$. It is called the \emph{de Rham resolution}. We define the  de Rham cohomology as the cohomology of this complex i.e. 
\[
\Koh^n_{\dR}(X,\BC) = \frac{\{\text{global $n$-forms }\phi\in\sC^n(X)\text{ on $X$ which are $d$-closed, i.e. }d\phi=0\}}{\{d\psi: \text{where }\psi\in\sC^{n-1}(X) \text{ is a global }(n-1)\text{-form on $X$}\}}.
\]

Let $V = \T_0 X$ be the tangent space to $ X$ at 0 (regarded as a complex vector space). Let $T\dual = \Hom_{\BC}(V, \BC)$ be the complex cotangent space to $X$ at 0 and $ \ol T\dual = \Hom_{\Cant}(V, \BC)$. Then from linear algebra 
\[ 
 \Hom_{\BR}(V, \BC) \cong \Hom_{\BC}(V, \BC) \oplus \Hom_{\Cant}(V, \BC)\ \text{ i.e.} \   \Hom_{\BR}(V, \BC)  \cong T\dual \oplus \ol T\dual,
\]
and 
\begin{align*}
 {\bigwedge}^r \Hom_{\BR}(V, \BC) \cong  \bigoplus_{p+q= r} {\bigwedge}^p T\dual \otimes {\bigwedge}^q \ol T\dual
\end{align*}
By translation under the group law on $X$  every complex co-vector $\phi \in {\wedge}^p T\dual \otimes {\wedge}^q \ol T\dual$  extends to a unique translation invariant  $\omega_{\phi}\in \sC^{p,q}$, and therefore  every complex co-vector $\phi\in\wedge^r\Hom_\BR(V,\BC)$ extends to a unique translation invariant form $\omega_\phi$ belonging to $\sC^n$. For all $d$-closed $n$-forms $\omega$, there is unique translation invariant $\omega_\phi$ for $\phi \in  {\wedge}^n \Hom_{\BR}(V, \BC)$, such that
\[
\omega-\omega_\phi = d\eta, \ \text{for some}\ (n-1)  \text{-form}\ \eta.
\]
 Therefore, $\Koh^r_{\dR}(X,\BC) \cong  {\bigwedge}^r \Hom_{\BR}(V, \BC)$, and has the decomposition 
\[
\Koh^r_{\dR}(X,\BC) \cong  \bigoplus_{p+q= r} {\bigwedge}^p T\dual \otimes {\bigwedge}^q \ol T\dual.
\]
For the sheaf $\Omega^p:=\ker(\bar\partial\colon\sC^{p,0}\to\sC^{p,1})$ of holomorphic $p$-forms on $X$ we know from \cite[Chapter~I, \S\,1, Theorem]{Mumford70} that 
\[
\Koh^q(X, \CO_X) \cong {\bigwedge}^{q}\ol T\dual \quad \text{and} \  \Koh^q(X, \Omega^p) \cong {\bigwedge}^{p} T\dual\otimes {\bigwedge}^{q}\ol T\dual,
\]
so 
\[
\Koh^r_{\dR}(X,\BC) \cong  \bigoplus_{p+q= r} \Koh^{p,q}(X), \quad \text{where}\quad\Koh^{p,q}(X) := \Koh^q(X, \Omega^p).
\]
This is the famous Hodge decomposition.

\bigskip
\noindent
Now we obtain the \emph{de Rham isomorphism}
\[
\Koh^1(X,\BC) =  \Koh^1(X,\BZ)\otimes_{\BZ} \BC = \Hom_\BZ(\Lambda, \BZ) \otimes_{\BZ} \BC = \Hom_{\BR}(V, \BC) \cong \Koh^1_{\dR}(X,\BC). 
\]

Then, $\Koh^n_{\dR}(X,\BC) \cong  \Koh^n(X,\BQ)\otimes_{\BQ} \BC$. Note that complex conjugation on the right tensor factor  of the target defines a conjugate-linear automorphism of $\Koh^n_{\dR}(X,\BC)$.  For more details see \cite[Chapter~I, \S\,1]{Mumford70}. Taking also exterior powers, we can summarize the results as a

\begin{Theorem}[De Rham isomorphism] \label{ThmDeRhamIsom}
For every abelian variety $X$ over $\BC$ there are canonical \emph{comparison isomorphisms between singular and de Rham cohomology}  
\[
\Koh^r_{\dR}(X, \BC) \cong \Koh^r(X, \BZ) \otimes_\BZ \BC\,. 
\]
\end{Theorem}

\begin{Example}\label{ExGmBettiDR}
Also for the multiplicative group scheme $\BG_m:=\BG_{m,\BC}=\Spec\BC[x,x^{-1}]$ there is a de Rham isomorphism between $\Koh^1(\BG_m(\BC), \BZ)$ and $\Koh^1_{\dR}(\BG_m,\BC)=\BC\tfrac{dx}{x}$. As in Example~\ref{ExGmBettiEt}, the singular homology $\Koh_1(\BG_m(\BC), \BZ)=\BZ\cdot u$, where $u\colon[0,1]\to\BG_m(\BC)=\BC\mal$ is the cycle given by $u(s)=\exp(2\pi is)$. The de Rham isomorphism is given as the pairing
\[
\Koh_1(\BG_m, \BZ)\times \Koh^1_{\dR}(\BG_m,\BC)\longto\BC\,,\qquad (nu,\omega)\longmapsto n{\TS\int_u}\omega\,,\qquad (u,\tfrac{dx}{x})\longmapsto {\TS\int_u}\tfrac{dx}{x}\;=\;2\pi i.
\]
The corresponding isomorphism $\Koh^1(\BG_m, \BZ)\otimes_\BZ\BC\isoto \Koh^1_{\dR}(\BG_m,\BC)$ sends the generator of $\Koh^1(\BG_m, \BZ)$, which is dual to $u$, to $(2\pi i)^{-1}\cdot\tfrac{dx}{x}$.
\end{Example}

\section{$p$-divisible Groups and the $p$-adic Period Isomorphism}\label{SectPAdicPeriodIsom}
\setcounter{equation}{0}

Let $R$ be a commutative ring. Let $p$ be a prime number, and $h$ an integer $\geq 0$. A $p$-divisible group $G$ over $R$ of height $h$ is an inductive system
\[
(G_n, i_n), \quad n\geq 0
\]
where
\begin{enumerate}
\item $G_n$ is a finite flat commutative group scheme of finite presentation over $R$ of order $p^{nh}$,

\item for each  $n\geq 0$,
\[
0 \to G_n \xrightarrow{i_n} G_{n+1}\xrightarrow{[p^n]} G_{n+1}
\]
is exact (i.e., $G_n$ can be identified via $i_n$ with the kernel of multiplication by $p^n$  in $G_{n+1}$).
\end{enumerate}

These axioms for ordinary abelian groups would imply
\[
G_n \cong  (\BZ/ p^n\BZ)^h \quad \text{and} \quad G = \varinjlim G_n = (\BQ_p/\BZ_p)^h.
\]
A homomorphism  $f: G\to  H$  of $p$-divisible groups is defined in the obvious way: if $G=(G_n, i_n), \ H=(H_n,  i_n)$ then $f$ is a system of homomorphisms $f_n: G_n \to H_n$ of group schemes over $R$, satisfying  $i_n\circ f_n=f_{n+1}\circ i_n$ for all $n\geq 1$.

\begin{Example}
Let $G$ be a commutative group variety over a field $K$, which is either an abelian variety or $\BG_m$. We can associate a $p$-divisible group with $G$:

Define $G[m]$ as the kernel of multiplication by $m$. Then $(G[p^n], i_n)$ is a $p$-divisible group, where $i_n$ denotes the obvious inclusion. This $p$-divisible group is sometimes denoted $G[p^\infty]$.
\begin{enumerate}
\item  If $G=X$ is an abelian variety, then the height of $G[p^\infty]$ is $2\dim X$.
\item If $G=\BG_m$ is the multiplicative group scheme, then $\BG_m[p^\infty] = \Bmu_{p^\infty} := (\Bmu_{p^n}, i_n)$ with height 1. Here $\Bmu_{p^n}=\Spec K[x]/(x^{p^n}-1)$ is the group scheme of $p^n$-th roots of unity.
\end{enumerate}
\end{Example}

Let us see how $p$-divisible groups generalize Tate modules. Suppose $p \neq \charakt(K)$. Then for a $p$-divisible group $(G_n, i_n)$ of height $h$ over $K$ each $G_n$ is a finite \'etale group scheme over $K$ and each $M_n: =G_n(K^\sep)$ is a discrete $\sG_K$-module of size $p^{nh}$ annihilated by $p^n$ and $M_{n+1}[p^n] = M_n$. It follows that $M_n=(\BZ/p^n\BZ)^h$. We can form two kinds of limits: \\
(i) the direct limit $M_\infty =\varinjlim M_n$ is $(\BQ_p/\BZ_p)^h$  with a continuous $\sG_K$-action for the discrete topology, and \\
(ii) multiplication by $p$ on $M_{n+1}$ provides a quotient map $M_{n+1}\twoheadrightarrow M_n$ of discrete $\sG_K$-modules yielding an inverse limit $T_p(M) = \invlim M_n$ that is a finite free $\BZ_p$-module of rank $h$ equipped with a continuous action of $\sG_K$ for the $p$-adic topology.

We can recover the direct system $(M_n,i_n)$ from both limits, namely $M_n = M_\infty[p^n]$ and $M_n = T_p(M)/(p^n)$. The viewpoint of $M_\infty$ explains the “$p$-divisible” aspect of the situation (since multiplication by $p$ is surjective on $(\BQ_p/\BZ_p)^h)$, whereas $T_p(M)$ has a nicer $\BZ_p$-module structure. Since the \'etale group scheme $G_n$ is uniquely determined by the $\sG_K$-module $M_n$, this proves:

\begin{Proposition}If $K$ is a field with $p \neq  \charakt(K)$, then the functor $G \to T_p(G)$ is an equivalence from the category of $p$-divisible groups over $K$ to the category of finite free $\BZ_p$-modules with continuous $\sG_K$-action.
\qed
\end{Proposition}

On the other hand let $K$ be a finite extension of $\BQ_p$ and let $X$ be an abelian variety over $K$. Assume that $X$ has \emph{good reduction}, i.e. there exists a smooth projective commutative group scheme $\CX$ over $\CO_K$ with $X \cong \CX\times_{\CO_K}\Spec K$. Then $X[p^n]$ admits an integral model $\CG_n:=\CX[p^n]$ with $\CG_n = \CG_{n+1}[p^n]$ for all $n \geq 1$ and $\CG=(\CG_n,i_n)$ is a $p$-divisible group over $\CO_K$ with $\CG_K:=\CG\times_{\CO_K}\Spec K\cong X[p^\infty]$.

Now  due to Tate \cite{TateP-DivisibleGroups} we know that if $\CG$ and $\CH$ are $p$-divisible groups over $\CO_K$ then 
\[
\Hom_{\CO_K}(\CG, \CH)\isoto \Hom_K(\CG_K, \CH_K).
\]

\begin{Remark}\label{RemDieudonneMod}
$p$-divisible groups over a perfect field $k$ of characteristic $p$ have a description via semi-linear algebra by their \emph{Dieudonn\'e module}. The latter is a finite free module $M$ over the ring $W(k)$ of $p$-typical Witt-vectors over $k$, equipped with a $\Frob_{p}$-semi-linear morphism $F:M\to M$, called \emph{Frobenius}, and a $\Frob_{p}^{-1}$-semi-linear morphism $V:M\to M$, called \emph{Verschiebung}, satisfying $FV=p=VF$.

This was generalized by Fontaine~\cite{Fon77} to $p$-divisible groups $G$ over the ring of integers $\CO_K$ of a finite field extension $K$ of $\BQ_p$. Those $p$-divisible groups are described by the Dieudonn\'e module $D$ of the special fiber $G\times_{\CO_K}\Spec k$ together with a decreasing exhaustive and separated filtration $\Fil^\bullet$ on $D_K = D\otimes_{K_0}K$ satisfying $\Fil^0(D_K) = D_K, \ \Fil^2(D_K) = (0)$, where $K_0=W(k)[p^{-1}]$ is the maximal unramified subextension of $K$.
\end{Remark}

\begin{Notation}\label{NotAinf}
Let $\CO_{\BC_p}^\flat:={\invlim}(\CO_{\BC_p},\Frob_p)=\{x=(x^{(n)})_{n\in\BN_0}\in(\CO_{\BC_p})^{\BN_0}\colon (x^{(n+1)})^p=x^{(n)}\}$ and $A_{\inf}:=W(\CO_{\BC_p}^\flat)$ be the ring of Witt vectors. Every element of $A_{\inf}$ can be written in the form $\sum_{i=0}^\infty[x_i]p^i$ where $[x]$ denotes the Teichm\"uller lift of the element $x=(x^{(n)})_n\in\CO_{\BC_p}^\flat$. Let $\Theta: A_{\inf}[\tfrac{1}{p}] \twoheadrightarrow \BC_p$ be the morphism sending $\sum_i[x_i]p^i$ to $\sum_i x_i^{(0)}p^i$. The \emph{de Rham period ring} $\BB^+_{p, \dR}$ is the completion of $A_{\inf}[\tfrac{1}{p}]$ at the maximal ideal $\ker\Theta$ and $\BB_{p,\dR}: = \Frac(\BB^+_{p,\dR})$ is \emph{the field of $p$-adic periods}. The de Rham period ring $\BB^+_{p,\dR}$ is a complete discrete valuation ring with residue field $\BC_p$ and maximal ideal $\ker \Theta$. One can show that the ideal $\ker \Theta\subset A_{\inf}$ is principal and generated by an element $[p^\flat] - p \in A_{\inf}$, where  $p^\flat = (p, p^{1/p}, p^{1/p^2}, \cdots)\in \CO_{\BC_p}^\flat$. Any other generator is of the form $([p^\flat]-p)\cdot u$ for $u\in A_{\inf}\mal$. For more details see \cite{Fon77}. There is a filtration on $\BB_{p,\dR}$ defined by putting $\Fil^i(\BB_{p,\dR}) = ([p^\flat]-p)^i\cdot\BB_{p,\dR}^+$ for $i\in \BZ$, and we define $\hat{v}_p(x)$ for $x\in \BB_{p,\dR}\setminus\{0\}$ by $\hat{v}_p(x)=i$ if $x \in \Fil^i(\BB_{p,\dR}) \setminus\Fil^{i+1}(\BB_{p,\dR})$. For $x \in \BB_{p,\dR}\setminus\{ 0\}$, the quantity 
\begin{equation}\label{EqNormBdR}
v_p(x)\;:=\;v_p\big(\Theta(x\cdot ([p^\flat]-p)^{-\hat{v}_p(x)})\big)\in\BQ
\end{equation}
does not depend on the choice of the generator $[p^\flat]-p$ of $A_{\inf}\cap \ker \Theta$. Indeed, if we replace the generator $[p^\flat]-p$ of $\ker\Theta\subset A_{\inf}$ by another generator $([p^\flat]-p)\cdot u$ with $u\in A_{inf}\mal$, because then $v_p\bigl(\Theta(x\cdot(([p^\flat]-p)\cdot u)^{-\hat v_p(x)}\bigr)=v_p\bigl(\Theta(x\cdot([p^\flat]-p)^{-\hat v_p(x)}\bigr)+ v_p(\Theta(u))^{-\hat v_p(x)}=v_p\bigl(\Theta(x\cdot([p^\flat]-p)^{-\hat v_p(x)}\bigr)$ as $\Theta (u) \in \CO_{\BC_p}\mal$. If $x$ and $y$ are two elements of $\BB_{p,\dR}$, then $\hat v_p(xy) = \hat v_p(x) + \hat v_p(y)$, and hence $v_p(xy) = v_p(x) + v_p(y)$. But note that $v_p$ does not satisfy the triangle inequality. 

Finally, if $K\subset\BC_p$ is a finite field extension of $\BQ_p$, then there is an action of $\sG_K$ on $\BB_{p,\dR}$ which respects the filtration, and $(\BB_{p,\dR})^{\sG_K} =K$. Also note that  there does not exist a lift of the absolute  Frobenius $\phi_p$ on $\BB_{p,\dR}$.
\end{Notation}

The \emph{$p$-adic period isomorphism} is provided by the following theorem which was proved by Fontaine and Messing~\cite{Fontaine-Messing} using the associated $p$-divisible group.

\begin{Theorem}\label{ThmPAdicPeriodIso}
Let $K_p\subset\BQ_p^\alg$ be a finite extension of $\BQ_p$ and let $X$ be an abelian variety over $K_p$. Then for every $n\ge0$ there is a period isomorphism from $p$-adic Hodge theory  
\[
h_{p,\dR}: \Koh^n_\et(X\times_{K_p}\Spec\BQ_p^\alg, \BZ_p)\otimes_{\BZ_p} \BB_{p,\dR} \isoto \Koh^n_{\dR}(X,K_p)\otimes_{K_p} \BB_{p,\dR},
\]
which is $\sG_{K_p}$-equivariant and compatible with the filtrations, where on the left hand side, $\sG_{K_p}$ acts on both factors and the filtration is induced only by $\BB_{p,\dR}$, and on the right hand side $\sG_{K_p}$ acts only on $\BB_{p,\dR}$ and the filtration is induced by the Hodge filtration on $\Koh^1_{\dR}(X,K_p)$ and the filtration on $\BB_{p,\dR}$, i.e.\ $\Fil^k\bigl(\Koh^1_{\dR}(X,K_p)\otimes_{K_p} \BB_{p,\dR}\bigr):=\sum\limits_{i+j=k}\Fil^i\Koh^1_{\dR}(X,K_p)\otimes_{K_p} \Fil^j\BB_{p,\dR}$.
\end{Theorem}

It was conjectured by Fontaine~\cite[A.6]{Fon82} and proved by Faltings~\cite[Theorem~8.1]{FaltingsCrystCohom}, Niziol~\cite{Niziol98} and Tsuji~\cite{Tsuji99}, that the theorem also holds for arbitrary smooth proper schemes over $K_p$.

\begin{Example}\label{ExGm}
Also for the multiplicative group scheme $\BG_m:=\BG_{m,\BQ_p}=\Spec\BQ_p[x,x^{-1}]$ there is a period isomorphism between $\Koh^1_\et(\BG_{m,\BQ_p^\alg}, \BZ_p)$ and $\Koh^1_{\dR}(\BG_m,\BQ_p)=\BQ_p\tfrac{dx}{x}$, see Example~\ref{ExGmBettiDR}. As in Example~\ref{ExGmBettiEt} let $\epsilon_p^{(n)}\in\BQ^\alg\subset\BQ_p^\alg$ be a primitive $p^n$-th root of unity with $(\epsilon_p^{(n+1)})^p=\epsilon_p^{(n)}$, such that $\epsilon_p=(\epsilon_p^{(n)})_n\in\CO_{\BC_p}^\flat$. Then $\Koh_{1,\et}(\BG_{m,\BQ_p^\alg}, \BZ_p)=T_p\BG_m=\epsilon_p^{\BZ_p}$ and $\Koh^1_\et(\BG_{m,\BQ_p^\alg}, \BZ_p)=(T_p\BG_m)\dual=(\epsilon_p^{-1})^{\BZ_p}$. On the latter $\Gal(\BQ_p^\alg/\BQ_p)$ acts through the inverse of the cyclotomic character. The series $t_p:=\log[\epsilon_p]:=-\sum_{n>0}\tfrac{1}{n}(1-[\epsilon_p])^n$ converges in $\BB_{p,\dR}$. Under the period isomorphism 
\[
h_{p,\dR}: \Koh^1_\et(\BG_{m,\BQ_p^\alg}, \BZ_p)\otimes_{\BZ_p} \BB_{p,\dR} \isoto \Koh^1_{\dR}(\BG_m,\BQ_p)\otimes_{\BQ_p} \BB_{p,\dR},
\]
of $\BG_m$ the element $\tfrac{dx}{x}\otimes 1$ is mapped to $\epsilon_p^{-1}\otimes t_p$. Therefore $t_p$ can be viewed as the $p$-adic analog of the complex period $2\pi i$ from Example~\ref{ExGmBettiDR}. It satisfies $\hat{v}_p(t_p)=1$ and $v_p(t_p)=v_p\bigl(\Theta(t_p\cdot ([p^\flat]-p)^{-1})\bigr)=\tfrac{1}{p-1}$, see \cite[\S\,0.2]{ColmezPeriods}.
\end{Example}

\section{Complex Multiplication}\label{SectCMAbVar}
\setcounter{equation}{0}

We follow \cite{MilneCM}. Complex conjugation on $\BC$ (or a subfield) is denoted by $c$ or simply by $a\to  \bar a$. A complex conjugation on a field $K$ is an involution induced by an embedding of $K$ into $\BC$ and by complex conjugation on $\BC$.  

\bigskip
\noindent
A number field $E$ is a \emph{CM-field} if it is a quadratic extension $E/F$ where the base field $F$ is totally real but $E$ is totally imaginary. i.e., every embedding of $F \into  \BC$  lies entirely within $\BR$, but there is no embedding of $E \into  \BR$ or equivalently there exists an automorphism $c_E\neq \id$ of $E$ such that  $\rho \circ c_E = c\circ \rho$ for all homomorphisms $\rho: E\into \BC$. In other words, there is a subfield $F$ of $E$ such that $E = F[\sqrt \alpha]$,\ $F$ totally real, $\alpha \in F$ and $\rho(\alpha) < 0$ for all homomorphisms $\rho: F \into \BC$.

\begin{Remark}A finite composite of CM-subfields of a field is CM. In particular, the Galois closure of a CM-field in any larger field is CM.
\end{Remark}
A \emph{CM-algebra} is a finite product of CM-fields. Equivalently, it is a finite product of number fields admitting an automorphism $c_E$ that is of order $2$ on each factor and such that $\rho \circ c_E = c\circ \rho$ for all homomorphisms $\rho: E\to \BC$. The fixed algebra of $c_E$ is a product of the largest totally real subfields of the factors.

\bigskip
\noindent
Let $E$ be a CM-algebra. The set $\Hom_\BQ(E, \BC)$ of $\BQ$-homomorphisms $E \to \BC$ is a union in complex conjugate pairs $\{\phi, c\circ \phi\}$. A \emph{CM-type} on $E$ is the choice of one element from each such pair. More formally:
\begin{Definition}\label{DefAbVarCMType}
A \emph{CM-type} on a CM-algebra $E$ is a subset $\Phi \subset \Hom_\BQ(E, \BC)$ such that
\[
\Hom_\BQ(E, \BC) = \Phi \sqcup c \Phi \quad (\text{disjoint union}).
\]
Here $c \Phi : = \{ c \circ \phi\ |\  \phi \in \Phi\})$.

\end{Definition}

Let $X$ be an abelian variety over the complex numbers $\BC$. We have seen that $\QEnd_\BC (X)$ is a semi-simple $\BQ$-algebra which acts faithfully on the $(2 \dim X)$-dimensional $\BQ$-vector space $\Koh_1(X,\BQ)$. Therefore, by Proposition~\ref{PropSemi-Simple}
\[
2 \dim X \geq [\QEnd_\BC(X) : \BQ]_\red
\]
and when equality holds, $\QEnd_\BC (X)$ is a product of matrix algebras over fields.

\begin{Definition} An abelian variety $X$ over a subfield $K\subset\BC$ is said to have \emph{complex multiplication} (or be of \emph{CM-type}, or be a \emph{CM abelian variety}) over $K$ if
\[
2 \dim X = [\QEnd_K(X) : \BQ]_\red.
\]
By Proposition~\ref{PropEtSubalg} this definition is equivalent to the statement that $\QEnd_K (X)$ contains an \'etale $\BQ$-subalgebra of degree  $2 \dim X$ over $\BQ$. Indeed, if the latter holds then $2\dim X$ is less or equal to the degree of a maximal \'etale $\BQ$-subalgebra. By Proposition~\ref{PropEtSubalg} the latter degree equals $[\QEnd_K(X) : \BQ]_\red$. And the inequality $[\QEnd_K(X) : \BQ]_\red\le 2\dim X$ proves the claim.
\end{Definition}

Note that when $X$ is a CM abelian variety over a field $K\subset \BC$ then $\QEnd_K(X)\subset\QEnd_\BC(X)$ implies that this inclusion is an equality.

\begin{Remark}\label{RemCM}
 Let $X \approx_K \prod_i X_i^{n_i}$ be the decomposition of $X$ (up to isogeny) into a product of isotypic abelian varieties over $K$. Then $D_i = \QEnd_K (X_i)$  is a division algebra, and  $\QEnd_K (X) \cong  \prod  M_{n_i}(D_i)$ is the decomposition of $\QEnd_K (X)$  into a product of simple $\BQ$-algebras. From the above definition and Proposition~\ref{PropSemi-Simple} we see that $X$ has complex multiplication if and only if $D_i$ is a commutative field of degree $2 \dim X_i$ for all $i$. In particular, a simple abelian variety $X$ has complex multiplication if and only if $\QEnd_K (X)$ is a field of degree $2\dim X$ over $\BQ$, and an arbitrary abelian variety has complex multiplication if and only if each simple isogeny factor does.
\end{Remark}

Let $X$ be an abelian variety over $\BC$.  An endomorphism $\alpha$ of $X$ defines an endomorphism of the vector space $\Koh_1(X, \BQ)$ of dimension $2\dim X$ over $\BQ$. Therefore, the characteristic polynomial $P_\alpha$ of $\alpha$ is defined as
\[
P_\alpha(T)  : = \det\big (\alpha - T{\mathrel |\Koh_1(X, \BQ)}\big).
\]
It is monic, of degree $2\dim X$, and has coefficients in $\BZ$. More generally, we define the characteristic polynomial of any element of $\QEnd(X)$ by the same formula.

\noindent 
Consider an endomorphism $\alpha$ of an abelian variety $X$ over $\BC$, and write $X = \BC^g/\Lambda$ with $\Lambda = \Koh_1(X,  \BZ)$. If $\alpha$ is an isogeny, then $\alpha: \Lambda \to \Lambda$ is injective,
and it defines an isomorphism 
\[
\ker(\alpha) =\alpha^{-1}\Lambda/ \Lambda \isoto \Lambda/\alpha \Lambda.
\]
Therefore, for an isogeny $\alpha : X\to X$
\[
\deg \alpha =\#\ker(\alpha)= \bigl| \det\big (\alpha  {\mathrel |\Koh_1(X, \BQ)}\big)\bigr| =  |P_\alpha(0)|.
\]
More generally, for any integer $r$ we have $\deg (\alpha-r)  =  \bigl| \det\big (\alpha -r {\mathrel |\Koh_1(X, \BQ)}\big)\bigr|= |P_\alpha(r)|$; compare \cite[Chap~5 \S~12]{Cornell86}.

For the convenience of the reader we reproduce the proof from \cite{MilneCM} of the following results.

\begin{Lemma}[{\cite[Lemma~3.7]{MilneCM}}]\label{LemmaCM}
Let $F$ be a subfield of $\QEnd(X)$, where $X$ is an abelian variety over $\BC$. If $F$ has a real prime, then $[F:\BQ]$ divides $\dim X$.
\end{Lemma}
\begin{proof} First note that $ \Koh_1(X, \BQ)$ is a vector space over $F$ of dimension $m : = 2 \dim X/ [F:\BQ]$. So for any $\alpha \in \End(X)\cap F$, the characteristic polynomial $P_\alpha(T)$  is the $m$-th-power of the characteristic polynomial of $\alpha$ in $F/\BQ$. In particular,
\[
\Nm_{F/\BQ}(\alpha)^m = \deg \alpha \geq 0.
\]
However, if $F$ has a real prime, then from the weak approximation theorem $\alpha$ can be chosen to be large and negative at that prime and close to 1 at the remaining primes so that $\Nm_{F/\BQ}(\alpha)<0$.  This gives a contradiction unless $m$ is even. 
\end{proof}

For the next proposition recall the definition of a Rosati involution on $\QEnd_K(X)$. By \cite[Chapter~III, \S\,13, Corollary~5]{Mumford70} there exist \emph{polarizations} on $X$, that is, isogenies $\lambda\colon X \to  X^\vee=\Pic^0(X)$ which over $K^\alg$ are of the form $\lambda(x)=x^*\CL\otimes\CL^{-1}$ for an ample line bundle $\CL$ on $X_{K^\alg}$. Every polarization $\lambda$ has an inverse $\lambda^{-1}\in \QHom_K(X^\vee, X)$. The \emph{Rosati involution} on $\QEnd_K(X)$ corresponding to $\lambda$ is
\begin{equation}\label{EqRosati}
\alpha\mapsto \alpha^\dag = \lambda^{-1}\circ \alpha^\vee \circ \lambda
\end{equation}

\begin{Proposition}[{\cite[Proposition~3.6]{MilneCM}}]\label{PropCM}
\begin{enumerate}
\item \label{PropCM_A}
A simple abelian variety $X$ has complex multiplication if and only if $\QEnd(X)$ is a CM-field of degree $2 \dim X$ over $\BQ$.
\item \label{PropCM_B}
An isotypic abelian variety $X$ has complex multiplication if and only if $\QEnd(X)$ contains a  field  of degree $2 \dim X$ over $\BQ$ (which can be chosen to be a CM-field invariant under some Rosati involution).
\item \label{PropCM_C}
An abelian variety $ X$ has complex multiplication if and only if $\QEnd(X) $ contains an \'etale $\BQ$-algebra $E$ (which can be chosen to be a CM-algebra invariant under some Rosati involution)  of degree $2 \dim X$ over $\BQ$. In this case $\Koh_1(X, \BQ)$ is free over $E$ of rank 1.
\end{enumerate}
\end{Proposition}

\begin{proof}
\ref{PropCM_A} $\QEnd_K(X)$ is a field extension of $\BQ$ of degree $2 \dim X$ by Remark~\ref{RemCM}. We know that it is either totally real or CM because it is stable under the Rosati involutions \eqref{EqRosati}. Now Lemma~\ref{LemmaCM} shows that $\QEnd_K(X)$ is a CM-field.

\medskip\noindent
For \ref{PropCM_B} and \ref{PropCM_C} see \cite[Proposition~3.6]{MilneCM}.
\end{proof}

\begin{Definition}\label{DefCMTypeAbVar}
Let $X$ be an abelian variety with complex multiplication, so that $\QEnd(X) $  contains a CM-algebra $E$ for which $\Koh_1(X, \BQ)$ is a free $E$-module of rank 1, and let $\Phi$  be the set of homomorphisms $E \to  \BC$ occurring in the representation of $E$ on $ \T_0(X)$, i.e.,  $ \T_0(X)\cong  \bigoplus_{\phi \in \Phi} \BC_\phi$ where $\BC_\phi$ is the one-dimensional $\BC$-vector space on which $a \in E$ acts as $\phi(a)$. Then, because
\begin{equation}\label{EqHodgeAbVar}
\Koh_1(X, \BR) \cong \T_0(X) \oplus \ol {\T_0(X)},
\end{equation}
$\Phi $ is  a CM-type on $E$, and we say that, $X$ together with the injective homomorphism $E \to \QEnd(X)$ is \emph{of CM-type $(E, \Phi)$}.
\end{Definition}

Let $e$ be a basis vector for $\Koh_1(X, \BQ)$ as an $E$-module, and let $\Fa$ be the $\CO_E$-lattice in $E$ such that $\Fa e = \Koh_1(X, \BZ)$. Under the above  isomorphism
\begin{alignat}{3}\label{EqHodgeCMAbVar}
& \Koh_1(X, \BR) & \;\isoto\;  & \bigoplus_{\phi \in \Phi} \BC_\phi \oplus  \bigoplus_{\phi \in c\Phi} \BC_\phi,\\
& \quad e\otimes 1 &\;\longmapsto\; & (\cdots, e_\phi, \cdots;\cdots,e_{c\circ \phi}, \cdots)\nonumber
\end{alignat}
where each $e_\phi$ is a $\BC$-basis for $\BC_\phi$. The $e_\phi$ determine an isomorphism
\[
 \T_0(X)\cong  \bigoplus_{\phi \in \Phi} \BC_\phi.
\]

Next we state two important results on abelian varieties  with complex multiplication from   \cite{ShimuraTaniyama} and \cite{SerreTate} which we will need later.

\begin{Proposition} \cite[Prop~26 \S12.4]{ShimuraTaniyama} Let $X$ be an abelian variety over $K = K^\sep\subset\BC$ with complex multiplication, then there exists an abelian variety isogenous to $X$ defined over a  field which is a finite extension of $\BQ$.
\end{Proposition}

\begin{Theorem}\cite[Thm~6]{SerreTate}\label{ThmPotGoodRed}
Let $X$ be an abelian variety over a finite extension $K/\BQ$ with complex multiplication, then there exists a finite extension $L/K$ such that $X$ has good reduction at every place of $\CO_L$.
\end{Theorem}

\section{The Faltings Height of an Abelian Variety}\label{SectFaltingsHeight}
\setcounter{equation}{0}

We recall the definition of the \emph{Faltings height} of an abelian variety. It was introduced by Faltings in his proof of the Mordell Conjecture and the Tate Conjecture~\ref{ThmTateConjAbVar} for abelian varieties; see \cite{FaltingsEndlichkeit} or \cite[Chapter~2, \S~3]{Cornell86} for the English translation. Let $K$ be a number field, $\CO_K$ the ring of integers in $K$. We define a metrized line bundle on $\Spec(\CO_K)$ to be a projective $\CO_K$-module $P$ of  rank 1, together with norms $\|\,.\,\|_v$  on $P \otimes_{\CO_K}  K_v$ for all infinite places $v$ of $K$, where $K_v$ denotes the completion of $K$ at $v$. We define $\epsilon_v = 1$ or $2$ according to whether $K_v \cong \BR$ or $K_v \cong\BC$. The \emph{degree} of the metrized line bundle is 
defined as  
\[
\deg (P, \|\,.\, \|) = \log (\#(P/\CO_K\cdot x)) - \sum_{v|\infty} \epsilon_v \log \|x\|_v,
\]
where $x$ is a nonzero element of $P$ and the sum runs over all infinite places of $K$. The right-hand side is of course independent of $x$ because of the product formula~\eqref{EqProdFormulaAbVar}.

Let now $X$ be an abelian variety of dimension $g$ over  $K$, and let $\CX$ be the relative identity component of the N\'eron model of $X$ over $\CO_K$. Assume that $\CX$ is semi-abelian, i.e.\ a smooth algebraic group $q : \CX \to \Spec \CO_K$, whose   fibers are connected of dimension $g$, and are extensions of an abelian variety  by a torus. Let  $s : \CX \to \Spec \CO_K$  be the zero section.  Let $\omega_{X/\CO_K} = s^\ast(\Omega^g_{\CX/\CO_K}),\ \omega_{X/\CO_K}$ is a line bundle on $\CO_K$. The metrics at the  infinite places $v$ of $K$ are given by 
\[
\|\alpha\|_v^2 : = \frac{1}{(2\pi)^g}\int_{X_v(\BC)}|\alpha \wedge \bar \alpha| \qquad\text{for}\qquad \alpha\in \omega(X_v) = \Gamma(X_v, \Omega^g_{X_v})\,,
\]
where $X_v$ denotes the base change of $X$ under the map $K\to K_v$. Then Faltings~\cite[Chapter~2, \S~3]{Cornell86} defines a moduli-theoretic height as follows.

\begin{Definition}\label{DefFaltingsHeight}
The (\emph{stable}) \emph{Faltings height} $ht_{\rm Fal}^{\rm st}(X)$ of $X$ is defined as
\begin{equation}
ht_{\rm Fal}^{\rm st}(X)\;:=\; \frac{1}{[K:\BQ]}\deg (\omega_{X/\CO_K},\|\,.\,\|).
\end{equation}
\end{Definition}
It is easy to check that $ht_{\rm Fal}^{\rm st}(X)$ is invariant under extension of the ground  field. Since every abelian variety is potentially semi-stable by Grothendieck~\cite[Expos\'e~IX, Th\'eor\`eme~3.6]{SGA7}, the Faltings height is defined for every abelian variety over a number field. It measures the arithmetic complexity of the abelian variety and is ``not far'' from an actual height on the moduli space of principally polarized abelian varieties.

\section{Colmez's Conjecture on Periods of CM Abelian Varieties}\label{SectColmezConjAbVar}
\setcounter{equation}{0}

In \cite{ColmezPeriods} P.~Colmez considers product formulas for periods of abelian varieties in the following

\begin{Situation}\label{SitCMAbVar}
Let $X$ be an abelian variety defined over a number field $K$ with complex multiplication by the ring of integers $\CO_E$ in a CM-field $E$ and of CM-type $(E,\Phi)$. Let $H_E:=\Hom_\BQ(E,\BQ^\alg)$ be the set of all ring homomorphisms $E\into\BQ^\alg$ and assume that $K$ contains $\psi(E)$ for every $\psi\in H_E$. By Theorem~\ref{ThmPotGoodRed} we may assume moreover, that $K$ is a finite Galois extension of $\BQ$ and that $X$ has good reduction at every prime of $\CO_K$. For a fixed $\psi\in H_E$ let $\omega_\psi\in\Koh^1_{\dR}(X,K)$ be a non-zero cohomology class such that $b^*\omega_\psi=\psi(b)\cdot\omega_\psi$ for all $b\in E$. For every embedding $\eta\colon K\into\BQ^\alg$, let $X^\eta:=X\times_{\Spec K,\Spec\eta}\Spec K$ and $\omega_\psi^\eta\in\Koh^1_{\dR}(X^\eta,K)$ be deduced from $X$ and $\omega_\psi$ by base extension. Let $(u_\eta)_\eta\in\prod_{\eta\in H_K}\Koh_1(X^\eta(\BC),\BZ)$ be a family of cycles compatible with complex conjugation $c$, that is $u_{c\eta}=c(u_\eta)$. Let $v$ be a place of $\BQ$. 

If $v=\infty$ the de Rham isomorphism (Theorem~\ref{ThmDeRhamIsom}) between Betti and de Rham cohomology yields a pairing
\[
\langle\,.\;,\,.\,\rangle_\infty\colon \Koh_1(X^\eta(\BC),\BZ) \times \Koh^1_{\dR}(X^\eta,K) \longto \BC\,,\quad (u_\eta,\omega_\psi^\eta)\longmapsto\langle u_\eta,\omega_\psi^\eta\rangle_\infty\,.
\]
We define the complex absolute value $\bigl|\int_{u_\eta}\omega_\psi^\eta\bigr|_\infty:=|\langle u_\eta,\omega_\psi^\eta\rangle_\infty|_\infty\in\BR$. 

If $v$ corresponds to a prime number $p\in\BZ$, the comparison isomorphism $\Koh^1(X^\eta(\BC),\BZ)\otimes_\BZ \BZ_p\isoto\Koh^1_\et(X^\eta_{\BQ_p^\alg},\BZ_p)$ together with the comparison isomorphism from $p$-adic Hodge theory (Theorems~\ref{ThmCompIsomBEt} and \ref{ThmPAdicPeriodIso}) yield a pairing
\[
\langle\,.\;,\,.\,\rangle_p\colon \Koh_1(X^\eta(\BC),\BZ) \times \Koh^1_{\dR}(X^\eta,K) \longto \BB_{p,\dR}\,,\quad (u_\eta,\omega_\psi^\eta)\longmapsto\langle u_\eta,\omega_\psi^\eta\rangle_p\,.
\]
We define the absolute value $\bigl|\int_{u_\eta}\omega_\psi^\eta\bigr|_p:=|\langle u_\eta,\omega_\psi^\eta\rangle_p|_p:=p^{-v_p\left(\langle u_\eta,\omega_\psi^\eta\rangle_p\right)}\in\BR$, where the ``valuation'' $v_p$ on $\BB_{p,\dR}$ was defined in \eqref{EqNormBdR} in Notation~\ref{NotAinf}. 
\end{Situation}

Colmez~\cite{ColmezPeriods} now considers the product $\prod\limits_v\prod\limits_{\eta\in H_K}\bigl|\int_{u_\eta}\omega_\psi^\eta\bigr|_v$, or equivalently $\tfrac{1}{\#H_K}$ times its logarithm 
\begin{equation}\label{EqSumAbVar}
\TS\tfrac{1}{\#H_K}\sum\limits_v\sum\limits_{\eta\in H_K}\log\bigl|\int_{u_\eta}\omega_\psi^\eta\bigr|_v\;=\;\tfrac{1}{\#H_K}\sum\limits_{\eta\in H_K}\log\bigl|\langle u_\eta,\omega_\psi^\eta\rangle_\infty\bigr|_\infty-\tfrac{1}{\#H_K}\sum\limits_{\;v=v_p\ne\infty\;}\sum\limits_{\eta\in H_K}v_p\bigl(\langle u_\eta,\omega_\psi^\eta\rangle_p\bigr)\log p\,.
\end{equation}
The right sum over all $v=v_p$ does not converge. Namely, Colmez~\cite[Theorem~II.1.1]{ColmezPeriods} proves the following Theorem~\ref{ThmColmezLocal} below. To formulate the theorem we need a

\begin{Definition}\label{DefArtinMeasure}
In this definition we denote by $Q$ the function field from the introduction or the field $\BQ$, and by $Q_v$ the completion of $Q$ at a place $v\ne\infty$. The case $Q=\BQ$ is relevant in the present section, and the other case will be relevant in Section~\ref{SectColmezConjAMot}. For $F=Q$ or $F=Q_v$ let $F^\sep$ be the separable closure of $F$ in $F^\alg$ and let $\sG_F:=\Gal(F^\sep/F)$. Let $\CC(\sG_F,\BQ)$ be the $\BQ$-vector space of locally constant functions $a\colon\sG_F\to\BQ$ and let $\CC^0(\sG_F,\BQ)$ be the subspace of those functions which are constant on conjugacy classes, that is, which satisfy  $a(h^{-1}gh)=a(g)$ for all $g,h\in\sG_F$. Then the $\BC$-vector space $\CC^0(\sG_F,\BQ)\otimes_\BQ\BC$ is spanned by the traces of representations $\rho\colon\sG_F\to\GL_n(\BC)$ with open kernel for varying $n$ by \cite[\S\,2.5, Theorem~6]{SerreLinRep}. Via the fixed embedding $Q^\sep\into Q_v^\sep$ we consider the induced inclusion $\sG_{Q_v}\subset\sG_Q$ and morphism $\CC(\sG_Q,\BQ)\to\CC(\sG_{Q_v},\BQ)$. If $\chi$ is the trace of a representation $\rho\colon\sG_Q\to\GL_n(\BC)$ with open kernel we let $L(\chi,s):=\prod_{\text{all }v}L_v(\chi,s)$, respectively $L^\infty(\chi,s):=\prod_{v\ne\infty}L_v(\chi,s)$ be the Artin $L$-function of $\rho$ with, respectively without the factor at $\infty$. Note that the latter factor involves the Gamma-function if $Q=\BQ$. These $L$-functions only depend on $\chi$ and converge for all $s \in \BC$ with $\CR e(s)>1$; see \cite[Chapter~XII, \S\,2]{LangANT} for $Q=\BQ$ and \cite[pp.~126ff]{Rosen02} for the function field case. We also let $q_v$ be the cardinality of the residue field of $Q_v$ (this means $q_v=p$ if $Q=\BQ$ and $Q_v=\BQ_p$) and we set
\begin{align}
\label{EqZFct1}
Z^\infty(\chi,s)\;:=\; & \frac{\tfrac{d}{ds}L^\infty(\chi,s)}{L^\infty(\chi,s)}\;=\;-\sum_{v\ne\infty}Z_v(\chi,s)\log q_v\qquad\text{with}\\
\label{EqZFct2}
Z_v(\chi,s)\;:=\; & \frac{\tfrac{d}{ds}L_v(\chi,s)}{-L_v(\chi,s)\cdot\log q_v}\;=\;\frac{\tfrac{d}{dq_v^{-s}}L_v(\chi,s)}{q_v^s\cdot L_v(\chi,s)}\;. 
\end{align}
Moreover, we let $\Ff_\chi$ be the Artin conductor of $\chi$. If $Q=\BQ$, it is a positive integer $\Ff_\chi=\prod_p p^{\,\mu_{\Art,p}(\chi)}\in\BZ$, and if $Q$ is the function field of the curve $C$ it is an effective divisor $\Ff_\chi=\sum_v \mu_{\Art,v}(\chi)\cdot[v]$ on $C$; see \cite[Chapter~VI, \S\S\,2,3]{SerreLF}, where $\mu_{\Art,v}(\chi)$ is denoted $f(\chi,v)$. In particular, only finitely many values $\mu_{\Art,v}(\chi)$ are non-zero. We set 
\begin{alignat}{3}\label{EqMuArt}
& \mu^\infty_\Art(\chi) \; := \; && \log(\Ff_\chi)\;=\;\sum_{v\ne\infty}\mu_{\Art,v}(\chi) \log q_v\qquad &\text{if }Q=\BQ\,,\text{ respectively}\\
& \mu_\Art(\chi) \; := \; && \deg(\Ff_\chi)\log q \; := \; \sum_{\text{all }v}\mu_{\Art,v}(\chi)[\BF_v:\BF_q] \log q\; = \; & \sum_{\text{all }v}\mu_{\Art,v}(\chi) \log q_v\qquad \text{and} \nonumber \\
\label{EqMuArtFF} & \mu^\infty_\Art(\chi) \; := \; && \sum_{v\ne\infty}\mu_{\Art,v}(\chi) \log q_v\qquad &\text{if $Q$ is a function field}\,.
\end{alignat}
By linearity we extend $Z^\infty(\,.\,,s)$ and $\mu^\infty_\Art$ to all $a\in\CC^0(\sG_Q,\BQ)$ and $Z_v(\,.\,,s)$ and $\mu_{\Art,v}$ to all $a\in\CC^0(\sG_{Q_v},\BQ)$. The map $Z_v(\,.\,,s)$ takes values in $\BQ(q_v^{-s})$. 
\end{Definition}

For our CM-type $(E,\Phi)$ and for every $\psi\in H_E$ we define the functions
\begin{align}
a_{E,\psi,\Phi}\colon\;\sG_\BQ\to\BZ, & \quad g\mapsto \begin{cases} 
1 & \text{when }g\psi\in\Phi \\
0 & \text{when }g\psi\notin\Phi             
\end{cases}
\qquad\text{and} \nonumber \\[2mm]
a^0_{E,\psi,\Phi}\colon\;\sG_\BQ\to\BQ, & \quad g\mapsto \tfrac{1}{\#H_K}\TS\sum\limits_{\eta\in H_K}a_{E,\eta\psi,\eta\Phi}(g)\;=\;\dfrac{\#\{\eta\in H_K\colon\eta^{-1}g\eta\psi\in\Phi\}}{\#H_K}\label{EqAbVar:a0_E}
\end{align}
which factor through $\Gal(K/\BQ)$ by our assumption that $\psi(E)\subset K$ for all $\psi\in H_E$. In particular, $a_{E,\psi,\Phi}\in\CC(\sG_\BQ,\BQ)$ and $a^0_{E,\psi,\Phi}\in\CC^0(\sG_\BQ,\BQ)$ is independent of $K$. 

We also define integers $v_p(\omega_\psi^\eta)$ which are all zero except for finitely many. Let $K_p$ be the $p$-adic completion of $K\subset \BQ^\alg\subset \BQ_p^\alg\subset\BC_p$ and let $\CX^\eta$ be an abelian scheme over $\CO_{K_p}$ with $\CX^\eta\times_{\CO_{K_p}}\Spec K_p\cong X^\eta\times_K\Spec K_p$. Then there is an element $x\in K_p\mal$, unique up to multiplication by $\CO_{K_p}\mal$, such that $x^{-1}\omega_\psi^\eta$ is an $\CO_{K_p}$-generator of the free $\CO_{K_p}$-module of rank one 
\[
\Koh^{\eta\psi}(\CX^\eta,\CO_{K_p})\;:=\;\bigl\{\omega\in\Koh^1_\dR(\CX^\eta,\CO_{K_p})\colon b^*\omega=\eta\psi(b)\cdot\omega\es\forall\;b\in\CO_E\bigr\}\,,
\]
and we set
\begin{equation}\label{EqValuationOfOmegaPsi}
v_p(\omega_\psi^\eta)\;:=\;v_p(x)\;\in\;\BZ\,.
\end{equation}
This value does not depend on the choice of the model $\CX^\eta$ with good reduction, because all such models are isomorphic over $\CO_{K_p}$. Now Colmez~\cite[Theorem~II.1.1]{ColmezPeriods} computed the terms in \eqref{EqSumAbVar} as follows.

\begin{Theorem}\label{ThmColmezLocal}
If the image of $u_\eta$ in $\Koh_1(X^\eta(\BC),\BQ_p)=\Koh_{1,\et}(X^\eta_{\BQ_p^\alg},\BZ_p)$ is a generator of the $\CO_E\otimes_\BZ \BZ_p$-module $\Koh_{1,\et}(X^\eta_{\BQ_p^\alg},\BZ_p)=T_pX^\eta$, then
\begin{alignat}{2}
&& \TS\tfrac{1}{\#H_K}\sum\limits_{\eta\in H_K}v_p(\langle \omega_\psi^\eta,u_\eta\rangle_v) & \;=\; Z_p(a^0_{E,\psi,\Phi},1)-\mu_{\Art,p}(a^0_{E,\psi,\Phi})
+\tfrac{1}{\#H_K}\sum\limits_{\eta\in H_K}v_p(\omega_\psi^\eta)\,.
\end{alignat}
\end{Theorem}

Since $-\mu_{\Art,p}(a^0_{E,\psi,\Phi})+\tfrac{1}{\#H_K}\sum\limits_{\eta\in H_K}v_p(\omega_\psi^\eta)$ vanishes for all but finitely many primes $p$ and $\sum\limits_{p}Z_p(a^0_{E,\psi,\Phi},1)$ diverges, the sum \eqref{EqSumAbVar} diverges. Colmez~\cite[Convention~0]{ColmezPeriods} assigns to this divergent sum a value by the following

\begin{Convention}\label{ConventionAbVar} Let $(x_p)_{p\ne\infty}$ be a tuple of complex numbers indexed by the prime numbers $p$ in $\BZ$. We will give a sense to the (divergent) series $\Sigma \stackrel{?}{=} \sum_{p<\infty} x_p$ in the following situation. We suppose that there exists an element $a\in\CC^0(\sG_\BQ,\BQ)$ such that $x_p = - Z_p(a,1) \log p$  for all $p$ except for finitely many. Then we let $a^*\in\CC^0(\sG_\BQ,\BQ)$ be defined by $a^*(g):=a(g^{-1})$. We further assume that $Z^\infty(a^*,s)$ does not have a pole at $s=0$, and we define the limit of the series $\sum_{p<\infty} x_p$ as
\begin{equation}\label{EqConvention}
\Sigma\;:=\;-Z^\infty(a^*,0) - \mu^\infty_\Art(a) +\sum_{p<\infty} \bigl(x_p + Z_p(a,1)\log p\bigr)
\end{equation}
inspired by the functional equation relating $L(a,s)$ with $L(a^*,1-s)$ deprived of the terms at $\infty$.
\end{Convention}

\begin{Example}\label{ExGmPeriods}
The convention allows to prove the product formula for the multiplicative group $\BG_m:=\BG_{m,\BQ}=\Spec\BQ[x,x^{-1}]$. Namely, for the generator $\omega=\frac{dx}{x}$ of $\Koh^1_\dR(\BG_m,\BQ)=\BQ\cdot\omega$ and for the cycle $u\colon[0,1]\to\BG_m(\BC)=\BC\mal$ given by $u(s)=\exp(2\pi is)$ with $\Koh_1(\BG_m(\BC), \BZ)=\BZ\cdot u$, we have computed in Examples~\ref{ExGmBettiEt}, \ref{ExGmBettiDR} and \ref{ExGm}
\begin{alignat*}{5}
& \langle \omega,u\rangle_\infty && \;=\; 2\pi i && \qquad \text{and}\qquad && \log\bigl|\langle \omega,u\rangle_\infty\bigr|_\infty && \;=\; \log(2\pi)\,,\\
& \langle \omega,u\rangle_p && \;=\; t_p && \qquad \text{and}\qquad && \log\bigl|\langle \omega,u\rangle_p\bigr|_p && \;=\; \log|t_p|_p \;=\; -\tfrac{\log p}{p-1} \;=\; -Z_p(\BOne,1)\log p\,,
\end{alignat*}
where $\BOne(g)=1$ for every $g\in\sG_\BQ$. So Convention~\ref{ConventionAbVar} implies $\sum_{p<\infty}\log|\langle \omega,u\rangle_p|_p=-\frac{\zeta'_\BZ(0)}{\zeta_\BZ(0)}=-\log(2\pi)$ for the Riemann Zeta-function $\zeta_\BZ$ and $\sum_v\log|\langle \omega,u\rangle_v|_v=0$. Therefore $\prod\limits_v|\langle \omega,u\rangle_v|_v=1$.
\end{Example}

The Convention~\ref{ConventionAbVar} and the Theorem~\ref{ThmColmezLocal} allow us to give to the divergent sum \eqref{EqSumAbVar} a convergent interpretation. In order to remove the dependency on the chosen cycles $(u_\eta)_\eta\in\prod_{\eta\in H_K}\Koh_1(X^\eta(\BC),\BZ)$, Colmez considers the value
\begin{equation}\label{EqDoubleOmega}
\langle \omega_\psi^\eta,\omega_{c\psi}^\eta,u_\eta\rangle_v \;:=\;\left(t_v\cdot\frac{\langle \omega_\psi^\eta,u_\eta\rangle_v}{\langle \omega_{c\psi}^\eta,u_\eta\rangle_v}\right)^{\frac{1}{2}},
\end{equation}
where $t_\infty=2\pi i$ and for $v=v_p\ne\infty$, $t_v=t_p$ is the $p$-adic analog of $2\pi i$ from Examples~\ref{ExGm} and \ref{ExGmPeriods}. Note that $\Phi\sqcup c\Phi=H_E$ implies $a^0_{E,\psi,\Phi}+a^0_{E,c\psi,\Phi}=\BOne$, and hence $Z_p(a^0_{E,\psi,\Phi},1)+Z_p(a^0_{E,c\psi,\Phi},1)=Z_p(\BOne,1)$ and $\mu_{\Art,p}(a^0_{E,\psi,\Phi})+\mu_{\Art,p}(a^0_{E,c\psi,\Phi})=\mu_{\Art,p}(\BOne)=0$. Therefore, Theorem~\ref{ThmColmezLocal} implies
\begin{alignat*}{2}
&& \TS\tfrac{1}{\#H_K}\sum\limits_{\eta\in H_K}v_p(\langle \omega_\psi^\eta,\omega_{c\psi}^\eta,u_\eta\rangle_v) \;= & \; \frac{1}{2}\Bigl(Z_p(\BOne,1)+Z_p(a^0_{E,\psi,\Phi},1)-\mu_{\Art,p}(a^0_{E,\psi,\Phi})+\tfrac{1}{\#H_K}{\TS\sum\limits_{\eta\in H_K}}v_p(\omega_\psi^\eta)\\
&& & \qquad\qquad\quad-Z_p(a^0_{E,c\psi,\Phi},1)+\mu_{\Art,p}(a^0_{E,c\psi,\Phi})-\tfrac{1}{\#H_K}{\TS\sum\limits_{\eta\in H_K}}v_p(\omega_{c\psi}^\eta)\Bigr)\\
&& \;= & \; Z_p(a^0_{E,\psi,\Phi},1) -\mu_{\Art,p}(a^0_{E,\psi,\Phi}) + \frac{1}{2}\Bigl(\tfrac{1}{\#H_K}{\TS\sum\limits_{\eta\in H_K}}v_p(\omega_\psi^\eta)-v_p(\omega_{c\psi}^\eta)\Bigr)\,.
\end{alignat*}
Using Convention~\ref{ConventionAbVar} one thus obtains
\begin{alignat}{1}\label{EqSumAbVarDouble}
& \TS\tfrac{1}{\#H_K}\sum\limits_v\sum\limits_{\eta\in H_K}\log\bigl|\langle \omega_\psi^\eta,\omega_{c\psi}^\eta,u_\eta\rangle_v\bigr|_v\\
& \qquad\;=\;-Z^\infty((a^0_{E,\psi,\Phi})^*,0)+\tfrac{1}{\#H_K}\sum\limits_{\eta\in H_K}\Bigl(\log\bigl|\langle \omega_\psi^\eta,\omega_{c\psi}^\eta,u_\eta\rangle_\infty\bigr|_\infty-\frac{1}{2}\sum_{p<\infty}\bigl(v_p(\omega_\psi^\eta)-v_p(\omega_{c\psi}^\eta)\bigr)\log p\Bigr) \nonumber
\end{alignat}
which is independent of the chosen $u_\eta$. Colmez formulated the following

\begin{Conjecture}[{\cite[Conjecture~0.1]{ColmezPeriods}}]\label{ConjColmezAbVar} 
The sum \eqref{EqSumAbVarDouble} is zero, or equivalently the \emph{product formula} holds:
\[
\prod\limits_v\prod\limits_{\eta\in H_K}\bigl|\langle \omega_\psi^\eta,\omega_{c\psi}^\eta,u_\eta\rangle_v\bigr|_v=1\,.
\]
\end{Conjecture}

He then proved

\begin{Lemma}[{\cite[Lemme~II.2.9]{ColmezPeriods}}]\label{LemmaHeightEAbVar}  
In Situation~\ref{SitCMAbVar} the value
\begin{equation}\label{EqHeightE}
ht(E,\psi,\Phi)\;:=\;\tfrac{1}{\#H_K}\sum\limits_{\eta\in H_K}\Bigl(\log\bigl|\langle \omega_\psi^\eta,\omega_{c\psi}^\eta,u_\eta\rangle_\infty\bigr|_\infty-\frac{1}{2}\sum_{p<\infty}\bigl(v_p(\omega_\psi^\eta)-v_p(\omega_{c\psi}^\eta)\bigr)\log p\Bigr)
\end{equation}
only depends on $E,\psi$ and $\Phi$ and not on the choice of $X,\omega_\psi,u_\eta$ and $K$.
\end{Lemma}

Colmez also relates the product formula to the Faltings height, see Definition~\ref{DefFaltingsHeight}.

\begin{Theorem}[{\cite[Th\'eor\`eme~II.2.10(ii)]{ColmezPeriods}}]\label{ThmFaltHeight}
In Situation~\ref{SitCMAbVar} the Faltings height $ht_{\rm Fal}^{\rm st}(X)$ of $X$ satisfies
\begin{equation}\label{EqFaltHeight}
ht_{\rm Fal}^{\rm st}(X)\;=\;-\sum_{\psi\in\Phi}\Bigl(ht(E,\psi,\Phi)+\tfrac{1}{2}\mu^\infty_{\Art}(a^0_{E,\psi,\Phi}) \Bigr)
\end{equation}
\end{Theorem}

This immediately implies the following

\begin{Corollary}\label{CorHeightEAbVar}
In Situation~\ref{SitCMAbVar} the following assertions are equivalent.
\begin{enumerate}
\item \label{CorHeightE_A}
$ht(E,\psi,\Phi)=Z^\infty((a^0_{E,\psi,\Phi})^*,0)$.
\item \label{CorHeightE_B}
The product formula holds, that is, the expression~\eqref{EqSumAbVarDouble} is zero and $\prod\limits_v\prod\limits_{\eta\in H_K}\bigl|\langle \omega_\psi^\eta,\omega_{c\psi}^\eta,u_\eta\rangle_v\bigr|_v=1$.
\end{enumerate}
If \ref{CorHeightE_A} and \ref{CorHeightE_B} hold for all $\psi\in\Phi$ then $ht_{\rm Fal}^{\rm st}(X)\;=\;-\sum_{\psi\in\Phi} \bigl( Z^\infty((a^0_{E,\psi,\Phi})^*,0)+\tfrac{1}{2}\mu^\infty_{\Art}(a^0_{E,\psi,\Phi})\bigr)$. \qed
\end{Corollary}

Colmez~\cite[Conjecture~II.2.11]{ColmezPeriods} conjectures that statements \ref{CorHeightE_A} and \ref{CorHeightE_B} of Corollary~\ref{CorHeightEAbVar} hold for all $E,\psi,\Phi$. There are various partial results in this direction. The first is due to Colmez himself who was able to prove the following theorem up to a rational multiple of $\log 2$, which was then removed by Obus:

\begin{Theorem}[{\cite[Th\'eor\`eme~0.5]{ColmezPeriods}, \cite[Theorem~4.9]{Obus13}}] \label{ThmColmezAbelianE}
If $E$ is abelian over $\BQ$, then the product formula holds true for every $\psi,\Phi$, and hence
\begin{equation}\label{EqFaltingsHeightWithZ}
ht_{\rm Fal}^{\rm st}(X)\;=\;-\sum_{\psi\in\Phi} \Bigl( Z^\infty((a^0_{E,\psi,\Phi})^*,0)+\tfrac{1}{2}\mu^\infty_{\Art}(a^0_{E,\psi,\Phi}) \Bigr)\,.
\end{equation}
\end{Theorem}

There has been much further work and progress on Colmez's conjecture by many people. For example, Yang~\cite{Yang13} proved it for a large class of CM-fields $E$ of degree $[E:\BQ]=4$, including the first known cases when $E/\BQ$ is non-abelian. Let us also mention the most recent results by Andreatta, Goren, Howard, Madapusi Pera \cite{AGHM}, Yuan, Shou-Wu Zhang~\cite{YuanZhang15} and Barquero-Sanchez, Masri \cite{BSM}. 

\begin{Theorem}[{\cite[Theorem~A]{AGHM}, \cite[Theorem~1.1]{YuanZhang15}}] \label{ThmColmezAveragedE}
For every CM-field $E$ Colmez's conjecture holds true on average over all CM-types $\Phi$, that is
\[
\sum_\Phi\sum_{\psi\in\Phi} ht(E,\psi,\Phi)\;=\;\sum_\Phi\sum_{\psi\in\Phi} Z^\infty((a^0_{E,\psi,\Phi})^*,0)\,.
\]
\end{Theorem}

\begin{Remark}\label{RemChowla}
In \cite{YuanZhang15} the averaged Colmez conjecture (Theorem~\ref{ThmColmezAveragedE}) follows from a generalized Chowla-Selberg formula \cite[Theorem~1.7]{YuanZhang15}. Moreover, (generalized) Chowla-Selberg formulas are special cases of generalized Gross-Zagier formulas. In the case when $[E:\BQ]=2$, the generalized Chowla-Selberg formula \cite[Theorem~1.7]{YuanZhang15} is actually equivalent to the classical Lerch-Chowla-Selberg formula \eqref{EqLerch}, and it is also equivalent to the Colmez conjecture for $E$, by using a result of Faltings \cite[Theorem~7.b)]{Faltings84}. See \cite[\S\,0.6]{ColmezPeriods} and \cite[\S\,4.3]{Zhang20} for additional explanations.
\end{Remark}

As a consequence of Theorem~\ref{ThmColmezAveragedE}, Barquero-Sanchez and Masri~\cite[Theorem~1.1]{BSM} proved that for any fixed totally real number field $F$ of degree $[F:\BQ]\ge3$ there are infinitely many effective, ``positive density'' sets of CM extensions $E/F$ such that $E/\BQ$ is non-abelian and Colmez's conjecture \eqref{EqFaltingsHeightWithZ} on the Faltings height holds true for $E$ and any $\Phi$. Moreover, they prove

\begin{Theorem}[{\cite[Theorem~1.4]{BSM}}]
In Situation~\ref{SitCMAbVar} if the Galois closure of $E$ has degree $2^{\dim X}\cdot(\dim X)!$ over $\BQ$, then
\[
ht_{\rm Fal}^{\rm st}(X)\;=\;-\sum_{\psi\in\Phi}  Z^\infty((a^0_{E,\psi,\Phi})^*,0)-\tfrac{1}{2}\mu^\infty_{\Art}(a^0_{E,\psi,\Phi})\,.
\]
\end{Theorem}

As another consequence of Theorem~\ref{ThmColmezAveragedE} and of previous work by Edixhoven~\cite[Problem~14]{EdixhovenTexel2001}, Pila, Wilkie, Yafaev, Zannier and many others \cite{EdixhovenYafaev,PilaTsimerman_Ax-Lindemann, PilaWilkie06,PilaZannier}, Tsimerman~\cite{Tsimerman18} proved the \emph{Andr\'e-Oort-Conjecture} for the Siegel modular varieties:

\begin{Theorem}[{\cite[Theorem~1.3]{Tsimerman18}}] \label{ThmAO}
Let $\CA_g$ be the Siegel modular variety parameterizing principally polarized abelian varieties of dimension $g$ over $\BC$. Let $X\subset\CA_g$ be an irreducible closed subvariety which contains a Zariski dense subset of special points of $\CA_g$. Then $X$ is a special subvariety.
\end{Theorem}

The averaged Colmez conjecture (Theorem~\ref{ThmColmezAveragedE}) enters in this result by implying that the Galois orbit of a special point, that is a CM abelian variety, is large. This result and the Andr\'e-Oort-Conjecture were previously obtained in several cases conditionally under assumption of the generalized Riemann Hypothesis.

\part{Drinfeld Modules and $A$-motives}\label{Amot}
\section{Basic Definitions}\label{SectBasicDefAMot}
\setcounter{equation}{0}

Following the general philosophy about similarities between number fields and function fields, we now transfer the contents of Part~\ref{PartAbVar} to characteristic $p$. Here Drinfeld modules replace elliptic curves and $A$-motives replace abelian varieties. We follow the expositions in \cite[Ch.4]{Goss}, \cite[Ch.2]{Thakur04} and begin with the analog of Notation~\ref{NotNumberFld}

\begin{Notation}\label{NotFcnFld}
Let $\BF_q$ be a finite field with $q$ elements and  characteristic $p$.  Let $C$ be a smooth projective, geometrically irreducible curve over $\BF_q$ with  function field $Q = \BF_q(C)$. Let $\infty\in C$ be a fixed closed point and let $A:=\Gamma(C\setminus\{\infty\},\CO_C)$ be the $\BF_q$-algebra of those rational functions on $C$ which are regular outside $\infty$. Let $v_\infty$ be the valuation associated with the prime $\infty$.

By a \emph{place} of $C$ we mean a closed point $v\in C$. So either $v=\infty$ or $v$ is a maximal ideal of $A$. It defines a normalized valuation on $Q$ which we also denote by $v$, respectively by $v_\infty$ and which takes the value $v(z_v)=1$ on a uniformizing parameter $z_v\in Q$ at $v$. We now fix such a uniformizer $z_v$ at every $v$ and if $v=\infty$ we abbreviate $z_\infty$ to $z$. We denote the residue field of $v$ by $\BF_v$, its degree over $\BF_q$ by $d_v=[\BF_v:\BF_q]$ and its cardinality by $q_v:=\#\BF_v=q^{d_v}$. Thus if $a\in A\setminus\BF_q$ then $v_\infty(a)<0$, because $\BF_q$ is the field of constants in $Q$ as $C$ is geometrically irreducible, see \cite[IV$_2$, 4.3.1 and Proposition~4.5.9c)]{EGA_IV2}. The ring $A$ and its fraction field $Q$ play the role of $\BZ$ and $\BQ$ in the arithmetic of function fields.

 Let $Q_v$ be the completion of $Q$ with respect to the valuation $v$ and let $A_v\subset Q_v$ be the valuation ring of $v$. Then there is a canonical isomorphism $A_v \cong \BF_v\dbl z_v \dbr$. Let $Q_v^\alg$ be a fixed algebraic closure of $Q_v$ and let $\BC_v$ be the completion of $Q_v^\alg$ with respect to the canonical extension of $v$. We also use $v$ to denote this extension to  $Q_v^\alg$ and thus to $\BC_v$. However, we denote the image of $z_v$ in $\BC_v$ by $\zeta_v$ and abbreviate $\zeta_\infty$ to $\zeta$. Note that $\BC_v$ is algebraically closed. On $\BC_v$ and all its subrings we consider the normalized absolute value $|\,.\,|_v\colon\BC_v\to\BR_{\ge0}$ given by $|x|_v=q_v^{-v(x)}$. We let $\CO_{\BC_v}=\{x\in\BC_v\colon|x|_v\le1\}$ be the valuation ring of $\BC_v$. We also fix an algebraic closure $Q^\alg$ of $Q$ and an embedding $Q^\alg\into Q_v^\alg$ for every place $v$ of $Q$.

Let $K$ be a field extension of $\BF_q$ and fix an $\BF_q$-morphism  $ \gamma: A \to K$. We will call the pair  $(K, \gamma: A \to K)$ an \emph{$A$-field}. The prime ideal $\ker(\gamma)\subset A$ is called the \emph{$A$-characteristic} of $K$ and is denoted $\AChar(K,\gamma)$ or simply $\AChar(K)$. If $\AChar(K) = (0)$ we say $K$ has \emph{generic $A$-characteristic}. Then $\gamma$ is injective and $K$ is via $\gamma$ a field extension of $Q$. If $\AChar(K)=v\subset A$ is a maximal ideal, we say that $\AChar(K)$ is \emph{finite} and $K$ has \emph{finite $A$-characteristic $v$}. Then $K$ is via $\gamma$ a field extension of $\BF_v$.
\end{Notation}

Let $\BG_{a,K} = \Spec({K[X]})$ be the additive group scheme over $K$ and let $\tau \in \End_K({\BG_{a,K}})$  be the $q$-th power Frobenius endomorphism given by $\tau^*(X)=X^q$. Also every $b\in K$ induces an endomorphism $\psi_b\in\End_K(\BG_{a,K})$ given by $\psi_b^*(X)=bX$. These endomorphisms satisfy $\tau\circ\psi_b=\psi_{b^q}\circ\tau$. Then the ring $\End_{K, \BF_q}({\BG_{a,K}})$ of $\BF_q$-linear endomorphisms of group schemes over $K$ equals the non-commutative polynomial ring over $K$ in $\tau$:
\[
K\{\tau\}\;:=\;\bigl\{\,\DS\sum\limits_{i=0}^n b_i\tau^i\colon n\in\BN_0, b_i\in K\,\bigr\}\qquad\text{with}\quad \tau b=b^q\tau\,.
\]
For $\sum\limits_{i=0}^n b_i\tau^i\in K\{\tau\}$ we set $\deg_\tau\Bigl(\sum\limits_{i=0}^n b_i\tau^i\Bigr)=\max\{i\colon b_i\ne0\}$.

\begin{Definition}\label{Defdrinfeld} Let $(K, \gamma: A \to K)$ be an $A$-field.  A \emph{Drinfeld $A$-module} over $K$ is a pair $\ulG=(G,\phi)$ with $G\cong\BG_{a,K}$ and $\phi$ is an $\BF_q$-algebra  homomorphism $\phi: A \to \End_{K, \BF_q}(G)\cong K\{\tau\} , \ a \mapsto \phi_a,$  such that
\begin{enumerate} 
\item $\Lie(\phi_a) = \gamma(a)$ i.e. $ (a - \gamma(a))\cdot \Lie(G) = 0$ in $K$ for all $a\in A$.
\item There exists an $a \in A$ such that $\phi_a \in K\{\tau\}\setminus K$ i.e. $\phi_a \neq \gamma(a)\cdot \tau^0$ i.e. $\deg_{\tau}(\phi_a)>0$.
\end{enumerate}
Then there is an integer $r>0$ such that $\deg_\tau(\phi_a) = -rd_\infty v_\infty(a)$ for every $a\in A$, see \cite[\S\,4.5]{Goss}. It is called the \emph{rank} of $(G,\phi)$ and is denoted $\rk\ulG$ or $\rk\phi$. Also sometimes a Drinfeld $A$-module $\ulG=(G,\phi)$ is simply denoted by $\phi$. 

A \emph{morphism} between Drinfeld $A$-modules  $(G,\phi)$ and  $(G',\phi')$ over $K$ is a homomorphism $f : G \to G'$ of group schemes such that $\phi'_a\circ f = f\circ \phi_a$ for every $a \in A$. We denote the set of morphisms between $\ulG$ and $\ulG'$ by $\Hom_K(\ulG,\ulG')$ and we write $\End_K(\ulG):=\Hom_K(\ulG,\ulG)$.

In particular, for every $c\in A$ the commutation $\phi_a\circ\phi_c=\phi_{ac}=\phi_{ca}=\phi_c\circ\phi_a$ implies that $\phi_c\in\End_K(\ulG)$. Thus $\End_K(\ulG)$ is an $A$-algebra via $A\to\End_K(\ulG)$, $c\mapsto\phi_c$ and $\Hom_K(\ulG,\ulG')$ is an $A$-module. So we may also define $\QHom_K(\ulG,\ulG'):=\Hom_K(\ulG,\ulG')\otimes_AQ$ and write $\QEnd_K(\ulG):=\QHom_K(\ulG,\ulG)=\End_K(\ulG)\otimes_AQ$. 
\end{Definition}

\begin{Remark}\label{RemAndersonAModule}
Drinfeld $A$-modules possess higher dimensional generalizations, which are called \emph{abelian Anderson $A$-modules}, see \cite[Definition~1.2]{HA17}. They were originally defined by Anderson~\cite{Anderson} for $A=\BF_q[t]$ under the name \emph{abelian $t$-modules}. These are group schemes which carry an action of the ring $A$ subject to certain conditions. Abelian Anderson $A$-modules are the function field analogs of abelian varieties. Although Anderson worked over a field, abelian Anderson $A$-modules also exist naturally over arbitrary $A$-algebras as base rings. They possess an (anti-)equivalent description by semi-linear algebra objects called \emph{$A$-motives}, which we will define next. Through the work of Drinfeld and Anderson it was realized very early on that a Drinfeld module or abelian Anderson $A$-module over a field is completely described by its $A$-motive. The same is true over an arbitrary $A$-algebra $R$, as is shown for example in \cite{HA17}. So in a way the situation in function field arithmetic is much better than in the arithmetic of abelian varieties (which only have a local $p$-adic semi-linear algebra description via the Dieudonn\'e module of the associated $p$-divisible group, see Remark~\ref{RemDieudonneMod}): the $A$-motive is a ``global'' Dieudonn\'e module which integrates the ``local'' Dieudonn\'e modules for every prime in a single object. We will return to this in Section~\ref{SectLocalShtukas} and Proposition~\ref{PropLocGlobDieudonne}.
\end{Remark}

Before we define $A$-motives we have to fix some

\begin{Notation}
For an $A$-field $(K,\gamma)$ we write $A_K:=A\otimes_{\BF_q}K$ and set $\CJ:=(a\otimes1-1\otimes\gamma(a)\colon a\in A)\subset A_K$. We consider the endomorphism $\sigma^*:=\id_A\otimes\Frob_{q,K}$ of $A_K$, where $\Frob_{q,K}(b)=b^q$ for $b\in K$. For an $A_K$-module $M$ we set $\sigma^*M:=M\otimes_{A_K,\sigma^*}A_K$ and we write $\sigma_M^*:M\to\sigma^*M,\,m\mapsto m\otimes1$ for the natural $\sigma^*$-semilinear map. For a homomorphism $f\colon M\to N$ of $A_K$-modules we set $\sigma^*f:=f\otimes\id_{A_K}\colon\sigma^*M\to\sigma^*N$. Note that the endomorphism $\sigma^*$ corresponds to a morphism of schemes
\begin{equation}\label{EqSigma}
\sigma:=\id_C\times\Spec(\Frob_{q,K})\colon \;C_K:=C\times_{\BF_q}\Spec K \to C_K
\end{equation}
which is the identity on points and on sections of $\CO_C$ and the $q$-Frobenius on $K$. It satisfies $\sigma|_{\Spec A_K}=\Spec(\sigma^*)\colon\Spec A_K \to\Spec A_K$.
\end{Notation}

\begin{Example}\label{ExAMotOfDriMod}
Before we give the general definition of $A$-motives, we define the $A$-motive associated to a Drinfeld $A$-module $\ulG=(G,\phi)$ over $K$ as in \cite{Anderson}. Namely, we set 
\[
 M:= M( \ulG):= M( \phi): = \Hom_{K, \BF_q}(G, \BG_{a,K}),
\]
where $\Hom_{K, \BF_q}(-,-)$ is the group of $\BF_q$-linear homomorphisms of group schemes over $K$. Every choice of an isomorphism $G\cong\BG_{a,K}$ induces an isomorphism $M(\ulG)\cong K\{\tau\}$. We make  $M( \ulG)$ into an $A_K\{\tau\} = A \otimes_{\BF_q}K \{\tau\}$ module in the fashion given below:
\begin{alignat}{2}
&(a,m)\mapsto m\circ \phi_a &&\qquad\text{for}\quad\ m \in M, \ a \in A;\\
&(b,m)\mapsto \psi_b\circ m  &&\qquad\text{for}\quad \ m \in M, \ b \in K; \\
\label{EqTauM} &(\tau,m)\mapsto \tau m = \Frob_{q, \BG_a}\circ\,  m &&\qquad\text{for}\quad \BG_{a,K} \to \BG_{a,K}: \ m \in M.
\end{alignat}
Since the actions of $a \in A$ and of $b\in K$ commute, i.e. $a( b\cdot m)= \psi_b\circ m\circ\phi_a = b(a\cdot m)$, this makes $M$ into a module over $A_K: = A\otimes_{\BF_q}K$. It is not difficult to see that $M$ is a locally free $A_K$-module of rank $r: = \rk \ulG$, see \cite[Lemma~5.4.1]{Goss}. Now for $a\in A$ and $b\in K$ we have
\[
\tau\circ(a\otimes b)(m)= \tau\circ (\psi_b\circ m\circ\phi_a) = \psi_{b^q}\circ \tau\circ m\circ \phi_a = (a\otimes b^q)\circ \tau m\,.
\]
Since  the action of $\tau$ is not $A_K$-linear but $\sigma^*$-semi linear, it induces an $A_K$-linear map $\tau_M: \sigma^\ast M\to M$ defined by $\tau_M(m\otimes 1) = \tau m$. Sending $m\in M:=\Hom_{K, \BF_q}(G, \BG_{a,K})$ to $\Lie m\in\Hom_K(\Lie G,\Lie \BG_{a,K})=\Hom_K(\Lie G,K)$ defines a canonical isomorphism of $A_K$-modules
\begin{equation}\label{EqLieG}
\coker \tau_M \;=\;M/\tau_M(\sigma^*M)\;\isoto\;\Hom_K(\Lie G,K),\quad m\mod\tau_M(\sigma^*M)\;\longmapsto\; \Lie m\,,
\end{equation}
where $a\in A$ acts on $\Lie E$ via $\Lie\phi_a$; see \cite[Lemma~1.3.4]{Anderson}. This implies $\dim_K(\coker \tau_M) =1$, which can also be seen directly from $M\cong K\{\tau\}$ and $\tau_M(\sigma^*M)\cong K\{\tau\}\cdot\tau$. 
\end{Example}

The above construction motivates the  definition of $A$-motives:

\begin{Definition}\label{DefAMotive}
An  \emph{(effective) $A$-motive of rank $r$ and dimension $d$ over $K$} is a pair $\ulM=(M,\tau_M)$ consisting of a locally free $A_K$-module $M$ of rank $r$ and an $A_K$-homomorphism $\tau_M\colon\sigma^*M \to M$ such that
\begin{enumerate} 
\item $\dim_K(\coker \tau_M)=d$.
\item $(a-\gamma(a))^d \cdot \coker \tau_M =0$ for all $a\in A$.
\end{enumerate}
 We write $\rk\ulM:=r$ and $\dim\ulM:=d$. 

A \emph{morphism} between $A$-motives $f\colon(M,\tau_M)\to(N,\tau_N)$ over $K$ is an $A_K$-homomorphism $f\colon M\to N$ with $f\circ\tau_M=\tau_N\circ\sigma^*f$. We denote the set of morphisms between $\ulM$ and $\ulN$ by $\Hom_K(\ulM,\ulN)$ and we write $\End_K(\ulM):=\Hom_K(\ulM,\ulM)$. Since $\sigma^*(a)=a$ for all $a\in A$ and $\tau_M$ is $A_K$-linear, we have $a\cdot\id_M\in\End_K(\ulM)$. Thus $\End_K(\ulM)$ is an $A$-algebra via $A\to\End_K(\ulM)$, $a\mapsto a\cdot\id_M$ and $\Hom_K(\ulM,\ulN)$ is an $A$-module. So we may also define $\QHom_K(\ulM,\ulN):=\Hom_K(\ulM,\ulN)\otimes_AQ$ and write $\QEnd_K(\ulM):=\QHom_K(\ulM,\ulM)=\End_K(\ulM)\otimes_AQ$. 
\end{Definition}

On the relation with Drinfeld $A$-modules we have the following theorem, see \cite{Anderson} or \cite[\S 5.4]{Goss}.

\begin{Theorem}\label{ThmEquivDriMod}
The contravariant functor $\ulG \mapsto \ulM(\ulG)$  from Drinfeld $A$-modules to $A$-motives  over $K$ is fully faithful. Its essential image consists of all $\ulM=(M, \tau_M)$ such that $M$ is free over $K\{\tau\}$ of rank $1$. The latter implies that $\dim\ulM=1$.
\end{Theorem}

In this sense we view $A$-motives as higher dimensional generalizations of Drinfeld $A$-modules. As an illustration of the claim that $A$-motives (and abelian Anderson $A$-modules) play the role of abelian varieties, see for example \cite{BoHa09} where the theory of $A$-motives over finite fields is developed in analogy with \cite{TateEndoms}.

\begin{Example} \label{CarMod} Let $C = \BP^1_{\BF_q}$, and set  $ A = \BF_q[t]$. Then $A_K = K[t]$. Let $K = \BF_q(\theta)$ be the rational function field in the variable $\theta$ and let  $\gamma : A\to K$ be given by $\gamma(t) = \theta$. The \emph{Carlitz module} over $K$ is given by $\ulG=(\BG_{a,K},\phi)$ with $\phi\colon \BF_q[t]\to K\{\tau\}$ defined by $\phi_t = \theta+\tau$. The $A$-motive associated with the Carlitz module is given by $\ulCC = (\CC = K[t], \tau_\CC = t-\theta)$ and is called the \emph{Carlitz motive}. Both $\ulG$ and $\ulCC$ have rank $1$. As we will see in Examples~\ref{ExCarlitzTorsion} and \ref{ExCarlitzVAdic} below, the Carlitz module is the function field analog of the multiplicative group $\BG_{m,\BQ}$ from Example~\ref{ExGmBettiEt}.
\end{Example}

\section{Isogenies and Semi-simple $A$-Motives}\label{SectEndomAMot}
\setcounter{equation}{0}

If we define the \emph{rank} of an abelian variety $X$ ad $\rk X:=2\cdot\dim X$, see Remark~\ref{RemRkAbVar} below, the analog of Theorem~\ref{ThmHomAbVar} is the following

\begin{Theorem}\label{ThmHomAMot}
For two $A$-motives $\ulM$ and $\ulN$ over an $A$-field $K$ the $A$-module $\Hom_k(\ulM,\ulN)$ is finite projective of rank $\le(\rk\ulM)\cdot(\rk\ulN)$. The same is true for Drinfeld $A$-modules over $K$.
\end{Theorem}

\begin{proof}
For $A$-motives this was proved by Anderson~\cite[Corollary~1.7.2]{Anderson} and for Drinfeld $A$-modules it can be found in \cite[Theorem~4.7.8]{Goss}.
\end{proof}

\begin{Definition} Let $\ulG=(G,\phi)$ and $\ulG'=(G',\phi')$ be two Drinfeld A-modules over $K$. A non zero morphism $f \in \Hom_K(\ulG, \ulG')$ is called an \emph{isogeny}.  If there is an isogeny $f: \ulG \to \ulG'$, then $\ulG$ and $\ulG'$ are \emph{isogenous}.
\end{Definition}
From \cite[4.7.13]{Goss}, we know that if there is an isogeny $f: \ulG \to \ulG'$, then there exists a some nonzero $a \in A$ and an isogeny $\hat f : \ulG' \to \ulG$ such that 
\[
\hat f f = \phi_a\quad\text{and}\quad f \hat f = \phi'_a .
\]
In particular, if  $0\ne f\in\End_K(\ulG)$, then $f$  is invertible in $ \QEnd(\ulG): =\End_K(\ulG) \otimes_A Q$, so $\QEnd(\ulG)$ is  a finite dimensional division algebra over $Q$.

\begin{Definition}Let $\ulM$ and $\ulN$ be two $A$-motives over $K$.  A morphism $f \in \Hom_K(\ulM, \ulN)$ is called an \emph{isogeny} if $f$ is injective and $\coker f$ is a finite dimensional $K$-vector space. If there exists an isogeny $f\in\Hom_K(\ulM,\ulN)$  then $\ulM$ and $\ulN$  are said to be \emph{isogenous} over $K$ and we write $\ulM\approx_K\ulN$. This defines an equivalence relation by Remark~\ref{RemIsogDriMod}(d) below. 
\end{Definition}

\begin{Remark}\label{RemIsogDriMod}
\begin{enumerate}
\item \label{RemIsogDriMod_A}
Two Drinfeld $A$-modules are isogenous if and only if their associated $A$-motives are isogenous, see \cite[Theorem~5.9 and Proposition~5.4]{HA17}.
\item \label{RemIsogDriMod_B}
If two $A$-motives $\ulM$ and $\ulN$ are isogenous then $\rk \ulM = \rk \ulN$ and $\dim \ulM = \dim \ulN$, see \cite[Proposition~5.8]{HA17}. 
\item \label{RemIsogDriMod_C}
Conversely, let $f\colon\ulM\to\ulN$ be a morphism of $A$-motives with $\rk\ulM=\rk\ulN$. Then $f$ is injective if and only if $\coker f$ is a finite dimensional $K$-vector space, and in this case $f$ is an isogeny. Indeed, since $M$ is locally free over $A_K$, it is contained in $M\otimes_{A_K}\Quot(A_K)$ where $\Quot(A_K)$ denotes the fraction field of $A_K$. Since $\rk\ulM=\rk\ulN$ the injectivity of $f$ is equivalent to $f$ inducing an isomorphism $M\otimes_{A_K}\Quot(A_K)\to N\otimes_{A_K}\Quot(A_K)$, and this in turn is equivalent to $\coker f$ being torsion, and hence finite.
\item \label{RemIsogDriMod_D}
If $f\colon\ulM\to\ulN$ is an isogeny between $A$-motives, then there exists non-canonically an isogeny $\hat f\colon\ulN\to\ulM$ and a non-zero element $a\in A$ with $\hat ff=a\cdot\id_\ulM$ and $f\hat f=a\cdot\id_\ulN$ by \cite[Corollary~5.15]{HA17}
\item \label{RemIsogDriMod_E}
Let $\ulM$ and $\ulN$ be $A$-motives over $K$. If $\ulM$ and $\ulN$ are isogenous over $K$ via an isogeny $f$, then
\[
\QEnd_K (\ulM) \cong \QHom_K(\ulM,\ulN) \cong \QEnd_K(\ulN),\ \ h\mapsto f\circ h \mapsto f\circ h\circ f^{-1}.
\]
More precisely, $\QHom_K(\ulM,\ulN)$ is a free right $\QEnd_K(\ulM)$-module of rank $1$ and a free left $\QEnd_K(\ulN)$-module of rank $1$. If $\ulM$ and $\ulN$ are not isogenous then $\QHom_K(\ulM,\ulN)=(0)$.
\end{enumerate}
\end{Remark}

\begin{Definition} Let $\ulM$ be an  $A$-motive over $K$. 
\begin{enumerate}
\item An \emph{$A$-factor-motive over $K$} of $\ulM$ is an $A$-motive $\ulM'$ together with a surjective morphism $\ulM \onto \ulM'$ of $A$-motives over $K$.
\item $\ulM$ is called \emph{simple over $K$} if $\ulM$ is non trivial and $\ulM$ has no $A$-factor-motives over $K$ other than $(0)$ and $\ulM$. 
\item $\ulM$ is called \emph{semi-simple over $K$} if $\ulM$ is isogenous to a direct sum of simple $A$-motives over $K$, i.e.\ $\ulM\approx_K \oplus_{i}\ulM_i$ with $\ulM_i$ simple.
\end{enumerate}
\end{Definition}

\begin{Remark}\label{RemEndOfSimpleAMot}
(a) In comparison to the analogous Definition~\ref{DefSimpleAbVar} for abelian varieties, $A$-motives behave dually. This is due to the fact that the functor from Drinfeld $A$-modules to $A$-motives is contravariant.

\medskip\noindent
(b) For any Drinfeld $A$-module $\phi$ over $K$ the $A$-motive $\ulM(\phi)$ is simple by \cite[Corollary~7.5]{BoHa11}.

\medskip\noindent
(c) But in contrast to abelian varieties (Remark~\ref{RemPoincareWeil}) not every $A$-motive is semi-simple up to isogeny. This was observed in \cite[Examples~6.1 and 6.13]{BoHa09}.

\medskip\noindent
(d) Let $\ulM$ and $\ulN$ be two $A$-motives over $K$ of same rank and let $\ulM$ be simple over $K$. Then every non-zero morphism $f\in\Hom_K(\ulM,\ulN)$ is an isogeny. Namely, the image of $f$ is a non-zero $A$-factor-motive of $\ulM$, and hence isomorphic to $\ulM$ via $f$, because $\ulM$ is simple. So $f$ is injective and hence an isogeny by Remark~\ref{RemIsogDriMod}\ref{RemIsogDriMod_C}.  

In particular, if $\ulM$ is simple over $K$ then every non-zero endomorphism $0\ne f\in\End_K(\ulM)$ is an isogeny and therefore invertible in $\QEnd_K(\ulM)$ by Remark~\ref{RemIsogDriMod}\ref{RemIsogDriMod_D}. This implies that $\QEnd_K(\ulM)$ is a division algebra over $Q$. 

Moreover, if  $\ulM$ is semi-simple over $K$ with decomposition $\ulM \approx_K \ulM_1 \oplus \cdots \oplus \ulM_n$ up to isogeny into simple $A$-motives $\ulM_i$ over $K$, then   $\QEnd_K(\ulM)$ decomposes into a finite direct product of full matrix algebras over the division algebras  $\QEnd_K(\ulM_i)$ over $Q$, compare Remark~\ref{RemEndOfSimpleAbVar}.
\end{Remark}

\section{Analytic Theory of Drinfeld Modules}\label{SectAnalytTh}
\setcounter{equation}{0}

In this section we consider Drinfeld $A$-modules over $\BC_\infty$, which is an $A$-field via the natural inclusion $A\subset Q\subset Q_\infty\subset \BC_\infty$ denoted by $\gamma$.

If $\ulG=(\BG_{a,\BC_\infty},\phi)$ with $\phi\colon A\to\BC_\infty\{\tau\}$ is a Drinfeld $A$-module over $\BC_\infty$ then there is a uniquely determined power series $\exp_\ulG(z)=\sum_{i=0}^\infty e_i z^{q^i}$ with $e_i\in\BC_\infty$, $e_0=1$ satisfying 
\[
\phi_a( \exp_\ulG(z)) = \exp_\ulG (\gamma(a)\cdot z)
\]
for all $a \in A$, see \cite[4.6.7]{Goss}. It is called the \emph{exponential function of $\ulG$}. The power series $\exp_\ulG$ converges for every $z\in\BC_\infty$ and its kernel $\Lambda(\ulG)$ is an \emph{$A$-lattice} in $\BC_\infty$ (that is, a finitely generated projective, discrete $A$-submodule) of the same rank as the Drinfeld $A$-module $\ulG$. Note that $\BC_\infty$ is infinite dimensional over $Q_\infty$ and therefore contains $A$-lattices of arbitrarily high rank. 

\medskip

Conversely, let $\Lambda\subset\BC_\infty$ be an $A$-lattice of rank $r$. Then the function
\begin{equation}\label{exp}
\exp_{\Lambda}(z) = z \prod_{0\ne\lambda \in \Lambda} (1-\tfrac{z}{\lambda})
\end{equation}
converges for every $z\in\BC_\infty$ and can be written as an everywhere convergent power series in $z$. Moreover $\exp_\Lambda\colon\BC_\infty\to\BC_\infty$ is a surjective $\BF_q$-linear map whose zeroes are simple and located at $\Lambda$. For more details see \cite[\S 4.2]{Goss}. For $a \in A\setminus\{0\}$ we can now define the polynomial 
\begin{equation}\label{conddefdri}
\phi_a^{\Lambda}(x)\;:=\;\gamma(a)\cdot x\cdot\prod_{0\ne\lambda\in \gamma(a)^{-1}\Lambda/\Lambda}\bigl(1-\tfrac{x}{\exp_\Lambda(\lambda)}\bigr)\;\in\;\BC_\infty[x]\,.
\end{equation}
It satisfies 
\begin{equation}\label{EqFunctEqPhi}
\exp_\Lambda(\gamma(a)\cdot z) =  \phi_a^{\Lambda}(\exp_\Lambda(z))
\end{equation}
and makes the following diagram with exact rows commutative
\begin{align}\label{commdefdri}
\xymatrix @C+1pc {
0 \ar[r]& \Lambda \ar[d] \ar[r] &\BC_\infty \ar[r]^{\exp_\Lambda}\ar[d]_{\gamma(a)} & \BC_\infty \ar[d]^{\phi_a^{\Lambda}}\ar[r] &0\\
0 \ar[r]& \Lambda \ar[r]& \BC_\infty \ar[r]^{\exp_\Lambda} & \BC_\infty\ar[r] & 0
}.
\end{align}
It is easy to see that
\begin{enumerate}
\item  $\phi_a^\Lambda(x)$ is an $\BF_q$-linear polynomial, i.e. $\phi_a^{\Lambda} \in \BC_\infty\{\tau\}$, of $\tau$-degree $\deg_\tau(\phi_a^\Lambda) = -rd_\infty v_\infty(a)$;
\item $\phi^\Lambda: a \mapsto \phi_a^{\Lambda} $ defines a ring homomorphism $\phi^\Lambda\colon A \to \BC_\infty\{\tau\}$.
\end{enumerate}
The additive group $\BC_\infty$, considered as the quotient $\BC_\infty/\Lambda$ via $\exp_\Lambda$, thus carries a new structure as an $A$-module given by $z\mapsto \phi_a^{\Lambda}(z)$ for $a\in A$. Therefore, for every $A$-lattice $\Lambda\subset\BC_\infty$ of rank $r$ we get a Drinfeld $A$-module $\ulG^\Lambda:=(\BG_{a,\BC_\infty}, \phi^\Lambda)$ of rank $r$ over $\BC_\infty$.

\begin{Definition}
Let $\Lambda_1,\  \Lambda_2$  be two $A$-lattices of the same rank. A morphism from $\Lambda_1 \to  \Lambda_2$ is an element $c \in \BC_\infty$, with $c\Lambda_1\subseteq  \Lambda_2$. If the ranks of $\Lambda_1$ and  $ \Lambda_2$ are different, then we only allow $0 \in \BC_\infty$ to be a morphism. 
\end{Definition}

\begin{Theorem}[{\cite[Proposition~3.1]{Drinfeld}}] \label{ThmDriModLattice}
The functors $\ulG\mapsto\Lambda(\ulG)$ and $\Lambda\mapsto\ulG^\Lambda$ give an equivalence of categories between the category of Drinfeld $A$-modules over $\BC_\infty$ and the category of $A$-lattices in $\BC_\infty$.
\end{Theorem}

\begin{Corollary}\label{CorEndDriMod}
If $\ulG$ is a Drinfeld $A$-module over a field $K$ of generic $A$-characteristic, then $\QEnd_K(\ulG)$ is a commutative field whose degree over $Q$ divides $\rk\ulG$.
\end{Corollary}

\begin{proof}
Since $\ulG$ and all elements of $\QEnd_K(\ulG)$ are defined over a finitely generated subfield $K_0$ of $K$, we can choose a $Q$-embedding $K_0\into \BC_\infty$ and it suffices to prove the corollary when $K=\BC_\infty$. In this case $\ulG\cong \ulG^\Lambda$ for an $A$-lattice $\Lambda\subset \BC_\infty$ of rank equal to $\rk\ulG$. By Theorem~\ref{ThmDriModLattice} we have isomorphisms $\End_K(\ulG)\isoto\{c\in\BC_\infty\colon c\Lambda\subset\Lambda\}$, $f\mapsto\Lie(f)$ and $\QEnd_K(\ulG)\isoto\{c\in\BC_\infty\colon c(Q\cdot\Lambda)\subset Q\cdot\Lambda\}$. In particular $\QEnd_K(\ulG)\subset\BC_\infty$ is a commutative field. Since $Q\cdot\Lambda\subset\BC_\infty$ is a $Q$-vector space of dimension $\rk\ulG$ and also a $\QEnd_K(\ulG)$-vector space, the formula $\rk\ulG=\dim_Q(Q\cdot\Lambda)=[\QEnd_K(\ulG):Q]\cdot\dim_{\QEnd_K(\ulG)}(Q\cdot\Lambda)$ tells us that $[\QEnd_K(\ulG):Q]$ divides $\rk\ulG$.
\end{proof}

We regard Drinfeld $A$-modules and particularly those of rank two as analogs of elliptic curves, where the functional equation \eqref{EqFunctEqPhi} for $\exp_\Lambda(z)$ corresponds to the group law derived from \eqref{ElEq2}. The point is that \eqref{ElEq2} defines a $\BZ$-module structure on the elliptic curve $\BC/\Lambda \isoto E_\Lambda(\BC)$,  while \eqref{conddefdri} and \eqref{EqFunctEqPhi} define the above $A$-module structure on the additive group scheme $\BG_{a_K}$.

\begin{Definition}\label{DefDriModBetti}
Let $\ulG$ be a Drinfeld $A$-module of rank $r$ over $\BC_\infty$. The \emph{Betti (co-)homology realization}  of $\ulG$ is defined by
\[
\Koh^1_\Betti(\ulG, R): = \Lambda (\ulG)\otimes_A R \qquad 
\text{and} \qquad \Koh_{1,\Betti}(\ulG, R): = \Hom_A(\Lambda (\ulG), R)
\]
for any $A$-algebra $R$. Both are free $R$-modules of rank $r$.
\end{Definition}

\section{Torsion Points and $v$-adic Cohomology of Drinfeld Modules}\label{SectTorsDriMod}
\setcounter{equation}{0}

\begin{Definition}
Let $\ulG=(G,\phi)$ be a Drinfeld  $A$-module over an $A$-field $K$ and let $G(K^\alg)$ be the set of $K^\alg$-valued points of $G$. For an element $ a \in A$, we set
\[
\ulG[a](K^\alg):=\phi[a](K^\alg): = \{ P \in G(K^\alg)\ |\ \phi_a(P) = 0 \},
\]
and we call $\ulG[a](K^\alg)$ the \emph{module of $a$-torsion points} of $\ulG=(G,\phi)$. If $\Fa\subseteq A$ is an ideal, we set 
\[
\ulG[\Fa](K^\alg):=\phi[\Fa](K^\alg): = \{ P \in G(K^\alg)\ |\ \phi_a(P) = 0 \ \text{for all} \ a \in \Fa\}.
\]
The latter are the $K^\alg$-valued points of a closed subgroup scheme $\ulG[\Fa]$ of $G$, which is an $A/\Fa$-module scheme via $\ol a \mapsto {\phi_a} |_{\ulG[\Fa]}$. If $\Fa=(a)$ then $\ulG[\Fa](K^\alg)=\ulG[a](K^\alg)$.
\end{Definition}

\begin{Remark}
We have the following observation, see \cite[\S\,4.5]{Goss}, where we denote the $A$-characteristic of $K$ by $\Fp=\AChar(K):=\ker(\gamma:A\to K)$:
\begin{enumerate} \item If $a \in A$ is prime to $\AChar(K)$, we see that the polynomial  $\phi_a$ is separable and $\# \ulG[a](K^\alg) =  (\#A/(a))^{\rk \ulG}$. Since this holds for every $a\in A$ and $\ulG[\Fa](K^\alg)$ is an $A/\Fa$-module, one obtains $\ulG[\Fa](K^\alg) \cong  (A/\Fa)^{\rk \ulG}$ as $A$-modules.

\item $\# \ulG[\Fp](K^\alg) =  (\#A/(\Fp))^{\rk \ulG - h}$ and $\ulG[\Fp](K^\alg) \cong  (A/(\Fp))^{\rk \ulG - h}$, where $h$ is the height of the Drinfeld $A$-module defined by $h: = \frac{w(a)}{v_{\Fp}(a) \cdot [\BF_\Fp:\BF_q]}$ for every $a\in A$, where $w(a)$ is the smallest integer $ i\geq 0$ with $\tau^i$ occurring in $\phi_a$, with nonzero coefficient.
\end{enumerate}
\end{Remark}

\begin{Example}\label{ExCarlitzTorsion}
The Carlitz module $\ulG=(\BG_{a,K},\phi)$ over $K=\BF_q(\theta)$ with $\phi_t = \theta+\tau$ from Example~\ref{CarMod} has rank $1$. For every $a=\sum_{i=0}^n a_i t^i$ with $a_i\in\BF_q$ and $a_n\ne0$, we have $\phi_a=\sum_{i=0}^n a_i\phi_t^n=\sum_{i=0}^n a_i(\theta+\tau)^n=(\sum_{i=0}^n a_i\theta^n)\cdot \tau^0+\ldots+a_n\tau^n=\gamma(a) \tau^0+\ldots+a_n\tau^n$. Therefore, the polynomial $\phi_a(x)=\gamma(a)\,x+\ldots+a_n\,x^{q^n}$ has degree $q^n$ and is separable, because $\gamma(a)\ne0$. From this it follows that $\#\ulG[a](K^\alg)=q^n=\#(A/(a))$ and that $\ulG[a](K^\alg)\cong A/(a)$ for every $a\in A$. This illustrates that the Carlitz module is the function field analog of the multiplicative group $\BG_m=\BG_{m,\BQ}$ from Example~\ref{ExGmBettiEt}, which for $a\in\BN_{>0}$ satisfies $\BG_m[a](\BQ^\alg):=\ker[a](\BQ^\alg)=\{x\in\BQ^\alg\colon x^a=1\}\cong\BZ/(a)$.
\end{Example}

\begin{Definition}
Let $v$ be a prime ideal of $A$. Let $\ulG=(G,\phi)$ be a Drinfeld $A$-module over $K$ of fixed rank $r$ and define the $A_v$-module $\ulG[v^\infty](K^\alg): = \cup_{n\geq1}\ulG[v^n](K^\alg)$. The $A_v$-module
\begin{equation}\label{EqTateModOfDriMod}
\Koh_{1,v}(\ulG,A_v):=T_v(\ulG) = \Hom_{A_v}\bigl(Q_v/A_v,\, G(K^\alg)\bigr)= \Hom_{A_v}\bigl(Q_v/A_v,\, \ulG[v^\infty](K^\alg)\bigr).
\end{equation}
is called the \emph{$v$-adic homology realization} or the \emph{$v$-adic Tate module} of $\ulG$. It carries a continuous $\sG_K$-action. Note that when $z=\frac{a}{c}\in Q$ is a uniformizing parameter of $A_v$ then the map $\phi_z:=\phi_c^{-1}\circ\phi_a\colon\ulG[v^n](K^\alg)\to\ulG[v^{n-1}](K^\alg)$ is well defined and
\[
T_v(\ulG) \cong \invlim \bigl(\ulG[v^n](K^\alg),\phi_z\bigr)\,;
\]
see for example \cite[after Definition~4.8]{HartlKim}. A morphism $f :\ulG \to \ulG' $ of Drinfeld $A$-modules gives a morphism $T_v(f): T_v(\ulG)\to  T_v(\ulG')$ of $A_v[\sG_K]$-modules. If $v$ is different from the $A$-characteristic $\AChar(K)$ of $K$, then $T_v(\ulG)$ is isomorphic to $A_v^{\oplus r}$. 
\end{Definition}

\begin{Remark}\label{RemRkAbVar}
The results of this section parallel Remark~\ref{RemRkTateModAbVar} for abelian varieties. Since the $\ell$-adic Tate module of an abelian variety $X$ is isomorphic to $(\BZ_\ell)^{2\dim X}$, while the $v$-adic Tate module of a Drinfeld $A$-module $\ulG$ is isomorphic to $A_v^{\rk\ulG}$ it is natural to call the number $\rk X:=2\dim X$ the \emph{rank} of the abelian variety $X$, compare also Theorems~\ref{ThmHomAbVar} and \ref{ThmHomAMot}.
\end{Remark}

There is a similar theory of Tate modules for $A$-motives which we will explain in the next section.

\section{Cohomology Realizations and Period Maps for $A$-Motives}\label{UniCoh}
\setcounter{equation}{0}

\subsection{Uniformizability and Betti Cohomology}\label{SectBettiAMot}

In this section we discuss the notion of uniformizability, cohomology realizations  and period maps for $A$-motives from \cite{HartlJuschka} and also we generalize the results to the case  $d_\infty = [\BF_\infty: \BF_q]\neq1$. For a field extension $K$ of $\BF_q$ we consider the closed subscheme $\infty_K:=\infty\times_{\BF_q}\Spec K\subset C_K:=C\times_{\BF_q}\Spec K$. If $K$ contains $\BF_\infty$, then $\infty_K$ is the disjoint union of $d_\infty$-many $K$-rational points of $C_K$.

In order to define the notion  of uniformizability for $A$-motives  we have to introduce some notation of rigid analytic geometry as in \cite{HP04}. For a general introduction to rigid analytic geometry see \cite{BGR}.

\begin{Notation}\label{NotRigidDiscs}
With the curve $C_{\BC_\infty}$ and its open affine part $C'_{\BC_\infty}: =C_{\BC_\infty}\setminus {\infty}_{\BC_\infty} $ one can associate by \cite[\S 9.3]{BGR} rigid analytic spaces  $\FC_{\BC_\infty}: = (C_{\BC_\infty})^\rig$ and ${\FC}'_{\BC_\infty}: = (C'_{\BC_\infty})^\rig = \FC_{\BC_\infty}\setminus \infty_{\BC_\infty}$. The underlying sets of $\FC_{\BC_\infty}$ and ${\FC}'_{\BC_\infty}$ are the sets of $\BC_\infty$-valued points of $C_{\BC_\infty}$ and $C_{\BC_\infty}\setminus {\infty}_{\BC_\infty}$, respectively. The endomorphism $\sigma$ of $C_{\BC_\infty}$ from \eqref{EqSigma} induces endomorphisms of $\FC_{\BC_\infty}$ and  $\FC'_{\BC_\infty}$ which we denote by the same symbol $\sigma$.

Let $\CO_{\BC_\infty}$ be the valuation ring of $\BC_\infty$  and let $\kappa_{\BC_\infty}$ be its residue field. By the valuative criterion of properness every point of $\FC_{\BC_\infty} = C_{\BC_\infty}(\BC_\infty) ={C}(\BC_\infty)$ extends uniquely to an $\CO_{\BC_\infty}$-valued point of ${C}$ and in the reduction gives rise to a $\kappa_{\BC_\infty}$-valued point of ${C}$. This gives us a reduction map 
\begin{equation}\label{EqRedMap}
\red : \FC_{\BC_\infty}={C}(\BC_\infty) \longto {C}(\kappa_{\BC_\infty})\,. 
\end{equation}
The subscheme ${\infty}_{\kappa_{\BC_\infty}}\subset C_{\kappa_{\BC_\infty}}$ contains $d_\infty$  points. We denote them by $\{\infty_i\text{ for }i\in\BZ/d_\infty\BZ\}$ in such a way that the map $\sigma$ from \eqref{EqSigma} transports  $\infty_i$ to $\infty_{i+1}$ and $(\sigma^{d_\infty})^\ast$  stabilizes  each $\infty_i$. Since the curve $C_{\kappa_{\BC_\infty}}$ is non-singular, \cite[Proposition 2.2]{Boschab}  implies for each $i$ that the preimage $\FD_{i}$ of $\infty_i \in {\infty}_{\kappa_{\BC_\infty}}$ under  $\red$ is an open rigid analytic unit disc in $\FC_{\BC_\infty}$ around ${\infty}_{i}$. Let $\FD'_{i}: = \FD_{i} \setminus {\infty}_{i}$ be the punctured open unit disc around ${\infty}_{i}$ in $\FC_{\BC_\infty}$. Then $\sigma$ maps $\FD_i$ isomorphically onto $\FD_{i+1}$. We let $\CO(\FD_{i})$ and $\CO(\FC_{\BC_\infty}\setminus \cup_{i} \FD_{i})$ be the coordinate rings of rigid analytic functions on the spaces $\FD_i$ and $\FC_{\BC_\infty}\setminus \cup_{i} \FD_{i}$, respectively. The uniformizer $z\in\CO(\FD_i)$ is a coordinate function on the disc $\FD_i$ for every $i$.
\end{Notation}

\begin{Example}
If $C = \BP^1_{\BF_q}, \ A = \BF_q[t]$,  and $[\BF_\infty: \BF_q]=1$, we can give the following explicit description. $\FD_0\subset\BP^1(\BC_\infty)$ is the open unit disc around $\infty$.
\[
\CO(\FC_{\BC_\infty}\setminus \FD_{0}): =\BC_\infty\langle t\rangle:= \bigg\{\sum_{i=0}^{\infty}a_i t^i,\  a_i \in \BC_\infty, \ a_i \to 0 \ \text{as}\ i\to \infty \bigg\}
\]
and $\FC_{\BC_\infty}\setminus \FD_{0}$ is the closed unit disc inside $C(\BC_\infty) \setminus \infty_{\BC_\infty}  = \BC_\infty$ on which the coordinate $t$ has absolute value less or equal to 1. Also we can take $z = 1/t$ as the coordinate on the disc $\FD_0$, and suggestively write $\FD_0=\{|z|<1\}$.
\end{Example}

\begin{Definition} For an $A$-motive $\ulM$ over $\BC_\infty$, we define the \emph{$\tau$-invariants}
\[
\Lambda (\ulM): = ( {M}\otimes_{A_{\BC_\infty}} \CO(\FC_{\BC_\infty}\setminus \cup_{i} \FD_{i}))^\tau: = \{ m \in {M}\otimes_{A_{\BC_\infty}} \CO(\FC_{\BC_\infty}\setminus\cup_i \FD_{i}): \ \tau_M(\sigma_M^{\ast}m) = m \}.
\]
\end{Definition}
\medskip
Since the ring of $\sigma^*$-invariants in $\CO(\FC_{\BC_\infty}\setminus  \cup_{i} \FD_{i})$ equals $A$, the set $\Lambda (\ulM)$ is an $A$-module. It was shown implicitly by Anderson~\cite[Proof of Lemma~2.10.6]{Anderson} that $\Lambda(\ulM)$ is finite projective of rank at most equal to $\rk \ulM$.

\begin{Definition}\label{Uniformizable}An $A$-motive  $\ulM$ is called \emph{uniformizable} (or \emph{rigid analytically trivial}) if the natural homomorphism 
\[
h_\ulM\colon \Lambda (\ulM) \otimes_A \CO(\FC_{\BC_\infty}\setminus  \cup_{i} \FD_{i}) \;\longto\; M \otimes_{A_{\BC_\infty}} \CO(\FC_{\BC_\infty}\setminus \cup_{i} \FD_{i}), \quad \lambda \otimes f \longmapsto f\cdot \lambda
\]
is an isomorphism. 
\end{Definition}

\begin{Example}\label{Cartau} We keep the notation from Example \ref{CarMod}.  We recall that the  Carlitz  motive over $\BC_\infty$ is given by $\underline \CC = \big( \CC =\BC_\infty[t], \tau_{\CC} = t-\theta \big)$.  We set  $\tminus{}: =  \prod_{i=0}^{\infty}\bigl(1- \frac{t}{\theta^{q^i}}\bigr)  \in \CO({{\FC}'_{\BC_\infty}})\subset  \CO(\FC_{\BC_\infty}\setminus \FD_{0})$  and  choose an $\eta \in \BC_\infty$ with $\eta^{1-q} = -\theta$. Then we see that $\eta \tminus{} \in \Lambda (\underline \CC)$, because 
\[
\tau_\CC(\sigma_\CC^*(\eta\tminus{}))\;=\;(t-\theta)\cdot\eta^q\cdot\sigma^*\prod_{i=0}^{\infty}\bigl(1- \frac{t}{\theta^{q^i}}\bigr)\;=\;\eta\cdot\frac{t-\theta}{-\theta}\cdot\prod_{i=1}^{\infty}\bigl(1- \frac{t}{\theta^{q^i}}\bigr)\;=\;\eta\tminus{}\, .
\]
Since $\eta \tminus{}$ has no zeroes outside $\FD_{0}$ it generates the $\CO(\FC_{\BC_\infty}\setminus \FD_{0})$-module $\CC \otimes_{A_{\BC_\infty}} \CO(\FC_{\BC_\infty}\setminus \FD_{0}) = \CO(\FC_{\BC_\infty}\setminus \FD_{0})$ and so $h_{\underline \CC}$  is an isomorphism and $\ulCC$ is uniformizable.
\end{Example}

Anderson~\cite{Anderson} proved the following criterion for uniformizability.
\begin{Lemma} Let $\ulM$ be an $A$-motive of rank $r$.
\begin{enumerate}
\item\label{LemmaAndCrit_A}
The homomorphism $h_\ulM$ is injective and it satisfies $h_\ulM \circ (\id_{\Lambda (\ulM)}\otimes \id)=  (\tau_M \otimes \id) \circ \sigma^\ast h_\ulM$.
\item\label{LemmaAndCrit_B}
$\ulM$ is uniformizable if and only if $\rk_A\Lambda (\ulM) = r$.
\end{enumerate}
\end{Lemma}
\begin{proof}
\ref{LemmaAndCrit_B} was proved by Anderson~\cite[Lemma~2.10.6]{Anderson}.

\medskip\noindent
\ref{LemmaAndCrit_A} is implicitly proved by Anderson~\cite{Anderson}. It is explicitly stated for example in \cite[Lemma~4.2]{BH07}.
\end{proof}

Next we state the  generalization of \cite[Proposition~3.25]{HartlJuschka}, which we will need to define period maps. The point $\Var(\CJ)\in C_{\BC_\infty}(\BC_\infty)$ lies in one of the discs $\FD_i$, because $|\gamma(a)|_\infty>1$ for all $a\in A\setminus\BF_q$. We normalize the indexing of the $\FD_i$ in such a way that $\Var(\CJ)\in\FD_0$. Then for any $i \in \BN_{0}$, we consider the pullbacks $\sigma^{i \ast}\CJ = (a\otimes 1 - 1 \otimes \gamma (a)^{q^i}:\ a \in A) \subset A_{\BC_\infty}$ and the points  $\Var(\sigma^{i \ast}\CJ)$ of $C'_{\BC_\infty}$ and $\FC'_{\BC_\infty}$. They correspond to the point $\Var(z-\zeta^{q^i})\in\FD_i$ and have $\infty_{\BC_\infty} = \{\infty_0,\cdots, \infty_{d_\infty-1}\}$  as accumulation points. More precisely, for each $k= 0,1, \cdots, d_\infty-1$ the point $\infty_k$ is the limit of the sequence $\Var(\sigma^{(k+d_\infty i)*}\CJ)=\Var(z-\zeta^{q^{k+d_\infty i}})$ for $i\in\BN_0$. Therefore, $\FC'_{\BC}\setminus \cup_{i\in\BN_0}  \Var(\sigma^{i \ast}\CJ)$  is an admissible open rigid analytic subspace of  $\FC'_{\BC_\infty}$.

\begin{Proposition} \cite[Proposition 3.25] {HartlJuschka}\label{PropExt} 
Let  $\ulM$ be a uniformizable effective $A$-motive over $\BC_\infty$. Then $\Lambda (\ulM)$ equals $\{ m \in  M \otimes_{A_{\BC_\infty}} \CO(\FC'_{\BC_\infty}): \ \tau_M(\sigma_M^{\ast}m) = m \}$ and the isomorphism $h_\ulM$ extends to an injective homomorphism 
\begin{equation*}\label{Eqh_M}
h_\ulM\colon \Lambda (\ulM) \otimes_A \CO(\FC'_{\BC_\infty}) \;\longto\; M \otimes_{A_{\BC_\infty}} \CO(\FC'_{\BC_\infty}), \ \lambda \otimes f \mapsto f\cdot \lambda
\end{equation*}
with $h_\ulM \circ (\id_{\Lambda (\ulM)}\otimes \id)=  (\tau_M \otimes \id) \circ \sigma^\ast h_\ulM$. At the point $\Var(\CJ)$ its cokernel satisfies $\coker h_\ulM \otimes \BC_\infty\dbl z - \zeta \dbr = M/ {\tau_M(\sigma^\ast M)}$. The morphism $h_\ulM$ is a local isomorphism away from $\cup_{i\in\BN_0}  \Var(\sigma^{i \ast}\CJ)$, and $\sigma^*h_\ulM$ is a local isomorphism away from $\cup_{i\in\BN_{>0}}  \Var(\sigma^{i \ast}\CJ)$.
\end{Proposition}

\begin{proof}
This follows in the same way as \cite[Proposition 3.25]{HartlJuschka}.
\end{proof}

\begin{Definition}\label{DefAMotBetti}
Let $\ulM$ be an $A$-motive of rank $r$ over $\BC_\infty$. Anderson defined the \emph{Betti cohomology realization}  of $\ulM$ by setting 
\[
\Koh^1_\Betti(\ulM, R): = \Lambda (\ulM)\otimes_A R \qquad 
\text{and} \qquad \Koh_{1,\Betti}(\ulM, R): = \Hom_A(\Lambda (\ulM), R)
\]
for any $A$-algebra $R$. This is most useful when $\ulM$ is uniformizable, in which case  both are locally  free $R$-modules of rank equal to $\rk \ulM$.
\end{Definition}

\begin{Example} \label{Cartau1}
We keep the notation from Example~\ref{Cartau}. There we have calculated $\Lambda (\underline \CC)$ as the $A$-module generated by $\eta \tminus{}$, so 
\[
\Koh^1_\Betti(\underline{\CC}, A) =\eta \tminus{}\cdot A \qquad \text{and} \qquad\Koh_{1,\Betti}(\ulM, A) =  (\eta \tminus{})^{-1}  \cdot A.
\]
\end{Example}

\begin{Remark}\label{RemOmegaQ}
To explain the compatibility with Definition~\ref{DefDriModBetti} let $\Omega^1_{A/\BF_q} $ be the module of K\"ahler differentials of $A$ over $\BF_q$. Then $\Omega^1_{A/\BF_q}\otimes_A Q=\Omega^1_{Q/\BF_q}=Q\,dz$ because the field extension $Q/\BF_q(z)$ is separable as it is unramified at $\infty$.
\end{Remark}

\begin{Proposition}[{\cite[Corollary~2.12.1]{Anderson}}]\label{PropCompBettiGandM}
Let $\ulG=(G,\phi)$ be a Drinfeld $A$-module over $\BC_\infty$ and let $\ulM=\ulM(\ulG)$ be the associated $A$-motive. Then $\ulM$ is uniformizable and there is a perfect pairing of $A$-modules
\[
\Koh_{1,\Betti}(\ulG,A)\times \Koh^1_\Betti(\ulM,A)\;\longto\;\Omega^1_{A/\BF_q}\,,\quad(\lambda,m)\longmapsto \omega_{A,\lambda,m}
\]
where $\omega_{A,\lambda,m}$ is determined by the residues $\Res_\infty(a\cdot\omega_{A,\lambda,m})=-m\bigl(\exp_\ulG(\Lie\phi_a(\lambda))\bigr)\in\BF_q$ for all $a\in Q$. The pairing yields a canonical isomorphism
\[
\Koh_{1,\Betti}(\ulM,A)\otimes_A\Omega^1_{A/\BF_q}\;\isoto\;\Koh_{1,\Betti}(\ulG,A)\,,
\]
which is functorial in $\ulG$.
\end{Proposition}

\subsection{$v$-adic Cohomology}\label{SectEtCohAMot}

\begin{Definition}\label{DefAMotVAdic}
For an $A$-field $K$ consider the $v$-adic completion $A_{v, K}: = \invlim A_{K}/v^n A_{K} $ of $A_{K}$. Let $\ulM$ be an $A$-motive over $K$ and let $v\subset A$ be a maximal ideal with $v\neq \AChar(K)$. Since $(A_{v, K^\sep})^{\tau = \id} = A_v$, we can define the $v$-$adic \ cohomology\ realizations$ of $\ulM$ as the $A_v$-modules 
\begin{alignat}{2}
\label{EqKoh_v}
&\Koh^1_v( \ulM, A_v) & & :=( M \otimes_{A_K}A_{v, K^\sep})^\tau:= \{m \in M\otimes_{A_K}A_{v,K^\sep} \ |\ \tau_M(\sigma_M^* m ) = m\} \qquad\text{and}\\ 
\nonumber
&\Koh_{1,v}( \ulM, A_v) & & :=\Hom_{A_v}( \Koh^1_v(\ulM,A_v), A_v).
\end{alignat}
They are free $A_v$-modules of rank equal to $\rk M$, carrying a continuous action of the Galois group $\sG_K$ by \cite[Proposition 6.1]{Tag96}, and the inclusion  $\Koh^1_v( \ulM, A_v) \subset M\otimes_{A_K}A_{v,K^\sep}$ induces a canonical isomorphism of $A_{v,K^\sep}$-modules 
\[
\Koh^1_v(\ulM,A_v) \otimes_{A_v}A_{v, K^\sep}\isoto M\otimes_{A_K}A_{v,K^\sep}
\]
which is both $\sG_K$ and $\tau$-equivariant, where on the left module $\sG_K$ acts on both factors and $\tau$ is $\id\otimes \sigma^\ast$ and on the right module $\sG_K$ acts  only on $A_{v, K^\sep}$ and $\tau$ is  $(\tau_M\circ \sigma_M^\ast)\otimes \sigma^\ast$. One also sometimes denotes  $\Koh^1_v( \ulM, A_v) $ by  $\check T_v(\ulM)$ and  calls this the \emph{$v$-adic dual Tate module} associated with $\ulM$ at $v$. We also define the $Q_v$-vector spaces with continuous $\sG_K$-action
\begin{alignat*}{2}
&\Koh^1_v( \ulM, Q_v) && := \Koh^1_v(\ulM,A_v) \otimes_{ A_v} Q_v \qquad \text{and}\\
& \Koh_{1,v}( \ulM, Q_v) && := \Hom_{A_v}(\Koh^1_v(\ulM,A_v) , Q_v)= \Koh_{1,v}(\ulM,A_v) \otimes_{ A_v} Q_v \,.
\end{alignat*}
The association  $\ulM \mapsto   \Koh^1_v( \ulM, A_v)$  or $\ulM \mapsto   \Koh^1_v( \ulM, Q_v)$  is a covariant functor which is exact and faithful. 
\end{Definition}

The analog of the Tate conjecture is the following theorem which was proved by Taguchi~\cite{TaguchiTateConj} and Tamagawa~\cite[\S\,2]{TamagawaTateConj}.

\begin{Theorem}[Tate conjecture for $A$-motives]\label{ThmTateConj}
If $K$ is a finitely generated $A$-field and $v\ne\AChar(K)$ then
\begin{align*} 
\Hom(\ulM , \ulM')\otimes_A A_v \isoto  \Hom_{A_v[\sG_K]}(\Koh^1_v( \ulM, A_v), \Koh^1_v( \ulM', A_v))
\end{align*}
is an isomorphism of $A_v$-modules for $A$-motives $\ulM$ and  $\ulM'$.
\end{Theorem}

Let us explain the relation between $T_v\ulG$ and $\check T_v\ulM(\ulG):=\Koh^1_v(\ulM(\ulG),A_v)$ for a Drinfeld $A$-module $\ulG$. The $A_v$-module $\Hom_{\BF_v}(Q_v/A_v,\BF_v)$ is canonically isomorphic to the $A_v$-module $\wh\Omega^1_{A_v/\BF_v}=A_v\,dz_v$ of continuous differential forms; see \cite[Equation (4.5)]{HartlKim}, and therefore, it is a free $A_v$-module of rank $1$. If $\ulG$ is a Drinfeld $A$-module over $K$ and $\ulM=\ulM(\ulG)$ is its associated $A$-motive, then there is a natural $\sG_K$-equivariant perfect pairing of $A_v$-modules
\begin{equation}\label{EqPairingTateMod}
\langle\,.\,,\,.\,\rangle\colon \;T_v \ulG\times \check T_v\ulM\;\longto\;\Hom_{\BF_v}(Q_v/A_v,\BF_v)\;\cong\;\wh\Omega^1_{A_v/\BF_v}\,,\quad\langle f, m\rangle:= m\circ f\,,
\end{equation}
which identifies $T_v \ulG$ with the contragredient $\sG_K$-representation $\Hom_{A_v}(\check T_v\ulM,\,\wh\Omega^1_{A_v/\BF_v})$ of $\check T_v\ulM$; see \cite[Proposition~4.9]{HartlKim}. Together with Theorems~\ref{ThmEquivDriMod} and \ref{ThmTateConj} this implies the following

\begin{Corollary}[Tate conjecture for Drinfeld $A$-modules] \label{MorTate}
Let $\ulG$ and $\ulG'$ be two Drinfeld $A$-modules over a finitely generated field $K$. Then the natural map 
\[
\Hom_{K}(\ulG, \ulG') \otimes_A A_v \to \Hom_{A_v[\sG_K]}(T_v\ulG, T_v\ulG'),\quad f\otimes a \mapsto a\cdot T_v(f)
\]
is an isomorphism of $A_v$-modules.
\end{Corollary}

\subsection{De Rham Cohomology and Period Isomorphisms}\label{SectDRAMot}

In this subsection let $(K,\gamma)$ be an $A$-field of generic $A$-characteristic. Then $K$ is a field extension of $Q$ via $\gamma$ and we set $\zeta:=\gamma(z)$. There is an identification $\invlim A_K/\CJ^n=K\dbl z-\zeta\dbr$ from \cite[Lemma~1.3]{HartlJuschka}.

\begin{Definition}\label{DefAMotDeRham}
Let $\ulM$ be an $A$-motive over an $A$-field $K$ of generic $A$-characteristic. The \emph{de Rham realization of $\ulM$} is defined as 
\begin{eqnarray*}\label{EqdRAMotive}
\Koh^1_{\dR}\bigl(\ulM,K\dbl z-\zeta\dbr\bigr) & := & \sigma^*M\otimes_{A_K}\invlim A_K/\CJ^n\,,\\[2mm]
\Koh^1_{\dR}\bigl(\ulM,K\dpl z-\zeta\dpr\bigr) & := & \Koh^1_{\dR}\bigl(\ulM,K\dbl z-\zeta\dbr\bigr)\otimes_{K\dbl z-\zeta\dbr}K\dpl z-\zeta\dpr\qquad\text{and}\\[2mm]
\Koh^1_{\dR}(\ulM,K) & := & \sigma^*M\otimes_{A_K}A_K/\CJ \\[1mm]
& = & \Koh^1_{\dR}\bigl(\ulM,K\dbl z-\zeta\dbr\bigr)\otimes_{K\dbl z-\zeta\dbr}K\dbl z-\zeta\dbr/(z-\zeta)\,.\nonumber
\end{eqnarray*}
The \emph{Hodge-Pink lattice of $\ulM$} is defined as $\Fq^\ulM:=\tau_M^{-1}(M\otimes_{A_K}\invlim A_K/\CJ^n)\subset\Koh^1_{\dR}\bigl(\ulM,K\dpl z-\zeta\dpr\bigr)$, and the descending \emph{Hodge-Pink filtration of $\ulM$} is defined via $\Fp^\ulM:=\Koh^1_{\dR}(\ulM,K\dbl z-\zeta\dbr)$ and
\begin{align*}
F^i\Koh^1_{\dR}(\ulM,K) \es := & \es \bigl(\Fp^\ulM\cap(z-\zeta)^i\Fq^\ulM\bigr)\big/\bigl((z-\zeta)\Fp^\ulM\cap(z-\zeta)^i\Fq^\ulM\bigr)\\[2mm]
= & \es \text{image of }\bigl(\sigma^*M\cap\tau_M^{-1}(\CJ^i M)\bigr)\otimes_RK  \es\text{in} \es \Koh^1_{\dR}(\ulM,K)\,;
\end{align*}
compare also with \cite[\S\,2.6]{Goss}. Since $\ulM$ is effective, we have $\Fp^\ulM\subset\Fq^\ulM$ with $\tau_M\colon\Fq^\ulM/\Fp^\ulM\isoto\coker\tau_M$ and $F^0\Koh^1_{\dR}(\ulM,K)=\Koh^1_{\dR}(\ulM,K)$. Note that the de Rham realization with Hodge-Pink lattice and filtration is a covariant functor on the category of $A$-motives over $K$ with quasi-morphisms.
\end{Definition}

\begin{Definition}
If $\ulG$ is a Drinfeld $A$-module over an $A$-field $K$ of generic characteristic, let $\ulM=(M,\tau_M)=\ulM(\ulG)$ be the associated $A$-motive. Then the \emph{de Rham cohomology realization} of $\ulG$ is defined to be
\begin{eqnarray*}
\Koh^1_\dR(\ulG,K) & := & \Hom_A(\Omega^1_{A/\BF_q},\,\sigma^*M/\CJ\cdot\sigma^*M)\,,\\[2mm]
\Koh^1_\dR(\ulG,K\dbl z-\zeta\dbr) & := & \Hom_A\bigl(\Omega^1_{A/\BF_q},\,\sigma^*M\otimes_{A_K}K\dbl z-\zeta\dbr\bigr)\,,\\[2mm]
\Koh_{1,\dR}(\ulG,K\dbl z-\zeta\dbr) & := & \Hom_{A_K}(\sigma^*M,\,\wh\Omega^1_{K\dbl z-\zeta\dbr/K})\quad\text{and}\\[2mm]
\Koh_{1,\dR}(\ulG,K) & := & \Hom_{A_K}(\sigma^*M,\,\wh\Omega^1_{K\dbl z-\zeta\dbr/K})\otimes_{K\dbl z-\zeta\dbr}K\dbl z-\zeta\dbr/(z-\zeta)\,,
\end{eqnarray*}
where $\Omega^1_{A/\BF_q} $ is the module of K\"ahler differentials of $A$ over $\BF_q$ and $\wh\Omega^1_{K\dbl z-\zeta\dbr/K}=K\dbl z-\zeta\dbr dz$ is the $K\dbl z-\zeta\dbr$-module of continuous differentials. We define the \emph{Hodge-Pink lattices} of $\ulG$ as the $K\dbl z-\zeta\dbr$-submodules 
\[
\begin{array}{rcccl}
\Fq^\ulG & := & \Hom_A\bigl(\Omega^1_{A/\BF_q},\,\tau_M^{-1}(M)\otimes_{A_K}K\dbl z-\zeta\dbr\bigr) & \subset & \Koh^1_\dR\bigl(\ulG,K\dpl z-\zeta\dpr\bigr) \quad\text{and}\\[2mm]
\Fq_\ulG & := & (\tau_M\dual\otimes\id_{K\dpl z-\zeta\dpr})\bigl(\Hom_{A_K}(M,\,\wh\Omega^1_{K\dbl z-\zeta\dbr/K})\bigr) & \subset & \Koh_{1,\dR}\bigl(\ulG,K\dpl z-\zeta\dpr\bigr)\,.
\end{array}
\]
In both cases the Hodge-Pink filtrations $F^i \Koh^1_\dR(\ulG,K)$ and $F^i \Koh_{1,\dR}(\ulG,K)$ of $\ulG$ are recovered as the images of $\Koh^1_\dR\bigl(\ulG,K\dbl z-\zeta\dbr\bigr)\cap(z-\zeta)^i\Fq^\ulG$ in $\Koh^1_\dR(\ulG,K)$ and of $\Koh_{1,\dR}\bigl(\ulG,K\dbl z-\zeta\dbr\bigr)\cap(z-\zeta)^i\Fq_\ulG$ in $\Koh_{1,\dR}(\ulG,K)$ like in Definition~\ref{DefAMotDeRham}. All these structures are compatible with the natural duality between $\Koh^1_\dR$ and $\Koh_{1,\dR}$.
\end{Definition}

\begin{Remark}\label{RemHodgeFiltrDriMod}
It was shown in \cite[Remark~4.45 and Lemma~5.46]{HartlJuschka} that this definition coincides with the definitions given by Deligne, Anderson, Gekeler and Jing Yu, see \cite[Definition~2.6.1]{Goss94}, \cite[\S\,2]{Gekeler89} and \cite{Yu90}. Moreover, it was shown in \cite[Diagram~(5.36) in the Proof of Theorem~5.40]{HartlJuschka} that the dual of the sequence of $K\dbl z-\zeta\dbr$-modules $0\to\Fp^\ulM\to\Fq^\ulM\to\coker\tau_M\to0$ is isomorphic to the sequence 
\[
0\longto\Fq_\ulG\longto\Koh_{1,\dR}(\ulG,K\dbl z-\zeta\dbr)\longto\Lie\ulG\longto0\,.
\]
Since $z-\zeta=0$ on $\Lie\ulG$ we obtain modulo $(z-\zeta)\Koh_{1,\dR}(\ulG,K\dbl z-\zeta\dbr)$ the exact sequence of $K$-vector spaces
\begin{equation}\label{EqHodgeSeqDriMod}
0\longto F^0\Koh_{1,\dR}(\ulG,K)\longto\Koh_{1,\dR}(\ulG,K)\longto\Lie\ulG\longto0\,,
\end{equation}
which is the analog of the decomposition \eqref{EqHodgeAbVar}.
\end{Remark}

For a uniformizable $A$-motive $\ulM$ over $\BC_\infty$ the morphism $h_\ulM$ from Proposition \ref{PropExt} induces comparison isomorphisms between the Betti and the $v$-adic, respectively the de Rham realizations as follows.

Since $v\neq \infty$ the points in the closed subscheme $\{v\}\times_{\BF_q}\Spec\BC_\infty\subset C_{\BC_\infty}$ do not specialize to $\infty_{\kappa_\BC}\in C_{\kappa_\BC}$ and so this closed subscheme lies in $C_{\BC_\infty}\setminus \cup_i\FD_i$.  This gives us isomorphisms $\CO(C_{\BC_\infty}\setminus \cup_i \FD_i)/v^n\CO(C_{\BC_\infty}\setminus \cup_i\FD_i) \isoto A_{\BC_\infty}/v^nA_{\BC_\infty}$ for all $n\in \BN$ and $\invlim\CO(C_{\BC_\infty}\setminus \cup_i \FD_i)/v^n\CO(C_{\BC_\infty}\setminus \cup_i\FD_i) \isoto\invlim A_{\BC_\infty}/v^nA_{\BC_\infty}=A_{v,\BC_\infty}$. The isomorphism $h_\ulM$ from Proposition~\ref{PropExt} induces a $\tau$-equivariant isomorphism
\[
 \Koh^1_\Betti(\ulM,A)\otimes_A \invlim \CO(C_{\BC_\infty}\setminus\cup_i\FD_i)/v^n\CO(C_{\BC_\infty}\setminus \cup_i\FD_i) \isoto M\otimes_{A_{\BC_\infty}}A_{v,\BC_\infty}\,.
\]
Taking $\tau$-invariant on both sides provides us with the isomorphism between the Betti and the $v$-adic  realization
\[
h_{\Betti,v}: \Koh^{1}_\Betti(\ulM, A_v) = \Koh^1_\Betti(\ulM,A)\otimes_A A_v \isoto \Koh^1_v(\ulM, A_v), \ \lambda \otimes f \mapsto (f\cdot\lambda \mod v^n)_{n\in \BN}.
\]

On the other hand, Proposition~\ref{PropExt} implies that $\sigma^\ast h_\ulM$ is an isomorphism locally at $\Var(\CJ)$ that is 
\[
\sigma^\ast h_\ulM\otimes \id_{{\BC_\infty}\dbl z-\zeta\dbr}: \Koh^1_\Betti(\ulM,A)\otimes_A {\BC_\infty}\dbl z-\zeta\dbr \isoto \sigma^\ast M\otimes_{A_{\BC_\infty}}{\BC_\infty}\dbl z-\zeta\dbr.
\]
This induces an isomorphism between the Betti and the de Rham realization 
\begin{alignat*}{2}
h_{\Betti,\dR} & := \sigma^\ast h_\ulM\otimes \id_{{\BC_\infty} \dbl z-\zeta\dbr}: \es && \Koh^{1}_\Betti(\ulM, {\BC_\infty}\dbl z-\zeta\dbr) \isoto  \Koh^1_{\dR}(\ulM, {\BC_\infty}\dbl z-\zeta\dbr), \\
h_{\Betti,\dR} & := \sigma^\ast h_\ulM \ \text{mod}\ \CJ: &&\Koh^{1}_\Betti(\ulM, {\BC_\infty}) \isoto  \Koh^1_{\dR}(\ulM, {\BC_\infty}).
\end{alignat*}
We summarize the above result as follows, compare \cite[Theorem 3.39]{HartlJuschka}.

\begin{Theorem}\label{PeriodIso}
If $\ulM$ is a uniformizable $A$-motive over $\BC_\infty$ there are canonical \emph{comparison isomorphisms}, sometimes also called \emph{period isomorphisms}
\begin{equation}\label{Eq1PeriodIso}
h_{\Betti,v}: \Koh^{1}_\Betti(\ulM, A_v) = \Koh^1_\Betti(\ulM,A)\otimes_A A_v \isoto \Koh^1_v(\ulM, A_v), \ \lambda \otimes f \mapsto (f\cdot\lambda \mod v^n)_{n\in \BN}
\end{equation}
and 
\begin{alignat}{2}
h_{\Betti,\dR} &:= \sigma^\ast h_\ulM\otimes \id_{{\BC_\infty}\dbl z-\zeta\dbr}: \es && \Koh^{1}_\Betti(\ulM, {\BC_\infty}\dbl z-\zeta\dbr) \isoto  \Koh^1_{\dR}(\ulM, {\BC_\infty}\dbl z-\zeta\dbr), \nonumber \\[2mm]
\label{Eq2PeriodIso} h_{\Betti,\dR} &:= \sigma^\ast h_\ulM \ \emph{mod}\ \CJ: && \Koh^{1}_\Betti(\ulM, {\BC_\infty}) \isoto  \Koh^1_{\dR}(\ulM, {\BC_\infty}). 
\end{alignat}
The latter yields a pairing
\begin{eqnarray}
\langle\,.\;,\,.\,\rangle_\infty\colon & \Koh_{1,\Betti}(\ulM,\BC_\infty) \times \Koh^1_{\dR}(\ulM,\BC_\infty) & \longto\es\BC_\infty\,, \label{Eq3PeriodIso} \\[1mm]
& \qquad(u\,,\,\omega)\; & \longmapsto\es \langle u,\omega\rangle_\infty\;:=\; u\otimes\id_{\BC_\infty}\bigl(h_{\Betti,\dR}^{-1}(\omega)\bigr)\,. \nonumber
\end{eqnarray}
\end{Theorem}

All these cohomology realizations and period isomorphisms are functorial in $\ulM$ and by \cite[Theorem~5.49]{HartlJuschka} compatible with the functor from Drinfeld $A$-modules to $A$-motives, Proposition~\ref{PropCompBettiGandM} and the pairing \eqref{EqPairingTateMod}.

\begin{Example}\label{ExCarBettiDR}
For the Carlitz motive $\ulCC = (\CC = \BF_q(\theta)[t], \tau_\CC = t-\theta)$ from Example~\ref{CarMod} the period isomorphism $h_{\Betti,\dR}$ is given as follows. By Example~\ref{Cartau} the generator $\eta\tminus{}$ of $\Koh^{1}_\Betti(\ulCC, {\BC_\infty})=A\cdot\eta\tminus{}$ is sent under $h_{\Betti,\dR}$ to the element $\sigma^*(\eta\tminus{})|_{t=\theta}=\eta^q\prod_{i=1}^{\infty}(1- \theta^{1-q^i})\in\BC_\infty$ which has absolute value $\bigl|\eta^q\prod_{i=1}^{\infty}(1- \theta^{1-q^i})\bigr|_\infty=|\eta^q|_\infty=|\theta|_\infty^{q/(1-q)}=q^{-q/(q-1)}$. This element is the analog of the period $(2\pi i)^{-1}$ from Example~\ref{ExGmBettiDR}, because the Carlitz module and Carlitz motive are the analogs of the multiplicative group $\BG_m$, see Example~\ref{ExCarlitzTorsion}.
\end{Example}

\begin{Theorem}\label{ThmDriModIntegral}
Let $\ulG$ be a Drinfeld $A$-module over $\BC_\infty$ and let $\ulM=(M,\tau_M)=\ulM(\ulG):=\Hom_{\BC_\infty, \BF_q}(G, \BG_{a,\BC_\infty})$ be the associated $A$-motive. Let $\Fq^\ulM$ and $\Fp^\ulM:=\Koh^1_\dR(\ulM,\BC_\infty\dbl z-\zeta\dbr)$ be as in Definition~\ref{DefAMotDeRham}. Let $m\in\Fq^\ulM$ be such that its image $\ol m$ under the isomorphism $\tau_M\otimes\id_{\BC_\infty\dbl z-\zeta\dbr}\colon\Fq^\ulM/\Fp^\ulM\isoto\coker\tau_M$ generates the one dimensional $\BC_\infty$-vector space $\coker\tau_M$. Let $\omega:=-(z-\zeta)\cdot m\in (z-\zeta)\Fq^\ulM\subset\Fp^\ulM$. Consider the pairing
\begin{equation}\label{EqThmDriModIntegral0}
\coker\tau_M\times \Lie\ulG\;\longto\;\Lie\BG_{a,\BC_\infty}\;=\;\BC_\infty\,,\qquad (\ol m,\lambda)\;\longmapsto\;\ol m(\lambda)
\end{equation}
induced from \eqref{EqLieG} and the isomorphism
\[
\beta_A\colon\Koh_{1,\Betti}(\ulG,Q)\;\isoto\;\Koh_{1,\Betti}(\ulM,Q)\otimes_Q\Omega^1_{Q/\BF_q}\;=\;\Koh_{1,\Betti}(\ulM,Q)\cdot dz
\]
from Proposition~\ref{PropCompBettiGandM} using Remark~\ref{RemOmegaQ}. Let $\lambda\in\Koh_{1,\Betti}(\ulG,Q)\subset\Lie\ulG$ and let $u\in\Koh_{1,\Betti}(\ulM,Q)$ be such that $\beta_A(\lambda)=u\,dz$. Then the pairing \eqref{Eq3PeriodIso} can be computed as
\begin{equation}\label{EqThmDriModIntegral}
\langle u,\omega\rangle_\infty\;=\;\ol m(\lambda)\,.
\end{equation}
\end{Theorem}

\begin{proof}
As in \cite[Diagram (5.36) in the proof of Theorem~5.39]{HartlJuschka} the isomorphism $\beta_A$ fits into a commutative diagram
\begin{equation}\label{EqEandM2}
\xymatrix @R+1pc @C+3pc { 
\Koh_{1,\Betti}(\ulM,Q)\otimes_Q\Omega^1_{Q/\BF_q} \ar[r]^-{\DS\tilde\gamma_A} & \Hom_{\BC_\infty}(\coker\tau_M,\BC_\infty) \\
\Koh_{1,\Betti}(\ulG,Q) \ar@{^{ (}->}[r] \ar[u]^{\DS\cong}_{\DS\beta_A} & \Lie G \ar[u]^{\DS\cong}_{\DS\alpha}
}
\end{equation}
where the isomorphism $\alpha$ is induced from the pairing \eqref{EqThmDriModIntegral0}, and the map $\tilde\gamma_A$ is given by 
\newcommand{\EltOfPMDual}{u}
\begin{eqnarray*}
& \tilde\gamma_A\colon \Koh_{1,\Betti}(\ulM,Q)\otimes_Q\Omega^1_{Q/\BF_q}\;=\;\Koh_{1,\Betti}(\ulM,Q)\cdot dz & \longto \es \Hom_{\BC_\infty}(\coker\tau_M,\BC_\infty)\,,\\[1mm]
& \mbox{\qquad\qquad\qquad\qquad\qquad\qquad\qquad}u\,dz & \longmapsto \es \bigl[\ol m\mapsto -\Res_{z=\zeta}u(\ol m)dz\bigr]\,.
\end{eqnarray*}
Here $u(\ol m)\in\BC_\infty\dpl z-\zeta\dpr$ is defined as 
\begin{eqnarray*}
u(\ol m) & := & (u\otimes\id_{\BC_\infty\dpl z-\zeta\dpr})\circ(h_\ulM\otimes\id_{\BC_\infty\dpl z-\zeta\dpr})^{-1}\circ(\tau_M\otimes\id_{\BC_\infty\dpl z-\zeta\dpr})(m) \\[2mm]
& = & (u\otimes\id_{\BC_\infty\dpl z-\zeta\dpr})\circ (h_{\Betti,\dR}^{-1}\otimes\id_{\BC_\infty\dpl z-\zeta\dpr})(m)
\end{eqnarray*}
where
\[
h_\ulM\otimes\id_{\BC_\infty\dpl z-\zeta\dpr}\colon \Koh^1_\Betti(\ulM,Q)\otimes_Q\BC_\infty\dpl z-\zeta\dpr\;\isoto\; M\otimes_{A_{\BC_\infty}}\BC_\infty\dpl z-\zeta\dpr
\]
is the isomorphism from Proposition~\ref{PropExt} with $h_\ulM=\tau_M\circ\sigma^*h_\ulM$ and $h_{\Betti,\dR}=\sigma^*h_\ulM\otimes\id_{\BC_\infty\dbl z-\zeta\dbr}$. Note that $u(\ol m)$ is only well defined up to adding elements of $\BC_\infty\dbl z-\zeta\dbr$, because the preimage $m$ of $\ol m$ is only well defined up to $\Fp^\ulM$ and $(u\circ h_{\Betti,\dR}^{-1})(\Fp^\ulM)=u\bigl(\Koh^1_\Betti(\ulM,\BC_\infty\dbl z-\zeta\dbr)\bigr)\subset \BC_\infty\dbl z-\zeta\dbr$. This shows that, nevertheless, the residue $-\Res_{z=\zeta}u(\ol m)dz$ is well defined and independent of the preimage $m$ of $\ol m$. We may thus compute
\[
\ol m(\lambda)\;=\;\alpha(\lambda)(\ol m)\;=\;(\tilde\gamma_A\circ\beta_A)(\lambda)(\ol m)\;=\;\tilde\gamma_A(u\,dz)(\ol m)\;=\;-\Res_{z=\zeta}u(\ol m)dz\,.
\]
Now $m=-(z-\zeta)^{-1}\cdot\omega$ and $u(\ol m)=(u\circ h_{\Betti,\dR}^{-1})(m)=-(z-\zeta)^{-1}\cdot\langle u,\omega\rangle_\infty$ in $\BC_\infty\dpl z-\zeta\dpr\big/\BC_\infty\dbl z-\zeta\dbr$. This yields
\[
\ol m(\lambda)\;=\;-\Res_{z=\zeta}u(\ol m)dz\;=\;\Res_{z=\zeta}\Bigl(\langle u,\omega\rangle_\infty\tfrac{d(z-\zeta)}{z-\zeta}\Bigr)\;=\;\langle u,\omega\rangle_\infty\,.
\]
\end{proof}

\section{Local Shtukas and the $v$-adic Period Isomorphism}\label{SectLocalShtukas}
\setcounter{equation}{0}

We next describe the function field analog of $p$-divisible groups.

\begin{Notation}
We fix a place $v\ne\infty$ of $Q$. Let $K\subset Q^\alg$ be an $A$-field which is a finite extension of $Q$ via $\gamma$. Under the fixed embedding $Q^\alg\into\BC_v$ let $L$ be the $v$-adic completion of $K\subset\BC_v$. Let $R$ be the valuation ring of $L$, let $\pi_L$ be a uniformizing parameter of $R$ and let $\kappa$ be the residue field of $R$. Then $R=\kappa\dbl\pi_L\dbr$ and $L=\kappa\dpl\pi_L\dpr$. The homomorphism $\gamma\colon A\to K$ extends by continuity to $\gamma\colon A_v\to L$ and factors through $\gamma\colon A_v\to R$ with $\zeta_v=\gamma(z_v)\in\pi_L R\setminus\{0\}$. Let $R\dbl z_v\dbr$ be the power series ring in the variable $z_v$ over $R$ and $\hat\sigma^*_{\!v}$ the endomorphism of $R\dbl z_v\dbr$ with $\hat\sigma^*_{\!v}(z_v)=z_v$ and $\hat\sigma^*_{\!v}(b)=b^{q_v}$ for $b\in R$, where $q_v=\#\BF_v$. For an $R\dbl z_v\dbr$-module $\hat M$ we let $ \hat\sigma_{\!v}^* \hat M:= \hat M\otimes_{R\dbl z_v \dbr, \hat\sigma_{\!v}^*}R\dbl z_v\dbr$ as well as  $\hat M[\frac{1}{z_v-\zeta_v}]:=\hat M\otimes_{R\dbl z_v\dbr}R\dbl z_v\dbr[\frac{1}{z_v-\zeta_v}]$ and  $\hat M[\frac{1}{z_v}]:=\hat M\otimes_{R\dbl z_v\dbr}R\dbl z_v\dbr[\frac{1}{z_v}]$. We obtain a canonical embedding $A_R:=A\otimes_{\BF_q}R\into R\dbl z_v\dbr$ by mapping $z_v\otimes1\mapsto z_v$ and $1\otimes\zeta_v\mapsto\zeta_v$.
\end{Notation}

The function field analog of $p$-divisible groups is given by the following

\begin{Definition}\label{DefZDLAM}
A \emph{$z_v$-divisible local Anderson module over $R$} is a sheaf of $\BF_q\dbl z_v\dbr$-modules $G$ on the big \fppf-site of $\Spec R$ such that
\begin{enumerate}
\item
$G$ is \emph{$z_v$-torsion}, that is $G = \dirlim G[z_v^n]$,
\item
$G$ is \emph{$z_v$-divisible}, that is $z_v\colon G \to G$ is an epimorphism,
\item 
for every $n$ the $\BF_q$-module $G[z_v^n]$ is representable by a finite locally free strict $\BF_q$-module scheme over $R$ in the sense of Faltings (see \cite{FaltingsStrict} or \cite[Definition~4.7]{HartlSingh}), and
\item \label{DefZDLAM_D}
locally on $\Spec R$ there exists an integer $d \in \BZ_{\geq 0}$, such that
$(z_v-\zeta_v)^d=0$ on $\omega_ G$  where $\omega_G := \invlim\omega_{G[z_v^n]}$ and $\omega_{G[z_v^n]}:=\epsilon^*\Omega^1_{G[z_v^n]/\Spec R}$ for the unit section $\epsilon$ of $G[z_v^n]$ over $R$.
\end{enumerate}
\end{Definition}

\begin{Example}\label{ExTorsionOfAModule}
Let $\ulG=(G,\phi)$ be a Drinfeld $A$-module over $R$ which is defined as in Definition~\ref{Defdrinfeld} by replacing $K$ by $R$. By \cite[Theorem~6.6]{HA17} the torsion module $\ulG[v^n]$ is a finite locally free strict $\BF_v$-module scheme and the inductive limit $\ulG[v^\infty]:=\dirlim\ulG[v^n]$ is a $z_v$-divisible local Anderson module over $R$ for which one can take $d=1$ in Definition~\ref{DefZDLAM}\ref{DefZDLAM_D}.
\end{Example}

Similarly to Remark~\ref{RemDieudonneMod}, divisible local Anderson modules have a description by semi-linear algebra. It is given by local $\hat\sigma^*_{\!v}$-shtukas.

\begin{Definition} \label{LocalshtDef}
A \emph{local $\hat\sigma^*_{\!v}$-shtuka of rank $r$} over $R$ is a pair $\ulHM = (\hat M,\tau_{\hat M})$ consisting of a free $R\dbl z_v\dbr$-module $\hat M$ of rank $r$, and an isomorphism $\tau_{\hat M}:\hat\sigma_{\!v}^* \hat M[\frac{1}{z_v-\zeta_v}] \isoto \hat M[\frac{1}{z_v-\zeta_v}]$. It is \emph{effective} if $\tau_{\hat M}(\hat\sigma_{\!v}^*\hat M)\subset \hat M$ and \emph{\'etale} if $\tau_{\hat M}(\hat\sigma_{\!v}^*\hat M)= \hat M$. We write $\rk\ulHM$ for the rank of $\ulHM$.

A \emph{morphism} of local shtukas $f:\ulHM=(\hat M,\tau_{\hat M})\to\ulHN=(\hat N,\tau_{\hat N})$ over $R$ is a morphism of the underlying modules $f:\hat M\to \hat N$ which satisfies $\tau_{\hat N}\circ \hat\sigma_{\!v}^* f = f\circ \tau_{\hat M}$. We denote the $A_v$-module of homomorphisms $f\colon\ulHM\to\ulHN$ by $\Hom_R(\ulHM,\ulHN)$ and write $\End_R(\ulHM)=\Hom_R(\ulHM,\ulHM)$.

A \emph{quasi-morphism} between local shtukas $f\colon(\hat M,\tau_{\hat M})\to(\hat N,\tau_{\hat N})$ over $R$ is a morphism of $R\dbl z_v\dbr[\tfrac{1}{z_v}]$-modules $f\colon M[\tfrac{1}{z_v}]\isoto N[\tfrac{1}{z_v}]$ with $\tau_{\hat N}\circ\hat\sigma_{\!v}^*f=f\circ\tau_{\hat M}$. It is called a \emph{quasi-isogeny} if it is an isomorphism of $R\dbl z_v\dbr[\tfrac{1}{z_v}]$-modules. We denote the $Q_v$-vector space of quasi-morphisms from $\ulHM$ to $\ulHN$ by $\QHom_R(\ulHM,\ulHN)$ and write $\QEnd_R(\ulHM)=\QHom_R(\ulHM,\ulHM)$. 
\end{Definition}

Note that $\Hom_R(\ulHM,\ulHN)$ is a finite free $A_v$-module of rank at most $\rk\ulHM\cdot\rk\ulHN$ by \cite[Corollary~4.5]{HartlKim} and $\QHom_R(\ulHM,\ulHN)=\Hom_R(\ulHM,\ulHN)\otimes_{A_v}Q_v$. Also every quasi-isogeny $f\colon\ulHM\to\ulHN$ induces an isomorphism of $Q_v$-algebras $\QEnd_R(\ulHM)\isoto\QEnd_R(\ulHN)$, $g\mapsto fgf^{-1}$, similarly to Remark~\ref{RemQHomAbVar}(a). 

\bigskip

The analog of the (``local'') Dieudonn\'e functor from Remark~\ref{RemDieudonneMod} is given by the following 

\begin{Theorem}[\bfseries {\cite[Theorem~8.3]{HartlSingh}}]
There is an anti-equivalence between the category of $z_v$-divisible local Anderson modules over $R$ and the category of effective local $\hat\sigma^*_{\!v}$-shtukas over $R$ given by the contravariant functor $\ulHM_{q_v}$ defined by $\ulHM_{q_v}(G):=\invlim[n] \ulHM_{q_v}\bigl(G[z_v^n]\bigr)$, where 
\[
\ulHM_{q_v}(G[z_v^n]) \;:=\; \bigl( \Hom_{R\text{\rm-groups},\BF_q\text{\rm-lin}}(G[z_v^n],\BG_{a,R}),\hat\tau_{M_q(G[z_v^n])}\bigr)
\]
and $\hat\tau_{M_q(G[z_v^n])}$ is provided by the relative $q_v$-Frobenius of the additive group scheme $\BG_{a,R}$ over $R$ like in \eqref{EqTauM}.
\end{Theorem}

It turns out that like with abelian Anderson $A$-modules, one can dispense with the notions of $z_v$-divisible local Anderson modules, because their equivalent description by local $\hat\sigma^*_{\!v}$-shtukas can be obtained purely from $A$-motives as in the following

\begin{Example}\label{AMotSectLocalShtukas} 
Let $\ulM=(M, \tau_M)$ be an $A$-motive over $K$ and assume that it has \emph{good reduction}, that is, there exist a pair $\ul\CM = (\CM, \tau_{\CM})$ consisting of a locally free module $\CM$ over $A_R:=A\otimes_{\BF_q} R$ of finite rank and a morphism $\tau_\CM:\sigma^*\CM\to \CM$ of $A_R$-modules whose cokernel is annihilated by a power of the ideal $\CJ: = (a\otimes1-1\otimes \gamma( a): \ a\in A)\subset A_R$, such that $\ul\CM\otimes_RL\cong\ulM\otimes_K L$. The reduction $\ul\CM \otimes_R \kappa$ is an $A$-motive over $\kappa$ of $A$-characteristic $v= \ker(\gamma: A\to \kappa)$. The pair $\ul\CM$ is called an \emph{$A$-motive over $R$} and a \emph{good model} of $\ulM$.

We consider the $v$-adic completions $A_{v, R}$ of $A_R$ and $\ul\CM \otimes_{A_R}A_{v, R}: = (\CM \otimes_{A_R}A_{v, R}, \tau_\CM\otimes \id)$ of $\ul\CM$. We let $d_v:=[\BF_v:\BF_q]$ and discuss the two cases $d_v = 1$ and $d_v> 1$ separately. If $d_v = 1$, and hence  $q_v = q$ and $\hat\sigma^*_{\!v} =\sigma^*$, we have $A_{v, R} = R\dbl z_v \dbr$, and $\ul\CM \otimes_{A_R}A_{v, R}$ is an effective local $\hat\sigma^*_{\!v}$-shtuka over $\Spec R$ which we denote by  $\ulHM_v(\ul\CM)$ and call the \emph{local $\hat\sigma^*_{\!v}$-shtuka at $v$ associated with $\ul\CM$}.

If $d_v> 1$, the situation is more complicated, because $\BF_v \otimes_{\BF_q}R$ and $A_{v,R}$ fail to be integral domains. Namely,
\[
\BF_v\otimes_{\BF_q} R = \prod_{\Gal(\BF_{v}/\BF_q)}\BF_v\otimes_{\BF_{v}}R = 
\prod_{i\in\BZ/d_v\BZ}\BF_v\otimes_{\BF_q} R\,/\,(a\otimes 1-1\otimes \gamma(a)^{q^i}:a\in \BF_{v})
\]
and $\sigma^*$ transports the $i$-th factor to the $(i+1)$-th factor. In particular $\hat\sigma_{\!v}^*$ stabilizes each factor. Denote by $\Fa_i$ the ideal of $A_{v,R}$ generated by $\{a\otimes 1-1\otimes \gamma(a)^{q^i}:a\in \BF_{v}\}$. Then
\[
A_{v,R} = \prod_{\Gal(\BF_{v}/\BF_q)}A_v\wh\otimes_{\BF_{v}}R = 
\prod_{i\in\BZ/d_v\BZ}A_{v,R}\,/\,\Fa_i.
\]
Note that each factor is isomorphic to $R\dbl z_v \dbr$ and the ideals $\Fa_i$ correspond precisely to the places $v_i$ of $C_{\BF_{v}}$ lying above $v$. The ideal $\CJ$ decomposes as follows $ \CJ\cdot A_{v,R}/\Fa_0 = (z_v-\zeta_v)$  and  $ \CJ\cdot A_{v,R}/\Fa_i = (1)$ for $i  \neq 0$. We define the \emph{local $\hat\sigma^*_{\!v}$-shtuka at $v$ associated with $\ul\CM$} as $\ulHM_v(\ul\CM):=(\hat M,\tau_{\hat M}):=\bigl(\CM\otimes_{A_R}A_{v,R}/\Fa_0\,,\,(\tau_\CM\otimes1)^{d_v}\bigr)$, where $\tau_\CM^{d_v}:=\tau_\CM\circ\sigma^*\tau_\CM\circ\ldots\circ\sigma^{(d_v-1)*}\tau_\CM$. Of course if $d_v=1$ we get back the definition of $\ulHM_v(\ul\CM)$ given above. Also note that $\CM/\tau_\CM(\sigma^*\CM)=\hat \CM/\tau_{\hat \CM}(\hat\sigma_{\!v}^*\hat \CM)$.

The local shtuka $\ulHM_v(\ul\CM)$ allows to recover $\ul\CM\otimes_{A_R}A_{v,R}$ via the isomorphism
\[
\bigoplus_{i=0}^{d_v-1}(\tau_\CM\otimes1)^i\mod\Fa_i\colon\Bigl(\bigoplus_{i=0}^{d_v-1}\sigma^{i*}(\CM\otimes_{A_R}A_{v,R}/\Fa_0),\;(\tau_\CM\otimes1)^{d_v}\oplus\bigoplus_{i\ne0}\id\Bigr)\;\isoto\;\ul\CM\otimes_{A_R}A_{v,R}\,,
\]
because for $i\ne0$ the equality $\CJ\!\cdot\! A_{v,R}/\Fa_i=(1)$ implies that $\tau_\CM\otimes 1$ is an isomorphism modulo $\Fa_i$; see \cite[Propositions~8.8 and 8.5]{BoHa11} for more details.
\end{Example}

\begin{Proposition}[{\cite[Theorem~7.6]{HA17}}]\label{PropLocGlobDieudonne}
Let $\ulG=(G,\phi)$ be a Drinfeld $A$-module over $R$ and let $\ulG[v^\infty]:=\dirlim\ulG[v^n]$ be its $z_v$-divisible local Anderson module over $R$ from Example~\ref{ExTorsionOfAModule}. Let $\ul\CM(\ulG)$ be the associated $A$-motive over $R$ and let $\ulHM_{q_v}(\ulG[v^\infty])$ be the associated local $\hat\sigma^*_{\!v}$-shtuka over $R$. Then $\ulHM_{q_v}(\ulG[v^\infty])$ is canonically isomorphic to the local $\hat\sigma^*_{\!v}$-shtuka $\ulHM_v(\ul\CM)$ from Example~\ref{AMotSectLocalShtukas}.

\end{Proposition}

\begin{Example}\label{ExCarlitzLocSht}
It was shown in \cite[Example~2.7]{HartlKim} that the local $\hat\sigma^*_{\!v}$-shtuka at $v$ associated with the Carlitz motive $\ulCC = (\CC = \BF_q(\theta)[t], \tau_\CC = t-\theta)$ from Example~\ref{CarMod} equals $\ulHM_v(\ulCC)=\bigl(\BF_v\dbl\zeta_v\dbr\dbl z\dbr,\tau_{\hat M}=(z_v-\zeta_v)\bigr)$. Here $L=\BF_v\dpl\zeta_v\dpr$ and $R=\CO_L=\BF_v\dbl\zeta_v\dbr$.
\end{Example}

Next we define the $v$-adic realization and the de Rham realization of a local shtuka $\ulHM = (\hat M, \tau_{\hat M})$ over $R$. Since $\tau_{\hat M}$ induces an isomorphism   $\tau_{\hat M}:  \hat\sigma_{\!v}^* \hat M  \otimes_{R\dbl z_v\dbr}L\dbl z_v\dbr \isoto \hat M  \otimes_{R\dbl z_v\dbr}L\dbl z_v\dbr$, we can think of $\ulHM\otimes_{R\dbl z_v\dbr}L\dbl z_v\dbr$  as an \'etale local shtuka over $L$.

 \begin{Definition}\label{dualTatemodule} The \emph{$v$-adic realization} $\Koh^1_v(\ulHM, A_v)$ of a local $\hat\sigma^*_{\!v}$-shtuka $\ulHM = (\hat M, \tau_{\hat M})$ is the $\sG_L$-module of $\tau$-invariants
\[
\Koh^1_v(\ulHM, A_v)\;:=\; (\hat M \otimes_{R\dbl z_v\dbr}L^\sep\dbl z_v\dbr)^{ \tau} : = \{m \in \hat M \otimes_{R\dbl z_v\dbr}L^\sep\dbl z_v\dbr:  \tau_{\hat M}( \hat\sigma_{\!\hat M}^* m) = m \} ,
\]
where we set $\hat\sigma_{\!\hat M}^*m:=m\otimes1\in\hat M\otimes_{R\dbl z_v\dbr,\hat\sigma^*_{\!v}}R\dbl z_v\dbr=:\sigma^*M$ for $m\in M$. One also writes sometimes $\check T_v\ulHM=\Koh^1_v(\ulHM, A_v)$ and calls this the \emph{dual Tate module of $\ulHM$}. By \cite[Proposition~4.2]{HartlKim} it is a free $A_v$-module of the same rank as  $\hat M$. We also write $\Koh^1_v(\ulHM,B):=\Koh^1_v(\ulHM,A_v)\otimes_{A_v}B$ for an $A_v$-algebra $B$.

If $\ulM = (M, \tau_M)$ is an $A$-motive over $L$ with good model $\ul\CM$ and $\ulHM = \ulHM_v(\ul\CM)$ is the local shtuka at $v$ associated with $\ul\CM$, then $\Koh^1_v(\ulHM, A_v)$ is by \cite[Proposition~4.6]{HartlKim} canonically isomorphic as a representation of $\sG_L$ to the $v$-adic realization $\Koh^1_v(\ulM, A_v)$ of $\ulM$.
\end{Definition}

\begin{Example}\label{ExCarlitzVAdic}
We describe the $v$-adic realization $\Koh^1_v(\ulCC, A_v)=\Koh^1_v(\ulHM_v(\ulCC), A_v)$ of the Carlitz module from Example~\ref{ExCarlitzLocSht} by using its local shtuka $\ulHM_v(\ulCC)=\bigl(\BF_v\dbl\zeta_v\dbr\dbl z_v\dbr,\tau_{\hat M}=(z_v-\zeta_v)\bigr)$ at $v$ computed there. For all $i\in\BN_0$ let $\tplusminus_i\in L^\sep$ be solutions of the equations $\tplusminus_0^{q_v-1}=-\zeta_v$ and $\tplusminus_i^{q_v}+\zeta_v \tplusminus_i=\tplusminus_{i-1}$. This implies $|\tplusminus_i|=|\zeta_v|^{q_v^{-i}/(q_v-1)}<1$. Define the power series $\tplus{v}=\sum_{i=0}^\infty \tplusminus_iz_v^i\in\CO_{L^\sep}\dbl z_v\dbr$. It satisfies $\hat\sigma^*_{\!v}(\tplus{v})=(z_v-\zeta_v)\cdot\tplus{v}$, but depends on the choice of the $\tplusminus_i$. A different choice yields a different power series $\ttplus{v}$ which satisfies $\ttplus{v}=u\tplus{v}$ for a unit $u\in(L^\sep\dbl z_v\dbr\mal)^{\hat\sigma^*_{\!v}=\id}=\BF_v\dbl z_v\dbr^{\SSC\times}=A_v^{^{\SSC\times}}$, because $\hat\sigma^*_{\!v}(u)=\frac{\hat\sigma^*_{\!v}(\ttplus{v})}{\hat\sigma^*_{\!v}(\tplus{v})}=\frac{\ttplus{v}}{\tplus{v}}=u$. The field extension $\BF_v\dpl\zeta_v\dpr(\tplusminus_i\colon i\in\BN_0)$ of $\BF_v\dpl\zeta_v\dpr$ is the function field analog of the cyclotomic tower $\BQ_p(\sqrt[p^i]{1}\colon i\in\BN_0)$; see \cite[\S\,1.3 and \S\,3.4]{HartlDict}. There is an isomorphism of topological groups called the \emph{$v$-adic cyclotomic character}
\[
\chi_v\colon\Gal\bigl(\BF_v\dpl\zeta_v\dpr(\tplusminus_i\colon i\in\BN_0)\big/\BF_v\dpl\zeta_v\dpr\bigr)\;\isoto\; A_v^{^{\SSC\times}},
\]
which satisfies $g(\tplus{v}):=\sum_{i=0}^\infty g(\tplusminus_i)z_v^i=\chi_v(g)\cdot \tplus{v}$ in $L^\sep\dbl z_v\dbr$ for $g$ in the Galois group. It is independent of the choice of the $\tplusminus_i$. The $v$-adic (dual) Tate module $\check T_v\ulHM=\Koh^1_v(\ulHM_v(\ulCC), A_v)$ of $\ulHM_v(\ulCC)$ and $\ulCC$ is generated by $(\tplus{v})^{-1}$ on which the Galois group acts by the inverse of the $v$-adic cyclotomic character. The reader should compare this to Example~\ref{ExGm}.
\end{Example}

\begin{Definition}\label{DefDeRham}
Let  $\ulHM = (\hat M, \tau_{\hat M})$  be a local $ \hat\sigma^*_{\!v}$-shtuka over $R$. We define the \emph{de Rham realizations} of  $\ulHM$ as
\begin{align*}
\Koh^1_{\dR}(\ulHM, R )\;:= \;& \hat\sigma_{\!v}^* \hat{ M}/(z_v-\zeta_v)\hat M \;=\;\hat\sigma_{\!v}^* \hat{ M} \otimes_{R\dbl z_v\dbr, z_v\mapsto \zeta_v} R\,, \quad\text{as well as}\\[1mm]
\Koh^1_{\dR}(\ulHM, L \dbl z_v -\zeta_v\dbr)\;:=\; & \hat\sigma_{\!v}^* \hat{ M} \otimes_{R\dbl z_v\dbr}  L\dbl z_v -\zeta_v\dbr \quad\text{and}  \\[1mm]
\Koh^1_{\dR}(\ulHM, L )\;:= \;& \hat\sigma_{\!v}^* \hat{ M} \otimes_{R\dbl z_v\dbr, z_v\mapsto \zeta_v} L \; =\;  \Koh^1_{\dR}(\ulHM, L \dbl z_v -\zeta_v\dbr)\otimes_{ L \dbl z_v -\zeta_v\dbr}\, L\dbl z_v -\zeta_v\dbr/(z_v -\zeta_v)\\[1mm]
=\; & \Koh^1_{\dR}(\ulHM, R )\otimes_RL\,.
\end{align*}
It carries the \emph{Hodge-Pink lattice} $\Fq^\ulHM:=\tau_{\hat M}^{-1}(\hat M\otimes_{R\dbl z_v\dbr}L\dbl z_v-\zeta_v\dbr)\subset \Koh^1_{\dR}(\ulHM, L \dbl z_v -\zeta_v\dbr)[\tfrac{1}{z_v-\zeta_v}]$. 

If $\ulM = (M, \tau_M)$ is an $A$-motive over $L$ with good model $\ul\CM$ and $\ulHM = \ulHM_v(\ul\CM)$ is the local shtuka at $v$ associated with $\ul\CM$ and $d_v = [\BF_v: \BF_q]$ is as in Example \ref{AMotSectLocalShtukas}, the map 
\[
\sigma^*\tau_M^{{d_v}-1}=\sigma^*\tau_M\circ\sigma^{2*}\tau_M\circ\cdots\circ\sigma^{({d_v}-1)*}\tau_M\colon\sigma^{{d_v}*}M\otimes_{A_R}A_{v,R}/\Fa_0\isoto\sigma^*M\otimes_{A_R}A_{v,R}/\Fa_0
\]
is an isomorphism, because $\tau_M$ is an isomorphism  over $A_{v, R}/\Fa_i$ for all  $i\neq 0$. Therefore, it defines canonical isomorphisms of the de Rham realizations 
\begin{align*}
\sigma^*\tau_M^{{d_v}-1}\colon & \Koh^1_{\dR}\bigl(\ulHM,L\dbl z_v-\zeta_v\dbr\bigr)\isoto\Koh^1_{\dR}\bigl(\ulM,L\dbl z_v-\zeta_v\dbr\bigr)\quad\text{and}\\
\sigma^*\tau_M^{{d_v}-1}\colon & \Koh^1_{\dR}(\ulHM,L)\isoto\Koh^1_{\dR}(\ulM,L)\,,
\end{align*}
which are compatible with the Hodge-Pink lattices and the Hodge-Pink filtrations.
\end{Definition}

The \emph{$v$-adic period isomorphism} for an $A$-motive $\ulM$ over a field $K\subset Q_v^\alg$ is provided by the following theorem by using the local $\hat\sigma^*_{\!v}$-shtuka $\ulHM:=\ulHM_v(\ulM)$.

\begin{Theorem}[{\cite[Theorem~4.14]{HartlKim}}] \label{ThmHv,dR}
If $\ulHM$ is a local $\hat\sigma^*_{\!v}$-shtuka over $R$ then there is a canonical comparison isomorphism 
\begin{equation*}
h_{v,\dR}\colon\;\Koh^1_v\bigl(\ulHM,Q_v)\otimes_{Q_v}\BC_v\dpl z_v-\zeta_v\dpr\;\isoto\;\Koh^1_{\dR}\bigl(\ulHM,L\dpl z_v-\zeta_v\dpr\bigr)\otimes_{L\dpl z_v-\zeta_v\dpr}\BC_v\dpl z_v-\zeta_v\dpr
\end{equation*}
If $\ulM$ is an $A$-motive over $L$ (which does not need to have good reduction) then there is a canonical comparison isomorphism 
\begin{equation}\label{EqHv,dR}
h_{v,\dR}\colon\;\Koh^1_v\bigl(\ulM,Q_v)\otimes_{Q_v}\BC_v\dpl z_v-\zeta_v\dpr\;\isoto\;\Koh^1_{\dR}\bigl(\ulM,L\dpl z_v-\zeta_v\dpr\bigr)\otimes_{L\dpl z_v-\zeta_v\dpr}\BC_v\dpl z_v-\zeta_v\dpr
\end{equation}
Both isomorphisms are equivariant for the action of $\sG_L$, where on the source this group acts on both factors of the tensor product and on the target it acts only on $\BC_v$. 
\end{Theorem}

In comparison with the $p$-adic comparison isomorphism for an abelian variety over a finite extension of $\BQ_p$ from Theorem~\ref{ThmPAdicPeriodIso}, the ring $\BC_v\dpl z_v-\zeta_v\dpr$ is the function field analog of $\BB_{p,\dR}$.

\begin{Example}\label{ExCarVDR}
For the Carlitz motive $\ulCC = (\CC = \BF_q(\theta)[t], \tau_\CC = t-\theta)$ from Example~\ref{CarMod} we have $\Koh^1_v(\ulCC,Q_v)=Q_v\cdot(\tplus{v})^{-1}\cong Q_v$ and \mbox{$\Koh^1_{\dR}(\ulCC,\BF_q(\theta)\dbl z_v-\zeta_v\dbr)=\BF_q(\theta)\dbl z_v-\zeta_v\dbr=:\Fp$}, see Example~\ref{ExCarlitzVAdic}. The Hodge-Pink lattice is $\Fq=(z_v-\zeta_v)^{-1}\Fp$ and the Hodge filtration satisfies $F^1=\Koh^1_{\dR}(\ulCC,\BF_q(\theta))\supset F^2=(0)$. With respect to the bases $(\tplus{v})^{-1}$ of $\Koh^1_v(\ulCC,Q_v)$ and $1$ of \mbox{$\Koh^1_{\dR}(\ulCC,\BF_q(\theta)\dbl z_v-\zeta_v\dbr)$} the comparison isomorphism $h_{v,\dR}$ from Theorem~\ref{ThmHv,dR} is given by the \emph{$v$-adic Carlitz period} \mbox{$(z_v-\zeta_v)^{-1}(\tplus{v})^{-1}=\hat\sigma_{\!v}^*(\tplus{v})^{-1}$}. It has a pole of order one at $z_v=\zeta_v$ because $\tplus{v}\in \BF_v\dpl\zeta_v\dpr^\sep\langle\tfrac{z_v}{\zeta_v}\rangle\mal\subset \BC_v\dbl z_v-\zeta_v\dbr\mal$. So $h_{v,\dR}\bigl(\Koh^1_v(\ulCC,Q_v)\otimes_{Q_v}\BC_v\dbl z_v-\zeta_v\dbr\bigr)=(z_v-\zeta_v)^{-1}\BC_v\dbl z_v-\zeta_v\dbr=\Fq\otimes_{K\dbl z_v-\zeta_v\dbr}\BC_v\dbl z_v-\zeta_v\dbr$.
\end{Example}

\begin{Definition}\label{DefNorm}
On the power series ring $\CO_{\BC_v}\dbl z_v\dbr$ we consider the $\CO_{\BC_v}$-embedding $\CO_{\BC_v}\dbl z_v\dbr\into\BC_v\dbl z_v-\zeta_v\dbr$ given by $z_v\mapsto z_v=\zeta_v+(z_v-\zeta_v)$. Let $\Theta : \BC_v\dbl z_v- \zeta_v\dbr \to \BC_v, \ z_v \mapsto \zeta_v$ be the residue map. Then $\CO_{\BC_v}\dbl z_v\dbr \cap \ker  \Theta$ is a principal ideal of  $\CO_{v}\dbl z_v\dbr$ generated by $z_v-\zeta_v$. Any other generator is of the form $(z_v-\zeta_v)\cdot u$ with $u \in \CO_{\BC_v}\dbl z_v\dbr\mal$. On $\BC_v\dpl z_v- \zeta_v\dpr$ we  define a valuation $\hat{v}$ by
\[
\hat{v}\left( \sum_{i = -N }^{\infty}b_i (z_v- \zeta_v)^i \right) : = \min\{i: \ b_i \neq 0\}.
\]
and we extend the  valuation $v$ on $ \BC_v$ to $\BC_v\dpl z_v- \zeta_v\dpr$ by
\begin{equation}\label{EqNorm}
v(f): = v\big(\Theta(f\cdot (z_v-\zeta_v)^{-\hat v(f)})\big).
\end{equation}
If $f$ and $g$ are two elements of $\BC_v((z_v- \zeta_v))$, then $\hat v(fg) = \hat v(f) + \hat v(g)$, and hence $v(fg) = v(f) + v(g)$. But note that $v$ does not satisfy the triangle inequality. The valuation $v(f)$ is unchanged, if we replace the generator $z_v-\zeta_v$ of $\CO_{\BC_v}\dbl z_v\dbr \cap \ker  \Theta$ by another generator $(z_v-\zeta_v)\cdot u$ with $u\in\CO_{\BC_v}\dbl z_v\dbr\mal$, because then $v\bigl(\Theta(f\cdot((z_v-\zeta_v)\cdot u)^{-\hat v(f)}\bigr)=v\bigl(\Theta(f\cdot(z_v-\zeta_v)^{-\hat v(f)}\bigr)+ v(\Theta(u))^{-\hat v(f)}=v\bigl(\Theta(f\cdot(z_v-\zeta_v)^{-\hat v(f)}\bigr)$ as $\Theta (u) \in \CO_{v}\mal$.
\end{Definition}

\begin{Example}\label{ExCarValuation}
The inverse $(z_v-\zeta_v)(\tplus{v})=\hat\sigma_{\!v}^*(\tplus{v})$ of the $v$-adic Carlitz period $\sigma_{\!v}^*(\tplus{v})^{-1}$ from Example~\ref{ExCarVDR} satisfies $\hat{v}\bigl((z_v-\zeta_v)(\tplus{v})\bigr)=1$ and $v_p(\hat\sigma_{\!v}^*(\tplus{v}))=v_p\bigl((z_v-\zeta_v)(\tplus{v})\bigr)=v_p(\Theta(\tplus{v}))=v_p(\sum_{i=0}^\infty\tplusminus_i\zeta_v^i)=v_p(\tplusminus_0)=\tfrac{1}{q_v-1}$, see Example~\ref{ExCarlitzVAdic}. The reader should compare this to Example~\ref{ExGm}.
\end{Example}

\section{Complex Multiplication}\label{SectCMAMot}
\setcounter{equation}{0}

\begin{Definition}\label{CMmotives} Let $\ulM$ be an $A$-motive over an $A$-field $K$. If $\QEnd_K(\ulM)$ contains a commutative  semi-simple  $Q$-algebra $E$  of dimension $\dim_{Q}E = \rk \ulM$, then we call $\ulM$ a \emph{CM $A$-motive over $K$} and we say that \emph{$\ulM$ has complex multiplication by $E$ over $K$}.
\end{Definition}

Here semi-simple means that $E$ is a product of fields. Note that we do not assume that $E$ is itself a field. By \cite[Theorem 4.2.5]{Schindler} any CM $A$-motive $\ulM$ is semi-simple. We know from \cite[Theorem 4.4.7]{Schindler}  if $\ulM$ is simple, uniformizable  then $\dim_{Q}\QEnd_K{(\ulM)} \leq \rk \ulM$ and  if in addition $\ulM$ has complex multiplication by $E$, then $E = \QEnd_K{(\ulM)}$ is a field.

Let $\ulM$ be an  $A$-motive over $K$ with complex multiplication through $E$ and let $\CO_E$ be  the integral closure of $A$ in $E$. If $E = \prod_i E_i$ is a product of finite field extensions of $Q$, then $\CO_E = \prod_i \CO_{E_i}$, where $\CO_{E_i}$ is the integral closure of $A$ in $E_i$. By \cite[Theorem 3.3.3]{Schindler} there exists an  $A$-motive $\ulM'$ isogenous to $\ulM$ such that $\CO_E \subseteq \End_K(\ulM')$. So for all aspects which only depend on the isogeny class of $\ulM$ we can assume that $\CO_E \subseteq \End_K(\ulM)$. Then  $M$ is a locally free module over the ring $\CO_E\otimes_{\BF_q} K$ and
\[
M = \bigoplus_i (M\otimes_{\CO_E}\CO_{E_i}).
\]
Since  $\CO_E \hookrightarrow \End_K(\ulM)$ is injective, $M\otimes_{\CO_E}\CO_{E_i}$  is a locally free module over the ring  $\CO_{E_i}\otimes_{\BF_q} K$ of rank $\geq 1$, because otherwise $\CO_{E_i}$ acts as $0$ on $\ulM$, which is a contradiction. Now the estimate
\begin{align*}
\rk_{A_K} M =  \sum_i \rk_{A_K} (M\otimes_{\CO_E}\CO_{E_i}) &= \sum_i \rk_{(\CO_{E_i}\otimes_{\BF_q} K)}(M\otimes_{\CO_E}\CO_{E_i})\cdot [E_i : Q]\\ 
&\geq\sum_i [E_i: Q] = [E:Q] = \rk_{A_K} M
\end{align*}
shows that $\rk_{(\CO_{E_i}\otimes_{\BF_q} L)}(M\otimes_{\CO_E}\CO_{E_i}) =1$  for all $i$. Therefore, $M$ is a locally free module over  $\CO_E\otimes_{\BF_q} K$ of rank $1$. Thus we have the following proposition.

\begin{Proposition}\cite[Proposition~3.3.5]{Schindler}\label{Proplocfreerank1}
Let $\ulM = (M, \tau_M)$ be an  $A$-motive over $K$ with complex multiplication $E$ such that $\CO_E \subseteq \End_K(\ulM)$, then 
\begin{enumerate} 
\item  M is a locally free $\CO_E\otimes_{\BF_q}K $-module of rank 1.
\item $\tau_M : \sigma^\ast M \to M$ is an  $\CO_E\otimes_{\BF_q}K $-linear injection.
\end{enumerate}
\end{Proposition}

\begin{Theorem}[{\cite[Theorem 6.3.6]{Schindler}}] \label{ThmSchindler}
Let $\ulM$ be an   $A$-motive over an $A$-field $K$ with complex multiplication $E$ such that $\CO_E \subseteq \End_K(\ulM)$ and $E$ is separable over $Q$. Then $\ulM$ is already defined over a finite separable extension $L$ of the $A$-field $\Quot(A/\AChar(K))$ which is $Q$ or a finite field, i.e.\ $\ulM\cong \ulM_L\otimes_L K$ for an $A$-motive $\ulM_L$ over $L$.
\end{Theorem}

\begin{Theorem}[{\cite[Section 3.6]{Pelzer09}}] \label{ThmPelzer}
If $\ulM$ is an $A$-motive defined over a finite extension $K/Q$ with complex multiplication by a separable $Q$-algebra $E$, then there exists a finite separable extension $L/K$ such that  $\ulM$ has good reduction at every prime of $\CO_L$.
\end{Theorem}

\begin{Remark}\label{RemCMDrinfMod}
If $\ulM=\ulM(\ulG)$ is the $A$-motive of a Drinfeld $A$-module $\ulG$ then both theorems are well known. Namely, in this case there is exactly one place of $E$ above $\infty$ by \cite[Proposition~4.7.17]{Goss}. Then $\ulG$ can be viewed as a Drinfeld $\CO_E$-module of rank $1$. All these are defined over the Hilbert class field of $E$ and have everywhere good reduction by \cite{Hayes79}, see \cite[Theorems~2.6.4 and 3.4.2]{Thakur04}.
\end{Remark}

\begin{Definition}\label{DefCMTypeAMot}
A \emph{CM-type} is a pair $(E,(d_\psi)_{\psi\in H_E})$ consisting of a finite dimensional, semi-simple, commutative $Q$-algebra $E$ and a tuple of integers $(d_\psi)_{\psi\in H_E}$ indexed by $H_E:=\Hom_Q(E,Q^\alg)$.

An \emph{isomorphism} $f\colon(E,(d_\psi)_{\psi\in H_E})\isoto(E',(d'_{\psi'})_{\psi'\in H_{E'}})$ of CM-types is an isomorphism $f\colon E\isoto E'$ of $Q$-algebras with $d_{\psi'\circ f}=d'_{\psi'}$ for all $\psi'\in H_{E'}$.
\end{Definition}

\begin{Remark}
The analog of a classical CM-type $(E,\Phi)$ as in Definition~\ref{DefAbVarCMType} would be a tuple $(d_\psi)_{\psi\in H_E}$ for which $d_\psi\in\{0,1\}$. Then one can set $\Phi:=\{\psi\in H_E\colon d_\psi=1\}$ and has $d_\psi=1$ for all $\psi\in \Phi$ and $d_\psi=0$ for all $\psi\in H_E\setminus\Phi$. But note, that we need a more flexible definition of CM-type here, due to the construction of the CM-type of a CM $A$-motive in Definition~\ref{DefCMTypeOfAMotive} below.
\end{Remark}

To prepare for this construction let $z \in Q$ be a uniformizer at $\infty$ and denote by $\zeta$ the image of $z$ in $Q^\alg$ under the natural inclusion $Q\subset Q^\alg$. We consider the power series ring $Q^\alg\dbl z-\zeta\dbr$ over $Q^\alg$ in the ``variable'' $z-\zeta$ as a $Q$-algebra via $z\mapsto\zeta+(z-\zeta)$. Let $E$ be a finite dimensional, semi-simple, commutative $Q$-algebra. Then by \cite[Lemma~A.3]{HartlSingh} there is a decomposition 
\begin{equation}\label{EqDecomp}
E\otimes_QQ^\alg\dbl z-\zeta\dbr\,=\,\prod_{\psi\in H_E}Q^\alg\dbl y_\psi-\psi(y_\psi)\dbr,
\end{equation}
where $y_\psi$ is a uniformizer at a place of $E$ such that $\psi(y_\psi)\ne0$. By \cite[Lemma~1.5]{HartlJuschka} the factors are obtained as the completion of $\CO_E\otimes_A A_{Q^\alg}=\CO_E\otimes_{\BF_q}Q^\alg$ along the kernels $(a\otimes1-1\otimes\psi(a)\colon a\in\CO_E)$ of the homomorphisms $\psi\otimes\id_{Q^\alg}\colon\CO_E\otimes_{\BF_q}Q^\alg\to Q^\alg$ for $\psi\in H_E$. If $(E,(d_\psi)_{\psi\in H_E})$ is a CM-type, then there is a finite free $Q^\alg\dbl z-\zeta \dbr $-submodule 
\begin{equation}\label{EqCMTypePowerSeries}
\Fq \,:=\,\prod_{\psi\in H_E}\bigl(y_\psi-\psi(y_\psi)\bigr)^{-d_\psi}\cdot Q^\alg\dbl y_\psi-\psi(y_\psi)\dbr\subseteq E \otimes_{Q}Q^\alg\dpl z-\zeta \dpr 
\end{equation}
with $\Fq\cdot Q^\alg\dpl z-\zeta \dpr= E\otimes_{Q}Q^\alg\dpl z-\zeta \dpr $. Conversely, every such $Q^\alg\dbl z-\zeta \dbr $-submodule $\Fq$ uniquely determines a tuple $(d_\psi)_{\psi\in H_E}$ of integers satisfying \eqref{EqCMTypePowerSeries}. So we could equivalently call $(E,\Fq)$ a ``CM-type''. In this description, an isomorphism $f\colon(E, \Fq)\isoto(E', \Fq')$ of CM-types is an isomorphism $f:E\isoto E'$ of $Q$-algebras which satisfies $(f\otimes\id_{Q^\alg\dpl z-\zeta \dpr})(\Fq)= \Fq'$.

\begin{Definition}\label{DefCMTypeOfAMotive}
Let $\ulM$ be an $A$-motive over a finite field extension $K\subset Q^\alg$ of $Q$ with complex multiplication through  $E$. We assume that $K$ contains $\psi(E)$ for all $\psi\in H_E$. Then the decomposition~\eqref{EqDecomp} exists already with $Q^\alg$ replaced by $K$. The $E\otimes_Q K\dbl z-\zeta\dbr$-module $\Koh^1_{\dR}(\ulM,K\dbl z-\zeta\dbr)$ is finite free of rank one, and correspondingly decomposes into eigenspaces 
\[
\Koh^\psi(\ulM,K\dbl y_\psi-\psi(y_\psi)\dbr)\;:=\;\Koh^1_{\dR}(\ulM,K\dbl z-\zeta\dbr)\otimes_{E\otimes_QK\dbl z-\zeta\dbr}\,K\dbl y_\psi-\psi(y_\psi)\dbr
\]
each of which is free of rank one over $K\dbl y_\psi-\psi(y_\psi)\dbr$, that is
\[
\Fp^{\ulM}\;:=\;\Koh^1_\dR(\ulM,K\dbl z-\zeta\dbr ) \;=\; \DS\prod_{\psi\in H_E}\Koh^\psi(\ulM,K\dbl y_\psi-\psi(y_\psi) \dbr)\,.
\]
Since the Hodge-Pink lattice $\Fq^\ulM$ from Definition~\ref{DefDeRham} is also an $E\otimes_Q K\dbl z-\zeta\dbr$-module and contains $\Fp^\ulM$, there are non-negative integers $d_\psi\in\BZ_{\ge0}$ such that 
\[
\Fq^\ulM\;=\;\prod_{\psi\in H_E}(y_\psi-\psi(y_\psi))^{-d_\psi}\Koh^\psi(\ulM,K\dbl y_\psi-\psi(y_\psi)\dbr)\,. 
\]
The tuple $(d_\psi)_{\psi\in H_E}$ is the \emph{CM-type} of $\ulM$. Since $\coker\tau_M=\Fq^\ulM/\Koh^1_{\dR}(\ulM,K\dbl z-\zeta\dbr)$ we see that $d_\psi$ is the dimension over $K$ of the generalized $\psi$-eigenspace of the action of $E$ on $\coker\tau_M$.

If we fix an isomorphism $\alpha\colon \Koh^1_{\dR}(\ulM,K\dbl z-\zeta\dbr)\isoto E\otimes_Q K\dbl z-\zeta\dbr$, then the CM-type of $\ulM$ can equivalently be described as $(E,\alpha(\Fq^\ulM))$. 
\end{Definition}

\begin{Example}\label{ExCMDriMod}
Let $\ulG$ be a Drinfeld $A$-module over an $A$-field $K$ of generic $A$-characteristic, such that $\ulM:=\ulM(\ulG)$ has CM by $\CO_E$ for a field extension $E$ of $Q$ with $[E:Q]=\rk\ulM=\rk\ulG$. By Remark~\ref{RemCMDrinfMod} we may assume that $K$ is a finite extension of $Q$, and we can fix an embedding $K\subset Q^\alg$. Theorem~\ref{ThmEquivDriMod} and Corollary~\ref{CorEndDriMod} imply that $\QEnd_K(\ulM)=\QEnd_K(\ulG)^{\rm op}$ is a (commutative) field extension of $Q$ of degree dividing $\rk\ulG$ and containing $E$. Thus, $E=\QEnd_K(\ulM)=\QEnd_K(\ulG)$. The field $E$ acts $K$-linearly on the one dimensional $K$-vector space $\Lie G$. Therefore, there is a $Q$-homomorphism $\psi_0\colon E\to \End_K(\Lie G)=K$, that is, an element $\psi_0\in H_E$ such that every $a\in E$ acts on $\Lie G$ via multiplication with $\psi_0(a)$. If $K$ contains $\psi(E)$ for all $\psi\in H_E$, then as $E$-modules, sequence \eqref{EqHodgeSeqDriMod} takes the form
\[
0\longto\bigoplus_{\psi\ne\psi_0}K_\psi \longto \Koh_{1,\dR}(\ulG,K) \longto K_{\psi_0}\longto0
\]
where $K_\psi$ denotes the $1$-dimensional $K$-vector space on which $E$ acts via $\psi$. In particular $\Lie\ulG=K_{\psi_0}$, and hence \eqref{EqHodgeSeqDriMod} is analogous to the decomposition \eqref{EqHodgeCMAbVar}. Since $\coker\tau_M\cong(\Lie\ulG)^\vee$ is $1$-dimensional with the induced $E$-action also given by $\psi_0$, the CM-type of $\ulG$ is $(E,(d_\psi)_{\psi\in H_E})$ with $d_{\psi_0}=1$ and $d_\psi=0$ for all $\psi\ne\psi_0$. This yields an isomorphism
\[
\tau_{M}\colon \big(y_{\psi_0}-{\psi_0}(y_{\psi_0})\big)^{-1}\Koh^{\psi_0}\big(\ulM,K\dbl y_{\psi_0}-{\psi_0}(y_{\psi_0})\dbr \big)\big/\Koh^{\psi_0}\big(\ulM,K\dbl y_{\psi_0}-{\psi_0}(y_{\psi_0})\dbr \big) \;=\; \Fq^{\ulM}/ \Fp^{\ulM} \;\isoto\; \coker \tau_{M}.
\]
Let $\omega_{\psi_0}\in\Koh^{\psi_0}(\ulM,K\dbl y_{\psi_0}-\psi_0(y_{\psi_0}) \dbr)$ be a $K\dbl y_{\psi_0}-\psi_0(y_{\psi_0})\dbr$-generator. Then $m:=(y_{\psi_0}-{\psi_0}(y_{\psi_0}))^{-1}\cdot\omega_{\psi_0}\in\Fq^\ulM$ and the image of $m$ in $\coker\tau_M\cong\Fq^\ulM/\Fp^\ulM$ generates the one dimensional $K$-vector space $\coker\tau_M$. In particular, if $E/Q$ is separable, we can take $y_{\psi_0}=z$ and $\psi_0(y_{\psi_0})=\zeta$ by \cite[Lemma~1.3]{HartlJuschka}. Then $y_{\psi_0}-{\psi_0}(y_{\psi_0})=z-\zeta$ and $K\dbl y_{\psi_0}-\psi_0(y_{\psi_0})\dbr=K\dbl z-\zeta\dbr$.
\end{Example}

\section{The Taguchi height of a Drinfeld module}\label{SectTagHeight}

Pushing the analogy between abelian varieties and Drinfeld modules forward, Taguchi~\cite[Section 5]{Tag} defined the analog of the Faltings height for Drinfeld modules. It is today called the \emph{Taguchi height}. Taguchi used it to prove the Tate Conjecture~\ref{MorTate} for Drinfeld modules. We follow the exposition of Wei~\cite[\S\,5.1]{Wei20}.

\begin{Definition}\label{DefLattice}
For an $A$-lattice $\Lambda\subset\BC_\infty$ of rank $r$, a $Q_\infty$-basis $\{\lambda_i\}_{1\leq i \leq r}$ of $Q_\infty \cdot \Lambda$ is called \emph{orthogonal} if  $\lambda_1,\ldots,\lambda_r$ satisfy that 
\begin{enumerate}
\item \label{DefLattice_A}
$\lambda_i \in \Lambda$ for $1\leq i \leq r$ ,
\item \label{DefLattice_B}
$|a_1 \lambda_1+\ldots + a_{r} \lambda_{r}|_{\infty} = \max\{ \,|a_i \lambda_i|_\infty; 1\leq i \leq r\,\}$ for all $a_1,\ldots,a_{r} \in Q_\infty$ ,
\item \label{DefLattice_C}
$Q_\infty \cdot \Lambda = \Lambda + (A_\infty \lambda_1 + \ldots + A_\infty \lambda_r)$ . 
\end{enumerate}
Note that if $\lambda_i \in Q\cdot\Lambda$ for $1\leq i \leq r$ such that $\DS\oplus_{i=1}^{r} Q\lambda_i = Q\cdot \Lambda$ and \ref{DefLattice_B} holds, then \ref{DefLattice_A} and \ref{DefLattice_C} can be achieved by multiplying all $\lambda_i$ with some $a \in A$ that has $v_\infty(a)\ll 0$. Then we define the \emph{$A$-volume} $D_A(\Lambda)$ of $\Lambda$ by 
\begin{equation}\label{EqD_A}
D_A( \Lambda) \;:=\; \left( \frac{\prod_{1\leq i \leq r} |\lambda_i|_\infty}{\#\big(\Lambda /(A \lambda_1 + \cdots +A \lambda_r)\big)}\right)^{1/r} =\; q^{1-g_Q}\cdot\left( \frac{ \prod_{1\leq i \leq r} |\lambda_i|_\infty}{\#\big(\Lambda \cap (A_\infty\lambda_1 + \cdots +A_\infty \lambda_r)\big)}\right)^{1/r} ,
\end{equation}
where $g_Q$ is the genus of $Q$
\end{Definition}

\begin{Example}\label{ExDA(OE)}
Let $E$ be a finite \emph{imaginary} field extension of $Q$, that is, $E_\infty:=E\otimes_Q Q_\infty$ is still a field. Then the absolute value $|\,.\,|_\infty$ on $Q_\infty$ extends in a unique way to an absolute value on $E_\infty$. The latter equals the restriction of the absolute value $|\,.\,|_\infty$ on $\BC_\infty$ for any $Q_\infty$-embedding $E_\infty\into\BC_\infty$. Under any such embedding $\CO_E$ is an $A$-lattice in $\BC_\infty$ of rank $[E:Q]$, and we can define $D_A(\CO_E)$, which is independent of the chosen embedding. If the ramification of $\infty$ in $E/Q$ is tame then 
\[
\log D_A(\CO_E) \;=\; \frac{1}{2[E:Q]}\cdot \log \#(A/\Fd_{\CO_E/A})
\]
by \cite[Remark~5.6]{Wei20} where $\Fd_{\CO_E/A}$ is the (relative) discriminant of $\CO_E$ over $A$.
\end{Example}

For the Taguchi height \cite[\S\,5]{Tag} of a Drinfeld module the following alternative, equivalent definition was given by Wei~\cite[\S\,5.1]{Wei20}.

\begin{Definition}[{\cite[\S\,5]{Tag}, \cite[\S\,5.1]{Wei20}}] \label{DefTagHeight}
Let $\ulG=(G,\phi)$ be a Drinfeld $A$-module of rank $r$ over a finite field extension $K\subset Q^\alg$ of $Q$. For every $\eta\in H_K:=\Hom_Q(K,Q^\alg)$ the embedding $\eta\colon K\into Q^\alg\subset Q_v^\alg$ allows to restrict the valuation $v$ on $Q_v^\alg$ to a valuation, that is, a place $\tilde v_\eta$ of $K$, such that the completion $K_{\tilde v_\eta}$ equals the closure of $\eta(K)$ in $Q_v^\alg$.  Conversely, for each place $\tilde v$ of $K$ with $\tilde v|v$, we let $K_{\tilde v}$ be the completion of $K$ at $\tilde v$. We choose a $Q_v$-embedding $\eta\colon K_{\tilde v}\into Q_v^\alg$ and the induced $Q$-embedding $\eta\colon K\into Q^\alg$. Then $\tilde v=\tilde v_\eta$. In this way the place $\tilde v$ is obtained $[K_{\tilde v}:Q_v]$-many times. We let $\ulG^\eta=(G^\eta,\phi^\eta)$ be the base change of $\ulG$ to $Q^\alg$ via $\eta\colon K\into Q^\alg$ and also to $\BC_\infty$ via the fixed inclusion $Q^\alg\subset\BC_\infty$. 

We choose an isomorphism $m\colon G\isoto\BG_{a,K}$ and consider the induced isomorphisms $m^\eta\colon G^\eta\isoto\BG_{a,Q^\alg}$ and $\Lie m^\eta\colon\Lie G^\eta\isoto Q^\alg$ for every $\eta\in H_K$. The \emph{local height of $\ulG$ at $\wt\infty_\eta$ with respect to $m$} is given by
\begin{equation}\label{EqTagHtInfinity}
ht_{{\rm Tag},\wt\infty_\eta}(\ul G/K) \;:=\; -[K_{\wt\infty_\eta}: Q_\infty] \cdot \log_q D_A\bigl( \Lie m^\eta\bigl(\Koh_{1,\Betti}(\ulG^\eta,A)\bigr)\bigr).
\end{equation}
To define the local height of $\ulG$ at a finite place ${\tilde v_\eta}$ of $K$ with ${\tilde v_\eta}|v\ne \infty$ we write
$$
m^\eta\circ \phi^\eta_a\circ (m^\eta)^{-1} \;=\; \gamma(a)+ \sum_{i=1}^{r \deg a} \phi^\eta_{a,i} \,\tau^i \;\in\; \End_{Q^\alg,\BF_q}(\BG_{a,Q^\alg})\;=\;Q^\alg\{\tau\}\qquad\text{with}\quad \phi^\eta_{a,i}\;\in\; Q^\alg\,.
$$
for each $a \in A$. We put $\ord_{\tilde v_\eta}(\ul G) := \min\biggl\{ \DS\frac{e(\tilde v_\eta|v)\cdot v(\phi^\eta_{a,i})}{q^i-1}: a \in A\setminus\BF_q, \ 1\leq i \leq r\deg a \biggr\}$, where $e(\tilde v_\eta|v)$ is the ramification index of $\tilde v_\eta$ in $K/Q$. The \emph{local height of $\ulG$ at ${\tilde v_\eta}$ with respect to $m$} is given by
\begin{equation}\label{EqTagHtInfinity2}
ht_{{\rm Tag},{\tilde v_\eta}}(\ul G/K) \;:=\; -[\BF_{\tilde v_\eta}:\BF_q] \cdot \lfloor\ord_{{\tilde v_\eta}}(\ul G)\rfloor\,,
\end{equation}
where $\lfloor x\rfloor$ denotes the largest integer $n\le x$, and $\BF_{\tilde v_\eta}$ is the residue field of ${\tilde v_\eta}$. 

Then the \emph{Taguchi height} $ht_{\rm Tag}(\ulG/K)$ of $\ulG$ is defined by taking the sum over all places of $K$
\begin{equation}\label{EqTagHtInfinity3}
ht_{\rm Tag}(\ul G/K)\;:=\; \frac{1}{[K:Q]} \cdot \biggl(\sum_{{\tilde v} \nmid \infty} ht_{{\rm Tag},{\tilde v}}(\ul G/K) + \sum_{\wt\infty\mid \infty} ht_{{\rm Tag},\wt\infty}(\ul G/K)\biggr).
\end{equation}
It does not depend on the isomorphism $m$.
\end{Definition}

\begin{Remark}
(1) Let $K'$ be a finite field extension of $K$. Let $\eta'\colon K'\into Q^\alg$ be a $Q$-homomorphism and let $\eta\colon K\into Q^\alg$ be its restriction to $K$. Let $\wt\infty'_{\eta'}$ and $\wt\infty_\eta$ be the corresponding places of $K'$ and $K$, respectively. It is clear that $\Lie m\bigl(\Koh_{1,\Betti}(\ulG^{\eta'},A)\bigr)=\Lie m\bigl(\Koh_{1,\Betti}(\ulG^{\eta},A)\bigr) \subset \BC_\infty$, and
$$ht_{{\rm Tag},\wt\infty'_{\eta'}}(\ul G/K') \;=\; [K'_{\wt\infty'_{\eta'}}: K_{\wt\infty_\eta}] \cdot ht_{{\rm Tag},\wt\infty_\eta}(\ul G/K).$$
For places ${\tilde v}$ of $K$ and ${\tilde v}'$ of $K'$ with ${\tilde v}'\mid {\tilde v} \nmid \infty$, one has $\ord_{{\tilde v}'}(\ul G) = e({\tilde v}'|{\tilde v}) \cdot \ord_{{\tilde v}}(\ul G)$, where $e({\tilde v}'|{\tilde v})$ is the ramification index of ${\tilde v}' / {\tilde v}$. Thus we get
$$
ht_{{\rm Tag},{\tilde v}'}(\ul G/K') \;\leq\; [K'_{{\tilde v}'}:K_{\tilde v}] \cdot ht_{{\rm Tag},{\tilde v}}(\ul G/K).
$$
In particular, assume that $\ul G$ has stable reduction at ${\tilde v}$, that is, there is an $x\in K_{\tilde v}$ such that $v(x^{q^i-1}\phi^\eta_{a,i})\ge0$ for all $i$ and $a$, and for every $a\in A\setminus\BF_q$ there is an $i\ge1$ such that $v(x^{q^i-1}\phi^\eta_{a,i})=0$. Then $\ord_{\tilde v}(\ul G)=-e(\tilde v|v)\cdot v(x)=-\tilde v(x)$ is an integer, which implies that
$ht_{{\rm Tag},{\tilde v}'}(\ul G/K') = [K'_{{\tilde v}'}:K_{\tilde v}] \cdot ht_{{\rm Tag},{\tilde v}}(\ul G/K)$. In conclusion, we have  $ht_{\rm Tag}(\ul G/K') \leq ht_{\rm Tag}(\ul G/K)$, and the equality holds when $\ulG$ has stable reduction everywhere.

\medskip\noindent
(2) Note that every Drinfeld $A$-module $\ul G$ over $K$ has potentially stable reduction everywhere by \cite[Proposition~7.1]{Drinfeld}. Define the \textit{stable Taguchi height of $\ul G$} as
$$
ht_{\rm Tag}^{\rm st}(\ul G)\;:=\; \log q \cdot \lim_{K'/K\text{ finite}} ht_{\rm Tag}(\ul G/K'),
$$
which is always convergent by (1).

\medskip\noindent
(3) Let $\ul G$ and $\ul G'$ be two Drinfeld $A$-modules over $Q^\alg$ which are isomorphic over $Q^\alg$. Then 
$$ht_{\rm Tag}^{\rm st}(\ul G) \;=\; ht_{\rm Tag}^{\rm st}(\ul G') \,.$$
\end{Remark}

\section{The Analog of Colmez's Conjecture for CM $A$-Motives}\label{SectColmezConjAMot}
\setcounter{equation}{0}

In \cite{HartlSingh} the authors have formulated the analog of Colmez's conjecture (Section~\ref{SectColmezConjAbVar}) for periods of CM $A$-motives. We consider the following

\begin{Situation}\label{SitCMAMot}
Let $\ulM$ be a uniformizable $A$-motive over a finite extension $K\subset Q^\alg$ of $Q$ with complex multiplication of CM-type $(E,(d_\psi)_{\psi\in H_E})$, in the sense of Definition~\ref{DefCMTypeAMot} such that $E$ is a product of \emph{separable} field extensions of $Q$ and $\ulM$ has complex multiplication by the ring of integers $\CO_E$ of $E$. As an abbreviation we denote the CM-Type of $\ulM$ by $(E,\Phi)$ with $\Phi=(d_\psi)_{\psi\in H_E}$. Let $H_E:=\Hom_Q(E,Q^\alg)$ be the set of all $Q$-homomorphisms $E\into Q^\alg$ and assume that $K$ contains $\psi(E)$ for every $\psi\in H_E$. By Theorems~\ref{ThmSchindler} and \ref{ThmPelzer} we may assume moreover, that $K$ is a finite Galois extension of $Q$ and that $\ulM$ has good reduction at every prime of $K$. For a fixed $\psi\in H_E$ let $\omega_\psi$ be a generator of the $K\dbl y_\psi-\psi(y_\psi)\dbr$-module $\Koh^\psi(\ulM,K\dbl y_\psi-\psi(y_\psi)\dbr)$. The image of $\omega_\psi$ in $\Koh^1_{\dR}(\ulM,K)$ is non-zero and satisfies $a^*\omega_\psi=\psi(a)\cdot\omega_\psi$ for all $a\in E$. For every embedding $\eta\colon K\into Q^\alg$, let $\ulM^\eta:=\ulM\otimes_{K,\eta}K$ and $\omega_\psi^\eta\in\Koh^{\eta\psi}(\ulM^\eta,K\dbl y_{\eta\psi}-\eta\psi(y_{\eta\psi})\dbr)$ be deduced from $\ulM$ and $\omega_\psi$ by base extension, and let $u_\eta\in\Koh_{1,\Betti}(\ulM^\eta,Q)=\Hom_A\bigl(\Koh^1_\Betti(\ulM^\eta,A),Q\bigr)$ be an $E$-generator. Let $v$ be a place of $Q$. 

If $v=\infty$ the pairing \eqref{Eq3PeriodIso} from Theorem~\ref{PeriodIso} between Betti and de Rham cohomology gives a pairing
\[
\langle\,.\;,\,.\,\rangle_\infty\colon \Koh_{1,\Betti}(\ulM^\eta,Q) \times \Koh^1_{\dR}(\ulM^\eta,K) \;\longto\;\BC_\infty\,,\quad (u_\eta,\omega_\psi^\eta)\; \longmapsto\; \langle u_\eta,\omega_\psi^\eta\rangle_\infty\;=:\;{\TS \int_{u_\eta}}\omega_\psi^\eta\;.
\]
We define the absolute value $\bigl|\int_{u_\eta}\omega_\psi^\eta\bigr|_\infty:=|\langle u_\eta,\omega_\psi^\eta\rangle_\infty|_\infty=q_\infty^{-v_\infty\left(\langle u_\eta,\omega_\psi^\eta\rangle_\infty\right)}\in\BR$. 

If $v\subset A$ is a maximal ideal, the comparison isomorphism $h_{\Betti,v}$ from \eqref{Eq1PeriodIso} in Theorem~\ref{PeriodIso} between Betti and $v$-adic cohomology together with the comparison isomorphism $h_{v,\dR}$ between $v$-adic and de Rham cohomology from \eqref{EqHv,dR} in Theorem~\ref{ThmHv,dR} yield a pairing
\begin{eqnarray*}
\langle\,.\;,\,.\,\rangle_v\colon & \Koh_{1,\Betti}(\ulM^\eta,Q) \times \Koh^1_{\dR}(\ulM^\eta,K) & \longto\es\BC_v\dpl z_v-\zeta_v\dpr\,,\\[1mm]
& \qquad(u_\eta,\omega_\psi^\eta)\; & \longmapsto\es \langle u_\eta,\omega_\psi^\eta\rangle_v\;:=\; u_\eta\otimes\id_{\BC_v\dpl z_v-\zeta_v\dpr}\bigl(h_{\Betti,v}^{-1}\circ h_{v,\dR}^{-1}(\omega_\psi^\eta)\bigr)\,.
\end{eqnarray*}
We define the absolute value $\bigl|\int_{u_\eta}\omega_\psi^\eta\bigr|_v:=|\langle u_\eta,\omega_\psi^\eta\rangle_v|_v:=q_v^{-v\left(\langle u_\eta,\omega_\psi^\eta\rangle_v\right)}\in\BR$, where the ``valuation'' $v$ on $\BC_v\dpl z_v-\zeta_v\dpr$ was defined in \eqref{EqNorm} in Definition~\ref{DefNorm}. 
\end{Situation}

In analogy with Section~\ref{SectColmezConjAbVar} we now consider the product $\prod\limits_v\prod\limits_{\eta\in H_K}\bigl|\int_{u_\eta}\omega_\psi^\eta\bigr|_v$ over all places $v$ of $Q$, or equivalently $\tfrac{1}{\#H_K}$ times its logarithm 
\begin{equation}\label{EqSumAMot}
\TS\tfrac{1}{\#H_K}\sum\limits_v\sum\limits_{\eta\in H_K}\log\bigl|\int_{u_\eta}\omega_\psi^\eta\bigr|_v\;=\;\tfrac{1}{\#H_K}\sum\limits_{\eta\in H_K}\log\bigl|\int_{u_\eta}\omega_\psi^\eta\bigr|_\infty-\tfrac{1}{\#H_K}\sum\limits_{v\ne\infty}\sum\limits_{\eta\in H_K}v\bigl(\int_{u_\eta}\omega_\psi^\eta\bigr)\log q_v\,.
\end{equation}
Again the right sum over all $v\ne\infty$ does not converge. Namely, we prove in \cite[Theorem~1.3]{HartlSingh} the following Theorem~\ref{ThmValueAtV} below. To formulate the theorem we recall Definition~\ref{DefArtinMeasure}. For our CM-type $(E,\Phi)$ and for every $\psi\in H_E$ we define the functions
\begin{align}\label{Eq:a_E}
a_{E,\psi,\Phi}\colon\;\sG_Q\to\BZ, & \quad g\mapsto d_{g\psi}\qquad\text{and}\\[2mm]
a^0_{E,\psi,\Phi}\colon\;\sG_Q\to\BQ, & \quad g\mapsto \tfrac{1}{\#H_K}\TS\sum\limits_{\eta\in H_K} a_{E,\eta\psi,\eta\Phi}(g) \;=\; \tfrac{1}{\#H_K}\TS\sum\limits_{\eta\in H_K}d_{\eta^{-1}g\eta\psi} \label{Eq:a0_E}
\end{align}
which factor through $\Gal(K/Q)$ by our assumption that $\psi(E)\subset K$ for all $\psi\in H_E$. In particular, $a_{E,\psi,\Phi}\in\CC(\sG_Q,\BQ)$ and $a^0_{E,\psi,\Phi}\in\CC^0(\sG_Q,\BQ)$ is independent of $K$. 

We also define integers $v(\omega_\psi^\eta)$ and $v_{\eta\psi}(u_\eta)$ for all $v\ne\infty$ which are all zero except for finitely many. Let $\CO_{E_v}:=\CO_E\otimes_AA_v$ and let $c\in E_v:=E\otimes_Q Q_v$ be such that $c^{-1}u_\eta$ is an $\CO_{E_v}$-generator of $\Koh_{1,\Betti}(\ulM^\eta,A)\otimes_A A_v$ $=\Koh_{1,v}(\ulM^\eta,A_v)$, which exists because $\CO_{E_v}$ is a product of discrete valuation rings. Then $c$ is unique up to multiplication by an element of $\CO_{E_v}\mal$ and we set
\begin{equation}\label{EqVPsi}
v_{\eta\psi}(u_\eta)\;:=\;v\bigl(\eta\psi(c)\bigr)\;\in\;\BQ\,,
\end{equation}
where we extend $\eta\psi\in H_E$ by continuity to $\eta\psi\colon E_v\to Q_v^\alg$. 

Also let $K_v$ be the $v$-adic completion of $K\subset Q^\alg\subset Q_v^\alg\subset\BC_v$ and let $\ul\CM^\eta=(\CM^\eta,\tau_{\CM^\eta})$ be an $A$-motive over $\CO_{K_v}$ with good reduction and $\ul\CM^\eta\otimes_{\CO_{K_v}}K_v\cong\ulM^\eta\otimes_K K_v$; see Example~\ref{AMotSectLocalShtukas}. On $\Koh^1_\dR(\ulM^\eta,K_v)$ we consider the following two integral structures arising from $\Koh^1_\dR(\ul\CM^\eta,\CO_{K_v}):=\sigma^*\CM^\eta\otimes_{A_{\CO_{K_v}},\,\gamma\otimes\id_{\CO_{K_v}}}\CO_{K_v}$
\begin{alignat*}{4}
& \wt\Koh{}^{\eta\psi}(\ul\CM^\eta,\CO_{K_v}) & \;:=\; & \bigl\{\omega\in\Koh^1_\dR(\ul\CM^\eta,\CO_{K_v})\colon [a]^*\omega=\eta\psi(a)\cdot\omega\es\forall\;a\in\CO_E\bigr\}\qquad\text{and}\\[2mm]
& \Koh^{\eta\psi}(\ul\CM^\eta,\CO_{K_v}) & \;:=\; & \Koh^1_\dR(\ul\CM^\eta,\CO_{K_v})\big/([a]^*-\eta\psi(a)\colon a\in\CO_E)\cdot\Koh^1_\dR(\ul\CM^\eta,\CO_{K_v})\,.
\end{alignat*}
By \cite[Lemma~1]{HartlSinghErr} (see also the arXiv version of \cite[Lemma~B.1]{HartlSingh}) these are free $\CO_{K_v}$-modules of rank one contained in
\[
\Koh^{\eta\psi}(\ulM^\eta,K_v) \;=\;\wt\Koh{}^{\eta\psi}(\ul\CM^\eta,\CO_{K_v})\otimes_{\CO_{K_v}} K_v \;=\;\Koh^{\eta\psi}(\ul\CM^\eta,\CO_{K_v})\otimes_{\CO_{K_v}} K_v
\]
and satisfying $\wt\Koh{}^{\eta\psi}(\ul\CM^\eta,\CO_{K_v})\subset\Koh^{\eta\psi}(\ul\CM^\eta,\CO_{K_v})$ with $\Koh^{\eta\psi}(\ul\CM^\eta,\CO_{K_v})\big/\wt\Koh{}^{\eta\psi}(\ul\CM^\eta,\CO_{K_v})\cong \CO_{K_v}/\eta\psi(\FD_{\CO_E/A})$, where $\FD_{\CO_E/A}$ is the different of $\CO_E$ over $A$.
Then there are elements $\tilde x,x\in K_v\mal$, unique up to multiplication by $\CO_{K_v}\mal$, such that 
\begin{eqnarray*}
& & \text{$\tilde x^{-1}\omega_\psi^\eta\mod y_{\eta\psi}-\eta\psi(y_{\eta\psi})$ is an $\CO_{K_v}$-generator of $\wt\Koh{}^{\eta\psi}(\ul\CM^\eta,\CO_{K_v})$ and}\\[1mm]
& & \text{$x^{-1}\omega_\psi^\eta\mod y_{\eta\psi}-\eta\psi(y_{\eta\psi})$ is an $\CO_{K_v}$-generator of $\Koh{}^{\eta\psi}(\ul\CM^\eta,\CO_{K_v})\,.$}
\end{eqnarray*}
We set
\begin{eqnarray}
\label{EqFalseValuationOfOmegaPsiAMot}
v^\sim(\omega_\psi^\eta) & := & v(\tilde x)\;\in\;\BQ\quad\text{and} \\[1mm]
\label{EqValuationOfOmegaPsiAMot}
v(\omega_\psi^\eta) & := & v(x)\;\in\;\BQ\,.
\end{eqnarray}
Then 
\begin{equation*}
v(\omega_\psi^\eta)-v^\sim(\omega_\psi^\eta)\;=\;v\bigl(\eta\psi(\FD_{\CO_E/A})\bigr)\;=\;v(\FD_{\eta\psi(E_v)/Q_v})
\end{equation*}
by \cite[Corollary~2]{HartlSinghErr} (see also the arXiv version of \cite[Corollary~B.2]{HartlSingh}), and consequently
\begin{align}
\nonumber \TS\sum\limits_{\eta\in H_K} v(\omega_\psi^\eta)-v^\sim(\omega_\psi^\eta) & \;=\; \TS\sum\limits_{\eta\in H_K} v\bigl(\eta\psi(\FD_{\CO_E/A})\bigr) \;=\;\TS v\Bigl(\,\prod\limits_{\eta\in H_K}\eta\psi(\FD_{\CO_E/A})\Bigr)\;=\;v\bigl(N_{K/Q}(\FD_{\psi(\CO_E)/A})\bigr)\\[1mm]
\nonumber &\;=\;v\Bigl(N_{\psi(E)/Q}\bigl(N_{K/\psi(E)}(\FD_{\psi(\CO_E)/A})\bigr)\Bigr)\;=\;[K:\psi(E)]\cdot v(\Fd_{\psi(\CO_E)/A})\quad\text{and}\\[2mm]
\TS\sum\limits_{\eta\in H_K}\sum\limits_{v\ne\infty} \bigl(v(\omega_\psi^\eta)-v^\sim & (\omega_\psi^\eta)\bigr)\log q_v \;=\; [K:\psi(E)]\cdot \log\#(A/\Fd_{\psi(\CO_E)/A})\,. \label{EqRelationBetweenThe_v}
\end{align}
These value only depend on the image of $\omega_\psi^\eta$ in $\Koh^1_\dR(\ulM^\eta,K)$. They also do not depend on the choice of the model $\ul\CM^\eta$ with good reduction, because all such models are isomorphic over $\CO_{K_v}$ by \cite[Proposition~2.13(ii)]{Gardeyn4}.

\begin{Remark}
In \cite[Formula~(1.13) and Definition~4.10]{HartlSingh} there is an error in the definition of $v(\omega_\psi^\eta)$. Namely, there $v(\omega_\psi^\eta)$ is defined to be $v^\sim(\omega_\psi^\eta)$ as in \eqref{EqFalseValuationOfOmegaPsiAMot}. However, in the rest of \cite{HartlSingh} the above definition \eqref{EqValuationOfOmegaPsiAMot} for $v(\omega_\psi^\eta)$ is used; see \cite{HartlSinghErr} or the arXiv version of \cite[Erratum~B]{HartlSingh}.
\end{Remark}

In \cite[Theorem~1.3]{HartlSingh} we computed the terms in \eqref{EqSumAMot} as follows, where we use \eqref{EqRelationBetweenThe_v} and the logarithmic derivative $Z_v$ of the Artin $L$-function from \eqref{EqZFct2} in Definition~\ref{DefArtinMeasure}.

\begin{Theorem}\label{ThmValueAtV}
Let $\Fd_{\psi(\CO_E)/A}$ denote the discriminant of the extension of Dedekind rings $\psi(\CO_E)/A$. Then for every $v\ne\infty$ we have
\begin{eqnarray*}
\TS\tfrac{1}{\#H_K}\sum\limits_{\eta\in H_K}v({\TS\int_{u_\eta}\omega_\psi^\eta}) & = & Z_v(a^0_{E,\psi,\Phi},1)-\mu_{\Art,v}(a^0_{E,\psi,\Phi})-\dfrac{v(\Fd_{\psi(\CO_E)/A})}{[\psi(E):Q]}+\tfrac{1}{\#H_K}\sum\limits_{\eta\in H_K}\bigl(v(\omega_\psi^\eta)+v_{\eta\psi}(u_\eta)\bigr) \\[2mm]
& = & Z_v(a^0_{E,\psi,\Phi},1)-\mu_{\Art,v}(a^0_{E,\psi,\Phi})+\tfrac{1}{\#H_K}\sum\limits_{\eta\in H_K}\bigl(v^\sim(\omega_\psi^\eta)+v_{\eta\psi}(u_\eta)\bigr)\,.
\end{eqnarray*}
This formula holds more generally for all tuples of $E_v$-generators $u_\eta\in\Koh_{1,\Betti}(\ulM^\eta,Q_v)=\Koh_{1,v}(\ulM^\eta,Q_v)$.
\end{Theorem}

Since $-\mu_{\Art,v}(a^0_{E,\psi,\Phi})-\dfrac{v(\Fd_{\psi(\CO_E)/A})}{[\psi(E):Q]}+\tfrac{1}{\#H_K}\sum\limits_{\eta\in H_K}\bigl(v(\omega_\psi^\eta)+v_{\eta\psi}(u_\eta)\bigr)$ vanishes for all but finitely many places $v$ and $\sum\limits_{v\ne\infty}Z_v(a^0_{E,\psi,\Phi},1)$ diverges, the sum \eqref{EqSumAMot} diverges. But as in Section~\ref{SectColmezConjAbVar} we can assign to this divergent sum a value by the following

\begin{Convention}\label{ConventionAMot} Let $(x_v)_{v\ne\infty}$ be a tuple of complex numbers indexed by the finite places $v$ of $Q$. We will give a sense to the (divergent) series $\Sigma \stackrel{?}{=} \sum_{v\ne \infty} x_v$ in the following situation. We suppose that there exists an element $a\in\CC^0(\sG_Q,\BQ)$ such that $x_v = - Z_v(a,1) \log q_v$  for all $v$ except for finitely many. Then we let $a^*\in\CC^0(\sG_Q,\BQ)$ be defined by $a^*(g):=a(g^{-1})$. We further assume that $Z^\infty(a^*,s)$ does not have a pole at $s=0$, and we define the limit of the series $\sum_{v\ne \infty} x_v$ as
\begin{equation}\label{EqConventionAMot}
\Sigma\;:=\;-Z^\infty(a^*,0) - \mu^\infty_\Art(a) +\sum_{v\neq \infty} \bigl(x_v + Z_v(a,1)\log q_v\bigr)
\end{equation}
inspired by Weil's \cite[p.~82]{WeilRH} functional equation 
\[
Z(\chi,1-s)\;=\;-Z(\chi^*,s)-(2\cdot genus(C)-2)\chi(1)\log q-\mu_\Art(\chi)
\]
deprived of the summands at $\infty$, where the genus term is considered as belonging to $\infty$.
\end{Convention}

Convention~\ref{ConventionAMot}, Theorem~\ref{ThmValueAtV} and \eqref{EqRelationBetweenThe_v} allow us to give to the divergent sum \eqref{EqSumAMot} the convergent interpretation 
\begin{align}\label{EqConvSumAMot}
\nonumber & -Z^\infty((a^0_{E,\psi,\Phi})^*,0)+\frac{\log\#(A/\Fd_{\psi(\CO_E)/A})}{[\psi(E):Q]}+\tfrac{1}{\#H_K}\sum\limits_{\eta\in H_K}\Bigl(\log\bigl|{\TS\int_{u_\eta}}\omega_\psi^\eta\bigr|_\infty-\sum_{v\ne\infty}\bigl(v(\omega_\psi^\eta)+v_{\eta\psi}(u_\eta)\bigr)\log q_v\Bigr) \\[2mm]
& =\es -Z^\infty((a^0_{E,\psi,\Phi})^*,0)+\tfrac{1}{\#H_K}\sum\limits_{\eta\in H_K}\Bigl(\log\bigl|{\TS\int_{u_\eta}}\omega_\psi^\eta\bigr|_\infty-\sum_{v\ne\infty}\bigl(v^\sim(\omega_\psi^\eta)+v_{\eta\psi}(u_\eta)\bigr)\log q_v\Bigr).
\end{align}

\newcommand{\secondeta}{{\tilde\eta}}
\begin{Remark}\label{RemConvention}
The problem arises that formulas~\eqref{EqSumAMot} and \eqref{EqConvSumAMot} depend on the choices of the $E$-generators $u_\eta$ of $\Koh_{1,\Betti}(\ulM^\eta,Q)$. Namely, multiplying one $u_\eta$ with an element $a\in E$ changes these sums by the summand $\tfrac{1}{\#H_K}\sum_{\text{all }v}v\bigl(\eta\psi(a)\bigr)\log q_v$, which may be different from zero. On the other hand, if all $u_\eta$ are simultaneously multiplied with the same $a\in E$ then the term $\tfrac{1}{\#H_K}\sum_{\eta\in H_K}\sum_{\text{all }v}v\bigl(\eta\psi(a)\bigr)\log q_v$ is added, which is zero by \eqref{EqProdFormulaAMot}.

Colmez~\cite{ColmezPeriods} faces the same problem and overcomes it by considering the terms \eqref{EqDoubleOmega} instead, which are independent of the chosen $u_\eta$. This is not possible for general $A$-motives, because it relies on the existence of a $Q$-automorphism $c$ of $Q^\alg$ such that the set of integers $\{d_\psi,d_{c\psi}\}$ does not depend on $\psi\in H_E$. In \eqref{EqDoubleOmega}, $c$ is complex conjugation and $\{d_\psi,d_{c\psi}\}=\{0,1\}$ for every $\psi\in H_E$. These conditions are not satisfied for the more general CM-types we considered so far for $A$-motives. 

It should also be noted, that it is in general \emph{not} possible to choose all $u_\eta$ in a compatible way, although this is possible for the generators $\omega_\psi^\eta$ by pulling back $\omega_\psi$ under $\eta$. However, it is possible for $A$-motives to pull back the induced $E_v$-generators $u_\eta\otimes 1\in \Koh_{1,\Betti}(\ulM^\eta,Q)\otimes_Q Q_v=\Koh_{1,v}(\ulM^\eta,Q_v)$ under additional automorphisms $\secondeta\in\sG_Q=\Gal(Q^\sep/Q)$. Namely, it follows from the definition in \eqref{EqKoh_v} that applying $\secondeta$ yields an $\CO_{E_v}$-isomorphism
\begin{equation*}
\secondeta\colon\Koh^1_v(\ulM^\eta,A_v)\,\isoto\,\Koh^1_v(\ulM^{\secondeta\eta},A_v)\,,\quad m \,\mapsto\,\secondeta(m)\,. 
\end{equation*}
If $\secondeta=\kappa\in\Gal(Q^\sep/K)$ then this isomorphism is just $\rho_{\ulM^\eta}(\kappa)$ where $\rho_{\ulM^\eta}\colon\sG_Q\to\Aut_{\CO_{E_v}}\Koh_{1,v}(\ulM^\eta,A_v)=\CO_{E_v}\mal$ is the Galois representation. Then $\secondeta(u_\eta\otimes 1)\in \Koh_{1,v}(\ulM^{\secondeta\eta},Q_v)=\Hom_{Q_v}( \Koh^1_v(\ulM^{\secondeta\eta},Q_v), Q_v)$ is defined by requiring 
\begin{equation}\label{EqRemConvention}
\secondeta(u_\eta\otimes 1)\bigl(\secondeta(m)\bigr)\;=\;(u_\eta\otimes1)(m)\qquad\text{for every}\quad m\in \Koh^1_v(\ulM^\eta,Q_v)\,.
\end{equation}
$\secondeta(u_\eta\otimes 1)$ is an $E_v$-generator of $\Koh_{1,v}(\ulM^{\secondeta\eta},Q_v)$. If $\secondeta$ is replaced by $\secondeta'=\secondeta\circ\kappa$ with $\kappa\in\Gal(Q^\sep/K)$ then $\ulM^{\secondeta'\eta}=\ulM^{\secondeta\eta}$ and $\secondeta'(m)=\rho_{\ulM^\eta}(\kappa)\cdot\secondeta(m)$, and hence $\secondeta'(u_\eta\otimes1)=\rho_{\ulM^\eta}(\kappa)^{-1}\cdot\secondeta(u_\eta\otimes1)=\rho_{\ulM^\eta}\dual(\kappa)\cdot\secondeta(u_\eta\otimes1)$. In particular, the value $v_{\secondeta\eta\psi}\bigl(\secondeta(u_\eta\otimes1)\bigr)$ only depends on the image of $\secondeta$ in $\Gal(K/Q)=H_K$. We abbreviate $\secondeta(u_\eta\otimes 1)$ to $u_\eta^\secondeta$. Allthough the notation is similar to $\omega_\psi^\eta$, it is understood, that $u_\eta^\secondeta$ does not exist in $\Koh_{1,\Betti}(\ulM^{\secondeta\eta},Q)$, but only in $\Koh_{1,\Betti}(\ulM^{\secondeta\eta},\BA_Q^\infty)=\prod'_{v\ne\infty}\Koh_{1,v}(\ulM^{\secondeta\eta},Q_v)$ where $\BA_Q^\infty$ is the ad\`ele ring of $Q$. Then for every fixed $\eta\in H_K$ Convention~\ref{ConventionAMot}, Theorem~\ref{ThmValueAtV} and \eqref{EqRelationBetweenThe_v} yield
\begin{align}
\label{EqConvSum2AMot}
& \log\bigl|{\TS\int_{u_\eta}}\omega_\psi^\eta\bigr|_\infty + \TS\tfrac{1}{\#H_K}\sum\limits_{\secondeta\in H_K}\sum\limits_{v\ne\infty}\log\bigl|\int_{u_\eta^\secondeta}\omega_\psi^{\secondeta\eta}\bigr|_v\;=\;\\
& \;=\;\log\bigl|{\TS\int_{u_\eta}}\omega_\psi^\eta\bigr|_\infty-Z^\infty((a^0_{E,\psi,\Phi})^*,0)+\frac{\log\#(A/\Fd_{\psi(\CO_E)/A})}{[\psi(E):Q]}-\tfrac{1}{\#H_K}\sum\limits_{\secondeta\in H_K}\sum_{v\ne\infty}\bigl(v(\omega_\psi^{\secondeta\eta})+v_{\secondeta\eta\psi}(u_\eta^\secondeta)\bigr)\log q_v \nonumber\\
& \;=\;\log\bigl|{\TS\int_{u_\eta}}\omega_\psi^\eta\bigr|_\infty-Z^\infty((a^0_{E,\psi,\Phi})^*,0)-\tfrac{1}{\#H_K}\sum\limits_{\secondeta\in H_K}\sum_{v\ne\infty}\bigl(v^\sim(\omega_\psi^{\secondeta\eta})+v_{\secondeta\eta\psi}(u_\eta^\secondeta)\bigr)\log q_v\,. \nonumber
\end{align}
If we restrict to \emph{imaginary} CM-fields $E$, which means that $E_\infty:=E\otimes_Q Q_\infty$ is still a field and carries a unique extension of the valuation $v_\infty$, then this sum is independent of the choice of the $E$-generator $u_\eta\in\Koh_{1,\Betti}(\ulM^\eta,Q)$. Indeed, if $u_\eta$ is multiplied with a unit $a\in E\mal$, then in \eqref{EqConvSum2AMot} the term 
\[
\TS -v_\infty(\eta\psi(a))\log q_\infty- \tfrac{1}{\#H_K}\sum\limits_{\secondeta\in H_K}\sum\limits_{v\ne\infty}v\bigl(\secondeta(\eta\psi(a))\bigr)\log q_v\;=\;- \tfrac{1}{\#H_K}\sum\limits_{\secondeta\in H_K}\sum\limits_{\text{all }v}v\bigl(\secondeta(\eta\psi(a))\bigr)\log q_v
\]
is added, which is zero by \eqref{EqProdFormulaAMot}. Imaginary CM-fields are particularly relevant for Drinfeld modules, see Theorem~\ref{ThmTagHeight} below. On the other hand, if $E$ has more than one place above $\infty$, then only the place induced from the embedding $\eta\psi\colon E\into Q_\infty^\alg\subset\BC_\infty$ contributes to \eqref{EqConvSum2AMot}, and then this formula is not invariant under changing $u_\eta$.
\end{Remark}

We thus propose to average twice over $\eta,\secondeta\in H_K$ and make the following

\begin{Conjecture}\label{ConjColmezAMot} 
Let $E$ be a finite \emph{imaginary} field extension of $Q$, which means that $E_\infty:=E\otimes_Q Q_\infty$ is still a field. Then the sum
\begin{align}\label{EqConvSum3AMot} 
& \TS \sum\limits_{\eta\in H_K}\biggl(\log\bigl|{\TS\int_{u_\eta}}\omega_\psi^\eta\bigr|_\infty-Z^\infty((a^0_{E,\psi,\Phi})^*,0)+{\DS\frac{\log\#(A/\Fd_{\psi(\CO_E)/A})}{[\psi(E):Q]}}-\tfrac{1}{\#H_K}\sum\limits_{\secondeta\in H_K}\sum\limits_{v\ne\infty}\bigl(v(\omega_\psi^{\secondeta\eta})+v_{\secondeta\eta\psi}(u_\eta^\secondeta)\bigr)\log q_v\biggr)\nonumber \\
& = \es \TS \sum\limits_{\eta\in H_K}\biggl(\log\bigl|{\TS\int_{u_\eta}}\omega_\psi^\eta\bigr|_\infty-Z^\infty((a^0_{E,\psi,\Phi})^*,0)-\tfrac{1}{\#H_K}\sum\limits_{\secondeta\in H_K}\sum\limits_{v\ne\infty}\bigl(v^\sim(\omega_\psi^{\secondeta\eta})+v_{\secondeta\eta\psi}(u_\eta^\secondeta)\bigr)\log q_v\biggr)
\end{align}
is zero, or equivalently the \emph{product formula} holds:
\[
\prod\limits_{\secondeta,\eta\in H_K}\biggl(\bigl|{\TS\int_{u_\eta}}\omega_\psi^{\eta}\bigr|_\infty\cdot\prod\limits_{v\ne\infty}\bigl|{\TS\int_{u_\eta^\secondeta}}\omega_\psi^{\secondeta\eta}\bigr|_v\biggr) \;:=\; \prod\limits_{\secondeta,\eta\in H_K}\biggl( \bigl|\langle u_\eta, \omega_\psi^{\eta}\rangle_v\bigr|_\infty\cdot\prod\limits_{v\ne\infty}\bigl|\langle u_\eta^\secondeta, \omega_\psi^{\secondeta\eta}\rangle_v\bigr|_v \biggr) \;=\; 1\,.
\]
\end{Conjecture}

\begin{Example}\label{ExCarlitzPeriods}
Similarly to Example~\ref{ExGmPeriods}, the convention allows to prove the product formula for the Carlitz motive $\ulCC = (\CC = \BF_q(\theta)[t], \tau_\CC = t-\theta)$ from Example~\ref{CarMod} over the field $K=\BF_q(\theta)=Q$ for which $H_K=\{\id_K\}$. We let $u\in\Koh_{1,\Betti}(\ulCC,A)$ be the generator which is dual to $\eta\tminus{}\in\Koh^{1}_\Betti(\ulCC,A)$ and we let $\omega=1\in\Koh^1_{\dR}(\ulCC, {\BC_\infty})$. Then we have computed in Examples~\ref{ExCarBettiDR}, \ref{ExCarVDR} and \ref{ExCarValuation} that
\begin{alignat*}{5}
& \langle u, \omega\rangle_\infty && \;=\; \eta^{-q}\prod_{i=1}^{\infty}(1- \theta^{1-q^i})^{-1} && \qquad \text{and}\qquad && \log\bigl|\langle u, \omega\rangle_\infty\bigr|_\infty && \;=\; \log(q^{q/(q-1)})\;=\;\tfrac{q}{q-1}\log q\,,\\
& \langle u, \omega\rangle_v && \;=\; \hat\sigma_{\!v}^*(\tplus{v}) && \qquad \text{and}\qquad && \log\bigl|\langle u, \omega\rangle_v\bigr|_v && \;=\; -v\bigl(\hat\sigma_{\!v}^*(\tplus{v})\bigr)\log q_v \;=\; -\tfrac{\log q_v}{q_v-1} \;=\; -Z_v(\BOne,1)\log q_v\,,
\end{alignat*}
where $\BOne(g)=1$ for every $g\in\sG_Q$. Here the CM-field is $E=Q$, $H_E=\{\id\}$ and the CM-type is given by $d_\id=1$. This implies that $a^0_{E,\id,\Phi}=\BOne$. So Convention~\ref{ConventionAMot} implies $\sum_{v\ne\infty}\log|\langle u, \omega\rangle_v|_v=-\frac{\zeta'_A(0)}{\zeta_A(0)}=-\tfrac{q}{q-1}\log q$ for the Riemann Zeta-function 
\[
\zeta_A(s) \;:=\; \prod_{  v \ne \infty}(1-(\#\BF_v)^{-s})^{-1} \;=\; \prod_{v \ne \infty}(1-q_v^{-s})^{-1} \;=\; \frac{1}{1-q^{1-s}}\,.
\]
We conclude $\sum_v\log|\langle u, \omega\rangle_v|_v=0$ and $\prod\limits_v|\langle u, \omega\rangle_v|_v=1$.
\end{Example}

In Section~\ref{SectExample} we will discuss an interesting example where $C$ and $Q$ have genus $1$. In the remainder of this section we focus on CM $A$-motives which come from Drinfeld modules. As analog of Colmez's Theorem~\ref{ThmFaltHeight} we have the following

\begin{Theorem}\label{ThmTagHeight}
Let $\ulG$ be a Drinfeld $A$-module over a finite separable field extension $K\subset Q^\alg$ of $Q$ with complex multiplication of CM-type $(E,\Phi)$ as in Example~\ref{ExCMDriMod}, where $\Phi=(d_\psi)_{\psi\in H_E}$ with $d_{\psi_0}=1$ for one $\psi_0\in H_E$ and $d_\psi=0$ for all $\psi\ne\psi_0$. Assume that $\ulG$ has complex multiplication by $\CO_E$ and that $E$ is a separable field extension of $Q$. Let $\ulM=\ulM(\ulG)$ and choose $\omega_{\psi_0}$ and $u_\eta$ as in Situation~\ref{SitCMAMot}. Then the stable Taguchi height $ht_{\rm Tag}^{\rm st}(\ulG)$ of $\ulG$ satisfies
\begin{eqnarray}
ht_{\rm Tag}^{\rm st}(\ulG) & = &\tfrac{1}{\#H_K}\sum\limits_{\eta\in H_K}\Bigl(-\log\bigl|{\TS\int_{u_\eta}}\omega_{\psi_0}^\eta\bigr|_\infty+\tfrac{1}{\#H_K}\sum\limits_{\secondeta\in H_K}\sum\limits_{v\ne\infty}\bigl(v(\omega_{\psi_0}^{\secondeta\eta})+v_{\secondeta\eta\psi_0}(u_\eta^\secondeta)\bigr)\log q_v \Bigr) \nonumber \\[2mm]
& & \qquad - \frac{\log\#(A/\Fd_{\CO_E/A})}{[E:Q]}-\log D_A(\CO_E) \label{EqThmTagHeight1} \\[2mm]
& = &\tfrac{1}{\#H_K}\sum\limits_{\eta\in H_K}\Bigl(-\log\bigl|{\TS\int_{u_\eta}}\omega_{\psi_0}^\eta\bigr|_\infty+\tfrac{1}{\#H_K}\sum\limits_{\secondeta\in H_K}\sum\limits_{v\ne\infty}\bigl(v^\sim(\omega_{\psi_0}^{\secondeta\eta})+v_{\secondeta\eta\psi_0}(u_\eta^\secondeta)\bigr)\log q_v \Bigr)-\log D_A(\CO_E)\,. \nonumber 
\end{eqnarray}
\end{Theorem}

\begin{proof}
1. Since both sides of the claimed equality \eqref{EqThmTagHeight1} are invariant under extending the field $K$, we may assume that $K$ is Galois over $Q$ and that $\ulG$ has good reduction at every finite place of $K$. Via the inclusion $K\subset Q^\alg\subset Q_\infty^\alg$ the restriction of the valuation $v_\infty$ on $Q_\infty^\alg$ to $K$ corresponds to a place $\wt\infty$ of $K$ such that the completion $K_{\wt\infty}$ equals the closure of $K$ in $Q_\infty^\alg$. For every $\eta \in H_K=\Gal(K/Q)$ we denote the image of $\wt\infty$ under $\eta$ by $\wt\infty_\eta$. Note that $\wt\infty_{\eta'}=\wt\infty_\eta$ if and only if $\eta'\eta^{-1}\in\Gal(K_{\wt\infty}/Q_\infty)$.

For $\eta \in \Gal(K/Q)$, we obtained  $\ulG^\eta=(G^\eta,\phi^\eta), \ulM^\eta$ and $\omega_{\psi_0}^\eta\in\Koh^{\eta\psi_0}(\ulM^\eta,K\dbl z-\zeta\dbr)$ from $\ulG=(G,\phi),\ulM$ and $\omega_{\psi_0}$ in Situation~\ref{SitCMAMot} by applying $\eta$ to the coefficients in $K$. Note that $K\dbl y_{\psi_0}-\psi_0(y_{\psi_0})\dbr=K\dbl z-\zeta\dbr$ by \cite[Lemma~1.3]{HartlJuschka} because $E/Q$ is separable. In addition, we chose $E$-generators $u_\eta\in\Koh_{1,\Betti}(\ulM^\eta,Q)$ and as in Remark~\ref{RemConvention} we obtain $E_v$-generators $u_\eta^\secondeta\in\Koh_{1,v}(\ulM^{\secondeta\eta},Q_v)$ for every $v\ne\infty$ and every $\secondeta\in H_K$. As in Example~\ref{ExCMDriMod} let $m^\eta:=-(z-\zeta)^{-1}\cdot\omega_{\psi_0}^\eta\in\Fq^{\ulM^\eta}$. The image $\ol m^\eta=\tau_{M^\eta}(m^\eta)$ of $m^\eta$ in $\coker\tau_{M^\eta}=\Hom_K(\Lie G^\eta,K)$ provides an isomorphism
\[
\ol m^\eta\colon\Lie G^\eta\;\isoto\;K\,.
\]
using \eqref{EqLieG}. We can lift $\ol m^\eta$ in a unique way to an element $\wt m^\eta\in M^\eta$ which is an isomorphism $\wt m^\eta\colon G^\eta\isoto\BG_{a,K}$. Indeed, if we choose any isomorphism $n\colon G^\eta\isoto\BG_{a,K}$ with $n\in M^\eta$, then $\ol m^\eta=b\cdot\Lie n$ for some $b\in K\mal$, and we may take $\wt m^\eta:=b\cdot n$. In particular, $\wt m^\eta$ is obtained from $\wt m:=\wt m^\id\colon G\isoto\BG_{a,K}$ by pull back under $\eta$. We recall the $E$-equivariant isomorphism for Betti-homology from Proposition~\ref{PropCompBettiGandM}
\begin{equation}\label{EqCompIsomBetti}
\Koh_{1,\Betti}(\ulM^\eta,A)\otimes_A\Omega^1_{A/\BF_q}\;\isoto\;\Koh_{1,\Betti}(\ulG^\eta,A)\,.
\end{equation}
We tensor it to $Q$ and observe that $\Omega^1_{A/\BF_q}\otimes_A Q=\Omega^1_{Q/\BF_q}=Q\,dz$; see Remark~\ref{RemOmegaQ}. Under the isomorphism \eqref{EqCompIsomBetti} we consider the element $\lambda_\eta:=u_\eta\,dz\in\Koh_{1,\Betti}(\ulG^\eta,Q)$. We may multiply $u_\eta$ by an element $a\in A$ such that we can assume $u_\eta\in\Koh_{1,\Betti}(\ulM^\eta,A)$ and $\lambda_\eta\in\Koh_{1,\Betti}(\ulG^\eta,A)$. Since $c\in E$ acts on $\Lie G^\eta$ as multiplication with $\eta\psi_0(c)$, Theorem~\ref{ThmDriModIntegral} implies for every $c\in E$
\[
\big|{\TS\int_{cu_\eta}} \omega_{\psi_0}^\eta\big|_\infty \;=\; \big|\langle cu_\eta,\omega_{\psi_0}^\eta\rangle_\infty\big|_\infty\;=\; \big|\ol m^\eta(c\lambda_\eta)\big|_\infty\;=\; \big|\eta\psi_0(c)\big|_\infty\cdot\big|\ol m^\eta(\lambda_\eta)\big|_\infty\,.
\]

\medskip\noindent
2. We want to compute $ht_{{\rm Tag},\wt\infty_\eta}(\ulG/K)$ as in Equations~\eqref{EqTagHtInfinity} and \eqref{EqD_A}. From \cite[Proposition~4.7.17]{Goss} we know that $E_\infty : = E\otimes_Q Q_\infty$ is still a field, that is $E/Q$ is imaginary in the sense of Example~\ref{ExDA(OE)} and Remark~\ref{RemConvention}. For every $\eta\in H_E$ we consider the $Q_\infty$-homomorphism $\eta\psi_0\otimes\id_{Q_\infty}\colon E_\infty\to K_{\wt\infty}\subset Q_\infty^\alg$ which is hence injective. Therefore, the restriction to $E_\infty$ of the valuation $v_\infty$ on $Q_\infty^\alg$ is the unique valuation on $E_\infty$ extending $v_\infty$ on $Q_\infty$. It is thus independent of $\eta$. By \cite[\S\,I.4, Proposition~10]{SerreLF} and \cite[\S\,3.6.2, Proposition~5]{BGR} there are elements $c_{1},\ldots,c_{r}\in E$ such that $E=\oplus_{i=1}^r Q\cdot c_{i}$ and
\begin{equation}\label{EqECartesian}
\Big|\sum_i a_i\cdot\eta\psi_0(c_{i})\Big|_\infty \;=\;\max\big\{\,\big|a_i\cdot\eta\psi_0(c_{i})\big|_\infty \,\big\}
\end{equation}
for every tuple $a_1,\ldots,a_r\in Q_\infty$. Under the isomorphism \eqref{EqCompIsomBetti} we consider the elements $\lambda_{\eta,i}:=c_{i}\cdot\lambda_\eta=c_{i}\cdot u_\eta\,dz\in \Koh_{1,\Betti}(\ulG^\eta,Q)$. Then $\Koh_{1,\Betti}(\ulG^\eta,Q)=\sum_{i=1}^r Q\cdot\lambda_{\eta,i}$, because $u_\eta$ is an $E$-generator of $\Koh_1(\ulM^\eta,Q)$. We will check whether the tuple $\ol m^\eta(\lambda_{\eta,1}),\ldots,\ol m^\eta(\lambda_{\eta,r})$ is orthogonal in the sense of Definition~\ref{DefLattice} for the $A$-lattice 
\[
\Lambda(\ulG^\eta)\;:=\;\ol m^\eta\bigl(\Koh_{1,\Betti}(\ulG^\eta,A)\bigr)\;\subset\;\ol m^\eta(\Lie G^\eta\otimes_K\BC_\infty)\;=\;\BC_\infty\;. 
\]
For $a_1,\ldots,a_r\in Q_\infty$ equation~\eqref{EqECartesian} implies
\begin{align*}
\Big|\sum_i a_i \,\ol m^\eta(\lambda_{\eta,i})\Big|_\infty & \;=\; \Big|\sum_i a_i\cdot\eta\psi_0(c_{i}) \cdot \ol m^\eta(\lambda_\eta)\Big|_\infty \\
& \;=\;\Big|\sum_i a_i\cdot\eta\psi_0(c_{i})\Big|_\infty\cdot\big|m^\eta(\lambda_\eta)\big|_\infty \;=\;\max\big\{\,\big|a_i \,\ol m^\eta(\lambda_{\eta,i})\big|_\infty \,\big\}\,.
\end{align*}
By multiplying all $c_{i}$ by the same element $a\in A$ with $v_\infty(a)\ll 0$, we may assume that $c_{i}\in\CO_E$ for all $i$ and that conditions \ref{DefLattice_A}, \ref{DefLattice_B} and \ref{DefLattice_C} from Definition~\ref{DefLattice} are satisfied for $\ol m^\eta(\lambda_{\eta,i})$. We observe that $\big(\sum_{i=1}^{r}A\,c_{i}\big) \lambda_\eta\subset\CO_E\,\lambda_\eta\subset\Koh_{1,\Betti}(\ulG^\eta,A)$, and hence
\begin{eqnarray*}
\#\biggl(\frac{\Koh_{1,\Betti}(\ulG^\eta,A)}{\big(\sum_{i=1}^{r}A\,c_{i}\big) \lambda_\eta}\biggr) & = & \#\biggl(\frac{\Koh_{1,\Betti}(\ulG^\eta,A)}{\CO_E\,\lambda_\eta}\biggr)\cdot\#\biggl(\frac{\CO_E\,\lambda_\eta}{\big(\sum_{i=1}^{r}A\,c_{i}\big) \lambda_\eta}\biggr) \\
& = & \#\biggl(\frac{\Koh_{1,\Betti}(\ulG^\eta,A)}{\CO_E\,\lambda_\eta}\biggr)\cdot\#\biggl(\frac{\CO_E}{\sum_{i=1}^{r}A\,c_{i}}\biggr)\;.
\end{eqnarray*}
Then
\begin{align}
\nonumber \frac{ht_{{\rm Tag},\wt\infty_\eta}(\ul G/K)}{-[K_{\wt\infty_\eta}:Q_\infty]} & = \log_q D_A\bigl(\ol m^\eta\bigl(\Koh_{1,\Betti}(\ulG^\eta,A)\bigr)\bigr)\\[2mm]
\nonumber &= \log_q\left( \frac{\prod_{1\leq i \leq r} |\ol m^\eta(\lambda_{\eta,i})|_\infty}{\#\big(\Lambda(G^\eta) \big/(A\cdot \ol m^\eta(\lambda_{\eta,1}) + \cdots +A\cdot \ol m^\eta(\lambda_{\eta,r}))\big)}\right)^{1/r}  \\[2mm]
\nonumber & = \log_q \left( \frac{\prod_{1\leq i \leq r} \big|\eta\psi_0(c_{i})\big|_\infty\cdot\big|\int_{ u_\eta} \omega_{\psi_0}^{\eta}\big|_\infty}{\#\big(\Koh_{1,\Betti}(\ulG^\eta,A)\big/(\sum_{i=1}^{r}A\,c_{i}) \lambda_\eta\big)}\right)^{1/r}\\[2mm]
\nonumber & = \log_q  \big|{\TS\int_{ u_\eta}} \omega_{\psi_0}^{\eta}\big|_\infty-\log_q\#\left(\frac{\Koh_{1,\Betti}(\ulG^\eta,A)}{\CO_E\, \lambda_\eta}\right)^{1/r}+\log_q\left( \frac{\prod_{1\leq i \leq r} \big|\eta\psi_0(c_{i})\big|_\infty}{\#\big(\CO_E \big/\sum_{i=1}^{r}A\,c_{i}\bigr)}\right)^{1/r}\\[2mm]
\label{EqThmTagHeight2} & = \log_q  \big|{\TS\int_{ u_\eta}} \omega_{\psi_0}^{\eta}\big|_\infty-\frac{1}{r}\,\log_q\#\left(\frac{\Koh_{1,\Betti}(\ulG^\eta,A)}{\CO_E\, \lambda_\eta}\right)+\log_q D_A(\CO_E)\;,
\end{align}
where the last equation is the definition of $D_A(\CO_E)$ from Example~\ref{ExDA(OE)}. In particular, this formula holds equally for all $\eta'\in H_K$ with $\wt\infty_{\eta'}=\wt\infty_\eta$ of $K$, that is for all $\eta'\in \Gal(K_{\wt\infty_\eta}/Q_\infty)\cdot\eta$. 

\medskip\noindent
3. We compute further 
\[
\Koh_{1,\Betti}(\ulG^\eta,A)\big/\CO_E\,\lambda_\eta \;=\; \prod_{v\ne\infty}\bigl(\Koh_{1,\Betti}(\ulG^\eta,A)\big/\CO_E\,\lambda_\eta\bigr)\otimes_A A_v \;=\; \prod_{v\ne\infty}\Koh_{1,\Betti}(\ulG^\eta,A_v)\big/\CO_{E_v}\,\lambda_\eta\,.
\]
Under the isomorphism \eqref{EqCompIsomBetti}, tensored to $A_v$ we have
\[
\xymatrix @R=2pc @C=5pc {
\CO_{E_v}\,\lambda_\eta \ar@{^{ (}->}[r] & \Koh_{1,\Betti}(\ulG^\eta,A_v) \\
\CO_{E_v}\,u_\eta\otimes_{A_v} A_v\,dz \ar@{<-->}[r] \ar[u]^-\cong & *!<1cm,0cm>{\es\CO_{E_v}\,u_\eta\otimes_A\Omega^1_{A/\BF_q}\;\subset\;\Koh_{1,\Betti}(\ulM^\eta,A_v)\otimes_A\Omega^1_{A/\BF_q}\;,} \ar[u]^\cong
}
\]
where the dashed arrow in the lower left corner comes from a comparison of $A_v$-modules of rank one, which is an inclusion $A_v\,dz \subset \Omega^1_{A/\BF_q}\otimes_A A_v$ or $A_v\,dz \supset \Omega^1_{A/\BF_q}\otimes_A A_v$ and even an equality for almost all $v$. Therefore, 
\[
\log_q \#\bigl(\Koh_{1,\Betti}(\ulG^\eta,A_v)\big/\CO_{E_v}\,\lambda_\eta\bigr) \;=\; r \,\ord_v(dz)\cdot[\BF_v:\BF_q] + \log_q \#\bigl(\Koh_{1,\Betti}(\ulM^\eta,A_v)\big/\CO_{E_v}\,u_\eta\bigr)\,.
\]
Here the factor $r=\rk_{A_v}\CO_{E_v}$ comes from the tensor product with $\CO_{E_v}\,u_\eta$, and $\ord_v(dz)$ is the order at $v$ of the rational section $dz$ of the line bundle $\Omega^1_{C/\BF_q}$. That is, if $A_v\,dz \subset \Omega^1_{A/\BF_q}\otimes_A A_v$ then $\log_q\#\bigl(\Omega^1_{A/\BF_q}\otimes_A A_v\big/A_v\,dz\bigr)=[\BF_v:\BF_q]\ord_v(dz)$. 
Adding over all places $v\ne\infty$ we obtain
\begin{equation}\label{EqThmTagHeight3}
\log_q \#\bigl(\Koh_{1,\Betti}(\ulG^\eta,A)\big/\CO_E\,\lambda_\eta\bigr) \;=\; \sum_{v\ne\infty}\Bigl(r\ord_v(dz)\cdot[\BF_v:\BF_q]+\log_q \#\bigl(\Koh_{1,\Betti}(\ulM^\eta,A_v)\big/\CO_{E_v}\,u_\eta\bigr)\Bigr)\,.
\end{equation}

\medskip\noindent
4. We now fix a place $v\ne\infty$ and let $e_\eta\in E_v:=E\otimes_Q Q_v$ such that $e_\eta^{-1}u_\eta$ is an $\CO_{E_v}$-generator of $\Koh_{1,\Betti}(\ulM^\eta,A_v)=\Koh_{1,v}(\ulM^\eta,A_v)$. Then $\CO_{E_v}/e_\eta\CO_{E_v}\isoto\Koh_{1,\Betti}(\ulM^\eta,A_v)\big/\CO_{E_v}\,u_\eta$ under $a\mapsto a\,e_\eta^{-1}u_\eta$. By the definition of $u_\eta^\secondeta$ in \eqref{EqRemConvention} also $e_\eta^{-1}u_\eta^\secondeta$ is an $\CO_{E_v}$-generator of $\Koh_{1,v}(\ulM^{\secondeta\eta},A_v)$. This means
\[
v_{\secondeta\eta\psi_0}(u_\eta^\secondeta)\;:=\;v\bigl(\secondeta\eta\psi_0(e_\eta)\bigr)\,.
\]
The $Q_v$-algebra $E_v$ decomposes into a product of fields $E_v=\prod_i E_{v,i}$. To compute the cardinality of $\CO_{E_v}/e_\eta\CO_{E_v}=\prod_i \CO_{E_{v,i}}/e_\eta\CO_{E_{v,i}}$, note that each $\CO_{E_{v,i}}/e_\eta\CO_{E_{v,i}}$ is an $\BF_v$-vector space. We denote its dimension by $n_i$. Let $K_v$ be the closure in $\BC_v$ of $K\subset Q^\alg\subset Q_v^\alg\subset\BC_v$, let $\CO_{K_v}$ be its valuation ring and $k_v$ its residue field. For every $Q_v$-homomorphism $\wt\psi_i\in H_{E_{v,i}}:= \Hom_{Q_v}(E_{v,i},Q_v^\alg)$ the $\BF_v$-vector space
\[
\bigl(\CO_{E_{v.i}}/e_\eta\CO_{E_{v,i}}\bigr) \otimes_{\CO_{E_{v,i}},\,\wt\psi_i}\CO_{K_v} \;=\;\CO_{K_v}\big/\wt\psi_i(e_\eta)\CO_{K_v}
\]
has dimension $n_i\cdot[K_v:\wt\psi_i(E_{v,i})]$, because $\CO_{K_v}$ is free over $\CO_{E_{v,i}}$ of rank $[K_v:\wt\psi_i(E_{v,i})]$. This dimension is equal to $[k_v:\BF_v]\cdot\ord_{K_v}(\wt\psi_i(e_\eta))=[K_v:Q_v]\cdot v(\wt\psi_i(e_\eta))$. We conclude that 
\[
n_i\;:=\;\dim_{\BF_v}\bigl(\CO_{E_{v.i}}/e_\eta\CO_{E_{v,i}}\bigr) \;=\;\frac{[K_v:Q_v]}{[K_v:\wt\psi_i(E_{v,i})]}\cdot v(\wt\psi_i(e_\eta)) \;=\;[E_{v,i}:Q_v] \cdot v(\wt\psi_i(e_\eta))
\]
\[
\text{and} \qquad\log_q \#\Bigl(\Koh_{1,\Betti}(\ulM^\eta,A_v)\big/\CO_{E_v}\,u_\eta\Bigr) \;=\; \sum_i n_i \cdot [\BF_v:\BF_q] \,.
\]
We now consider the following maps
\[
\xymatrix @R=0pc 
{
*=<1cm,0.7cm>{H_K} \ar@{->>}[r] & *=<2.5cm,0.7cm>{H_E} \ar[r]^-\sim & *=<3.5cm,0.7cm>{\Hom_{Q_v}(E_v,Q_v^\alg)}  \\
*=<1cm,0.7cm>{\secondeta} \ar@{|->}[r] & *=<2.5cm,0.7cm>{\secondeta\eta\psi_0 \;=:\; \wt\psi} \ar@{|->}[r] & *=<3.5cm,0.7cm>{\wt\psi\otimes\id_{Q_v}\quad}
}
\]
The set $\Hom_{Q_v}(E_v,Q_v^\alg)$ is equal to $\coprod_i H_{E_{v,i}}$, because every $\wt\psi\otimes\id_{Q_v}$ factors in a unique way
\begin{equation}\label{EqDecompE_v}
\xymatrix @R=1pc @C=0.6pc {
\wt\psi\otimes\id_{Q_v}\colon E_v \ar@{=}[r] & \prod_i E_{v,i} \ar@{->>}[rd] \ar[rr] & & Q_v^\alg \\
& & E_{v,i(\wt\psi)} \ar@{^{ (}->}[ur]_(0.6){\TS\es\wt\psi_i\;:=\;(\wt\psi\otimes\id_{Q_v})|_{E_{v,i(\wt\psi)}}}
}
\end{equation}
for an index $i(\wt\psi)$. The number of elements $\secondeta\in H_K$ which are mapped to the same $\wt\psi:=\secondeta\eta\psi_0\in H_E$ equals $\#\Gal(K/\eta\psi_0(E))=[K:\eta\psi_0(E)]=\tfrac{[K:Q]}{[E:Q]}$, and the number of $\secondeta\in H_K$ which are mapped into the set $H_{E_{v,i}}$ equals 
\begin{equation}\label{EqNumberOfeta}
\#H_{E_{v,i}}\cdot\tfrac{[K:Q]}{[E:Q]}\;=\;[E_{v,i}:Q_v]\cdot\tfrac{[K:Q]}{[E:Q]}\,.
\end{equation}
For each of the latter $\secondeta$ the valuation $v(\secondeta\eta\psi_0(e_\eta))=v(\wt\psi_i(e_\eta))=\tfrac{n_i}{[E_{v,i}:Q_v]}$ is the same. This implies
\begin{eqnarray}
\nonumber\tfrac{1}{\#H_K}\sum\limits_{\secondeta\in H_K}v_{\secondeta\eta\psi_0}(u_\eta^\secondeta)\cdot[\BF_v:\BF_q] & = & \tfrac{1}{\#H_K}\sum\limits_{\secondeta\in H_K}v(\secondeta\eta\psi_0(e_\eta))\cdot[\BF_v:\BF_q] \\[2mm]
\nonumber & = & \tfrac{[K:Q]}{\#H_K}\sum\limits_i \frac{n_i}{[E_{v,i}:Q_v]}\cdot \frac{[E_{v,i}:Q_v]}{[E:Q]} \cdot[\BF_v:\BF_q] \\[2mm]
\label{EqThmTagHeight4} & = & \frac{1}{r}\,\log_q \#\Bigl(\Koh_{1,\Betti}(\ulM^\eta,A_v)\big/\CO_{E_v}\,u_\eta\Bigr)\,.
\end{eqnarray}
Putting equations~\eqref{EqThmTagHeight2}, \eqref{EqThmTagHeight3} and \eqref{EqThmTagHeight4} together we can compute
\begin{equation}\label{EqThmTagHeight5}
 \frac{ht_{{\rm Tag},\wt\infty_\eta}(\ul G/K)}{[K_{\wt\infty_\eta}:Q_\infty]} = -\log_q \big|{\TS\int_{ u_\eta}} \omega_{\psi_0}^{\eta}\big|_\infty + \tfrac{1}{\#H_K}{\TS\sum\limits_{\secondeta\in H_K}}{\TS\sum\limits_{v\ne\infty}}\bigl(\ord_v(dz)+v_{\secondeta\eta\psi_0}(u_\eta^\secondeta)\bigr)\,[\BF_v:\BF_q] -\log_q D_A(\CO_E)\;.
\end{equation}

\medskip\noindent
5. Now we take a finite place ${\tilde v_\eta}$ of $K$ and let $v\ne\infty$ be the place of $Q$ with ${\tilde v_\eta}|v$. We choose an $\eta\in H_K$ such that ${\tilde v_\eta}$ is the place induced from $v$ via $\eta\colon K\into Q^\alg\subset Q_v^\alg\subset\BC_v$ and view $\ulG^\eta$ as a Drinfeld module over $\BC_v$. We use the isomorphism $\wt m^\eta$ from Step 1 above to write 
\[
\wt m^\eta\circ \phi^\eta_a\circ (\wt m^\eta)^{-1} \;=\; \gamma(a)+ \sum_{i=1}^{r \deg a} \phi^\eta_{a,i} \,\tau^i \;\in\; \End_{\BC_v,\BF_q}(\BG_{a,\BC_v})\;=\;\BC_v\{\tau\}\quad\text{with}\quad \phi^\eta_{a,i}\;\in\; \BC_v\,.
\]
Since $\ulG$ has good reduction at $\tilde v_\eta$ there exists an element $x_\eta \in K_{\tilde v_\eta}\mal$ such that 
\[
x_\eta \wt m^\eta\circ \phi^\eta_a\circ (\wt m^\eta)^{-1} x_\eta^{-1} \;=\; \gamma(a) + \sum_{i=1}^{r \deg a} \phi^\eta_{a,i}\cdot x_\eta^{1-q^i}\,\tau^i \;\in\; \CO_{\BC_v}\{\tau\} \quad\text{and}\quad \phi^\eta_{a,\,r \deg a}\cdot x_\eta^{1-q^{r \deg a}}\;\in\; \CO_{\BC_v}\mal.
\]
We have $\DS\frac{e(\tilde v_\eta|v)\cdot v(\phi^\eta_{a,i})}{q^i-1} =\frac{e(\tilde v_\eta|v)\cdot v(\phi^\eta_{a,i}\cdot x_\eta^{1-q^i})}{q^i-1} + e(\tilde v_\eta|v)\cdot v(x_\eta)$. Note that $\DS\frac{e(\tilde v_\eta|v)\cdot v(\phi^\eta_{a,i}\cdot x_\eta^{1-q^i})}{q^i-1} \geq 0$ for all $i$ and equal to $0$ for $i = r \deg a$. So
\[
\ord_{\tilde v_\eta}(\ulG) \;:=\; \min\left\{ \frac{e(\tilde v_\eta|v)\cdot v(\phi^\eta_{a,i})}{q^i-1}: a \in A\setminus\BF_q, \ 1\leq i \leq r\deg a \right\} \;=\; e(\tilde v_\eta|v)\cdot v(x_\eta)\;\in\; \BZ\,.
\]
Then 
\begin{equation}\label{EqThmTagHeight6a}
ht_{{\rm Tag},\tilde v_\eta}(\ul G/K) \;:=\; -[\BF_{\tilde v_\eta}:\BF_q] \cdot  e(\tilde v_\eta|v)\cdot v(x_\eta) \;=\; - [K_{\tilde v_\eta}:Q_v]\cdot v(x_\eta)\cdot [\BF_v:\BF_q]\,.
\end{equation}
It remains to relate $v(x_\eta)$ to $v(\omega^\eta_{\psi_0})$. For this let $\ul\CG^\eta$ be the good model of $\ulG^\eta$ over $\CO_{\BC_v}$ and let $\ul\CM^\eta$ be the $A$-motive of $\ul\CG^\eta$. The latter is the good model of $\ulM^\eta$ over $\CO_{\BC_v}$. Then $x_\eta\wt m^\eta$ extends to a coordinate system $x_\eta\wt m^\eta\colon \ul\CG^\eta\isoto\BG_{a,\CO_{\BC_v}}$ over $\CO_{\BC_v}$ of $\ul\CG^\eta$ and induces an isomorphism 
\[
\End_{\CO_{\BC_v}, \BF_q}(\BG_{a,\CO_{\BC_v}})\;=\;\CO_{\BC_v}\{\tau\}\;\isoto\; \ul\CM^\eta \;:=\; \Hom_{\CO_{\BC_v}, \BF_q}(\ul\CG^\eta,\BG_{a,\CO_{\BC_v}})\,,\quad f\;\longmapsto\; f\circ x_\eta\wt m^\eta \,. 
\]
This implies that $x_\eta\ol m^\eta$ generates the $\CO_{\BC_v}$-module $\coker\tau_{\CM^\eta}$. Next let $w=w_\eta$ be the place of $E$ which is induced from the place $\tilde v_\eta$ of $K$ under the embedding $\psi_0\colon E\into K$. Then $w_\eta$ is induced from the valuation $v$ on $\BC_v$ under the embedding $\eta\psi_0\colon E\into\BC_v$ and lies above the place $v$ of $Q$. Let $y_w\in\CO_E$ be an element which is a uniformizing parameter at $w$, that is, which satisfies $w(y_w)=1$. Set $\theta_w:=\eta\psi_0(y_w)\in \CO_{\BC_v}$. We use the isomorphism induced from $\tau_{M^\eta}$
\[
(y_w-\theta_w)^{-1}\Koh^{\eta\psi_0}(\ulM^\eta,\BC_v\dbl y_w-\theta_w\dbr) \big/ \Koh^{\eta\psi_0}(\ulM^\eta,\BC_v\dbl y_w-\theta_w\dbr) \;\isoto\;\Fq^{\ulM^\eta}/\Fp^{\ulM^\eta}\;\underset{\tau_{M^\eta}}{\isoto}\;\coker\tau_{M^\eta}\,.
\]
In the source of this isomorphism the elements $x_\eta m^\eta$ and $\tau_{M^\eta}^{-1}(x_\eta \wt m^\eta)$ are equal, because both have the same image $x_\eta \ol m^\eta$ in the target $\coker\tau_{M^\eta}$. Therefore, $x_\eta m^\eta$ is a generator of the canonical $\CO_{\BC_v}$-module structure on the source induced from $\ul\CM^\eta$. Multiplication with $y_w-\theta_w$ maps this $\CO_{\BC_v}$-structure isomorphically onto the $\CO_{\BC_v}$-module $\Koh^{\eta\psi_0}(\ul\CM^\eta,\CO_{\BC_v})$, which is hence generated by $(y_w-\theta_w)x_\eta m^\eta$. On the other hand, after multiplication with $-(z-\zeta)\mod(z-\zeta)^2$ we obtain $x_\eta\omega_{\psi_0}^\eta=-(z-\zeta)x_\eta m^\eta$ in
\[
\Koh^{\eta\psi_0}(\ulM^\eta,\BC_v) \;=\; \Koh^{\eta\psi_0}(\ulM^\eta,\BC_v\dbl y_w-\theta_w\dbr) \big/ (y_w-\theta_w)\Koh^{\eta\psi_0}(\ulM^\eta,\BC_v\dbl y_w-\theta_w\dbr)\,.
\]
All these are one dimensional $\BC_v$-vector spaces. Note that $y_w-\theta_w$ and $z-\zeta$ are not equal. Namely, if we write $I:=\ker\bigl(\CO_E\otimes_{\BF_q}\CO_E\to \CO_E,\, a\otimes a'\mapsto aa'\bigr)=(a\otimes 1-1\otimes a\colon a\in \CO_E)$, the element $(z-\zeta)\mod(z-\zeta)^2$ of $\BC_v$ is the image of $dz:=(z\otimes1-1\otimes z)\mod I^2 \in \Omega^1_{\CO_E/\BF_q}:=I/I^2$ under the $\CO_E$-homomorphism 
\[
\Omega^1_{\CO_E/\BF_q}\;\longto\;\Omega^1_{\CO_E/\BF_q}\hspace{-2em}\underset{\;\qquad \CO_E\otimes \CO_E/I,\,\id_{\CO_E}\otimes\eta\psi_0}{\otimes}\,(\CO_E\otimes_{\BF_q}\BC_v)/(a\otimes 1-1\otimes\eta\psi_0(a)\colon a\in \CO_E) \;=\; \Omega^1_{\CO_E/\BF_q}\,\underset{\CO_E,\,\eta\psi_0}{\otimes} \BC_v\,.
\]
On the other hand, $y_w-\theta_w$ is the image of $dy_w:=(y_w\otimes1-1\otimes y_w)\mod I^2$ and is a generator of the $\CO_{\BC_v}$-module $\Omega^1_{\CO_E/\BF_q}\otimes_{\CO_E,\,\eta\psi_0} \CO_{\BC_v}$. Therefore, $x_\eta\tfrac{y_w-\theta_w}{z-\zeta}\cdot\omega_{\psi_0}^\eta$ is an $\CO_{\BC_v}$-generator of $\Koh^{\eta\psi_0}(\ul\CM^\eta,\CO_{\BC_v})$, and hence
\begin{equation*}
v(\omega_{\psi_0}^\eta)\;=\;v\bigl(x_\eta^{-1}\tfrac{z-\zeta}{y_w-\theta_w}\bigr)\;=\;v\bigl(x_\eta^{-1}\cdot\eta\psi_0(\tfrac{dz}{dy_w})\bigr)\;=\;-v(x_\eta)+\frac{\ord_{w_\eta}(dz)}{e(w_\eta|v)}\,,
\end{equation*}
where again $\ord_{w_\eta}(dz)\in\BZ$ is the order at $w_\eta$ of the rational section $dz$ of the line bundle $\Omega^1_{\CO_E/\BF_q}$. From \eqref{EqThmTagHeight6a} we obtain for the local Taguchi height at $\tilde v_\eta$
\begin{eqnarray}\label{EqThmTagHeight6}
\frac{ht_{{\rm Tag},\tilde v_\eta}(\ul G/K)}{[K_{\tilde v_\eta}:Q_v]} \;=\; -v(x_\eta)\cdot [\BF_v:\BF_q] \;=\; v(\omega_{\psi_0}^\eta)\cdot [\BF_v:\BF_q]-\frac{\ord_{w_\eta}(dz)\cdot [\BF_{w_\eta}:\BF_q]}{[E_{w_\eta}:Q_v]}\,. 
\end{eqnarray}

\medskip\noindent
6. The summand on the right is related to the different $\FD_{\CO_E/A}$. Namely, by \cite[\S\,III.7, Proposition~14]{SerreLF} the $\CO_E$-module of relative differentials $\Omega^1_{\CO_E/A}$ is generated by one element and is isomorphic to $\CO_E/\FD_{\CO_E/A}$. This gives rise to the exact sequence \cite[0$_{\rm IV}$, Th\'eor\`eme~0.20.5.7]{EGA_IV1}
\[
\xymatrix {
0 \ar[r] & \Omega^1_{A/\BF_q}\otimes_A \CO_E \ar[r] & \Omega^1_{\CO_E/\BF_q} \ar[r] & \CO_E/\FD_{\CO_E/A} \ar[r] & 0\,.
}
\]
There is an element $0\ne a\in A$ with $a\,dz\in\Omega^1_{A/\BF_q}$. Dividing out $\CO_E\cdot a\,dz$ yields the exact sequence
\[
\xymatrix {
0 \ar[r] & \bigl(\Omega^1_{A/\BF_q}\otimes_A \CO_E\bigr)\big/\CO_E\cdot a\,dz  \ar[r] & \Omega^1_{\CO_E/\BF_q}\big/\CO_E\cdot a\,dz \ar[r] & \CO_E/\FD_{\CO_E/A} \ar[r] & 0\,.
}
\]
Counting elements, and denoting the places of $E$ by $w$ and their residue fields by $\BF_w$, we obtain
\begin{eqnarray*}
\prod_{w\nmid\infty}(\#\BF_w)^{\ord_w(a\,dz)} & = & \#\bigl(\Omega^1_{\CO_E/\BF_q}\big/\CO_E\cdot a\,dz\bigr) \\
& = &  \#\bigl(\CO_E/\FD_{\CO_E/A}\bigr)\cdot \#\bigl((\Omega^1_{A/\BF_q}/A\cdot a\,dz)\otimes_A \CO_E\bigr)\\[2mm]
& = &  \#\bigl(\CO_E/\FD_{\CO_E/A}\bigr)\cdot \#\bigl(\Omega^1_{A/\BF_q}/A\cdot a\,dz\bigr)^{[\CO_E:A]}\\[2mm]
& = &  \#\bigl(\CO_E/\FD_{\CO_E/A}\bigr)\cdot \Bigl(\prod_{v\ne\infty}(\#\BF_v)^{\ord_v(a\,dz)}\Bigr)^r.
\end{eqnarray*}
We observe $\ord_w(a\,dz)=w(a)+\ord_w(dz)$ and that for every place $v\ne\infty$ of $Q$
\[
\prod_{w|v}(\#\BF_w)^{w(a)}\;=\;\prod_{w|v}(\#\BF_v)^{[\BF_w:\BF_v]\cdot e(w|v)\cdot v(a)}\;=\;(\#\BF_v)^{\sum_{w|v}[\BF_w:\BF_v]\cdot e(w|v)\cdot v(a)}\;=\;(\#\BF_v)^{r\cdot v(a)}.
\]
Taking $\log_q$ this yields
\begin{equation}\label{EqThmTagHeight7}
\sum_{w\nmid\infty}[\BF_w:\BF_q]\cdot\ord_w(dz) - r\cdot\sum_{v\ne\infty}[\BF_v:\BF_q]\cdot\ord_v(dz) \;=\; \log_q \#\bigl(\CO_E/\FD_{\CO_E/A}\bigr) \;=\; \log_q \#\bigl(A/\Fd_{\CO_E/A})\,,
\end{equation}
where $\Fd_{\CO_E/A}=N_{E/Q}(\FD_{\CO_E/A})$ is the discriminant of $\CO_E$ over $A$, and the last equality comes from the fact that for all maximal ideals $\FP\subset\CO_E$ and $\Fp:=A\cap\FP\subset A$ with residue fields $\BF_\FP$, respectively $\BF_\Fp$, and for every $n\in\BN$ we have $N_{E/Q}(\FP^n)=\Fp^{[\BF_\FP:\BF_\Fp]n}$ and $\#(\CO_E/\FP^n)=\#(\BF_\FP)^n=(\#\BF_\Fp)^{[\BF_\FP:\BF_\Fp]n}=\#\bigl(A/N_{E/Q}(\FP^n)\bigr)$.

\medskip\noindent
7. Fix a place $w$ of $E$ above $v$. In terms of the decomposition $E_v:=E\otimes_Q Q_v=\prod_i E_{v,i}$ from diagram \eqref{EqDecompE_v} the completion $E_w$ of $E$ at $w$ equals $E_{v,i}$ for some $i$ and the number of $\eta\in H_K$ which give rise to the same $w_\eta=w$ equals $[E_w:Q_v]\cdot\tfrac{[K:Q]}{[E:Q]}$ by \eqref{EqNumberOfeta}. This together with \eqref{EqThmTagHeight6}, \eqref{EqThmTagHeight5} and \eqref{EqThmTagHeight7} finally implies
\begin{eqnarray*}
ht_{\rm Tag}^{\rm st}(\ulG) &=& \frac{\log q}{[K:Q]} \cdot \Bigl( \sum_{{\tilde v} \nmid \infty} ht_{{\rm Tag},{\tilde v}}(\ul G/K) + \sum_{\wt\infty\mid \infty} ht_{{\rm Tag},\wt\infty}(\ul G/K) \Bigr) \\
&=& \frac{\log q}{[K:Q]} \cdot \sum_{\eta\in H_K} \biggl( \sum_{v\ne\infty} \frac{ht_{{\rm Tag},{\tilde v_\eta}}(\ul G/K)}{[K_{\wt v_\eta}:Q_v]} + \frac{ht_{{\rm Tag},\wt\infty_\eta}(\ul G/K)}{[K_{\wt\infty_\eta}:Q_\infty]} \biggr) \\
& = & \frac{\log q}{[K:Q]} \cdot \sum_{\eta\in H_K} \biggl( \sum_{v\ne\infty}\Bigl( v(\omega_{\psi_0}^\eta)\cdot [\BF_v:\BF_q] - \frac{\ord_{w_\eta}(dz)\cdot [\BF_{w_\eta}:\BF_q]}{[E_{w_\eta}:Q_v]} \Bigr) \\
& & \qquad -\log_q \big|{\TS\int_{ u_\eta}} \omega_{\psi_0}^{\eta}\big|_\infty +\tfrac{1}{\#H_K}{\TS\sum\limits_{\secondeta\in H_K}}{\TS\sum\limits_{v\ne\infty}} \bigl(\ord_v(dz)+v_{\secondeta\eta\psi_0}(u_\eta^\secondeta)\bigr)\,[\BF_v:\BF_q] -\log_q D_A(\CO_E) \biggr)\\
& = & \tfrac{1}{\#H_K}\sum\limits_{\eta\in H_K}\Bigl(-\log\bigl|{\TS\int_{u_\eta}}\omega_{\psi_0}^\eta\bigr|_\infty+\tfrac{1}{\#H_K}\sum\limits_{\secondeta\in H_K}\sum\limits_{v\ne\infty}\bigl(v(\omega_{\psi_0}^{\secondeta\eta})+v_{\secondeta\eta\psi_0}(u_\eta^\secondeta)\bigr)\log q_v \Bigr)-\log D_A(\CO_E)\\
& & \qquad +\frac{\log q}{[K:Q]}\cdot \Bigl(-\frac{[K:Q]}{[E:Q]}\sum_{w\nmid\infty}[\BF_w:\BF_q]\cdot\ord_w(dz) + [K:Q]\sum_{v\ne\infty}[\BF_v:\BF_q]\cdot\ord_v(dz)\Bigr)\\
& = & \tfrac{1}{\#H_K}\sum\limits_{\eta\in H_K}\Bigl(-\log\bigl|{\TS\int_{u_\eta}}\omega_{\psi_0}^\eta\bigr|_\infty+\tfrac{1}{\#H_K}\sum\limits_{\secondeta\in H_K}\sum\limits_{v\ne\infty}\bigl(v(\omega_{\psi_0}^{\secondeta\eta})+v_{\secondeta\eta\psi_0}(u_\eta^\secondeta)\bigr)\log q_v \Bigr)\\
& & \qquad -\frac{\log\#(A/\Fd_{\CO_E/A})}{[E:Q]}-\log D_A(\CO_E)
\end{eqnarray*}
which finishes the proof.
\end{proof}

\begin{Remark}\label{RemArtinCharDriMod}
For a Drinfeld module $\ulG$ of rank $r$ over a finite Galois extension $K/Q$ with CM by $\CO_E$ for a separable field extension $E/Q$ with CM type as in Theorem~\ref{ThmTagHeight}, the functions from \eqref{Eq:a_E} and \eqref{Eq:a0_E} are
\begin{align*}
a_{E,\psi_0,\Phi}(g) & \;=\; \left\{\begin{array}{cl} 1 & \text{if }g\in\Gal(K/\psi_0(E)) \\ 0 & \text{else} \end{array}\right\}\;=\; \BOne_{\Gal(K/\psi_0E)}(g)\qquad\text{and} \\[2mm]
a^0_{E,\psi_0,\Phi}(g) & \;=\; \tfrac{1}{\#H_K} {\TS\sum\limits_{\eta\in H_K}}\BOne_{\Gal(K/\eta\psi_0E)}(g) \;=\; \Bigl(\tfrac{1}{r}\cdot \Ind^{\Gal(K/Q)}_{\Gal(K/\psi_0E)}\BOne_{\Gal(K/\psi_0E)}\Bigr)(g) \,,
\end{align*}
where $\BOne_{\Gal(K/\eta\psi_0E)}$ is the characteristic function of the subset $\Gal(K/\eta\psi_0(E))\subset \Gal(K/Q)$ and $\Ind$ denotes the induction of characters; see \cite[Chapter~VIII, \S\,3, Property~(V), page~222]{CasselsFroehlich}. Then $(a^0_{E,\psi_0,\Phi})^*=a^0_{E,\psi_0,\Phi}$ and \cite[loc.\ cit.]{CasselsFroehlich} implies 
\begin{eqnarray*}
L^\infty\bigl((a^0_{E,\psi_0,\Phi})^*,s,K/Q\bigr)^r & = & L^\infty(\Ind^{\Gal(K/Q)}_{\Gal(K/\psi_0E)}\BOne_{\Gal(K/\psi_0E)},s,K/Q) \\
& = & L^\infty(\BOne_{\Gal(K/\psi_0E)},s,K/\psi_0E) \\
& = & \zeta_{\CO_E}(s)\,,
\end{eqnarray*}
and hence
\[
r\cdot Z^\infty\bigl((a^0_{E,\psi_0,\Phi})^*,0\bigr) \;=\; \frac{\zeta'_{\CO_E}(0)}{\zeta_{\CO_E}(0)}\,.
\]
If $\infty$ is tamely ramified in $E/Q$ then Example~\ref{ExDA(OE)} and \cite[Lemma~5.17 and Proposition~5.18]{HartlSingh} imply that
\[
\log D_A(\CO_E) \;=\;\frac{\log\#(A/\Fd_{\CO_E/A})}{2r}\;=\;\frac{1}{2}\cdot\mu_{\Art}^\infty(a^0_{E,\psi_0,\Phi})\,,
\]
where $\mu_{\Art}^\infty$ was defined in \eqref{EqMuArtFF}. This puts Theorem~\ref{ThmTagHeight} in a form analogous to Colmez's Theorem~\ref{ThmFaltHeight}. \end{Remark}

Thus to establish the product formula in Conjecture~\ref{ConjColmezAMot} for a CM Drinfeld $A$-module $\ulG$ it suffices to relate the Taguchi height of $\ulG$ to the logarithmic derivative of the Zeta-function $\zeta_{\CO_E}$. This was achieved by Fu-Tsun Wei~\cite{Wei20}:

\begin{Theorem}[{\cite[Theorem~1.6]{Wei20}}] \label{ThmWei}
In Situation~\ref{SitCMAMot} let $\ulM=\ulM(\ulG)$ for a Drinfeld $A$-module $\ulG$ of rank $r$ with complex multiplication by $\CO_E$ over $K$ which has everywhere good reduction. Then the stable Taguchi height (Definition~\ref{DefTagHeight}) satisfies 
\[
ht_{\rm Tag}^{\rm st}(\ulG)\;=\;-\frac{1}{r}\cdot\frac{\zeta'_{\CO_E}(0)}{\zeta_{\CO_E}(0)}-\log D_A(\CO_E)
\]
\end{Theorem}

Theorems~\ref{ThmWei} and \ref{ThmTagHeight} and Remark~\ref{RemArtinCharDriMod} imply the following

\begin{Corollary}
The product formula from Conjecture~\ref{ConjColmezAMot} holds for CM Drinfeld $A$-modules. \qed
\end{Corollary}

In \cite{Wei20} Theorem~\ref{ThmWei} follows from the function field analogs of Kronecker's limit theorem and Lerch's formula \eqref{EqLerch}. In that sense, Wei's theorem can be viewed as the analog of Colmez's Theorem~\ref{ThmColmezAbelianE} in the abelian case. Analogously to Remark~\ref{RemChowla}, it would be interesting to describe, also in the function field case, the relation on the one hand between the Kronecker limit and the Lerch-type formulas in \cite{Wei20}, and on the other hand Gross-Zagier formulas like the ones proved by Yun, Wei Zhang, Howard and Shnidman \cite{YunZhang,YunZhang2,HowardShnidman} for the intersection numbers of Heegner cycles on moduli spaces of global $\PGL_2$-shtukas.

In the direction of the Andr\'e-Oort conjecture over function fields there is the following analog of Theorem~\ref{ThmAO} by Breuer and Hubschmid.

\begin{Theorem} \label{ThmAOBreuer}
The Andr\'e-Oort-Conjecture holds for irreducible closed subvarieties $X$ in Drinfeld modular varieties $M$ in the following cases:
\begin{enumerate}
\item \cite{BreuerTAMS}
$M$ is a product of Drinfeld modular curves which parameterize Drinfeld $A$-modules of rank $2$.
\item \cite{Breuer12}
$M$ is a Drinfeld modular variety parameterizing Drinfeld $A$-modules of rank $r$ and $X$ is a curve.
\item \cite{Hubschmid13}
$M$ is a Drinfeld modular variety parameterizing Drinfeld $A$-modules of rank $r$ such that $(q,r)=1$.
\end{enumerate}
That is, in both cases $X\subset M$ is a special subvariety if and only if it contains a dense set of CM points.
\end{Theorem}
 
Like in Theorem~\ref{ThmAO} one crucial ingredient is to show that the Galois orbit of a special point, that is a CM Drinfeld module, is large. This is done by following the strategy of Edixhoven~\cite{EdixhovenTexel2001,Edixhoven05}, who proved cases of the original Andr\'e-Oort-Conjecture for Shimura varieties conditionally under assuming the generalized Riemann Hypothesis. Over function fields various zeta functions are known to satisfy the Riemann Hypothesis by Deligne~\cite{DeligneWeilI}. So this approach to the Andr\'e-Oort-Conjecture over function fields can become unconditional. One the other hand, Conjecture~\ref{ConjColmezAMot} might also imply lower bounds for Galois orbits once it is related to heights of $A$-motives.

\section{Example}\label{SectExample}
\setcounter{equation}{0}

We give an example for Conjecture~\ref{ConjColmezAMot} in case of an $A$-motive $\ulM$ of rank $1$ where the curve $C$ has genus $1$. In this case, Conjecture~\ref{ConjColmezAMot} follows from Theorem~\ref{ThmWei}. This example was studied in detail by Green and Papanikolas~\cite{PapasGreen}. It is a beautiful exercise in computing with elliptic curves.

\begin{Point}\label{Point18.1}
Let $C$ be an elliptic curve over $\BF_q$, given by the (non-homogeneous) Weierstra{\ss} equation
\[
F\;:=\;F(t,y)\;:=\;y^2 + a_1ty + a_3y - t^3 - a_2t^2 - a_4t - a_6,\qquad\text{with}\quad  a_i \in \BF_q,
\]
 in the variables $t=\tfrac{X}{Z}$ and $y=\tfrac{Y}{Z}$, compare \eqref{EqWeierstrass}. Let $\infty\in \Var(Z^3\cdot F)\subset\BP^2_{\BF_q}$ be the $\BF_q$-rational point with $(X:Y:Z)=(0:1:0)$ at which $t$ and $y$ have pole order given by
\[
v_\infty (t) = -2,  \quad v_\infty(y) = -3\,.
\]
We have $A = \Gamma(C\setminus\{\infty\}, \CO_C)=\BF_q[t, y]/(F(t,y))$. For any field extension $L$ of $\BF_q$ there is exactly one point $\infty_L$ on $C_L$ above $\infty$, because $\infty$ is $\BF_q$-rational. To shorten the notation we sometimes denote the point $\infty_L$ again by $\infty$.

We consider a second copy of the ring $A$ given by $\BF_q[\theta,\epsilon]/(F(\theta,\epsilon))$ in the variables $\theta$ and $\epsilon$, and its fraction field $\BF_q(\theta, \epsilon)$. This is the function field of a second copy of the elliptic curve $C$, which we denote by $X_0$ and which has coordinates $\theta$ and $\epsilon$. That is $\BF_q(\theta, \epsilon)=\BF_q(X_0)$. Let  $\gamma : A\to \BF_q(\theta, \epsilon)$ be given by $\gamma(t) = \theta$ and $\gamma(y) = \epsilon$. This makes $\BF_q(\theta, \epsilon)$ into an $A$-field. We use the isomorphism $\gamma:Q\isoto\BF_q(\theta, \epsilon)$ to embed $\BF_q(\theta, \epsilon)$ canonically into $\BC_v$ for all places $v$ of $Q$. We note that 
\[
\Xi \;=\; \Var(t-\theta, y-\epsilon)\;=\;\Var(\CJ) \qquad\text{for the ideal}\quad \CJ\;:=\;(a\otimes1-1\otimes\gamma(a)\colon a\in A)\;=\;(t-\theta, y-\epsilon)
\]
is an $\BF_q(\theta, \epsilon)$-rational point of $C$. Furthermore, $\Xi\in C\bigl(\BF_q(\theta, \epsilon)\bigr)\subset C(\BC_\infty)$ specializes to $\infty\in C(\kappa_\infty)$ under the reduction map $red:C(\BC_\infty)\to C(\kappa_\infty)$ from \eqref{EqRedMap}. Recall the rigid analytic space $\FC:=\FC_{\BC_\infty}=(C_{\BC_\infty})^\rig$ and the disc $\FD\subset\FC$, which is defined in Notation~\ref{NotRigidDiscs} as the preimage in $\FC=C(\BC_\infty)$ of $\infty\in C(\kappa_\infty)$. This disc $\FD$ is the formal group of the elliptic curve $C_{\BC_\infty}$ over $\BC_\infty$, see \cite[Example~IV.3.1.3]{Silverman86}, where this formal group is denoted $\hat C(\Fm_\infty)$ for the maximal ideal $\Fm_\infty\subset\CO_{\BC_\infty}$. 

For any field extension $L$ of $\BF_q$ the relative $q$-Frobenius isogeny ${\rm Fr}_{q,C_L/L}:C_L\to C_L$ of $C_L$ over $L$ is given on $\Spec A_L\subset C_L$ by the $L$-homomorphism ${\rm Fr}_{q,C_L/L}^*\colon A_L\to A_L$, $t\mapsto t^q$, $y\mapsto y^q$. For any point $P\in C_L(L)$ we denote by $P^{(1)}:={\rm Fr}_{q,C_L/L}(P)\in C_L(L)$ the image of $P$. The composition $\sigma\circ{\rm Fr}_{q,C_L/L}={\rm Fr}_{q,C_L/L}\circ\sigma$ with the morphism $\sigma\colon C_L\to C_L$ from \eqref{EqSigma} equals the absolute $q$-Frobenius on $C_L$, which is the identity on points and the $q$-power map on the structure sheaf. For example, the morphism ${\rm Fr}_{q,C/\BF_q}$ sends $\Xi$ to $\Xi^{(1)}={\rm Fr}_{q,C/\BF_q}(\Xi)=\Var(t-\theta^q, y-\epsilon^q)$. 

The  isogeny $1-{\rm Fr}_{q,C/\BF_q} : C \to C$  is separable by \cite[Corollary~III.5.5]{Silverman86} and it induces an isomorphism of formal groups  $1-{\rm Fr}_{q,C/\BF_q} : \hat C(\Fm_\infty) \to \hat C(\Fm_\infty)$ by \cite[Corollary~IV.4.3 and Lemma~IV.2.4]{Silverman86}. Therefore, we can pick a unique point  $V\in \hat C(\Fm_\infty)=\FD\subset C(\BC_\infty)$ so that under the group law of $C$
\begin{equation}\label{EqDefOfV}
(1-{\rm Fr}_{q,C/\BF_q})(V) = V - V^{(1)} = \Xi, 
\end{equation}
and moreover, $(1-{\rm Fr}_{q,C/\BF_q})^{-1}(\Xi) = \{ V+ P \ |\ P \in C(\BF_q)\}$.

If we set $V = \Var(t-\alpha, y-\beta)$ with $\alpha,\beta\in\BC_\infty$ then $K:=\BF_q(\theta, \epsilon)(\alpha, \beta)=\BF_q(\alpha, \beta)\subset\BC_\infty$ is  the Hilbert class field of $\BF_q(\theta, \epsilon)$ by \cite[Proposition~3.3]{PapasGreen}. We view $K$ as the function field of a third copy of the elliptic curve $C$, which we denote by $X_1$ and which has coordinates $\alpha$ and $\beta$. The inclusion of fields $\BF_q(\theta,\epsilon)\subset K$ corresponds to a morphism $X_1\to X_0$ which is equal to the morphism $1-{\rm Fr}_{q,C/\BF_q} : C \to C$ under the identifications $X_1=C=X_0$. In particular, the set $X_1(\BF_q)$ equals the preimage of $\infty=(0:1:0)\in X_0$ under this map. This set consists of the points with $\alpha,\beta\in\BF_q$ together with the point $P=\infty_1\in X_1$ where $\alpha$ and $\beta$ have poles of order $2$ and $3$ respectively. It follows that $X_1\setminus X_1(\BF_q)=\Spec \CO_K$ for the integral closure $\CO_K$ of $A$ in $K$.
\end{Point}

\begin{Point}\label{Point18.2}
Now by \eqref{EqDefOfV} and the definition of the group law on $C$, see \cite[\S\,III.2]{Silverman86}, the $K$-valued points $V^{(1)}=\Var(t-\alpha^q,y-\beta^q)$ and $-V=\Var(t-\alpha,y+\beta+a_1\alpha+a_3)$ and $\Xi$ in $C(K)$ are collinear. We take $m$ to be the slope of the line connecting them:
\begin{equation}\label{EqDefOfm}
m  \;=\; \frac{\epsilon -\beta^q }{\theta-\alpha^q} \;=\; \frac{\epsilon+\beta+a_1\alpha +a_3}{\theta-\alpha} \;=\; \frac{\beta^q+\beta+a_1\alpha +a_3}{\alpha^q-\alpha}\;\in\; K.
\end{equation}
With respect to the valuation $v_\infty$ on $K\subset\BC_\infty$ we compute $v_\infty(\theta)=v_\infty(\alpha)=-2$ and $v_\infty(\epsilon)=v_\infty(\beta)=-3$, and hence obtain $v_\infty(m)=v_\infty(\frac{\epsilon -\beta^q }{\theta-\alpha^q})=-q$. We extend this to the following
\end{Point}

\begin{Lemma} \label{LemmaPolesOfm}
Let $P\in X_1$ be a closed point. Then the element $m\in K$ has a pole at $P$ if and only if $P\in X_1(\BF_q)=X_1\setminus\Spec\CO_K$. In particular, $m\in\CO_K$. Moreover, for the normalized valuation $v_P$ corresponding to $P$ we have
\[
v_P(m) =
\begin{cases}
-1 & \text{when}\quad P\in X_1(\BF_q), P\ne\infty_1\,, \\
-q & \text{when}\quad P=\infty_1\,.
\end{cases}
\]
\end{Lemma}

\begin{proof}
This can be proved by computing a uniformizing parameter at $P$, but we use the following different strategy. The element $m\in K$ was defined as the slope of the line through $V^{(1)}$, $-V$ and $\Xi$. This also holds over $X_1$ for the canonical extensions of $V^{(1)}$, $-V$ and $\Xi$ to $X_1$-valued points of $C\times_{\BF_q}X_1$. We now specialize to the residue field $L:=\kappa(P)$ of $P$. If $m(P) =\infty$, that is $\tfrac{1}{m}(P)=0$ then on the elliptic curve $C_L:=C\times_{\BF_q}\Spec L$ the line through $V^{(1)}$, $-V$ and   $\Xi$ contains the neutral element $\infty_{L}$, so $V^{(1)}= \infty_{L}$ or  $ - V =\infty_{L}$ or  $\Xi =\infty_{L}$. If $V^{(1)}= \infty_{L}$ or  $ - V =\infty_{L}$ then $V = \infty_{L}$, because $\infty_{L}=-\infty_{L}$ and this is the only point in ${\rm Fr}_{q,C_L/L}^{-1}(\infty_{L})$. From $V=\Var(t-\alpha,y-\beta)$ it follows that $P=\infty_1\in X_1(\BF_q)$. In this case $v_P(\theta)=v_P(\alpha)=-2$ and $v_P(\epsilon)=v_P(\beta)=-3$, and we obtain $v_P(m)=v_P(\frac{\epsilon -\beta^q }{\theta-\alpha^q})=-q$ as above. If $\infty_{L} = \Xi = V-V^{(1)}$ and $V\ne\infty_L$, then $V^{(1)}=V=\Var(t-\alpha,y-\beta)$ lies in $C(\BF_q)$. Thus $\alpha,\beta\in\BF_q$ and $P\in X_1(\BF_q)$. In this case $v_P(\alpha),v_P(\beta)\ge0$, and $\Xi=\Var(t-\theta,y-\epsilon)=\infty_L$ implies $v_P(\theta)=-2$ and $v_P(\epsilon)=-3$. We obtain $v_P(m)=v_P(\frac{\epsilon -\beta^q }{\theta-\alpha^q})=-1$. Conversely, if $P\in X_1(\BF_q)$, then $V=\Var(t-\alpha,y-\beta)\in C(\BF_q)$ and $\Xi = V-V^{(1)} = V-V = \infty_{L}$ and so the line through $V^{(1)}$, $-V$ and   $\Xi$  has slope $m = \infty$.
\end{proof}

\begin{Point}\label{Point18.3b}
By \eqref{EqDefOfV} and \cite[Corollary~III.3.5]{Silverman86} the divisor $[ V^{(1)}]-[V]+ [\Xi]-[\infty]$ on $C_K$ is principal. So there is a function $f \in K(t,y)=\Quot(A_K)$, called the \emph{shtuka function} for $A$ with
\begin{equation}\label{EqDivisorOff}
\di(f) = [V^{(1)}]-[V]+ [\Xi]-[\infty]\,.
\end{equation}
The shtuka function $f$ can be written as
\begin{equation}\label{EqDefOff}
f \;=\; \frac{\nu(t,y)}{\delta(t)} \;=\; \frac{y-\epsilon-m(t-\theta)}{t-\alpha} \;=\; \frac{y+\beta+a_1\alpha +a_3-m(t-\alpha)}{t-\alpha}\;=\;  \frac{y+\beta+a_1\alpha +a_3}{t-\alpha} -m,
\end{equation}
for
\begin{align*}
\nu \;:=\;\nu(t,y)\;:=\; y-\epsilon-m\cdot(t-\theta)\;\in\;\CO_K[t,y]\qquad\text{and} \qquad
\delta\;:=\;\delta(t)\;:=\;  t-\alpha\;\in\;\CO_K[t,y],
\end{align*}
with divisors on $C_K$ given by
\begin{align}\label{EqDivisors}
\di(\nu) \;=\; [V^{(1)}]+ [-V]+ [\Xi]-3[\infty] \qquad\text{and}\qquad
\di (\delta)\;=\; [V]+ [-V]-2[\infty].
\end{align}
The formulas \eqref{EqDivisorOff} and \eqref{EqDivisors} also hold for the Cartier divisors of $f$, $\nu$ and $\delta$ on the two dimensional scheme $C_{\CO_K}:=C\times_{\BF_q}\Spec\CO_K$, because $\nu$ and $\delta$ do not vanish on an entire fiber of $C_{\CO_K}$ over a closed point of $\Spec\CO_K$. Here we consider the $\CO_K$-valued points $\infty:=\Var(\tfrac{1}{t},\tfrac{t}{y})=\{\infty\}\times_{\BF_q}\Spec\CO_K$ and $V=\Var(t-\alpha,y-\beta)$ and $\Xi=\Var(t-\theta,y-\epsilon)$, etc.\ as Cartier divisors on $C_{\CO_K}$.
\end{Point}

\begin{Point}\label{Point18.3}
We consider the invertible sheaf $\CO_{C_K}([V])$ on $C_K$ with
\begin{align*}
\Gamma(\Spec A_K, \CO_{C_K}([V])) \;=\;& \bigl\{\,x\in\Quot(A_K)\colon \ord_P(x)\ge0 \ \forall\,P\in C_K\setminus\{V,\infty\}\text{ and }\ord_V(x)\ge-1\,\bigr\}\\
\;=\;& \bigl\{\,x\in\Quot(A_K)\colon \ord_P(x)\ge0 \ \forall\,P\ne V,\infty\text{ and }(t-\alpha)x,(y-\beta)x \in A_K\,\bigr\}\,.
\end{align*}
Then we compute $\Gamma(\Spec A_K, \sigma^*\CO_{C_K}([V]))$ as the $A_K$-module
\begin{align}
& \bigl\{\,x\otimes b\in\Quot(A_K)\otimes_{A_K,\sigma^*}A_K\colon \ord_P(x)\ge0 \ \forall\,P\ne V,\infty\text{ and }(t-\alpha)x,(y-\beta)x \in A_K\,\bigr\} \nonumber \\
\;=\;& \bigl\{\,x\otimes b\in\Quot(A_K)\otimes_{A_K,\sigma^*}A_K\colon \ord_P(x)\ge0 \ \forall\,P\ne V,\infty\text{ and }x\otimes b(t-\alpha^q),x\otimes b(y-\beta^q) \in A_K\,\bigr\} \nonumber \\
\;=\;& \Gamma(\Spec A_K,\CO_{C_K}([V^{(1)}]))\,.  \label{EqSigma*M}
\end{align}
We define an $A$-motive $\ulM=(M,\tau_M)$ over $K$ of rank $1$ and dimension $1$ as follows.
\begin{align*}
M \;=\;& \Gamma(\Spec A_K, \CO_{C_K}([V])) \\
\sigma^*M \;=\;& \Gamma(\Spec A_K, \CO_{C_K}([V^{(1)}]))\\
\tau_M \;:=\;& f\colon \sigma^*M \isoto M\otimes \CO_{C_K}(-[\Xi])\;\subset\; M\\
\coker\tau_M\;\cong\;& \CO_{C_K}/\CO_{C_K}(-[\Xi])\;\cong\; K\,.
\end{align*}
This $A$-motive corresponds to a Drinfeld $A$-module of rank $1$ over $K$, which is described more explicitly in \cite[\S\,3]{PapasGreen}. In particular, $\ulM$ is uniformizable. Moreover, $\ulM$ has CM through $\CO_E:=A$. We set $E=Q$ and then $H_E=\Hom_Q(E,Q^\alg)=\{\id_E\}$ consists of one single element $\psi=\id_E$. Correspondingly we drop all occurrences of $\psi$ from the notation used in Section~\ref{SectColmezConjAMot}. The de Rham cohomology of $\ulM$ is
\[
 \Koh^1_{\dR}(\ulM, K\dbl t-\theta \dbr) = \sigma^\ast M \otimes_{\CO_{C_K}}\invlim A_K/\CJ^n= \Gamma(\Spec A_K, \CO_{C_K}(V^{(1)})) \otimes_{\CO_{C_K}}K\dbl t-\theta \dbr=K\dbl t-\theta \dbr ,
\]
because $\invlim A_K/\CJ^n=K\dbl t-\theta\dbr$, and $\CO_{C_K}(V^{(1)})$ equals $\CO_{C_K}$ on the neighborhood $C_K\setminus\{V^{(1)}\}$ of $\Xi$.
For the unique element $\psi=\id_E$ in $H_E$ we have $\Koh^\psi(\ulM,K\dbl y_\psi-\psi(y_\psi)\dbr)=\Koh^1_\dR(\ulM,K\dbl t-\theta\dbr)$ and the Hodge-Pink lattice $\Fq^\ulM:=\tau_M^{-1}(M\otimes_{A_R}\invlim A_K/\CJ^n)\subset\Koh^1_{\dR}\bigl(\ulM,K\dpl t-\theta\dpr\bigr)$ of $\ulM$ satisfies 
\[
\Fq^\ulM=f^{-1}\cdot\Koh^1_{\dR}\bigl(\ulM,K\dbl t-\theta\dbr\bigr)=(t-\theta)^{-1}\cdot\Koh^1_{\dR}\bigl(\ulM,K\dbl t-\theta\dbr\bigr)
\]
by \eqref{EqDivisorOff}. So according to Definition~\ref{DefCMTypeOfAMotive} the CM-type of $\ulM$ is $\Phi=(d_{\id_E})$ with $d_{\id_E}=1$. 
\end{Point}

\begin{Point}\label{Point18.4}
We will next see that $\ulM$ has a good integral model $\ul\CM$ over $\CO_K$. Namely, by a similar computation as in \eqref{EqSigma*M} the invertible sheaf $\CO_{C_{\CO_K}}([V])$ on $C_{\CO_K}$ satisfies
\begin{align*}
\sigma^*\CO_{C_{\CO_K}}([V]) \;=\; \CO_{C_{\CO_K}}([V^{(1)}]).
\end{align*}
Then the good model $\ul\CM=(\CM,\tau_\CM)$ of $\ulM$ over $\CO_K$ is given by
\begin{align*}
\CM \;=\;& \Gamma(\Spec A_{\CO_K},\CO_{C_{\CO_K}}([V])) \\
\sigma^*\CM \;=\;& \Gamma(\Spec A_{\CO_K},\CO_{C_{\CO_K}}([V^{(1)}]))\\
\tau_\CM \;:=\;& f\colon \sigma^*\CM \isoto \CM\otimes\CO_{C_{\CO_K}}(-[\Xi])\;\subset\; \CM\\
\coker\tau_\CM\;\cong\;& A_{\CO_K}/A_{\CO_K}(-[\Xi])\;\cong\; \CO_K\,.
\end{align*}
\end{Point}

\begin{Point}\label{Point18.5}
With respect to the inclusion $K\subset\BC_\infty$ Papanikolas and Green~\cite[\S\,4]{PapasGreen} calculate $\Koh^1_\Betti(\ulM, A)$ as follows. They fix $(q-1)$-st roots of $-\alpha$ and $m\theta -\epsilon$, and set 
\begin{align*}
  \nu_\phi &:= (m\theta - \epsilon)^{1/(1-q)} \prod_{i=0}^\infty \Biggl( 1 - \biggl( \frac{m}{m\theta - \epsilon} \biggr)^{q^i} t + \biggl( \frac{1}{m\theta -\epsilon} \biggr)^{q^i} y \Biggr),\\
  \delta_\phi &:= (-\alpha)^{1/(1-q)} \prod_{i=0}^\infty \biggl( 1 - \frac{t}{\alpha^{q^i}}
  \biggr). 
\end{align*}
Since $v_\infty(\alpha) = -2$ in $\BC_\infty$, it follows that the product for $\delta_{\phi}$ converges in $\Gamma(\FC\setminus\{\infty\},\CO_\FC)$, is invertible on $\FC\setminus\FD$ and has zeroes of order $1$ at $V^{(i)}$ and $-V^{(i)}$ for all $i\in\BN_0$. Since  $v_\infty(m) = -q$, and so $v_\infty(m\theta-\epsilon)=-q-2$ and $v_\infty(\frac{m}{m\theta-\epsilon})=2$ it similarly follows that $\nu_{\phi} $ converges in $\Gamma(\FC\setminus\{\infty\},\CO_\FC)$ and is invertible on $\FC\setminus\FD$. Moreover, $\nu_\phi$ has zeroes of order $1$ at $\Xi^{(i)}$ and $-V^{(i)}$ and $V^{(i+1)}$ for all $i\in\BN_0$, because $1- \frac{m}{m\theta - \epsilon} \theta + \frac{1}{m\theta -\epsilon} \epsilon=0$ and $1- \frac{m}{m\theta - \epsilon} \alpha - \frac{1}{m\theta -\epsilon} (\beta+a_1\alpha+a_3)=0$ and $1- \frac{m}{m\theta - \epsilon} \alpha^q + \frac{1}{m\theta -\epsilon} \beta^q=0$. These functions satisfy the equations
\[
\nu_\phi \;=\; \nu\cdot \sigma^* \nu_\phi \;=\;(y-\epsilon-m\cdot(t-\theta))\cdot\sigma^*\nu_\phi\qquad \text{and}\qquad  \delta_\phi \;=\; \delta\cdot  \sigma^*\delta_\phi\;=\;(t-\alpha)\cdot \sigma^*\delta_\phi\,.
\]
Thus with the corresponding $(q-1)$-st root $\xi^{1/(q-1)}$ of $\xi = -\frac{m\theta -\epsilon}{\alpha} = -( m + \frac{\beta+a_1\alpha +a_3}{\alpha})$ we set
\begin{equation}\label{EqlambdaM}
\lambda_\ulM \;:=\; \frac{\nu_\phi}{\delta_\phi}\;=\;\xi^{1/(1-q)} \prod_{i=0}^\infty \frac{\sigma^{i*}f}{\xi^{q^i}}\;\in\; \Gamma(\FC\setminus\FD,\CO_\FC)^{\times}.
\end{equation}
Then $\tau_M(\sigma^* \lambda_\ulM) = f\cdot\sigma^*\lambda_\ulM=\lambda_\ulM$, and $\lambda_\ulM$ is a meromorphic function on $\FC\setminus\{\infty\}$ without poles or zeroes on $\FC\setminus\FD$. (By looking at the product decomposition of $\lambda_\ulM$ one even sees that it has a simple pole at $V$ and simple zeroes at $\Xi^{(i)}$ for all $i\in\BN_0$.) So we obtain
\begin{equation} \label{omegarhoprod}
\Koh^1_\Betti(\ulM, A) \;=\; \lambda_\ulM\cdot A.
\end{equation}
Let $u \in  \Koh_{1,\Betti}(\ulM, A)$ be the generator such that  $\langle u, \lambda_\ulM\rangle =1$. We also write $u_{\id_K}:=u$.
\end{Point}

\begin{Point}\label{Point18.6}
We can take $\omega:=\omega_\psi:= \sigma^*\delta^{-1}=(t-\alpha^q)^{-1}$ as a generator of $\Koh^1_{\dR}(\ulM, K\dbl t-\theta \dbr)$. Then the comparison isomorphism $h_{\Betti, \dR} = \sigma^\ast h_\ulM$ from Theorem~\ref{PeriodIso} sends the generator $\lambda_\ulM$ of $\Koh^1_\Betti(\ulM, A)$ to $\sigma^*\lambda_\ulM=\sigma^*(\lambda_\ulM\delta)\cdot\omega\in\Koh^1_{\dR}(\ulM, K\dbl t-\theta \dbr)$ and the comparison isomorphism $h_{\Betti, \dR} = \sigma^\ast h_\ulM\mod\CJ$ from \eqref{Eq2PeriodIso} sends the generator $\lambda_\ulM$ of $\Koh^1_\Betti(\ulM, A)$ to $\sigma^*(\lambda_\ulM\delta)(\Xi)\cdot\omega\in\Koh^1_{\dR}(\ulM, K)$. Therefore,
\begin{align*}
\langle u, h_{\Betti, \dR}^{-1}(\omega)\rangle_\infty \;=\; \langle u, \sigma^*(\lambda_\ulM\delta)(\Xi)^{-1}\cdot \lambda_\ulM\rangle_\infty \;=\;\frac{\xi^{q/(q-1)}}{(\sigma^*\delta)(\Xi)} \prod_{i=1}^\infty \frac{\xi^{q^i}}{(\sigma^{i*}f)(\Xi)}\,.
\end{align*}
To compute the absolute value of $\langle u, h_{\Betti, \dR}^{-1}(\omega)\rangle_\infty$ we observe that for every $i\in\BN_{>0}$
\[
\biggl|\frac{\xi^{q^i}}{(\sigma^{i*}f)(\Xi)}\biggr|_\infty\;=\;\biggl|\frac{1-\frac{\theta}{\alpha^{q^i}}}{1 - ( \frac{m}{m\theta - \epsilon} )^{q^i} \theta + ( \frac{1}{m\theta -\epsilon})^{q^i} \epsilon}\biggr|_\infty\;=\;1\,,
\]
as well as $v_\infty(\xi)=-q$, whence $|\xi^{q/(q-1)}|_\infty=q^{q^2/(q-1)}$, and $|(\sigma^*\delta)(\Xi)|_\infty=|(t-\alpha^q)(\Xi)|_\infty=|\theta-\alpha^q|_\infty=|\alpha^q|_\infty=q^{2q}$. Thus we obtain
\begin{align}\label{EqValueAtInfinity}
& \mathrel{\Big |}{\TS\int_{u}\omega}\mathrel{\Big |_\infty}\;:=\; \mathrel{\Big |}{\langle u, h_{\Betti, \dR}^{-1}(\omega)\rangle_\infty}\mathrel{\Big |_\infty}\;=\;q^{\frac{q^2}{q-1}-2q}\;=\;q^{\frac{q}{q-1}-q}\qquad\text{and} \nonumber \\
& \fbox{$\log\mathrel{\Big |}{\TS\int_{u}\omega}\mathrel{\Big |_\infty}\;=\;\bigl(\frac{q}{q-1}-q\bigr)\log q$\,.}
\end{align}
\end{Point}

\begin{Point}\label{Point18.7}
We consider the set $H_K:=\Hom_Q(K,Q^\alg)=\Gal\bigl(K/\BF_q(\theta,\epsilon)\bigr)$ which actually is a group, because $K$ is Galois over $\BF_q(\theta,\epsilon)$. It is isomorphic to the group $C(\BF_q)$ under the map $\eta\mapsto P_\eta:=V-\eta(V)$. Indeed, since $\eta(\Xi)=\Xi\in C(K)$ is fixed by $\eta$ we see that $\eta(V)$ still satisfies $\eta(V)-\eta(V)^{(1)}=\eta(V)-\eta(V^{(1)})=\eta(\Xi)=\Xi=V-V^{(1)}$. Therefore, the point $P_\eta=V-\eta(V)$ satisfies $P_\eta^{(1)}=P_\eta$, and hence $P_\eta\in C(\BF_q)$. Since the coordinates $(\alpha,\beta)$ of $V$ generate the field extension $K/\BF_q(\theta,\epsilon)$, the map $\eta\mapsto P_\eta$ is bijective. It is a group homomorphism, because $P_{\tilde\eta\eta}=V-\tilde\eta\eta(V)=V-\tilde\eta(V)+\tilde\eta(V)-\tilde\eta\eta(V)=P_{\tilde\eta}+\tilde\eta(P_\eta)=P_{\tilde\eta}+P_\eta$, as $P_\eta\in C(\BF_q)$ is fixed by $\tilde\eta$. In particular, $\#H_K=\#C(\BF_q)$.

We now fix an element $\eta\in H_K$ with $\eta\ne\id_K$ and let the $A$-motive $\ul\CM^\eta$ over $\CO_K$ and $\omega^\eta\in\Koh^1_\dR(\ulM^\eta,K\dbl t-\theta\dbr)$ be deduced from $\ul\CM$ and $\omega$ by base extension. Then $\ul\CM^\eta$ is isogenous to $\ul\CM$ by the theory of complex multiplication, which was developed for Drinfeld modules by Hayes~\cite{Hayes79} and for general $A$-motives by Pelzer~\cite{Pelzer09}. We give an elementary and explicit treatment for our $\ul\CM$. We claim that there is an isomorphism
\begin{equation}\label{EqIsomgeta}
g_\eta\colon\ul\CM^\eta\;\isoto\;\ul\CM\otimes\CO(-[P_\eta])\;=:\;\ul\CM(-[P_\eta])\,, 
\end{equation}
where $\CO(-[P_\eta])$ denotes the invertible sheaf on $\Spec A_{\CO_K}$ associated to the divisor $-[P_\eta]\times_{\BF_q}\Spec\CO_K$. Namely, the $A$-motives $\ul\CM^\eta$ and $\ul\CM(-[P_\eta])$ correspond to the invertible sheaves $\CO_{C_K}([\eta(V)])=\CO_{C_K}([V-P_\eta])$ and $\CO_{C_K}([V])\otimes\CO_{C_K}(-[P_\eta])=\CO_{C_K}([V]-[P_\eta])$ on $C_K$, respectively.

By \eqref{EqDefOfV} and \cite[Corollary~III.3.5]{Silverman86} the divisor $[V-P_\eta]-[V]+[P_\eta]-[\infty]$ on $C_K$ is principal and there is a function $g_\eta \in K(t,y)=\Quot(A_K)$ with
\begin{equation}\label{EqDivisorOfgeta}
\di(g_\eta) \;=\; [V-P_\eta]-[V]+[P_\eta]-[\infty] \;=\; [V-P_\eta]+[-V]+[P_\eta]-[V]-[-V]-[\infty]\,.
\end{equation}
It can be written explicitly as follows. By construction of the group law on $C$, the three points $V-P_\eta=\eta(V)$ and $-V$ and $P_\eta$ lie on a single line whose slope is
\[
\frac{\eta(\beta)-y(P_\eta)}{\eta(\alpha)-t(P_\eta)}\;=\;\frac{\eta(\beta)+\beta+a_1\alpha+a_3}{\eta(\alpha)-\alpha}\;=\;\frac{y(P_\eta)+\beta+a_1\alpha+a_3}{t(P_\eta)-\alpha}\;\in\;\CO_K\,.
\]
This slope is a priory an element of $K$, but we see that it lies in $\CO_K$ by reasoning like in Lemma~\ref{LemmaPolesOfm}. Indeed, the slope has a pole if and only if one of the points $P_\eta$ or $-V$ or $V-P_\eta=\eta(V)$ equals $\infty$. If $P_\eta=\infty$, then the bijectivity of the map $\eta\mapsto P_\eta$ implies $\eta=\id_K$ which was excluded. If $V-P_\eta=\infty$, and hence $V=P_\eta\in C(\BF_q)$, or if $-V=\infty$, then $\Xi=\infty$, and so the poles of the slope do not lie in $\Spec\CO_K$. That is, the slope lies in $\CO_K$ as claimed.  Then we can take
\begin{equation}\label{EqDefOfgeta}
g_\eta \;=\; \frac{y-\eta(\beta)-\frac{\eta(\beta)+\beta+a_1\alpha+a_3}{\eta(\alpha)-\alpha}(t-\eta(\alpha))}{t-\alpha}
\end{equation}
as an isomorphism $\CM^\eta\isoto\CM\otimes\CO(-[P_\eta])$. Here we use that formula~\eqref{EqDivisorOfgeta} for the divisor of $g_\eta$ also holds on $C_{\CO_K}$, because both numerator and denominator of $g_\eta$ lie in $\CO_K[t,y]$ and do not vanish on an entire fiber of $C_{\CO_K}$ over a closed point of $\Spec\CO_K$.

In order to see that $g_\eta$ is an isomorphism of $A$-motives, it remains to prove that $g_\eta\circ\eta(f)=f\circ\sigma^*g_\eta$. Since the divisor on both sides equals $[\eta(V)^{(1)}]+[P_\eta]-[V]+[\Xi]-2[\infty]$, both sides differ by multiplication with an element of $K\mal$. Multiplying both sides with the common denominator and comparing the coefficients of $t^2y$ shows that both sides are equal as desired.
\end{Point}

\begin{Point}\label{Point18.8}
The isomorphism $g_\eta\colon\ul\CM^\eta\isoto\ul\CM(-[P_\eta])$ induces isomorphisms on (co-)homology
\begin{alignat*}{2}
g_\eta\colon & \Koh^1_\dR(\ul\CM^\eta,\CO_K) && \;\isoto\;\Koh^1_\dR(\ul\CM(-[P_\eta]),\CO_K)\,, \\
g_\eta\colon & \Koh^1_\Betti(\ulM^\eta,A) && \;\isoto\;\Koh^1_\Betti(\ulM(-[P_\eta]),A)\,,\qquad\text{and} \\
g_\eta\colon & \Koh_{1,\Betti}(\ulM^\eta,A) && \;\isoto\;\Koh_{1,\Betti}(\ulM(-[P_\eta]),A)\,.
\end{alignat*}
These are compatible with the period isomorphisms $h_{\Betti,\dR}$ and the pairing between $\Koh^1_\Betti$ and $\Koh_{1,\Betti}$. So we may replace $\ul\CM^\eta$ by $\ul\CM(-[P_\eta])$ in the rest of our computation.

Since $\omega=(t-\alpha^q)^{-1}$ and $\omega\mod(t-\theta)=(\theta-\alpha^q)^{-1}\in\Koh^1_\dR(\ul\CM,\CO_K)$ we obtain $\omega^\eta=(t-\eta(\alpha)^q)^{-1}$ and $\omega^\eta\mod(t-\theta)=(\theta-\eta(\alpha)^q)^{-1}$, and we set $\wt\omega^\eta:=g_\eta(\omega^\eta)\in\Koh^1_{\dR}(\ulM(-[P_\eta]),K\dbl t-\theta\dbr)$ and $\wt\omega^\eta\mod(t-\theta)=g_\eta(\omega^\eta)\mod(t-\theta)\in\Koh^1_{\dR}(\ul\CM(-[P_\eta]),\CO_K)$. By definition, $\Koh^1_{\dR}(\ul\CM,\CO_K):=\sigma^*\CM/\CJ\sigma^*\CM=\sigma^*\CM|_\Xi$, with $\CJ=(t-\theta,y-\epsilon)$ being the vanishing ideal of the $\CO_K$-valued point $\Xi\in C(\CO_K)$. We compute
\begin{align*}
\wt\omega^\eta&\;=\;\sigma^*(g_\eta)\cdot(t-\eta(\alpha)^q)^{-1}\\[2mm]
&\;=\;\frac{y-\eta(\beta)^q-\frac{\eta(\beta)^q+\beta^q+a_1\alpha^q+a_3}{\eta(\alpha)^q-\alpha^q}(t-\eta(\alpha)^q)}{t-\alpha^q}\,\cdot\,(t-\eta(\alpha)^q)^{-1}\\[2mm]
&\;=\;\frac{y-\eta(\beta)^q-\frac{\eta(\beta)^q+\beta^q+a_1\alpha^q+a_3}{\eta(\alpha)^q-\alpha^q}(t-\eta(\alpha)^q)}{t-\eta(\alpha)^q}\,\cdot\,(t-\alpha^q)^{-1} \qquad\text{and}\\[2mm]
\wt\omega^\eta\mod(t-\theta)&\;=\;\frac{\epsilon-\eta(\beta)^q-\frac{\eta(\beta)^q+\beta^q+a_1\alpha^q+a_3}{\eta(\alpha)^q-\alpha^q}(\theta-\eta(\alpha)^q)}{\theta-\eta(\alpha)^q}\,\cdot\,(\theta-\alpha^q)^{-1}\\[2mm]
&\;=\;\Bigl(\frac{\epsilon-\eta(\beta)^q}{\theta-\eta(\alpha)^q}-\frac{\eta(\beta)^q+\beta^q+a_1\alpha^q+a_3}{\eta(\alpha)^q-\alpha^q}\Bigr)\,\cdot\,\omega\mod(t-\theta)\,.
\end{align*}
The element $\sigma^*g_\eta|_\Xi:=\frac{\epsilon-\eta(\beta)^q}{\theta-\eta(\alpha)^q}-\frac{\eta(\beta)^q+\beta^q+a_1\alpha^q+a_3}{\eta(\alpha)^q-\alpha^q}$ has absolute value
\begin{equation}\label{EqValueOfgeta}
\bigl|\sigma^*g_\eta|_\Xi\bigr|_\infty\;=\;q^{q}\,,\qquad\text{and hence}\qquad \log\bigl|\sigma^*g_\eta|_\Xi\bigr|_\infty\;=\;q\,\log q\,,
\end{equation}
because the first summand has absolute value $q$ and is dominated by the second summand which has absolute value $q^{q}$.
\end{Point}

\begin{Point}\label{Point18.9}
We now compute $v(\omega^\eta)$ for all places $v\ne\infty$ of $Q$ and for all $\eta\in H_K$. Observe that by \eqref{EqDivisors} the multiplication with $t-\alpha$ induces an isomorphism $\CO_{C_{\CO_K}}([V])\isoto\CO_{C_{\CO_K}}(2[\infty]-[-V])$ and the multiplication with $t-\alpha^q$ induces an isomorphism $\CO_{C_{\CO_K}}([V^{(1)}])\isoto\CO_{C_{\CO_K}}(2[\infty]-[-V^{(1)}])$. We restrict this morphism to the $\CO_K$-valued point $\Xi$, that is, we pull it back under the corresponding morphism $h_\Xi\colon\Spec\CO_K\to C_{\CO_K}$. To do so we first claim that $h_\Xi$ factors through the open subscheme of $C_{\CO_K}$ which is the complement of $\{\infty\}\cup\{-V^{(1)}\}$. Indeed, the locus on $C_{\CO_K}$ where $\Xi=-V^{(1)}$ is equal to the locus where $V=\infty$, and the latter locus does not lie above $\Spec\CO_K$. The same is true for the locus where $\Xi=\infty$. We conclude that multiplication with $\theta-\alpha^q$ induces an isomorphism
\begin{alignat*}{3}
\theta-\alpha^q\colon\;\Koh^1_\dR(\ul\CM,\CO_K)\;=\;&&\;h_\Xi^*\,\CO_{C_{\CO_K}}([V])&\;\isoto\;&& h_\Xi^*\,\CO_{C_{\CO_K}}(2[\infty]-[-V])\;=\;h_\Xi^*\,\CO_{C_{\CO_K}}\;=\;\CO_K \\[1mm]
\omega\mod(t-\theta)\;=\;&&\;(\theta-\alpha^q)^{-1} \quad & \;\longmapsto\; && \quad 1\;.
\end{alignat*}
This shows that $\Koh^1_\dR(\ul\CM,\CO_K)=\CO_K\cdot\omega\mod(t-\theta)$, and by base extension under $\eta$, also $\Koh^1_\dR(\ul\CM^\eta,\CO_K)=\CO_K\cdot\omega^\eta\mod(t-\theta)$. This yields
\begin{equation}\label{EqvOfomega}
\fbox{$v(\omega^\eta)\;=\;0\qquad\text{for every place }v\ne\infty\text{ and every }\eta\in H_K$\,.}
\end{equation}
\end{Point}

\begin{Point}\label{Point18.10}
We next compute $\Koh^1_\Betti(\ulM(-[P_\eta]),A)$ for the $A$-motive $\ulM(-[P_\eta])=\bigl(\CO_{C_K}([V]-[P_\eta]), \tau=f\bigr)$. The function $\lambda_\ulM$ from \eqref{EqlambdaM} satisfies $\tau(\sigma^* \lambda_\ulM) = f\cdot\sigma^*\lambda_\ulM=\lambda_\ulM$, but it does not have a zero at $P_\eta$, and hence does not lie in $\ulM(-[P_\eta])\otimes_{A_K}\CO_{\FC\setminus\FD}$ and not in $\Koh^1_\Betti(\ulM(-[P_\eta]),A)$. Instead, 
\[
\Koh^1_\Betti(\ulM(-[P_\eta]),A)\;=\;\lambda_\ulM\cdot\Gamma(\Spec A,\CO_C(-[P_\eta]))\;=\;\lambda_\ulM\cdot\Fp_\eta\,, 
\]
where $\Fp_\eta\subset A$ is the maximal ideal defining the $\BF_q$-valued point $P_\eta\in C$. Correspondingly, when we take $\tilde u_\eta:=u\in\Koh_{1,\Betti}(\ulM(-[P_\eta]),Q)=\Koh_{1,\Betti}(\ulM,Q)$, which pairs with $\lambda_\ulM$ to $\langle \tilde u_\eta,\lambda_\ulM\rangle=\langle u,\lambda_\ulM\rangle=1$, we obtain
\[
\Koh_{1,\Betti}(\ulM(-[P_\eta]),A)\;=\;\tilde u_\eta\cdot\Gamma(\Spec A,\CO_C([P_\eta]))\;=\;\tilde u_\eta\cdot\Fp_\eta^{-1}\,.
\]
This yields
\begin{equation}\label{EqvetaOfueta}
\fbox{$v_\eta(\tilde u_\eta)\cdot\log q_v\;=\;
\begin{cases}
0\quad&\text{if }v\ne\Fp_\eta\text{ or }\eta=\id_K\,,\\[1mm]
\log q\quad&\text{if }v=\Fp_\eta\text{ and }\eta\ne\id_K\,.   
\end{cases}
$}
\end{equation}
Also from \eqref{EqValueAtInfinity} and \eqref{EqValueOfgeta} we compute the absolute value
\begin{equation}\label{EqValuationPeriodeta}
\fbox{$\log\Bigl|{\TS\int_{\tilde u_\eta}}\wt\omega^\eta\Bigr|_\infty=\log\Bigl|\langle u,\sigma^*g_\eta|_\Xi\cdot\omega\rangle_\infty\Bigr|_\infty=\log\bigl|\sigma^*g_\eta|_\Xi\bigr|_\infty+\log\bigl|\langle u,\omega\rangle_\infty\bigr|_\infty=\tfrac{q}{q-1}\log q\,.$}
\end{equation}
\end{Point}

\begin{Point}\label{Point18.12}
Finally, we recall the zeta functions for the elliptic curve $C$, which are defined as the following products which converge for $s \in \BC$ with $\CR e(s)>1$ 
\begin{align*}
& \zeta_C(s) \;:=\; \prod_{\text{all }v}(1-(\#\BF_v)^{-s})^{-1} \;=\; \prod_{\text{all }v}(1-q_v^{-s})^{-1} \;=\; \frac{1-(q+1 -\#C(\BF_q))q^{-s}+q^{1-2s}}{(1-q^{-s})(1-q^{1-s})} \qquad\text{and}\\
& \zeta_A(s) \;:=\; \prod_{  v \ne \infty}(1-(\#\BF_v)^{-s})^{-1} \;=\; \prod_{v \ne \infty}(1-q_v^{-s})^{-1} \;=\; \frac{1-(q+1 -\#C(\BF_q))q^{-s}+q^{1-2s}}{1-q^{1-s}}.
\end{align*}
Since the CM-field is $E=Q$, $H_E=\{\id\}$ and the CM-type is given by $d_\id=1$, we have $a^0_{E,\id,\Phi}=\BOne$. Since $L^\infty(\BOne,s)=\zeta_A(s)$ we obtain
\begin{equation}\label{EqZetaInExample18}
\fbox{$\DS Z^\infty(\BOne,0)=\frac{\zeta'_A(0)}{\zeta_A(0)} = \Bigl(\frac{q+1-\#C(\BF_q)-2q}{1-(q+1-\#C(\BF_q))+q} -\frac{q}{1-q}\Bigl)\log q= \Bigl(\frac{1-\#C(\BF_q)-q}{\#C(\BF_q)} +\frac{q}{q-1}\Bigl)\log q\,.$}
\end{equation}
\end{Point}
We now put everything together using Theorem~\ref{ThmValueAtV} and formula~\eqref{EqConvSumAMot} to compute
\begin{alignat*}{2}
\TS\tfrac{1}{\#H_K}\sum\limits_v\sum\limits_{\eta\in H_K}\log\bigl|\int_{\tilde u_\eta}\omega_\psi^\eta\bigr|_v\;=\;&\Bigl(\frac{q+\#C(\BF_q)-1}{\#C(\BF_q)}-\frac{q}{q-1}\Bigl)\cdot\log q &&\text{from \eqref{EqZetaInExample18}} \\
& +\frac{1}{\#C(\BF_q)}\Bigl(\frac{q}{q-1}-q\Bigr)\cdot\log q \qquad\qquad\qquad&& \text{from \eqref{EqValueAtInfinity}} \\
& +\frac{\#C(\BF_q)-1}{\#C(\BF_q)}\,\frac{q}{q-1}\cdot\log q && \text{from \eqref{EqValuationPeriodeta}}
\\
& -\frac{\#C(\BF_q)-1}{\#C(\BF_q)}\cdot\log q && \text{from \eqref{EqvOfomega} and \eqref{EqvetaOfueta}}\\
\;=\;&0\,.
\end{alignat*}
Miraculously, all terms cancel and this shows that in the present example our Conjecture~\ref{ConjColmezAMot} holds true.


\bigskip\noindent
{\it Acknowledgments.} We thank F.~Breuer, F.~Pellarin and D.~Thakur for helpful suggestions and for remarks about the early history of the subject. We thank the unknown referees for their careful reading and their useful suggestions. Both authors acknowledge the support received from the German Science Foundation (DFG) in form of grant SFB 878 and Germany's Excellence Strategy EXC 2044--390685587 ``Mathematics M\"unster: Dynamics--Geometry--Structure''.


\addcontentsline{toc}{section}{References}

\bibliography{references1}

\providecommand{\bysame}{\leavevmode\hbox to3em{\hrulefill}\thinspace}
\providecommand{\MR}{\relax\ifhmode\unskip\space\fi MR }
\providecommand{\MRhref}[2]{%
  \href{http://www.ams.org/mathscinet-getitem?mr=#1}{#2}
}
\providecommand{\href}[2]{#2}
\begin{thebibliography}{AGHMP18}

\bibitem[EGA]{EGA_IV1}
A.~Grothendieck, \emph{\'{E}l\'{e}ments de g\'{e}om\'{e}trie alg\'{e}brique.
  {IV}. \'{E}tude locale des sch\'{e}mas et des morphismes de sch\'{e}mas.
  {I}}, Inst. Hautes \'{E}tudes Sci. Publ. Math. (1964), no.~20, 259.
  \MR{173675}

\bibitem[SGA 7]{SGA7}
Alexander Grothendieck, Michel Raynaud, and Dock~Sang Rim, \emph{Groupes de
  monodromie en g\'eom\'etrie alg\'ebrique. {I}}, Lecture Notes in Mathematics,
  Vol. 288, Springer-Verlag, 1972, S{\'e}minaire de G{\'e}om{\'e}trie
  Alg{\'e}brique du Bois-Marie 1967--1969 (SGA 7 I).

\bibitem[AGHMP18]{AGHM}
Fabrizio Andreatta, Eyal~Z. Goren, Benjamin Howard, and Keerthi Madapusi~Pera,
  \emph{Faltings heights of abelian varieties with complex multiplication},
  Ann. of Math. (2) \textbf{187} (2018), no.~2, 391--531. \MR{3744856}

\bibitem[And82]{AndersonLogarithmic}
Greg~W. Anderson, \emph{Logarithmic derivatives of {D}irichlet {$L$}-functions
  and the periods of abelian varieties}, Compositio Math. \textbf{45} (1982),
  no.~3, 315--332. \MR{656608}

\bibitem[And86]{Anderson}
\bysame, \emph{{$t$}-motives}, Duke Math. J. \textbf{53} (1986), no.~2,
  457--502. \MR{850546 (87j:11042)}

\bibitem[BGR84]{BGR}
S.~Bosch, U.~G{\"u}ntzer, and R.~Remmert, \emph{Non-{A}rchimedean analysis},
  Grundlehren der Mathematischen Wissenschaften [Fundamental Principles of
  Mathematical Sciences], vol. 261, Springer-Verlag, Berlin, 1984, A systematic
  approach to rigid analytic geometry. \MR{746961 (86b:32031)}

\bibitem[BH07]{BH07}
Gebhard B{\"o}ckle and Urs Hartl, \emph{Uniformizable families of
  {$t$}-motives}, Trans. Amer. Math. Soc. \textbf{359} (2007), no.~8,
  3933--3972. \MR{2302519 (2008e:11070)}

\bibitem[BH09]{BoHa09}
Matthias Bornhofen and Urs Hartl, \emph{Pure {A}nderson motives over finite
  fields}, J. Number Theory \textbf{129} (2009), no.~2, 247--283. \MR{2473877
  (2010d:11060)}

\bibitem[BH11]{BoHa11}
\bysame, \emph{Pure {A}nderson motives and abelian {$\tau$}-sheaves}, Math. Z.
  \textbf{268} (2011), no.~1-2, 67--100. \MR{2805425}

\bibitem[BL85]{Boschab}
Siegfried Bosch and Werner L{\"u}tkebohmert, \emph{Stable reduction and
  uniformization of abelian varieties. {I}}, Math. Ann. \textbf{270} (1985),
  no.~3, 349--379. \MR{774362 (86j:14040a)}

\bibitem[Bre07]{BreuerTAMS}
Florian Breuer, \emph{C{M} points on products of {D}rinfeld modular curves},
  Trans. Amer. Math. Soc. \textbf{359} (2007), no.~3, 1351--1374. \MR{2262854}

\bibitem[Bre12]{Breuer12}
\bysame, \emph{Special subvarieties of {D}rinfeld modular varieties}, J. Reine
  Angew. Math. \textbf{668} (2012), 35--57. \MR{2948870}

\bibitem[BSM18]{BSM}
Adrian Barquero-Sanchez and Riad Masri, \emph{On the {C}olmez conjecture for
  non-abelian {CM} fields}, Res. Math. Sci. \textbf{5} (2018), no.~1, Paper No.
  10, 41. \MR{3761437}

\bibitem[Cas67]{CasselsFroehlich}
\emph{Algebraic number theory}, Proceedings of an instructional conference
  organized by the London Mathematical Society (a NATO Advanced Study
  Institute) with the support of the International Mathematical Union. Edited
  by J. W. S. Cassels and A. Fr\"ohlich, Academic Press, London; Thompson Book
  Co., Inc., Washington, D.C., 1967. \MR{0215665}

\bibitem[Col93]{ColmezPeriods}
Pierre Colmez, \emph{P\'eriodes des vari\'et\'es ab\'eliennes \`a
  multiplication complexe}, Ann. of Math. (2) \textbf{138} (1993), no.~3,
  625--683. \MR{1247996}

\bibitem[CS86]{Cornell86}
Gary Cornell and Joseph~H. Silverman (eds.), \emph{Arithmetic geometry},
  Springer-Verlag, New York, 1986, Papers from the conference held at the
  University of Connecticut, Storrs, Connecticut, July 30--August 10, 1984.
  \MR{861969}

\bibitem[Del74]{DeligneWeilI}
Pierre Deligne, \emph{La conjecture de {W}eil. {I}}, Inst. Hautes \'{E}tudes
  Sci. Publ. Math. (1974), no.~43, 273--307. \MR{340258}

\bibitem[Dri76]{Drinfeld}
V.~G. Drinfeld, \emph{Elliptic modules}, Math. USSR-Sb \textbf{23} (1976),
  561--592.

\bibitem[DS05]{DiamondShurman}
Fred Diamond and Jerry Shurman, \emph{A first course in modular forms},
  Graduate Texts in Mathematics, vol. 228, Springer-Verlag, New York, 2005.
  \MR{2112196}

\bibitem[Edi05]{Edixhoven05}
Bas Edixhoven, \emph{Special points on products of modular curves}, Duke Math.
  J. \textbf{126} (2005), no.~2, 325--348. \MR{2115260}

\bibitem[EMO01]{EdixhovenTexel2001}
S.~J. Edixhoven, B.~J.~J. Moonen, and F.~Oort, \emph{Open problems in algebraic
  geometry}, Bull. Sci. Math. \textbf{125} (2001), no.~1, 1--22. \MR{1812812}

\bibitem[EY03]{EdixhovenYafaev}
Bas Edixhoven and Andrei Yafaev, \emph{Subvarieties of {S}himura varieties},
  Ann. of Math. (2) \textbf{157} (2003), no.~2, 621--645. \MR{1973057}

\bibitem[Fal83]{FaltingsEndlichkeit}
G.~Faltings, \emph{Endlichkeitss\"atze f\"ur abelsche {V}ariet\"aten \"uber
  {Z}ahlk\"orpern}, Invent. Math. \textbf{73} (1983), no.~3, 349--366.
  \MR{718935}

\bibitem[Fal84a]{Faltings84}
Gerd Faltings, \emph{Calculus on arithmetic surfaces}, Ann. of Math. (2)
  \textbf{119} (1984), no.~2, 387--424. \MR{740897}

\bibitem[Fal84b]{FaltingsComplements}
\bysame, \emph{Complements to {M}ordell}, Rational points ({B}onn, 1983/1984),
  Aspects Math., E6, Friedr. Vieweg, Braunschweig, 1984, pp.~203--227.
  \MR{766574}

\bibitem[Fal89]{FaltingsCrystCohom}
\bysame, \emph{Crystalline cohomology and {$p$}-adic {G}alois-representations},
  Algebraic analysis, geometry, and number theory ({B}altimore, {MD}, 1988),
  Johns Hopkins Univ. Press, Baltimore, MD, 1989, pp.~25--80. \MR{1463696}

\bibitem[Fal02]{FaltingsStrict}
\bysame, \emph{Group schemes with strict {$\cal O$}-action}, Mosc. Math. J.
  \textbf{2} (2002), no.~2, 249--279, Dedicated to Yuri I. Manin on the
  occasion of his 65th birthday. \MR{1944507}

\bibitem[FM87]{Fontaine-Messing}
Jean-Marc Fontaine and William Messing, \emph{{$p$}-adic periods and {$p$}-adic
  \'etale cohomology}, Current trends in arithmetical algebraic geometry
  ({A}rcata, {C}alif., 1985), Contemp. Math., vol.~67, Amer. Math. Soc.,
  Providence, RI, 1987, pp.~179--207. \MR{902593}

\bibitem[Fon77]{Fon77}
Jean-Marc Fontaine, \emph{Groupes {$p$}-divisibles sur les corps locaux},
  Soci\'et\'e Math\'ematique de France, Paris, 1977, Ast\'erisque, No. 47-48.
  \MR{0498610}

\bibitem[Fon82]{Fon82}
\bysame, \emph{Sur certains types de repr\'esentations {$p$}-adiques du groupe
  de {G}alois d'un corps local;\ construction d'un anneau de
  {B}arsotti-{T}ate}, Ann. of Math. (2) \textbf{115} (1982), no.~3, 529--577.
  \MR{657238}

\bibitem[Gar03]{Gardeyn4}
Francis Gardeyn, \emph{The structure of analytic {$\tau$}-sheaves}, J. Number
  Theory \textbf{100} (2003), no.~2, 332--362. \MR{1978461}

\bibitem[Gek89]{Gekeler89}
Ernst-Ulrich Gekeler, \emph{On the de {R}ham isomorphism for {D}rinfel'd
  modules}, J. Reine Angew. Math. \textbf{401} (1989), 188--208. \MR{1018059}

\bibitem[Gos94]{Goss94}
David Goss, \emph{Drinfel'd modules: cohomology and special functions}, Motives
  ({S}eattle, {WA}, 1991), Proc. Sympos. Pure Math., vol.~55, Amer. Math. Soc.,
  Providence, RI, 1994, pp.~309--362. \MR{1265558}

\bibitem[Gos96]{Goss}
\bysame, \emph{Basic structures of function field arithmetic}, Ergebnisse der
  Mathematik und ihrer Grenzgebiete (3) [Results in Mathematics and Related
  Areas (3)], vol.~35, Springer-Verlag, Berlin, 1996. \MR{1423131 (97i:11062)}

\bibitem[GP16]{PapasGreen}
Nathan Green and Mathew~A. Papanikolas, \emph{Special {L}-values and shtuka
  functions for drinfeld modules on elliptic curves}, Preprint available as
  arXiv:1607.04211 (2016).

\bibitem[Gro65]{EGA_IV2}
\bysame, \emph{{\'E}l\'ements de {G}\'eom\'etrie {A}lg\'ebrique. {IV}.
  {\'e}tude locale des sch\'emas et des morphismes de sch\'emas. {II}}, Inst.
  Hautes \'Etudes Sci. Publ. Math. (1965), no.~24, 231. \MR{0199181}

\bibitem[Gro78]{Gross78}
Benedict~H. Gross, \emph{On the periods of abelian integrals and a formula of
  {C}howla and {S}elberg}, Invent. Math. \textbf{45} (1978), no.~2, 193--211,
  With an appendix by David E. Rohrlich. \MR{480542}

\bibitem[Gro18]{GrossSurvey}
\bysame, \emph{On the periods of abelian varieties}, 2018,
  http://www.math.harvard.edu/$\sim$gross/preprints/cs.pdf.

\bibitem[GvKM19]{Zhang20}
Z.~Gao, R.~van K{\"a}nel, and L.~Mocz, \emph{Faltings heights and
  {$L$}-functions: {M}inicourse given by shou-wu zhang}, {A}rithmetic and
  {G}eometry: {T}en {Y}ears in {A}lpbach, Annals of {M}athematics {S}tudies,
  vol. 202, Princeton University Press, Princeton, 2019, pp.~102--174.

\bibitem[Har09]{HartlDict}
Urs Hartl, \emph{A dictionary between {F}ontaine-theory and its analogue in
  equal characteristic}, J. Number Theory \textbf{129} (2009), no.~7,
  1734--1757. \MR{2524192}

\bibitem[Har17]{HA17}
\bysame, \emph{{I}sogenies of abelian {A}nderson {$A$}-modules and
  {$A$}-motives}, Preprint, available as arXiv:math/1706.06807 (2017).

\bibitem[Hay79]{Hayes79}
David~R. Hayes, \emph{Explicit class field theory in global function fields},
  Studies in algebra and number theory, Adv. in Math. Suppl. Stud., vol.~6,
  Academic Press, New York-London, 1979, pp.~173--217. \MR{535766}

\bibitem[HJ20]{HartlJuschka}
Urs Hartl and Ann-Kristin Juschka, \emph{Pink's theory of {H}odge structures
  and the {H}odge conjecture over function fields}, to appear in Proceedings of
  the conference on ``$t$-motives: Hodge structures, transcendence and other
  motivic aspects", BIRS, Banff, Canada 2009, eds. G. B\"ockle, D. Goss, U.
  Hartl, M. Papanikolas, EMS 2018; also available as arxiv:1607.01412. (2020).

\bibitem[HK20]{HartlKim}
Urs Hartl and Wansu Kim, \emph{{L}ocal {S}htukas, {H}odge-{P}ink structures and
  {G}alois representations}, to appear in Proceedings of the conference on
  ``$t$-motives: Hodge structures, transcendence and other motivic aspects",
  BIRS, Banff, Canada 2009, eds. G. B\"ockle, D. Goss, U. Hartl, M.
  Papanikolas, EMS 2018; also available as arxiv:1512.05893. (2020).

\bibitem[Hoc65]{Hochschild}
G.~Hochschild, \emph{The structure of {L}ie groups}, Holden-Day, Inc., San
  Francisco-London-Amsterdam, 1965. \MR{0207883}

\bibitem[HP04]{HP04}
Urs Hartl and Richard Pink, \emph{Vector bundles with a {F}robenius structure
  on the punctured unit disc}, Compos. Math. \textbf{140} (2004), no.~3,
  689--716. \MR{2041777 (2005a:13008)}

\bibitem[HS19]{HowardShnidman}
Benjamin Howard and Ari Shnidman, \emph{A {G}ross-{K}ohnen-{Z}agier formula for
  {H}eegner-{D}rinfeld cycles}, Adv. Math. \textbf{351} (2019), 117--194.
  \MR{3950427}

\bibitem[HS20]{HartlSingh}
Urs Hartl and Rajneesh~Kumar Singh, \emph{Periods of {D}rinfeld modules and
  local shtukas with complex multiplication}, J. Inst. Math. Jussieu
  \textbf{19} (2020), no.~1, 175--208. \MR{4045083}

\bibitem[HS21]{HartlSinghErr}
\bysame, \emph{Erratum to {P}eriods of {D}rinfeld modules and local shtukas
  with complex multiplication}, submitted to J. Inst. Math. Jussieu (2021).
  \MR{4045083}

\bibitem[Hub13]{Hubschmid13}
Patrik Hubschmid, \emph{The {A}ndr\'{e}-{O}ort conjecture for {D}rinfeld
  modular varieties}, Compos. Math. \textbf{149} (2013), no.~4, 507--567.
  \MR{3049695}

\bibitem[Hus04]{Husemoller04}
Dale Husem\"oller, \emph{Elliptic curves}, second ed., Graduate Texts in
  Mathematics, vol. 111, Springer-Verlag, New York, 2004, With appendices by
  Otto Forster, Ruth Lawrence and Stefan Theisen. \MR{2024529}

\bibitem[Lan94]{LangANT}
Serge Lang, \emph{Algebraic number theory}, second ed., Graduate Texts in
  Mathematics, vol. 110, Springer-Verlag, New York, 1994. \MR{1282723}

\bibitem[Ler97]{Lerch1897}
Matthias Lerch, \emph{Sur quelques formules relatives au nombre des classes},
  Bulletin des Sciences Math{\'e}matiques (1897).

\bibitem[Liu02]{Liu_AlgGeom}
Qing Liu, \emph{Algebraic geometry and arithmetic curves}, Oxford Graduate
  Texts in Mathematics, vol.~6, Oxford University Press, Oxford, 2002,
  Translated from the French by Reinie Ern\'e, Oxford Science Publications.
  \MR{1917232}

\bibitem[Mil86]{Milne84}
J.~S. Milne, \emph{Abelian varieties}, Arithmetic geometry ({S}torrs, {C}onn.,
  1984), Springer, New York, 1986, pp.~103--150. \MR{861974}

\bibitem[Mil06]{MilneCM}
James~S. Milne, \emph{{C}omplex {M}ultiplication}, Course notes (2006),
  \href{http://www.jmilne.org/math/CourseNotes/cm.html}{http://www.jmilne.org/math/CourseNotes/cm.html}.

\bibitem[Mil08]{MilneAbVar}
\bysame, \emph{{A}belian {V}arieties}, Course notes (2008),
  \href{http://www.jmilne.org/math/CourseNotes/av.html}{http:/\!/www.jmilne.org/math/CourseNotes/av.html}.

\bibitem[Mum70]{Mumford70}
David Mumford, \emph{Abelian varieties}, Tata Institute of Fundamental Research
  Studies in Mathematics, No. 5, Published for the Tata Institute of
  Fundamental Research, Bombay; Oxford University Press, London, 1970.
  \MR{0282985}

\bibitem[Niz98]{Niziol98}
Wies\l~awa Nizio\l, \emph{Crystalline conjecture via {$K$}-theory}, Ann. Sci.
  \'Ecole Norm. Sup. (4) \textbf{31} (1998), no.~5, 659--681. \MR{1643962}

\bibitem[Obu13]{Obus13}
Andrew Obus, \emph{On {C}olmez's product formula for periods of {CM}-abelian
  varieties}, Math. Ann. \textbf{356} (2013), no.~2, 401--418. \MR{3048601}

\bibitem[Pel09]{Pelzer09}
Antje Pelzer, \emph{{D}er {H}auptsatz der {K}omplexen {M}ultiplikation f\"ur
  {A}nderson {A}-motive}, Diploma thesis, University of Muenster (2009).

\bibitem[PT14]{PilaTsimerman_Ax-Lindemann}
Jonathan Pila and Jacob Tsimerman, \emph{Ax-{L}indemann for {$\cal A_g$}}, Ann.
  of Math. (2) \textbf{179} (2014), no.~2, 659--681. \MR{3152943}

\bibitem[PW06]{PilaWilkie06}
J.~Pila and A.~J. Wilkie, \emph{The rational points of a definable set}, Duke
  Math. J. \textbf{133} (2006), no.~3, 591--616. \MR{2228464}

\bibitem[PZ08]{PilaZannier}
Jonathan Pila and Umberto Zannier, \emph{Rational points in periodic analytic
  sets and the {M}anin-{M}umford conjecture}, Atti Accad. Naz. Lincei Rend.
  Lincei Mat. Appl. \textbf{19} (2008), no.~2, 149--162. \MR{2411018}

\bibitem[Ros02]{Rosen02}
Michael Rosen, \emph{Number theory in function fields}, Graduate Texts in
  Mathematics, vol. 210, Springer-Verlag, New York, 2002. \MR{1876657}

\bibitem[SC67]{ChowlaSelberg67}
Atle Selberg and S.~Chowla, \emph{On {E}pstein's zeta-function}, J. Reine
  Angew. Math. \textbf{227} (1967), 86--110. \MR{215797}

\bibitem[Sch09]{Schindler}
Anne Schindler, \emph{Anderson {$A$}-{M}otive mit komplexer {M}ultiplikation},
  Diploma thesis, University of Muenster (2009).

\bibitem[Ser77]{SerreLinRep}
Jean-Pierre Serre, \emph{Linear representations of finite groups},
  Springer-Verlag, New York-Heidelberg, 1977, Translated from the second French
  edition by Leonard L. Scott, Graduate Texts in Mathematics, Vol. 42.
  \MR{0450380}

\bibitem[Ser79]{SerreLF}
\bysame, \emph{Local fields}, Graduate Texts in Mathematics, vol.~67,
  Springer-Verlag, New York-Berlin, 1979, Translated from the French by Marvin
  Jay Greenberg. \MR{554237 (82e:12016)}

\bibitem[Sil86]{Silverman86}
Joseph~H. Silverman, \emph{The arithmetic of elliptic curves}, Graduate Texts
  in Mathematics, vol. 106, Springer-Verlag, New York, 1986. \MR{817210}

\bibitem[ST61]{ShimuraTaniyama}
Goro Shimura and Yutaka Taniyama, \emph{Complex multiplication of abelian
  varieties and its applications to number theory}, Publications of the
  Mathematical Society of Japan, vol.~6, The Mathematical Society of Japan,
  Tokyo, 1961. \MR{0125113}

\bibitem[ST68]{SerreTate}
Jean-Pierre Serre and John Tate, \emph{Good reduction of abelian varieties},
  Ann. of Math. (2) \textbf{88} (1968), 492--517. \MR{0236190}

\bibitem[Tag93]{Tag}
Yuichiro Taguchi, \emph{Semi-simplicity of the {G}alois representations
  attached to {D}rinfeld modules over fields of ``infinite characteristics''},
  J. Number Theory \textbf{44} (1993), no.~3, 292--314. \MR{1233291}

\bibitem[Tag95]{TaguchiTateConj}
\bysame, \emph{The {T}ate conjecture for {$t$}-motives}, Proc. Amer. Math. Soc.
  \textbf{123} (1995), no.~11, 3285--3287. \MR{1286009}

\bibitem[Tam94]{TamagawaTateConj}
Akio Tamagawa, \emph{Generalization of {A}nderson's {$t$}-motives and {T}ate
  conjecture}, S\=urikaisekikenky\=usho K\=oky\=uroku (1994), no.~884,
  154--159, Moduli spaces, Galois representations and $L$-functions (Japanese)
  (Kyoto, 1993, 1994). \MR{1333454}

\bibitem[Tat66]{TateEndoms}
John Tate, \emph{Endomorphisms of abelian varieties over finite fields},
  Invent. Math. \textbf{2} (1966), 134--144. \MR{0206004}

\bibitem[Tat67]{TateP-DivisibleGroups}
J.~T. Tate, \emph{{$p$}-divisible groups}, Proc. {C}onf. {L}ocal {F}ields
  ({D}riebergen, 1966), Springer, Berlin, 1967, pp.~158--183. \MR{0231827}

\bibitem[Tha91]{Thakur91}
Dinesh~S. Thakur, \emph{Gamma functions for function fields and {D}rinfel'd
  modules}, Ann. of Math. (2) \textbf{134} (1991), no.~1, 25--64. \MR{1114607}

\bibitem[Tha04]{Thakur04}
\bysame, \emph{Function field arithmetic}, World Scientific Publishing Co.,
  Inc., River Edge, NJ, 2004. \MR{2091265}

\bibitem[Tsi18]{Tsimerman18}
Jacob Tsimerman, \emph{The {A}ndr\'e-{O}ort conjecture for {$\cal A_g$}}, Ann.
  of Math. (2) \textbf{187} (2018), no.~2, 379--390. \MR{3744855}

\bibitem[Tsu99]{Tsuji99}
Takeshi Tsuji, \emph{{$p$}-adic \'etale cohomology and crystalline cohomology
  in the semi-stable reduction case}, Invent. Math. \textbf{137} (1999), no.~2,
  233--411. \MR{1705837}

\bibitem[TW96]{Tag96}
Y.~Taguchi and D.~Wan, \emph{{$L$}-functions of {$\phi$}-sheaves and {D}rinfeld
  modules}, J. Amer. Math. Soc. \textbf{9} (1996), no.~3, 755--781. \MR{1327162
  (96j:11082)}

\bibitem[VS06]{VillaSalvador}
Gabriel~Daniel Villa~Salvador, \emph{Topics in the theory of algebraic function
  fields}, Mathematics: Theory \& Applications, Birkh\"auser Boston, Inc.,
  Boston, MA, 2006. \MR{2241963}

\bibitem[Wei48]{WeilRH}
Andr\'e Weil, \emph{Sur les courbes alg\'ebriques et les vari\'et\'es qui s'en
  d\'eduisent}, Actualit\'es Sci. Ind., no. 1041 = Publ. Inst. Math. Univ.
  Strasbourg {\bf 7} (1945), Hermann et Cie., Paris, 1948. \MR{0027151}

\bibitem[Wei20]{Wei20}
Fu-Tsun Wei, \emph{On kronecker terms over global function fields}, Invent.
  Math. (2020), https://doi.org/10.1007/s00222-019-00944-8.

\bibitem[Yan13]{Yang13}
Tonghai Yang, \emph{Arithmetic intersection on a {H}ilbert modular surface and
  the {F}altings height}, Asian J. Math. \textbf{17} (2013), no.~2, 335--381.
  \MR{3078934}

\bibitem[Yu90]{Yu90}
Jing Yu, \emph{On periods and quasi-periods of {D}rinfel'd modules}, Compositio
  Math. \textbf{74} (1990), no.~3, 235--245. \MR{1055694}

\bibitem[Yua19]{YuanSurvey}
Xinyi Yuan, \emph{On {F}altings heights of abelian varieties with complex
  multiplication}, Proceedings of the {S}eventh {I}nternational {C}ongress of
  {C}hinese {M}athematicians, {V}ol. {I}, Adv. Lect. Math. (ALM), vol.~43, Int.
  Press, Somerville, MA, 2019, pp.~521--536. \MR{3971887}

\bibitem[YZ17]{YunZhang}
Zhiwei Yun and Wei Zhang, \emph{Shtukas and the {T}aylor expansion of
  {$L$}-functions}, Ann. of Math. (2) \textbf{186} (2017), no.~3, 767--911.
  \MR{3702678}

\bibitem[YZ18]{YuanZhang15}
Xinyi Yuan and Shou-Wu Zhang, \emph{On the averaged {C}olmez conjecture}, Ann.
  of Math. (2) \textbf{187} (2018), no.~2, 533--638. \MR{3744857}

\bibitem[YZ19]{YunZhang2}
Zhiwei Yun and Wei Zhang, \emph{Shtukas and the {T}aylor expansion of
  {$L$}-functions ({II})}, Ann. of Math. (2) \textbf{189} (2019), no.~2,
  393--526. \MR{3919362}

\bibitem[Zar75]{Zarhin75}
Ju.~G. Zarhin, \emph{Endomorphisms of {A}belian varieties over fields of finite
  characteristic}, Izv. Akad. Nauk SSSR Ser. Mat. \textbf{39} (1975), no.~2,
  272--277, 471. \MR{0371897}

\end{thebibliography}

\vfill

\begin{minipage}[t]{0.5\linewidth}
\noindent
Urs Hartl\\
Universit\"at M\"unster\\
Mathematisches Institut \\
Einsteinstr.~62\\
D -- 48149 M\"unster
\\ Germany
\\[1mm]
\href{https://www.uni-muenster.de/Arithm/hartl/}{https:/\hspace{-1mm}/www.uni-muenster.de/Arithm/hartl/}
\end{minipage}
\begin{minipage}[t]{0.45\linewidth}
\noindent
Rajneesh Kumar Singh\\
Indian Institute of Technology Goa\\
Goa Engineering College Campus\\
Farmagudi, Ponda-403401, Goa\\
India
\\[1mm]
\end{minipage}

\end{document}